%% file: main.tex
\documentclass{article}
\usepackage[preprint]{neurips_2026}

\usepackage[utf8]{inputenc} 
\usepackage[T1]{fontenc}    

\usepackage{url}            
\usepackage{booktabs}       
\usepackage{bm}
\usepackage{amsfonts}       
\usepackage{nicefrac}       
\usepackage{microtype}      
\usepackage{xcolor}         
\usepackage{tikz}
\usepackage{pgfplots}
\pgfplotsset{compat=1.18}
\usetikzlibrary{shapes.geometric, arrows.meta, positioning, fit, calc}
\usepackage{ifthen}
\usepackage{amssymb}
\usepackage{hyperref}
\usepackage{aliascnt}
\usepackage{algorithm}
\usepackage[noend]{algpseudocode}
\usepackage{appendix}
\usepackage{enumitem}
\usepackage{subcaption}
\usepackage{natbib}
\usepackage{tcolorbox}
\tcbuselibrary{breakable}
\usepackage{multirow}
\usepackage{mathtools}
\usepackage{cleveref}
\usepackage{adjustbox}
\usepackage{amsmath}
\usepackage{amsthm}
\usepackage{xspace}
\usepackage{pifont}
\usepackage[colorinlistoftodos]{todonotes}
\usepackage[framemethod=TikZ]{mdframed}
\usepackage{soul}
\usepackage[normalem]{ulem}

\DeclareRobustCommand{\stagecirc}[1]{%
  \tikz[baseline=(char.base)]{%
    \node[draw,circle,inner sep=0pt,minimum size=1.5ex] (char) {\scriptsize #1};%
  }%
}
\sethlcolor{lightgray!25}

\newcommand{\grayhl}[1]{\hl{#1}}
\mdfdefinestyle{takeaway}{%
    middlelinecolor =   none,
    middlelinewidth =   1pt,
    backgroundcolor =   teal!10,
    roundcorner     =   5pt,
}
\usepackage{subcaption} 

\usepackage{minitoc}

\input{defns}

\newtheoremstyle{normaltheorem}
  {3pt}   
  {3pt}   
  {\upshape} 
  {}      
  {\bfseries} 
  {.}     
  {.5em}  
  {}      
  
\theoremstyle{normaltheorem}
\newtheorem{theorem}{Theorem}[section]
\newtheorem*{theorem*}{Theorem}

\newaliascnt{lemma}{theorem}
\newtheorem{lemma}[lemma]{Lemma}
\aliascntresetthe{lemma}
\crefname{lemma}{lemma}{lemmas}
\Crefname{lemma}{Lemma}{Lemmas}

\newaliascnt{proposition}{theorem}
\newtheorem{proposition}[proposition]{Proposition}
\aliascntresetthe{proposition}
\crefname{proposition}{proposition}{propositions}
\Crefname{proposition}{Proposition}{Propositions}

\newaliascnt{corr}{theorem}
\newtheorem{corr}[corr]{Corollary}
\aliascntresetthe{corr}
\crefname{corr}{corollary}{corollaries}
\Crefname{corr}{Corollary}{Corollaries}

\newaliascnt{remark}{theorem}
\newtheorem{remark}[remark]{Remark}
\aliascntresetthe{remark}
\crefname{remark}{remark}{remarks}
\Crefname{remark}{Remark}{Remarks}

\newaliascnt{definition}{theorem}
\newtheorem{definition}[definition]{Definition}
\aliascntresetthe{definition}
\crefname{definition}{definition}{definitions}
\Crefname{definition}{Definition}{Definitions}

\newaliascnt{example}{theorem}
\newtheorem{example}[example]{Example}
\aliascntresetthe{example}
\crefname{example}{example}{examples}
\Crefname{example}{Example}{Examples}

\newtheorem{assumption}{Assumption}
\crefname{assumption}{assumption}{assumptions}
\Crefname{assumption}{Assumption}{Assumptions}

\newcommand{\hyperobjective}{F}

\newcommand{\gibbs}{\mu}

\newcommand{\gibbstheta}{\gibbs_\theta^\lambda}
\newcommand{\poly}{\text{poly}}
\newcommand{\ud}{\mathrm{d}}
\newcommand{\PLcirc}{\texttt{P\L$^\circ$}}

\newcommand{\dimk}{\texttt{k}}

\newcommand{\R}{\mathbb{R}}
\newcommand{\calS}{\mathcal{S}}
\newcommand{\calT}{\mathcal{T}}
\newcommand{\calN}{\mathcal{N}}

\renewcommand{\d}{\mathrm{d}}

\newcommand{\BoN}{\texttt{BoN}\xspace}
\newcommand{\holder}{{H\"older}}
\newcommand{\softmax}{\operatorname{softmax}}

\newcommand{\selectedminimizerset}{{\texttt{O}}}

\newcommand{\minimaselection}{{minima-selection}\xspace}

\title{Select-then-differentiate: Solving Bilevel Optimization with Manifold Lower-level Solution Sets
}
\author{
Saeed Masiha\textsuperscript{1}
\quad Zebang Shen\textsuperscript{2}
\quad Negar Kiyavash\textsuperscript{1}
\quad Niao He\textsuperscript{2}\\
\textsuperscript{1}EPFL School of Management of Technology, Station 5, 1015 Lausanne, Switzerland\\
\textsuperscript{2}ETH Department of Computer Science, Universit\"atstrasse 6, 8092 Z\"urich, Switzerland\\
\texttt{\{mohammadsaeed.masiha,negar.kiyavash\}@epfl.ch}\\
\texttt{\{zebang.shen,niao.he\}@inf.ethz.ch}
}
\date{}

\begin{document}

\doparttoc
\faketableofcontents
\maketitle

\begin{abstract}
We study optimistic bilevel optimization when the lower-level problem has a non-isolated manifold of minimizers. In this setting, the hyper-objective may be non-differentiable because the upper-level criterion must choose among multiple lower-level solutions. Under a local Polyak--{\L}ojasiewicz (P{\L}) condition, we show that differentiability does not require the lower-level solution set to be a singleton: uniqueness of the optimistic selection is sufficient. This yields an explicit pseudoinverse-based hyper-gradient formula extending the classical singleton-minimizer result. We further characterize the regularity of the hyper-objective: non-degeneracy of the selected minimizer along the solution manifold yields local smoothness, while failure of uniqueness can create many non-differentiable points and failure of non-degeneracy can destroy all positive H\"older regularity of the hyper-gradient. Motivated by this theory, we propose HG-MS, a select-then-differentiate method combining explicit optimistic selection with efficient pseudoinverse-based hyper-gradient computation. Despite the nonconvex nature of optimistic selection over the lower-level solution manifold, we show that HG-MS converges to a stationary point of the optimistic objective with complexity governed by the intrinsic dimension of the solution manifold rather than its ambient dimension.
Empirically, we test a practical variant of HG-MS for matched-budget LLM source reweighting. This variant preserves the select-then-differentiate principle and obtains the best GSM8K/MATH scores across the tested backbones, along with competitive or best MT-Bench instruction-following results.
\end{abstract}

\section{Introduction}\label{sec:intro}
Many modern machine learning tasks involve an upper objective that evaluates the outcome of an inner training procedure. Such tasks can be formulated as bilevel optimization. Canonical examples include hyperparameter optimization and meta-learning \citep{feurer2019hyperparameter,liu2021investigating,bertinetto2018meta,finn2017maml}, differentiable neural architecture search \citep{liu2019darts}, bilevel formulations in reinforcement learning \citep{hong2023two}, and more recently,  bilevel formulations for LLM data reweighting, safe fine-tuning, and preference alignment \citep{pan2025scalebio,shen2025seal,jian2025spo}.
The progress in aforementioned applications has been driven by the ability to compute
\emph{hyper-gradients}, which allows us to directly optimize the hyper-objective.

A standard assumption underlying much of the bilevel optimization literature is that the lower-level problem has a \emph{unique} minimizer. In this singleton regime, the lower-level solution map is single-valued, and implicit differentiation yields a tractable hyper-gradient formula \citep{pedregosa2016hypergrad,franceschi2017forward,lorraine2020optimizing,huang2024optimal,liu2022bome,kwon2023penalty}. However, this assumption is often violated in the nonconvex and overparameterized settings that dominate modern applications. Neural-network training objectives admit many global minimizers, which often are not isolated points but rather come from connected sets with manifold structure \citep{draxler2018essentially,garipov2018loss,nguyen2019connected}. In such regimes, the lower-level solution map is set-valued, and the hyper-objective critically depends on lower-level \textit{minima selection}.
For instance, in recent LLM data reweighting formulations \citep{pan2025scalebio,shen2025seal}, the lower-level objective is typically the training loss, while the upper-level objective is evaluated on a validation set. Distinct lower-level minimizers can therefore attain nearly identical training loss but result in very different validation performance. This makes \emph{optimistic bilevel optimization} a natural formulation: among all lower-level minimizers, select the one that is best for the upper level.

 In this paper, we focus on the optimistic bilevel optimization problem  
\begin{equation}\label{eq:intro_blo}
\min_{\theta \in \Theta}\ F(\theta)
\qquad\text{where}\qquad
F(\theta):=\min_{x\in\calS(\theta)} f(\theta,x),
\qquad
\calS(\theta):=\arg\min_{x\in\R^d} g(\theta,x),
\end{equation}
with compact convex $\Theta \subseteq \R^m$. We refer to the inner optimization over $\calS(\theta)$ as \emph{minima selection} and define $\selectedminimizerset(\theta) := \arg\min_{x\in\calS(\theta)} f(\theta,x)$.
In particular, when $\selectedminimizerset(\theta)$ is a singleton, we denote the unique optimistic selection as $x^\star(\theta)$.
 In general, without additional structure, $F$ need not be continuous, let alone differentiable~\citep{dontchev2009implicit,arbel2022non}.

\textbf{Why existing theory is insufficient.} Most existing hyper-gradient analyses establish differentiability of the hyper-objective only when the lower-level solution set is a \emph{singleton}, \(\calS(\theta)=\{x^\star(\theta)\}\) (so minima selection is trivial), and when the lower-level loss \(g(\theta,\cdot)\) satisfies {locally quadratic growth} around \(x^\star(\theta)\) \citep{pedregosa2016hypergrad,franceschi2017forward,lorraine2020optimizing,huang2024optimal,liu2022bome,kwon2023penalty} (see Appendix~\ref{append:penalty-diff} for details). Recent extensions go beyond the singleton-minimizer setting \citep{xiao2023generalized,shen2025penalty}, but require \(g\) to be \emph{convex} in the lower-level variable \(x\), a strong and often impractical assumption. In overparameterized models, \(\calS(\theta)\) is typically neither singleton nor convex; consequently, existing results leave open whether the optimistic hyper-objective \(F\) is differentiable, and hence whether its hyper-gradient is well defined. Due to space constraints, a more detailed comparison with related work is deferred to Appendix~\ref{sec:related_work}.

\textbf{Our starting point: manifold-structured lower-level minimizers.} 
Motivated by the loss landscapes of overparameterized deep networks \citep{venturi2018spurious,liu2022loss}, we assume  the lower-level loss \(g\) satisfies the \(\PLcirc\) condition of \citep{gong2024poincare} with respect to the lower-level variable \(x\). This assumption implies a structured non-singleton regime: the lower-level solution set \(S(\theta)\) is an \emph{embedded submanifold} of \(\mathbb{R}^d\), and existing results further imply that the optimistic hyper-objective \(F\) is Lipschitz continuous \citep{masiha2025superquantile}. But Lipschitzness alone does not justify hyper-gradient methods \citep{chen2023bilevel}.
This leads to the main technical question of the paper: \grayhl{When lower-level minimizers form a manifold, what additional conditions make the optimistic hyper-objective differentiable or smooth, and hence a good candidate for hyper-gradient optimization?}

\textbf{Main insight: uniqueness of the optimistic selection suffices.}
The main source of nonsmoothness is the instability of the minima selection. When $\selectedminimizerset(\theta)$ contains multiple points, small perturbations of $\theta$ can change which minimizer is selected, making $F$ non-differentiable even if the map $\theta\mapsto\calS(\theta)$ varies smoothly. Thus, uniqueness of the optimistic selection is the key regularity condition.
Under the local \PLcirc\ condition, we show that if $\selectedminimizerset(\theta)=\{x^\star(\theta)\}$, then $F$ is continuously differentiable, even when $\calS(\theta)$ is non-singleton and non-convex. The resulting hyper-gradient admits an explicit pseudoinverse-based formula that strictly generalizes the classical singleton-minimizer formula.
\begin{theorem*}[Regularity; Informal]\label{thm:diff_F}
    Suppose that $g$ satisfies the local \(\PLcirc\) condition (see \Cref{ass:plcirc}).
    \begin{itemize}[leftmargin=*]
        \item If the optimistic minimizer is unique, i.e., $\selectedminimizerset(\theta)=\{x^\star(\theta)\}$, $F$ is continuously differentiable, and
        \begin{align}\label{eq:hg_formula}
            \nabla F(\theta)
            =
            \nabla_{\theta}f(\theta,x^\star(\theta))
            -
            \nabla^{2}_{\theta x}g(\theta,x^\star(\theta))
            \big[\nabla^{2}_{xx}g(\theta,x^\star(\theta))\big]^{\dagger}
            \nabla_{x}f(\theta,x^\star(\theta)),
        \end{align}
        where
$\big[\nabla^{2}_{xx}g(\theta,x^\star(\theta))\big]^{\dagger} = \mathbf{N}_\theta \big[ \mathbf{N}_\theta^\top \nabla^{2}_{xx}g(\theta,x^\star(\theta)) \mathbf{N}_\theta \big]^{-1} \mathbf{N}_\theta^\top,$
and \(\mathbf{N}_\theta\) is an orthonormal basis for the normal space of \(\calS(\theta)\) at \(x^\star(\theta)\); see \Cref{remark_pseudoinverse}.

        \item If, in addition, $x^\star(\theta)$ is non-degenerate (see \Cref{def:degenerate-optimistic}), then $\nabla\hyperobjective$ is Lipschitz continuous.
    \end{itemize}
\end{theorem*}

Our uniqueness condition is \emph{strictly weaker} than the standard singleton lower-level solution assumption. We further show that our conditions are nearly necessary: if $\selectedminimizerset(\theta)$ is not a singleton, then $F$ can be non-differentiable at \emph{infinitely many} points in an \emph{arbitrarily small} neighborhood; even under unique selection, if $x^\star$ is degenerate, then $\nabla\hyperobjective$ can fail to be $\alpha$-\holder\ continuous for \emph{any} $\alpha>0$. Together, these results give a sharp regularity picture for optimistic bilevel optimization in the manifold regime.


\textbf{Algorithmic consequence: select, then differentiate.}
Our regularity theory suggests a simple, practical algorithmic principle for bilevel optimization in the regime when the lower-level solution set exhibits manifold structure and the minima selection yields a unique minimizer:  one should first \emph{select} the optimistic minimizer, and only then \emph{differentiate}.
To this end, we propose \emph{Hyper-Gradient with Minima-Selection} (\textsc{HG-MS}); see \Cref{alg:hg-minsel-lmc}. 
The method proceeds iteratively in two stages.
    First, it approximates the minima selection using a \emph{Best-of-N} (\BoN) strategy: it draws $N$ samples from the Gibbs measure $\gibbstheta(\mathrm{d}x)\propto \exp(-g(\theta,x)/\lambda)\,\mathrm{d}x,$ and outputs $\hat{x}_N^\lambda(\theta)$, the one achieving the smallest upper-level value $f(\theta, \cdot)$, as a proxy for $x^\star(\theta)$.
    Note that when $\lambda \to 0$, $\gibbstheta$ concentrates around $\calS(\theta)$, allowing the procedure to explore the lower-level solution manifold. 
    Second, the method computes the hyper-gradient $\nabla \hyperobjective$ by plugging $\hat{x}_N^\lambda(\theta)$ in \Cref{eq:hg_formula}, which is then used to update $\theta$.

For convergence analysis, we interpret \textsc{HG-MS} as an inexact projected gradient method applied to the hyper-objective \(F\). The hyper-gradient error decomposes into three terms: (a) sampling error from \(\gibbstheta\), (b) \BoN selection error with finite \(N\), and (c) approximation error of the pseudo-inverse in \cref{eq:hg_formula}.  In particular, we show that the \BoN selector satisfies $\forall \theta \in \Theta$, $\mathbb{E}[\|x^\star(\theta) - \hat{x}_N^\lambda(\theta)\|^2] = \mathcal{O}(N^{-\frac{1}{\dimk}} + \lambda^{\frac{1}{2}}).$
 Here, $\dimk$ denotes the common dimension of the lower-level solution manifold $\{\calS(\theta)\}_{\theta \in \Theta}$,  rather than the ambient dimension $d$. There is an accuracy-efficiency tradeoff:  smaller \(\lambda\) and larger \(N\) make the \BoN selection more accurate, yet increase the  computational cost due to slower sampling from the Gibbs measure and the need for more samples from the Gibbs measure.  Balancing the tradeoff yields our complexity guarantee: we show that \textsc{HG-MS} finds an $\epsilon$-stationary point of the hyper-objective $\hyperobjective$ using $\mathcal{O}(\poly(\epsilon^{-\dimk}))$  queries to the oracle for $\partial_x g(\theta,\cdot)$. Note that the $\epsilon^{-\dimk}$ dependence is necessary in the worst case due to the nonconvex nature of the problem \citep{nemirovskij1983problem}.

\textbf{Empirical evidence.} 
HG-MS arises naturally from our hypergradient formula, which also underlies the canonical update analyzed in our theory. At LLM scale, however, fully computing this hypergradient is costly. We therefore evaluate a scalable select-and-differentiate instantiation of HG-MS, which preserves the core mechanism while making large-scale experiments feasible. This allows us to assess the practical impact of the HG-MS selection mechanism in crucial real-world applications.
We evaluate HG-MS primarily on LLM source-reweighting tasks and two MNIST bilevel benchmarks. The LLM experiments are the main practical testbed: under matched LoRA fine-tuning budgets, HG-MS attains the best GSM8K/MATH scores across the Llama-3-8B, Qwen-2-7B, Gemma-2-9B, and Qwen2.5-32B math settings, and is competitive or best on MT-Bench instruction following. In two-source GPT-2 data-composition experiments it improves final test loss over all baselines; MNIST hyper-cleaning \citep{ren2018reweight} and imbalanced-loss tuning provide multi-seed evidence that explicit selection improves upper-level optimization and test performance.

\textbf{Contributions.} In summary, we make three contributions:
\begin{itemize}[leftmargin=*]
\item \textbf{Differentiability of the hyper-objective under manifold-structured lower-level minimizers.}
We characterize when the optimistic hyper-objective is differentiable (and smooth) in the regime where the lower-level problem has a manifold of minimizers, showing that unique optimistic selection (and non-degeneracy) is  the key condition.

\item \textbf{A hyper-gradient method with \BoN minima selection and oracle complexity guarantee.}
We propose \textsc{HG-MS}, that combines \BoN minima selection with approximate hyper-gradient computation, achieving first-order stationarity guarantees with  oracle complexity governed by the intrinsic dimension of the lower-level solution manifold. 

\item \textbf{{Empirical support for the select-then-differentiate principle.}}
We instantiate HG-MS in LLM source-reweighting tasks, where explicit minima selection changes the learned source mixtures and improves matched-budget downstream performance in several mathematical-reasoning and instruction-following settings. Complementary MNIST hyper-cleaning and imbalanced-loss tuning experiments provide multi-seed evidence that the same select-then-differentiate principle improves optimization stability and test performance relative to standard hyper-gradient and penalty/alternating baselines.
\end{itemize}
\paragraph{Notations}
For any nonempty closed convex set \(C\subset\R^m\), we write $\mathrm{Proj}_{C}(u) :=\arg\min_{v\in C}\|v-u\|$.
For any embedded \(\mathcal{C}^{2}\) submanifold \(\mathcal{M}\subset\R^{d}\) and \(x\in\mathcal{M}\), we write
\(\calT_{x}\mathcal{M}\) and \(\calN_{x}\mathcal{M}\) for the tangent and normal spaces, and
\(P_{\calT_x\mathcal{M}}\), \(P_{\calN_x\mathcal{M}}\) for the corresponding orthogonal projectors.
When \(\mathcal{M}=\calS(\theta)\), we abbreviate \(\calT_x^\theta := \calT_x\calS(\theta)\) and \(\calN_x^\theta := \calN_x\calS(\theta)\).
For probability measures \(\nu\) and \(\mu\), we write \(R_2(\nu\|\mu)\) for the order-$2$ R\'enyi divergence.
Additional notations are collected in \Cref{append:notation-guide}.

\section{Local PL Condition and the Manifold Solution Set}\label{sec:prelim}
We now describe the \(\PLcirc\) condition introduced by \citet{gong2024poincare}, which is our central structural assumption.
Besides the standard {local} Polyak--{\L}ojasiewicz (P\L) inequality around the local minimizers, the extra requirement in the \(\PLcirc\) condition ensures that all local minima are connected (hence are all global minima). 
This supports efficient sampling-based approximations of the global optimal set\footnote{{If $\calS(\theta)$ has at least two connected components, then the mixing time of Langevin sampling of $\gibbstheta$ scales exponentially in $\lambda^{-1}$; see, e.g., \citet{menz2014poincare}.}}.

\begin{definition}[Local P\L]\label{def_local_PL}
Let \(\mathcal{M}\) be the collection of all local minima of \(g\in\mathcal{C}^{1}(\R^{d},\R)\).
We say that \(g\) is \emph{locally P\L} if there exists \(\mu>0\) such that for any connected component
\(\mathcal{M}'\subseteq \mathcal{M}\), there exists an open neighborhood \(\mathcal{N}(\mathcal{M}')\supset \mathcal{M}'\) satisfying
\[
g(x) - \min_{x' \in \mathcal{N}(\mathcal{M}')} g(x')
\;\leq\;
(2\mu)^{-1}\,\|\nabla g(x)\|^{2},
\, \forall x \in \mathcal{N}(\mathcal{M}').
\]
\end{definition}

\begin{definition}[\(\PLcirc\) condition]\label{def_mu_PL_assum}
A function \(g:\R^d\to\R\) satisfies the \(\PLcirc\) condition if: (1) \(g\in\mathcal{C}^{2}\) is locally P\L, and \(g\in\mathcal{C}^{4}\) on every neighborhood \(\mathcal{N}(\mathcal{M}')\); (2) Let \(\mathcal{N}(\mathcal{M}) := \cup_{\mathcal{M}'}\mathcal{N}(\mathcal{M}')\). For any \(x\in \R^{d}\setminus \mathcal{N}(\mathcal{M})\), if \(\nabla g(x)=0\), then \(\nabla^{2}g(x)\prec 0\); (3) The collection of all local minima \(\mathcal{M}\) is contained in a compact set $\mathcal{V}$.
\end{definition}

We impose \(\PLcirc\) uniformly over the parameter \(\theta\) in the lower-level problem.
\begin{tcolorbox}[colback=gray!20, colframe=gray!50, boxrule=0pt, arc=0mm, left=0mm, right=0mm, top=0mm, bottom=0mm]
\begin{assumption}[Lower-level \(\PLcirc\)]\label{ass:plcirc}
There exists a constant \(\mu>0\) and a compact domain $\mathcal{V} \subseteq \mathbb{R}^d$, both independent of $\theta$, such that for every \(\theta\in\Theta\), the function \(g(\theta,\cdot)\) satisfies the \(\PLcirc\) condition.
\end{assumption}
\end{tcolorbox}

\Cref{ass:plcirc} has both geometrical (manifold structure) and analytical (normal nondegeneracy) implications on the lower-level solution set $\calS(\theta)$, summarized as follows.

\begin{proposition}[Characterizations of $\calS(\theta)$]\label{prop:geometry}
Under \Cref{ass:plcirc}, the following hold for every \(\theta\in\Theta\):
(i) \(\calS(\theta)\) is a connected compact \(\mathcal{C}^{2}\) embedded submanifold of \(\R^{d}\) without boundary
\citep[Prop.~3, Cor.~1]{gong2024poincare}. In particular, all local minima of \(g(\theta,\cdot)\) are global minima.
(ii) For any \(x\in\calS(\theta)\), the Hessian \(H(\theta,x):=\nabla_{xx}^{2} g(\theta,x)\) satisfies $\ker H(\theta,x) \;=\; \calT_x^\theta$ and $\langle v, H(\theta,x)v\rangle \;\ge\; c\,\|v\|^2$ for all $v\in \calN_x^\theta$
	for some constant \(c>0\) that is uniform over \(\theta\in\Theta\) and \(x\in\calS(\theta)\) \citep[Cor.~2.17]{rebjock2024fast}.
Moreover, assuming $g\in \mathcal{C}^{3}$, there exists an integer \(\dimk\) such that
\(
\dim(\calS(\theta))=\dimk, \forall \theta\in\Theta
\), \citep[Lem.~3.5]{masiha2025superquantile}.
\end{proposition}
\begin{remark}
Under the lower-level local \PLcirc\ condition and mild regularity on $f$ and $g$, the hyper-objective function is Lipschitz continuous; see
\citep[Lem.~3.2]{masiha2025superquantile}. 
While this implies that $F$ is differentiable almost everywhere, it remains insufficient for our purpose.  
\end{remark}

\section{Local Differentiability of the Hyper-objective}\label{sec:hyp_diff}
Under the \PLcirc\ condition, we identify the \emph{worst-case} necessary conditions for the hyper-objective $\hyperobjective$ to be continuously differentiable and smooth, i.e. Lipschitz gradient.
To preserve as much generality as possible, all assumptions are made only locally around a fixed point $\theta_0$; Consequently, all conclusions below are local, holding only in a neighborhood of $\theta_0$.
All formal proofs for this section are given in Appendix~\ref{append_hyper_diff}.
This section is structured as follows:
\begin{itemize}[leftmargin=*]
    \item Through the counterexample in \Cref{ex:kink}, we show that a tie in \minimaselection, i.e., a non-singleton set $\selectedminimizerset(\theta)$, can make $F$ non-differentiable, even if both $f$ and $g$ are $\mathcal{C}^\infty$.
    \item We show in \Cref{thm:F-diff-unique} that the uniqueness of optimistic minimizer is sufficient to obtain $F\in\mathcal C^1$, together with the explicit hyper-gradient formula (\ref{eq:hg_formula}).
    \item Through the counterexample in \Cref{ex:unique-nonsmooth}, we show that the unique optimistic minimizer condition alone does not guarantee $\alpha$-\holder\ continuity of $\nabla F$ for any $\alpha>0$ (including Lipschitz continuity, which corresponds to $\alpha=1$).
    \item We show that when enhancing the unique optimistic minimizer condition with the non-degeneracy condition (\Cref{def:degenerate-optimistic}), the selected minimizer admits a $\mathcal C^1$ local branch and $\nabla \hyperobjective$ is locally Lipschitz (see \Cref{thm_hyp_obj_diff}).
\end{itemize}

\textbf{Non-unique minima-selection can create kinks in the hyper-objective.}
Even when $\calS(\theta)$ is a $\mathcal{C}^\infty$ connected manifold and varies smoothly with $\theta$, minimizing $f(\theta,x)$ over $\calS(\theta)$ can introduce kinks in $F$ whenever the optimistic minimizer is not unique.
The example below shows that, in an \emph{arbitrarily small} neighborhood, $F$ may have \emph{infinitely} many non-differentiable points despite $f$, $g$, and $\calS(\theta)$ being $\mathcal{C}^\infty$; the singularities arise solely from switching among non-unique minimizers.

    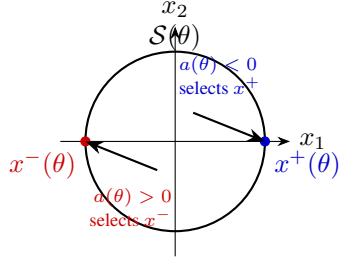
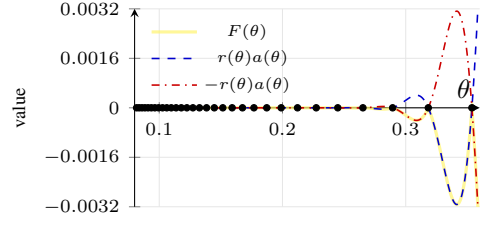
\begin{figure*}[t]
		\centering
		\begin{subfigure}{0.43\linewidth}
			\centering
			\begin{tikzpicture}[scale=0.95, >=Stealth]
			\def\r{1.25}
			\draw[->] (-1.6,0) -- (1.6,0) node[right] {$x_1$};
			\draw[->] (0,-1.6) -- (0,1.6) node[above] {$x_2$};
			\draw[thick, black] (0,0) circle (\r);
			\node[black] at (0,1.45) {$\calS(\theta)$};
			\filldraw[blue!80!black] (\r,0) circle (1.8pt) node[below right] {$x^+(\theta)$};
			\filldraw[red!80!black] (-\r,0) circle (1.8pt) node[below left] {$x^-(\theta)$};
			\draw[black, ->, thick] (0.25,0.40) -- (\r,0);
			\node[black, align=left, font=\scriptsize, color=blue!80!black] at (0.62,0.92) {$a(\theta)<0$\\ selects $x^+$};
			\draw[black, ->, thick] (-0.25,-0.40) -- (-\r,0);
			\node[black, align=right, font=\scriptsize, color=red!80!black] at (-0.62,-0.92) {$a(\theta)>0$\\ selects $x^-$};
		\end{tikzpicture}
		\caption{Two competing optimistic minimizers.}
	\end{subfigure}
		\hfill
		\begin{subfigure}{0.5\linewidth}
			\centering
				\begin{tikzpicture}
						\begin{axis}[
							width=0.88\linewidth,
							height=4.2cm,
							xmin=0.08, xmax=0.36,
							declare function={
								r(\th)=sqrt(1+\th*\th);
								a(\th)=exp(-1/(\th*\th))*sin(deg(10/\th));
							},
							ymin=-3.2e-3, ymax=3.2e-3,
							xlabel={$\theta$},
							ylabel={value},
							tick label style={font=\scriptsize},
							label style={font=\scriptsize},
							legend style={font=\tiny, draw=none, fill=none, at={(0.02,0.98)}, anchor=north west},
							samples=900,
							domain=0.08:0.36,
							grid=major,
							grid style={line width=.1pt, draw=gray!20},
							axis x line=middle,
							axis y line=left,
							unbounded coords=jump,
							ytick={-3.2e-3,-1.6e-3,0,1.6e-3,3.2e-3},
							scaled y ticks=false,
							yticklabel style={/pgf/number format/fixed,/pgf/number format/precision=5},
						]
                        \addplot[yellow!70!white, very thick, opacity=0.7] {-r(x)*abs(a(x))};
							\addlegendentry{$F(\theta)$}
						\addplot[blue!80!black, semithick, dashed] {r(x)*a(x)};
						\addlegendentry{$\;\;r(\theta)a(\theta)$}
						\addplot[red!80!black, semithick, dashdotted] {-r(x)*a(x)};
						\addlegendentry{$-r(\theta)a(\theta)$}
							\pgfplotsinvokeforeach{9,...,39}{
								\addplot[black, only marks, mark=*, mark size=1.2pt, forget plot]
									coordinates {({10/(#1*pi)},0)};
							}
						\end{axis}
					\end{tikzpicture}
					\caption{Two smooth branches and the resulting hyper-objective (kinks at ties $\theta_k=10/(k\pi)$, a countably infinite set accumulating at $0$, marked in black).}
				\end{subfigure}
			\caption{\textbf{Illustration of \Cref{ex:kink}.}
			Left: selection switches between two endpoints of $\calS(\theta)$ according to the sign of $a(\theta)$.
		Right: $F(\theta)=-r(\theta)|a(\theta)|$ is obtained by taking the smaller value between the two branches $\pm r(\theta)a(\theta)$, creating kinks at ties $a(\theta)=0$ (black dots).}
		\label{fig:ex-kink}
	\end{figure*}
\begin{example}\label{ex:kink}
Let $\theta\in\R$ and $x=[x_1,x_2]\in\R^2$ and write $r(\theta)=\sqrt{1+\theta^2}$.
Define $a(\theta):= e^{-1/\theta^2}\!\sin(10/\theta)$ for $\theta \neq 0$ and $a(0) = 0$.
Consider the functions $f(\theta,x):=a(\theta)\,x_1 + x_2^2$ and $g(\theta,x):=\big(\|x\|^2-(r(\theta))^2\big)^2$, both of which belong to $\mathcal C^\infty(\R \times \R^2)$.
The lower-level solution set is the circle
\(
\calS(\theta)=\{x\in\R^2:\ \|x\|=r(\theta)\}.
\)
On $\calS(\theta)$,
$f(\theta,x)=a(\theta) x_1+r(\theta)^2-x_1^2$, which is a concave quadratic in $x_1$. Therefore, for $a(\theta)\neq 0$, the minimizer is unique and given by
\(
x^\star(\theta)= [ - \mathrm{sign}(a(\theta)) r(\theta), 0 ];
\)
at any $\theta$ with $a(\theta)=0$, we have
\(
\arg\min_{x\in\calS(\theta)} f(\theta,x)=\{[-r(\theta),0],[r(\theta),0]\},
\) i.e., the selected optimistic minimizer is \emph{not unique}. Consequently, $F(\theta)=\min_{x\in\calS(\theta)} f(\theta,x)=-r(\theta)\,|a(\theta)|$. Since $a(\theta)=0$ at $\theta_k:=10/(k\pi)$ for each $k\in\mathbb{Z}\setminus\{0\}$, $F$ is not differentiable at every $\theta_k$ (a countably infinite set accumulating at $0$); see \Cref{fig:ex-kink}.
\end{example}


\textbf{Uniqueness in \minimaselection implies differentiability.}
The preceding example suggests that non-uniqueness of the optimistic minimizer is the main obstruction to differentiability. Our first positive result shows that, under the \PLcirc\ geometry summarized in \Cref{prop:geometry}, uniqueness is sufficient to recover $\mathcal C^1$ regularity of $F$ and an explicit hyper-gradient formula.%

\begin{theorem}[$\mathcal C^1$ regularity of the hyper-objective under uniqueness]\label{thm:F-diff-unique}
Let \Cref{ass:plcirc} hold and assume $f \in \mathcal C^1$ and $g \in \mathcal C^3$.
Fix $\theta_0\in\Theta$ and assume there exist an open set $\mathcal U$ containing $\theta_0$ such that, for all $\theta\in\mathcal U$, \(\arg\min_{x\in\calS(\theta)} f(\theta,x)\) is a singleton.
Then there exists an open neighborhood $\mathcal U_0\subseteq\mathcal U$ of $\theta_0$ such that $F \in \mathcal{C}^1(\mathcal U_0)$.
Moreover, for all $\theta\in\mathcal U_0$, the hyper-gradient is given by \cref{eq:hg_formula}.
\end{theorem}
\begin{proof}[Proof sketch]
We sketch the proof in three steps; see Appendix~\ref{append:F-diff-unique} for the full argument.

\smallskip\noindent
\emph{Step 1: parameterize \(\calS(\theta)\) locally by solving the normal equations.}
Let \(x_0:=x^\star(\theta_0)\), and abbreviate
\(\calT_0:=\calT_{x_0}^{\theta_0}\) and \(\calN_0:=\calN_{x_0}^{\theta_0}\).
Choose orthonormal bases \(U_{\calT_0}\) and \(U_{\calN_0}\) spanning \(\calT_0\) and \(\calN_0\). Write
\[
x=x_0+U_{\calT_0}u+U_{\calN_0}v
\]
and consider
\[
\Phi(\theta,u,v)
:=
U_{\calN_0}^{\top}\nabla_x g(\theta,x_0+U_{\calT_0}u+U_{\calN_0}v).
\]
By the normal nondegeneracy implied by \Cref{ass:plcirc} (see \Cref{prop:geometry}),
\[
D_v\Phi(\theta_0,0,0)
=
U_{\calN_0}^{\top}\nabla^2_{xx}g(\theta_0,x_0)U_{\calN_0}
\]
is invertible. The implicit function theorem gives a \(\mathcal C^2\) map \(v=\varphi(\theta,u)\) solving
\(\Phi(\theta,u,\varphi(\theta,u))=0\) near \((\theta_0,0)\). Defining
\[
\psi(\theta,u):=x_0+U_{\calT_0}u+U_{\calN_0}\varphi(\theta,u),
\]
we obtain a local chart for the normal critical set. The \(\PLcirc\) geometry, the fixed local dimension of
\(\calS(\theta)\), and local Hausdorff continuity then identify this chart with the nearby branch of the true minimizer
manifold \(\calS(\theta)\).

\smallskip\noindent
\emph{Step 2: reduce to a fixed-domain minimization.}
By uniqueness of the optimistic minimizer and compactness of the lower-level solution set, \(x^\star(\theta)\) remains in
this chart patch for all \(\theta\) close to \(\theta_0\). After possibly shrinking the parameter neighborhood, choose a
compact coordinate set \(U\subset\mathbb R^{\dimk}\), contained in the chart domain, that contains the coordinate
\(u^\star(\theta)\) of the selected minimizer for all such \(\theta\). The optimistic problem becomes the fixed-domain
minimization
\[
F(\theta)=\min_{u\in U}\tilde f(\theta,u),
\qquad
\tilde f(\theta,u):=f(\theta,\psi(\theta,u)).
\]

\smallskip\noindent
\emph{Step 3: apply an envelope argument and compute the hyper-gradient.}
We apply the Danskin-type envelope lemma in Appendix~\ref{append:F-diff-unique}, namely
\Cref{lem:envelope-unique}. For a value function \(F(\theta)=\min_{u\in U}\phi(\theta,u)\) over a fixed compact set,
uniqueness of the minimizer allows one to differentiate \(F\) by differentiating only \(\phi\) with respect to
\(\theta\) at the minimizing point; no derivative of the argmin map is needed. Applying this lemma pointwise, and then
using continuity of the unique minimizer on the fixed compact domain, yields the local \(\mathcal C^1\) regularity of
\(F\). Finally, differentiating the lower-level stationarity identity
\(\nabla_x g(\theta,\psi(\theta,u))=0\) gives the pseudoinverse hyper-gradient formula \eqref{eq:hg_formula}.
\end{proof}
\begin{remark} \label{remark_pseudoinverse}
For compactness, denote \(\mathbf{H}:=\nabla^{2}_{xx}g(\theta,x^\star(\theta))\), and let $\mathbf{N}_\theta$ and $\mathbf{T}_\theta$ be orthogonal bases that span \(\calN_{x^\star(\theta)}^\theta\) and \(\calT_{x^\star(\theta)}^\theta\), respectively.
Under the \(\PLcirc\) condition, \Cref{prop:geometry} implies \(\mathbf{N}_\theta^\top H \mathbf{N}_\theta \succ 0\);
Besides, we have \(\mathbf{T}_\theta^\top \mathbf{H} \mathbf{T}_\theta = 0\).
Consequently, $\mathbf{N}_\theta\left(\mathbf{N}_\theta^\top \mathbf{H} \mathbf{N}_\theta\right)^{-1} \mathbf{N}_\theta^\top$ coincides with the Moore--Penrose pseudoinverse of $\mathbf{H}$.
\end{remark}

\textbf{Uniqueness in \minimaselection alone is insufficient for \holder\ continuity of $\nabla F$.}
From an optimization perspective, however, $\mathcal C^1$ regularity is not enough.
For a non-asymptotic analysis, we typically need $\nabla F$ to be $\alpha$-\holder\ continuous for some $\alpha > 0$.
However, \Cref{ex:unique-nonsmooth} shows that uniqueness alone is \emph{insufficient} to guarantee this property.
This occurs when the upper-level objective restricted to the minimizer manifold has zero intrinsic second-order
curvature at the selected point\footnote{One might instead try to require positivity of the ambient Hessian \(\nabla^2_{xx}f\) on tangent vectors, but this is not
an intrinsic condition, as \Cref{ex:unique-nonsmooth} demonstrates.}. Therefore, the restricted objective
\(f(\theta,\cdot)|_{\calS(\theta)}\) must be non-degenerate, for example through its Riemannian Hessian.

\begin{definition}[Degenerate and non-degenerate optimistic minimizers]\label{def:degenerate-optimistic}
Consider the setting of \Cref{thm:F-diff-unique}.
Let \(\bar f_\theta:=f(\theta,\cdot)|_{\calS(\theta)}\). We say that $x^*(\theta)$ is a
\emph{non-degenerate optimistic minimizer} for $\theta \in \mathcal{U}_0$ if the Riemannian Hessian\footnote{See
\Cref{eq:riem-hess-restricted} in Appendix~\ref{append:regularity-assumption} for the convention used here.} of the restricted
objective is positive definite at \(x^*(\theta)\), i.e., $\big\langle v,\mathrm{Hess}_{\calS(\theta)}\bar f_\theta(x^*(\theta))[v]\big\rangle>0$ for all $v\in\calT_{x^*(\theta)}^{\theta}\setminus\{0\}$.
Equivalently, \(f(\theta,\cdot)\) has strictly positive intrinsic second-order derivative along every tangent direction
of the manifold \(\calS(\theta)\) at \(x^*(\theta)\).
If this condition fails, we call $x^*(\theta)$ a \emph{degenerate optimistic minimizer}. 
\end{definition}

The next example shows that this intrinsic condition is not cosmetic: uniqueness can hold for every \(\theta\) while
\(\nabla F\) has no positive H\"older modulus, and a condition based only on the ambient tangent Hessian would miss the
obstruction.

\begin{example}[Uniqueness does not imply local \texorpdfstring{$\alpha$}{alpha}-\holder\ continuity of $\nabla F$]\label{ex:unique-nonsmooth}
Consider the following smooth construction. For $\theta\in\R$ and $x=(x_1,x_2)\in\R^2$, let
\[
a(\theta):=
\begin{cases}
\frac14 \exp({-1/\theta^2})\sin(10/\theta), & \theta\neq 0,\\
0, & \theta=0,
\end{cases}
\qquad
\eta(t):=
\begin{cases}
\operatorname{sign}(t)\exp({-1/t^2}), & t\neq 0,\\
0, & t=0,
\end{cases}
\]
and define $\phi(t):=\int_0^t \eta(s)\,\mathrm ds$. Then $\eta,\phi\in\mathcal{C}^{\infty}(\R)$, $\eta=\phi'$ is strictly increasing, and $\phi''(0)=\eta'(0)=0$. Set
\[
g(\theta,x):=\big(\|x\|^2-(1+\theta^2)\big)^2,
\qquad
f(\theta,x):=\phi(x_2)+a(\theta)x_2+\rho(-x_1)+b\big(\|x\|^2-(1+\theta^2)\big),
\]
where \(b>0\) is arbitrary, $\rho(s)=\exp({-1/s})$ for $s>0$, and $\rho(s)=0$ for $s\le 0$. Thus $\calS(\theta)=\{x\in\R^2:\|x\|=r(\theta)\}$ with $r(\theta):=\sqrt{1+\theta^2}$.

The added \(b\)-term is identically zero on \(\calS(\theta)\), so it does not change the constrained minimization problem.
As in \Cref{ex:kink}, the term $\rho(-x_1)$ forces the optimistic minimizer to lie on the right semicircle, so minimizing $f(\theta,\cdot)$ over $\calS(\theta)$ reduces to
\[
\min_{|t|\le r(\theta)} \bigl(\phi(t)+a(\theta)t\bigr).
\]
Since $\eta$ is strictly increasing, this scalar problem has a unique critical point determined by $\eta(t)+a(\theta)=0$, namely
\[
t^\star(\theta)=
\begin{cases}
-\operatorname{sign}(a(\theta))\bigl(\log(1/|a(\theta)|)\bigr)^{-1/2}, & a(\theta)\neq 0,\\
0, & a(\theta)=0.
\end{cases}
\]
Moreover, $|t^\star(\theta)|\le(\log 4)^{-1/2}<1\le r(\theta)$, so this critical point is feasible and therefore is the unique minimizer. Hence the optimistic minimizer is unique for every $\theta$ and is
\[
x^\star(\theta)=\bigl(\sqrt{r(\theta)^2-(t^\star(\theta))^2},\,t^\star(\theta)\bigr).
\]

By \Cref{thm:F-diff-unique}, the hyper-objective $F$ is $\mathcal C^1$. Since the \(b\)-term vanishes on \(\calS(\theta)\), the scalar envelope formula gives $F'(\theta)=a'(\theta)t^\star(\theta)$. Let $\theta_k:=10/(k\pi)$, so $a(\theta_k)=0$ and $a'(\theta_k)\neq 0$. Then $F'(\theta_k)=0$. Since \(a'\) is continuous and nonzero at \(\theta_k\), and since \(\theta_k\) is a simple zero of \(a\), we have \(|a'(\theta)|\asymp1\) and \(|a(\theta)|\asymp|\theta-\theta_k|\) near \(\theta_k\). Therefore
\[
|F'(\theta)-F'(\theta_k)|
=|F'(\theta)|
\asymp
\frac{1}{\sqrt{\log(1/|\theta-\theta_k|)}}
\qquad\text{as }\theta\to\theta_k.
\]
Consequently, for every $\alpha>0$,
\[
\frac{|F'(\theta)-F'(\theta_k)|}{|\theta-\theta_k|^\alpha}\to\infty
\qquad\text{as }\theta\to\theta_k,
\]
because $|\theta-\theta_k|^\alpha\sqrt{\log(1/|\theta-\theta_k|)}\to 0$. Thus $\nabla F$ is not locally $\alpha$-\holder\ continuous near $\theta_k$ for any $\alpha>0$, even though the optimistic minimizer is unique for every $\theta$.

This example also shows why non-degeneracy must be intrinsic to \(\calS(\theta)\). At any zero \(\theta_k\), we have \(x^\star(\theta_k)=(r(\theta_k),0)\), and the unit tangent direction is \(v=(0,1)\). The ambient tangent Hessian satisfies
\[
\big\langle v,\nabla^2_{xx}f(\theta_k,x^\star(\theta_k))v\big\rangle=2b>0,
\]
because the vanishing term \(b(\|x\|^2-r(\theta_k)^2)\) contributes \(2bI\) to the ambient Hessian. However, this term is zero on \(\calS(\theta_k)\), so it contributes nothing to the restricted objective. The Riemannian Hessian of \(f(\theta_k,\cdot)|_{\calS(\theta_k)}\) in direction \(v\) is therefore \(0\), since the restricted scalar objective has second derivative \(\phi''(0)=0\). A condition based only on the ambient tangent Hessian would incorrectly classify this example as non-degenerate, even though the intrinsic restricted problem is degenerate and \(\nabla F\) has no positive H\"older modulus.
\end{example}

\textbf{Non-degeneracy of the unique optimistic minimizer implies local smoothness of $F$.}
To establish the stronger stability needed for the subsequent complexity analysis, we rule out this pathological case by assuming that the unique optimistic minimizer $x^\star(\theta)$ is non-degenerate.
To prove the local smoothness of $F$, the main missing ingredient is to establish a $\mathcal C^1$ local branch for $x^\star(\theta)$.
Once this is available, we can write $F(\theta)=f(\theta,x^\star(\theta))$ locally and then control $\nabla F$ using the regularity of $f,g$ together with the pseudoinverse hyper-gradient formula.
The next theorem proves this branch regularity by first deriving a parameter-dependent local chart for \(\calS(\theta)\) from \Cref{prop:geometry}, and then applying the implicit function theorem to the restricted problem on this chart.
\begin{theorem}[Smoothness of the hyper-objective under unique non-degenerate optimistic minimizer]\label{thm_hyp_obj_diff}
Let $f \in \mathcal C^2$, $g \in \mathcal C^3$ and satisfy \Cref{ass:plcirc}.
Fix $\theta_0\in\Theta$ and set \(x_0:=x^\star(\theta_0)\). Suppose that \(x_0\) is the unique optimistic minimizer and is non-degenerate in the sense of \Cref{def:degenerate-optimistic}.
Then there exists a neighborhood $\mathcal U$ of $\theta_0$ and a (unique) $\mathcal C^1$ selection
$x^\star:\mathcal U\to\R^d$ with $x^\star(\theta)\in\calS(\theta)$ such that $F(\theta)=f(\theta,x^\star(\theta))$ for all $\theta\in\mathcal U$.
Consequently, \(\nabla F\) is locally Lipschitz.
\end{theorem}
\begin{proof}[Proof sketch]
Let \(x_0:=x^\star(\theta_0)\), and abbreviate
\(\calT_0:=\calT_{x_0}^{\theta_0}\) and \(\calN_0:=\calN_{x_0}^{\theta_0}\).
Let \(\dim(\calT_0)=\dimk\), and choose orthonormal bases
\(U_{\calT_0}\in\R^{d\times\dimk}\) and
\(U_{\calN_0}\in\R^{d\times(d-\dimk)}\) spanning \(\calT_0\) and \(\calN_0\).

\smallskip\noindent
\emph{Step 1: parameterize \(\calS(\theta)\) locally by solving the normal equations.}
This is the same normal-chart construction as in the proof sketch of \Cref{thm:F-diff-unique}. The implicit function
theorem yields a \(\mathcal C^2\) map \(v=\varphi(\theta,u)\), and hence the local chart
\[
\psi(\theta,u):=x_0+U_{\calT_0}u+U_{\calN_0}\varphi(\theta,u).
\]
On a sufficiently small neighborhood, \(u\mapsto\psi(\theta,u)\) parameterizes the local branch of \(\calS(\theta)\)
near \(x_0\).

\smallskip\noindent
\emph{Step 2: reduce the optimistic minimization to intrinsic coordinates.}
Define the pullback objective \(\tilde f(\theta,u):=f(\theta,\psi(\theta,u))\). A point
\(x=\psi(\theta,u)\) is a constrained critical point of \(f(\theta,\cdot)\) on \(\calS(\theta)\) if and only if
\(\nabla_u\tilde f(\theta,u)=0\). At \((\theta_0,0)\), this stationarity holds. The non-degeneracy assumption in
\Cref{def:degenerate-optimistic} says precisely that the second-order variation of
\(\tilde f(\theta_0,\cdot)\) at \(u=0\) is positive definite, so
\(D_u(\nabla_u\tilde f)(\theta_0,0)\) is invertible. Applying the implicit function theorem again gives a
\(\mathcal C^1\) map \(u^\star(\theta)\), unique among nearby solutions, satisfying
\[
\nabla_u\tilde f(\theta,u^\star(\theta))=0.
\]
Set \(x_{\mathrm{loc}}(\theta):=\psi(\theta,u^\star(\theta))\). This gives the unique local optimistic branch. The full
proof then uses the strict value gap from the uniqueness of \(x_0\) on the compact set \(\calS(\theta_0)\) to rule out
competitors outside the chart patch, yielding the local identity
\[
x^\star(\theta)=x_{\mathrm{loc}}(\theta),
\qquad
F(\theta)=f(\theta,x^\star(\theta)).
\]

\smallskip\noindent
\emph{Step 3: obtain the hyper-gradient by the chain rule.}
Since \(F(\theta)=f(\theta,x^\star(\theta))\), the hyper-gradient follows from the chain rule. At a constrained minimizer,
the tangential component of \(\nabla_x f(\theta,x^\star(\theta))\) vanishes, so only the normal component of
\(D_\theta x^\star(\theta)\) is needed. Differentiating the lower-level stationarity condition
\(\nabla_x g(\theta,x^\star(\theta))=0\) and applying the pseudoinverse of
\(\nabla^2_{xx}g(\theta,x^\star(\theta))\) gives this normal sensitivity. Substitution gives
\eqref{eq:hg_formula}. The same formula, together with constant rank and the normal spectral gap from
\Cref{prop:geometry}, implies local Lipschitz continuity of \(\nabla F\).
\end{proof}

\section{\textsc{HG-MS}: a Select-then-Differentiate Algorithm}\label{sec:alg}

The previous section gives a local regularity result. For algorithm design and convergence analysis on $\Theta$, we impose uniqueness and non-degeneracy globally.

\begin{tcolorbox}[colback=gray!20, colframe=gray!50, boxrule=0pt, arc=0mm, left=0mm, right=0mm, top=0mm, bottom=0mm]
\begin{assumption}[Unique and non-degenerate optimistic minimizer; \textbf{Global}]\label{ass:unique_min}
For every $\theta\in\Theta$, the selected optimistic minimizer $x^\star(\theta)$ is unique and non-degenerate in the sense of \Cref{def:degenerate-optimistic}.
\end{assumption}
\end{tcolorbox}

This yields global smoothness of the hyper-objective.

\begin{proposition}\label{prop:F-smooth}
Let $f \in \mathcal C^2$, $g \in \mathcal C^3$, and suppose \Cref{ass:plcirc,ass:unique_min} hold. There exists $L_F<\infty$ such that $\|\nabla F(\theta)-\nabla F(\theta')\|\le L_F\|\theta-\theta'\|$, for all $\theta,\theta'\in\Theta$. In particular, $F$ is $L_F$-smooth on $\Theta$.
\end{proposition}

Thus, the natural upper-level update is projected gradient descent, $\theta_{t+1}:=\mathrm{Proj}_{\Theta}\big(\theta_t-\alpha_t\nabla F(\theta_t)\big),$ with $\nabla F$ given by \cref{eq:hg_formula}. Implementing this update requires two primitives: \stagecirc{A} an approximate optimistic selector $x^\star(\theta)$;
    \stagecirc{B} an approximate pseudoinverse action
    $\big[\nabla^2_{xx}g(\theta,x^\star(\theta))\big]^\dagger\nabla_x f(\theta,x^\star(\theta))$.
Both tasks are challenging because $\calS(\theta)$ is an implicit non-convex manifold, and because the selected approximate lower-level solution may lie slightly off $\calS(\theta)$, resulting in unstable Hessian pseudoinverse computations.

\subsection{Algorithm Design}

We propose an algorithm, Hyper-Gradient with Minima-Selection (\textsc{HG-MS}), that addresses the aforementioned two challenges by: \emph{Best-of-N} (\BoN) selection step from a Gibbs measure concentrated near $\calS(\theta)$, and a stabilized implicit step which computes the pseudoinverse action through a ridge-regularized linear solver that uses Hessian--vector products.
Algorithm~\ref{alg:hg-minsel-lmc} is the ULA/Gibbs instantiation used for the formal convergence analysis. In practice, the sampler may be replaced by a finite-budget branch generator that preserves the same select-then-differentiate structure.

\textbf{\stagecirc{A} \BoN selection near $\calS(\theta)$.}
Since $\calS(\theta)$ is implicit, we approximate it using the Gibbs measure $\gibbstheta(\d x)\propto \exp\{-g(\theta,x)/\lambda\}\,\d x$.
Under the manifold regularity induced by \PLcirc, $\gibbs_\theta^\lambda$ concentrates in a thin tube around $\calS(\theta)$ as $\lambda\to 0$ (see \Cref{lem:gibbs-tube-width}). Thus, samples from $\gibbstheta$ serve as proxies for points on $\calS(\theta)$.
Given samples $X_1,\ldots,X_N\sim\gibbstheta$, we approximate optimistic selection by choosing $i^\star\in\arg\min_{1\le i\le N} f(\theta,X_i)$ and setting $\hat{x}_N^\lambda(\theta):=X_{i^\star}$.
Equivalently, \BoN acts as a lower-tail operation at level $\delta=1/N$, selecting from the best $\delta$ fraction of candidates under $f(\theta,\cdot)$. The resulting selection error and its geometric dependence on $\calS(\theta)$ are analyzed in Appendix~\ref{append:selection-error}. We sample from $\gibbstheta$ using ULA primitive\footnote{The oracle-complexity guarantee uses \(\lambda\)-scaled Gaussian initializations for the ULA chains.} (see Appendix~\ref{append:lmc-ula}) \citep{chewi2024analysis}, whose convergence is controlled by the Poincar\'e inequality of \citep{gong2024poincare}; this yields a trade-off between selection accuracy and sampling cost.

\begin{remark}
Approximating $\calS(\theta)$ via $\gibbstheta$ also appears in \citep{masiha2025superquantile}, but their goal is to approximate only the \textbf{value} of $F$. Here we must approximate the selected \textbf{solution} itself in order to compute a hyper-gradient, which requires a different analysis.
\end{remark}

\textbf{\stagecirc{B} Hyper-gradient evaluation at the selected point.}
After selecting $\hat{x}_N^\lambda(\theta_t)$, we estimate the hyper-gradient by substituting it for $x^\star(\theta_t)$ in \cref{eq:hg_formula}. Let $H_t:=\nabla^2_{xx}g(\theta_t,\hat{x}_N^\lambda(\theta_t))$ and $b_t:=\nabla_x f(\theta_t,\hat{x}_N^\lambda(\theta_t))$.
We approximate $H_t^\dagger b_t$ by solving the ridge-regularized system $(H_t+\gamma I)\tilde v_t=b_t$ with conjugate gradients, using only Hessian--vector products, until $\|b_t-(H_t+\gamma I)\tilde v_t\|\le \eta_t$.
The ridge parameter $\gamma\ge 0$ stabilizes the solve when $\hat{x}_N^\lambda(\theta_t)$ is close to, but not exactly in $\calS(\theta_t)$: tangential zero eigenvalues of the exact Hessian can perturb  near-zero eigenvalues off the manifold. 
For the theoretical guarantee, we also allow clipping, \(v_t:=\mathrm{Proj}_{\mathbb B_d(0;R_v)}(\tilde v_t)\).
When this optional safeguard is omitted in practice, we set \(v_t=\tilde v_t\). The hyper-gradient estimate is then
\begin{align}\label{eq:003}
    \widehat h_t
    :=
    \nabla_\theta f(\theta_t,\hat{x}_N^\lambda(\theta_t))
    -
    \nabla^2_{\theta x}g(\theta_t,\hat{x}_N^\lambda(\theta_t))v_t,
\end{align}
followed by a projected step on $\Theta$. \Cref{alg:hg-minsel-lmc} summarizes our approach.

\begin{algorithm}[t]
\caption{\textsc{HG-MS}: Hyper-gradient with minima selection}
\begin{algorithmic}[1]
\Require initial $\theta_0\in\Theta$; stepsizes $\{\alpha_t\}$; temperature $\lambda$; number of chains $N$; ULA steps $K$ and stepsize $h$; chain starts $\{X^{(0)}_{t,i}\}$; CG tolerances $\{\eta_t\}$; ridge $\gamma\ge 0$; clip radius $R_v>0$
\For{$t=0,1,2,\ldots$}
    \State Sample in parallel $X_{t,i}\gets \textsc{ULA}(\theta_t,g,X^{(0)}_{t,i},\lambda,K,h)$
    \State $i_t^\star\gets\arg\min_{1\le i\le N} f(\theta_t,X_{t,i})$; \quad
    $\hat{x}_N^\lambda(\theta_t)\gets X_{t,i_t^\star}$
    \State Solve
    \(
        \big(\nabla^2_{xx}g(\theta_t,\hat{x}_N^\lambda(\theta_t))+\gamma I\big)\tilde v_t
        =
        \nabla_x f(\theta_t,\hat{x}_N^\lambda(\theta_t))
    \)
    by CG up to residual $\eta_t$
    \State \textbf{(Optional safeguard)} $v_t\gets \mathrm{Proj}_{\mathbb B_d(0;R_v)}(\tilde v_t)$
    \State $\widehat h_t\gets
    \nabla_\theta f(\theta_t,\hat{x}_N^\lambda(\theta_t))
    -
    \nabla^2_{\theta x}g(\theta_t,\hat{x}_N^\lambda(\theta_t))v_t$; $\quad \theta_{t+1}\gets
    \mathrm{Proj}_{\Theta}\big(\theta_t-\alpha_t\widehat h_t\big)$
\EndFor
\end{algorithmic}
\label{alg:hg-minsel-lmc}
\end{algorithm}

\input{Section_Convergence_HG_MS_Neurips}

\section{Empirical Evaluations}\label{sec:experiments}
\paragraph{Overview.}
We study minima selection for hyper-gradients across LLM source-reweighting and MNIST~\citep{lecun1998mnist} benchmarks. The main text reports both controlled two-source GPT-2 reweighting and real-world LLM
data-composition benchmarks; the two MNIST benchmarks are reported in Appendix~\ref{subsec:dhc} (data hyper-cleaning)
and Appendix~\ref{subsec:imb} (imbalanced loss tuning). The two LLM subsections below state the corresponding
source-reweighting bilevel objectives explicitly.
In this section, \textsc{HG-MS} refers to practical finite-budget implementations of the same algorithmic template as the analyzed method.
While the convergence analysis is based on the ULA/Gibbs instantiation of \textsc{HG-MS}, the LLM experiments use a small number of finite-budget lower-level optimizer branches, implemented with SGD or AdamW, instead of ULA samples.
The main LLM comparisons use a common implementation and matched downstream fine-tuning budgets. We compare against the vanilla hyper-gradient baseline,  \textsc{HG-baseline}, \textsc{ScaleBiO}
\citep{pan2025scalebio}, \textsc{LESS} \citep{xia2024less}, and \textsc{RHO-LOSS}
\citep{mindermann2022prioritized}.

\subsection{Controlled LLM source reweighting}\label{subsec:llm-controlled-main}
We first test the source-reweighting mechanism in the controlled two-source GPT-2 tasks adapted from
\citet{pan2025scalebio}.

\textbf{Bilevel formulation.}
Let \(\mathcal D^{\mathrm{tr}}_1,\mathcal D^{\mathrm{tr}}_2\) be the two training sources and let
\(\mathcal D^{\mathrm{val}}\) be the target validation set. The outer variable is a two-dimensional source-logit vector
\(\theta\), with source weights \(\omega=\softmax(\theta)\), and the lower variable \(y\) is the full trainable GPT-2
(124M) parameter state. With token-level next-token loss \(\ell(y;\xi)\), the controlled source-reweighting objective is
\[
g_{\mathrm{ctrl}}(\theta,y)
:=
\sum_{s=1}^{2}\omega_s\,|\mathcal D^{\mathrm{tr}}_s|^{-1}
\sum_{\xi\in\mathcal D^{\mathrm{tr}}_s}\ell(y;\xi),
\qquad
f_{\mathrm{ctrl}}(y)
:=
|\mathcal D^{\mathrm{val}}|^{-1}
\sum_{\xi\in\mathcal D^{\mathrm{val}}}\ell(y;\xi).
\]
Thus \(\calS_{\mathrm{ctrl}}(\theta)=\arg\min_y g_{\mathrm{ctrl}}(\theta,y)\) is the set of lower-level GPT-2 states
obtained under the current source mixture, and the optimistic upper objective is
\(F_{\mathrm{ctrl}}(\theta)=\min_{y\in\calS_{\mathrm{ctrl}}(\theta)} f_{\mathrm{ctrl}}(y)\). In practice, the algorithms
approximate \(\calS_{\mathrm{ctrl}}(\theta)\) with finite lower-level optimizer runs; \textsc{HG-MS} selects the candidate
state with the smallest validation loss before forming the outer update. Unlike the real-world LLM experiments below, this
controlled code path trains the full GPT-2 trainable parameter state rather than LoRA adapters;
Appendix~\ref{append:exp-llm} gives the full protocol and hyperparameters.

\Cref{tab:llm-controlled} reports two tasks. In \emph{denoising}, the sources are clean Alpaca examples and synthetically
corrupted Alpaca examples, so the target source mixture is \(\omega^\star=(1,0)\). In \emph{multilingual training}, the
sources are Alpaca-GPT4-ZH and Alpaca-GPT4-EN, with validation/test composition
\(\omega^\star=(0.6,0.4)\). We report learned source weights, source-weight recovery error, validation loss, and test
loss. On denoising, all reweighting methods improve over \textsc{Uniform}; \textsc{HG-MS} gives the best test loss and
the best clean/noisy recovery. On multilingual training, \textsc{HG-MS} gives the best validation and test losses, while
\textsc{ScaleBiO} is closest to the nominal source ratio. Since the bilevel objective optimizes held-out loss rather than
source-ratio recovery, the multilingual result shows that minima selection can favor branches with better validation
behavior even when the learned mixture is not closest to \(\omega^\star\).

\begin{table}[t]
  \centering
  \caption{\textbf{Controlled two-source LLM source reweighting adapted from ScaleBiO.}
  All methods use the same GPT-2 backbone and single-GPU code path. \textsc{HG-baseline} is the one-branch hyper-gradient variant, and \textsc{HG-MS} is the minima-selection variant. In denoising, \(\omega=(\omega_{\mathrm{clean}},\omega_{\mathrm{noisy}})\) and \(\omega^\star=(1,0)\); in multilingual training, \(\omega=(\omega_{\mathrm{ZH}},\omega_{\mathrm{EN}})\) and \(\omega^\star=(0.6,0.4)\). Lower is better.}
  \label{tab:llm-controlled}
  \begin{adjustbox}{max width=\textwidth}
  \begin{tabular}{llcccc}
    \toprule
    Task & Method & Learned \(\omega\) & \(\|\omega-\omega^\star\|_1\) & Val loss & Test loss \\
    \midrule
    \multirow{4}{*}{Denoising}
      & Uniform   & \((0.500,\ 0.500)\) & 1.000 & 2.571 & 2.610 \\
      & ScaleBiO  & \((0.922,\ 0.078)\) & 0.156 & \textbf{2.336} & 2.443 \\
      & HG-baseline & \((0.930,\ 0.070)\) & 0.140 & 2.427 & 2.454 \\
      & HG-MS     & \((0.962,\ 0.038)\) & \textbf{0.077} & 2.403 & \textbf{2.437} \\
    \midrule
    \multirow{4}{*}{Multilingual}
      & Uniform   & \((0.500,\ 0.500)\) & 0.200 & 2.430 & 2.407 \\
      & ScaleBiO  & \((0.623,\ 0.377)\) & \textbf{0.047} & 2.339 & 2.371 \\
      & HG-baseline & \((0.540,\ 0.460)\) & 0.119 & 2.327 & 2.307 \\
      & HG-MS     & \((0.514,\ 0.486)\) & 0.172 & \textbf{2.323} & \textbf{2.298} \\
    \bottomrule
  \end{tabular}
  \end{adjustbox}
\end{table}

\subsection{Real-world LLM applications}\label{subsec:llm-realworld}
We test the source-reweighting template on the LLM data-composition benchmarks of \citet{pan2025scalebio}: mathematical reasoning and instruction following. These experiments use multi-source LoRA fine-tuning portfolios; Appendix~\ref{append:exp-llm-real}
gives the task definitions, source inventories, budgets, inference protocol, and judging details. Here we summarize the
matched-budget outcomes.

\textbf{Bilevel formulation.}
Both real-world benchmarks use the multi-source version of the controlled formulation above. Let \(p\) be the number of
training sources, let \(\theta\in\R^p\) be source logits, let \(\omega=\softmax(\theta)\in\Delta^{p-1}\), and let \(y\)
denote the LoRA adapter parameters for a fixed base model. With source datasets
\(\mathcal D^{\mathrm{tr}}_1,\ldots,\mathcal D^{\mathrm{tr}}_p\) and held-out validation set
\(\mathcal D^{\mathrm{val}}\), the reweighting stage solves the bilevel problem induced by
\[
g_{\mathrm{rw}}(\theta,y)
:=
\sum_{s=1}^{p}\omega_s\,|\mathcal D^{\mathrm{tr}}_s|^{-1}
\sum_{\xi\in\mathcal D^{\mathrm{tr}}_s}\ell(y;\xi),
\qquad
f_{\mathrm{rw}}(y)
:=
|\mathcal D^{\mathrm{val}}|^{-1}
\sum_{\xi\in\mathcal D^{\mathrm{val}}}\ell_{\mathrm{val}}(y;\xi),
\]
namely \(F_{\mathrm{rw}}(\theta)=\min_{y\in\calS_{\mathrm{rw}}(\theta)}f_{\mathrm{rw}}(y)\) with
\(\calS_{\mathrm{rw}}(\theta)=\arg\min_y g_{\mathrm{rw}}(\theta,y)\). Here \(\ell\) is the token-level next-token
cross-entropy on source examples, and \(\ell_{\mathrm{val}}\) is the same loss on validation examples used only during
source reweighting. After reweighting, the learned \(\omega\) is converted into source counts for a final one-epoch
supervised fine-tuning run, and the resulting checkpoint is evaluated on the task metrics reported below (GSM8K/MATH
accuracy or MT-Bench score).

\textbf{Mathematical reasoning benchmark.}
Following \citep[Section~4.2.2]{pan2025scalebio}, this benchmark uses a ten-source portfolio and evaluates the final fine-tuned checkpoint on GSM8K and MATH. All methods share the source portfolio, final one-epoch fine-tuning stage, and seed.
\begin{table*}[t]
  \centering
  \caption{\textbf{Mathematical reasoning.}
Scores for the ten-source math-reasoning portfolio.}
  \label{tab:llm-realworld-math}
  \begin{adjustbox}{max width=\textwidth}
  \begin{tabular}{lcccccc}
    \toprule
    & \multicolumn{3}{c}{GSM8K} & \multicolumn{3}{c}{MATH} \\
    \cmidrule(lr){2-4}\cmidrule(lr){5-7}
    Method & Llama-3-8B & Qwen-2-7B & Gemma-2-9B & Llama-3-8B & Qwen-2-7B & Gemma-2-9B \\
    \midrule
    Uniform Weighting & $60.06$ & $60.03$ & $78.08$ & $30.24$ & $30.26$ & $37.29$ \\
    RHO-LOSS          & $65.43$ & $63.15$ & $78.47$ & $25.89$ & $33.40$ & $37.35$ \\
    LESS              & $74.45$ & $65.28$ & $79.28$ & $29.93$ & $29.24$ & $38.00$ \\
    ScaleBiO          & $66.19$ & $75.74$ & $79.62$ & $30.58$ & $36.93$ & $37.98$ \\
    HG-baseline       & $73.77$ & $68.46$ & $80.14$ & $30.49$ & $34.27$ & $38.23$ \\
    HG-MS             & $\textbf{74.85}$ & $\mathbf{82.03}$ & $\mathbf{80.86}$ & $\mathbf{30.86}$ & $\mathbf{38.98}$ & $\mathbf{38.48}$ \\
    \bottomrule
  \end{tabular}
  \end{adjustbox}
\end{table*}
\Cref{tab:llm-realworld-math} reports final test accuracies in percent after the shared fine-tuning: GSM8K covers grade-school word problems, while MATH covers competition-style mathematical reasoning. For all iterative methods, the matched reweighting stage is approximately $9$h for \texttt{Qwen-2-7B}, $12$h for \texttt{Llama-3-8B}, and $15$h for \texttt{Gemma-2-9B}. Under these budgets, \textsc{HG-MS} improves over \textsc{HG-baseline} and is best in every reported GSM8K and MATH column. As supplementary diagnostics for \Cref{tab:llm-realworld-math}, Appendix~\ref{append:exp-llm-real} reports a \texttt{Qwen-2-7B} budget sweep in \Cref{tab:llm-budget-sensitivity-qwen} and a fixed-budget branch-count sensitivity study in \Cref{tab:llm-branch-ablation-qwen}. The budget sweep changes only the iterative reweighting time ($3$h, $9$h, $15$h, and $21$h), while the branch-count study keeps the nine-hour budget fixed and varies the number of candidate branches \(N\). Across these diagnostics, \textsc{HG-MS} learns source distributions distinct from \textsc{HG-baseline} and \textsc{ScaleBiO}; the main \(N=4\), 9h configuration gives the strongest completed GSM8K/MATH results.

\begin{table*}[t]
  \centering
  \caption{\textbf{Large-model math and instruction following.}
  Left: \texttt{Qwen2.5-32B} math results. Right: MT-Bench scores for the eighteen-source instruction portfolio.}
  \label{tab:llm-realworld-math-large}
  \label{tab:llm-realworld-instr-mtbench}
  \begin{minipage}[t]{0.34\textwidth}
  \centering
  \textbf{(a) \texttt{Qwen2.5-32B} math}\\[2pt]
  \begin{adjustbox}{max width=\linewidth}
  \begin{tabular}{lcc}
    \toprule
    Method & GSM8K & MATH \\
    \midrule
    Uniform Weighting     & 78.19 & 43.30 \\
    RHO-LOSS              & \texttt{OOM} & \texttt{OOM} \\
    LESS                  & \texttt{OOM} & \texttt{OOM} \\
    ScaleBiO              & 88.4 & 51.29 \\
	HG-baseline       	  & 86.81 & 46.82\\
    HG-MS                 & \textbf{92.58} & \textbf{53.32} \\
    \bottomrule
  \end{tabular}
  \end{adjustbox}
  \end{minipage}
  \hfill
  \begin{minipage}[t]{0.55\textwidth}
  \centering
  \textbf{(b) Instruction following}\\[2pt]
  \begin{adjustbox}{max width=\linewidth}
  \begin{tabular}{lccc}
    \toprule
    Method & Llama-3-8B & Qwen-2-7B & Gemma-2-9B \\
    \midrule
    Uniform Weighting & 7.07 & 6.54 & 8.04 \\
    RHO-LOSS          & 7.09 & 6.37 & 8.29 \\
    LESS              & 7.06 & 6.80 & 8.19 \\
    ScaleBiO          & \textbf{7.36} & 6.58 & 8.11 \\
    HG-baseline       & 7.32 & 6.58 & 8.09 \\
    HG-MS             & 7.33 & \textbf{6.92} & \textbf{8.30} \\
    \bottomrule
  \end{tabular}
  \end{adjustbox}
  \end{minipage}
\end{table*}

\textbf{Large-model results.}
\Cref{tab:llm-realworld-math-large} (left panel) reports source-reweighting results on \texttt{Qwen2.5-32B} under the same ten-source math portfolio. Each completed non-\texttt{OOM}\footnote{Here \texttt{OOM} denotes out-of-memory under the same large-model memory budget.
} row applies the corresponding method to the 32B model itself, with a fixed \(15\)h outer reweighting budget for the iterative methods. It then uses the resulting source mixture to form the full \(20\)K-example final fine-tuning set, and evaluates the final checkpoint with the same decoding pipeline on the full GSM8K/MATH test sets (\(1319\) and \(4989\) examples). In this large-model setting, \textsc{HG-MS} is strongest among the completed rows on both GSM8K and MATH, scoring \(92.58\%\) and \(53.32\%\), respectively; this improves over \textsc{ScaleBiO} by \(4.18\) points on GSM8K and \(2.03\) points on MATH.

\textbf{Instruction-following benchmark.}
Following Section~4.2.1 and Table~2 of \citet{pan2025scalebio}, this benchmark uses the released eighteen-source ScaleBiO portfolio, mixing general-chat data with multilingual Alpaca variants. Evaluation uses MT-Bench: $80$ multi-turn questions scored by a GPT-4o judge on a $0$--$10$ scale.
\Cref{tab:llm-realworld-instr-mtbench} (right panel) reports the completed MT-Bench comparison. \textsc{HG-MS} is best on \texttt{Qwen-2-7B} ($6.92$) and \texttt{Gemma-2-9B} ($8.30$), while \textsc{ScaleBiO} is best on \texttt{Llama-3-8B} ($7.36$). The Llama gap is small ($7.33$ for \textsc{HG-MS}), thus, the table provides matched-budget evidence that explicit minima selection remains competitive in instruction-following source reweighting.

\section{Limitations and Future Directions}\label{sec:limitations-future}
The main theoretical contribution of \Cref{sec:hyp_diff} is an exact-selection regularity result for optimistic bilevel
optimization with manifold lower-level solution sets. Under local \PLcirc, the lower-level minimizers form a structured
manifold and the Hessian of \(g\) has a normal spectral gap. We show that this non-singleton regime still admits a
differentiable optimistic hyper-objective when the optimistic selection is unique, and that intrinsic non-degeneracy of the
selected point yields local smoothness and the pseudoinverse hyper-gradient formula. These assumptions mark the clean
smooth regime for exact optimistic selection: uniqueness prevents branch switching, non-degeneracy stabilizes the
restricted minimizer on \(\calS(\theta)\), and local \PLcirc implies the geometry needed to differentiate through the
moving solution set. A natural theoretical direction is to develop analogous exact-selection results under other
geometric descriptions of the lower-level landscape, while preserving the same principle that regularity is governed by
the stability of the selected lower-level solution.

\Cref{sec:alg,sec:theory} turn this exact formula into a select-then-differentiate algorithm. The analyzed version of
\textsc{HG-MS} uses ULA/Gibbs candidates because this gives a mathematically controlled candidate distribution near
\(\calS(\theta)\). The resulting oracle-complexity bound is therefore best viewed as a worst-case existence guarantee for
this analyzed instantiation, not as a practical runtime model. Its exponent is conservative because the proof separately
uses low-temperature sampling, R\'enyi accuracy of finite-step ULA, finite-candidate selection, and the stabilized
pseudoinverse solve. The dependence on the intrinsic dimension \(\dimk\) in the selection term is unavoidable in the worst
case, since selecting the best point on a \(\dimk\)-dimensional nonconvex minimizer manifold can be reduced to a
\(\dimk\)-dimensional global optimization problem. The larger dimension-independent exponent comes from combining several
uniform worst-case controls in a single end-to-end guarantee. Improving this part of the rate would require new
techniques that preserve the present selection-error and hyper-gradient-stability mechanisms while coupling the sampling,
selection, and linear-solve accuracies more tightly.

The experiments use the same selection principle with practical finite-budget lower-level branches. In the LLM
source-reweighting tasks, each branch performs only one lower-level SGD update because the lower-level model is large and
source reweighting must be run under a fixed wall-clock budget; in the MNIST experiments, the branches use longer
AdamW/SGD lower-level runs. This practical version should not be understood as trying to sample the full Gibbs measure.
Rather, the branch distribution is induced by initialization, minibatch randomness, optimizer state, perturbations, and the
finite training budget. A promising next step is a candidate-generator-centered theory: if \(\nu_{\theta,B}\) denotes the
law of a \(B\)-step optimizer branch, one can study when \(\nu_{\theta,B}\) places enough mass on lower-level candidates
with strong upper-level validation performance. Then \BoN selection amplifies this validation-good mass across branches.
Such a theory would connect the exact optimistic hyper-gradient picture developed here with the finite-time optimizer
dynamics used in large-scale bilevel applications.

\bibliographystyle{plainnat}
\bibliography{main}

\input{Appendix_neurips}
\end{document}

%% file: defns.tex
\def\d{\delta}

\def\textiid{i.i.d.\@\xspace}
\newcommand\iid{\ifmmode\text{ i.i.d. } \else \textiid \fi}

\newcommand{\beqs}{\begin{equation*}}
\newcommand{\eeqs}{\end{equation*}}
\newcommand{\beq}{\begin{equation}}
\newcommand{\eeq}{\end{equation}}

%% file: Section_Convergence_HG_MS_Neurips.tex
\section{Convergence Analysis of HG-MS}\label{sec:theory}

\paragraph{Proof roadmap.}
This section analyzes the ULA/Gibbs instantiation of \textsc{HG-MS} in \Cref{alg:hg-minsel-lmc}.
The proof has three layers. First, \Cref{prop:outer-inexact} treats \textsc{HG-MS} as an inexact projected-gradient
method for the optimistic hyper-objective \(F\); all algorithmic approximation effects enter through the hyper-gradient
error \(e_t:=\widehat h_t-\nabla F(\theta_t)\). Second, \Cref{thm:et2-bound} is a stability estimate: it bounds
\(\mathbb E\|e_t\|^2\) by the selected-point error \(\|\tilde x_t-x^\star(\theta_t)\|\), the CG residual, the ridge bias,
and a clipped off-tube event. Third, \Cref{thm:selection-error-bound} controls the selected-point error for the
Best-of-\(N\) ULA/Gibbs candidate generator. Substituting these two estimates gives \Cref{thm:main-et2-bound}, and one
conservative parameter choice yields the oracle-complexity consequence in \Cref{corr:oracle-complexity}.

\paragraph{Additional regularity assumptions.}
The convergence proof uses several assumptions beyond the local differentiability theory in \Cref{sec:hyp_diff}; they
are collected formally in Appendix~\ref{append:regularity-assumption}. Their roles are the following.
\Cref{ass:curvature_g} gives uniform third-derivative control of \(g(\theta,\cdot)\) near \(\calS(\theta)\); this gives a
quantitative tube radius \(r(\gamma)\) on which the ridge-regularized Hessian is invertible and also controls the
curvature and volume terms used in the selection analysis. \Cref{assum_sample_comp} collects global smoothness,
polynomial-growth, and coercive growth conditions; these make the Gibbs measures well-defined, give tube-tail and
off-tube moment bounds for ULA/Gibbs candidates, and keep the hyper-gradient stability constants finite.
Finally, the uniform intrinsic non-degeneracy result in \Cref{prop:uniform-riem-nondeg} is used in the selection proof to
convert a small validation-value gap on \(\calS(\theta)\) into small distance to \(x^\star(\theta)\).

\textbf{Outer-loop guarantee.}
This section analyzes \textsc{HG-MS} as an inexact projected-gradient method for the hyper-objective \(F\).
Recall the hyper-gradient estimate $\widehat h_t$ in \eqref{eq:003}, and define the
hyper-gradient error and the projected gradient mapping as $e_t:=\widehat h_t-\nabla F(\theta_t)$ and $\mathcal G_{\Theta}(\theta,\nabla F(\theta);\alpha)
:=\alpha^{-1}\big(\theta-\mathrm{Proj}_{\Theta}(\theta-\alpha\nabla F(\theta))\big)$, respectively.
Under \Cref{ass:plcirc,ass:unique_min}, \Cref{prop:F-smooth} provides the required global smoothness of \(F\) on \(\Theta\).
The following inexact projected-gradient bound is standard when $F$ is uniformly smooth; we state it for completeness
\citep[e.g.,][]{ghadimi2016mini,nesterov2003introductory}.
\begin{proposition}[Convergence to near-stationarity with inexact hyper-gradients]\label{prop:outer-inexact}
Suppose that $F$ is $L_F$-smooth on $\Theta$.
Let $F_\star:=\inf_{\theta\in\Theta}F(\theta)>-\infty$.
Initialize $\theta_0\in\Theta$ and run the projected update
$\theta_{t+1}=\mathrm{Proj}_{\Theta}(\theta_t-\alpha \widehat h_t)$ with a constant
stepsize $\alpha\le 1/L_F$ (last line of \Cref{alg:hg-minsel-lmc}), then
\begin{align}
\forall\ T \geq 1, \quad \frac{1}{T}\sum_{t=0}^{T-1}\mathbb E\big[\|\mathcal G_{\Theta}(\theta_t,\nabla F(\theta_t);\alpha)\|^2\big]\le
\frac{8\,(F(\theta_0)-F_\star)}{\alpha T}+ \frac{10}{T}\sum_{t=0}^{T-1}\mathbb E\big[\|e_t\|^2\big].
\label{eq:outer-inexact}
\end{align}
\end{proposition}


\textbf{Hyper-gradient error bound.}
We now bound the error term \(\mathbb E\|e_t\|^2\).
Write \(\tilde x_t:=\hat{x}_N^\lambda(\theta_t)\) for the hard-selected lower-level point returned by
\Cref{alg:hg-minsel-lmc}. The first ingredient is a stability bound for the stabilized hyper-gradient computation. It
does not yet use the fact that \(\tilde x_t\) came from Best-of-\(N\) selection; it only says that if the selected point is
close to \(\calS(\theta_t)\), the ridge-regularized and CG-approximated hyper-gradient is close to the exact
pseudoinverse hyper-gradient.

\begin{lemma}[Second-moment hyper-gradient error bound]\label{thm:et2-bound}
Let \Cref{ass:plcirc,ass:curvature_g,ass:unique_min,assum_sample_comp} hold.
Fix \(\gamma>0\), choose a tube radius \(r(\gamma)=\mathcal{O}(\gamma)\) as in \Cref{lem:tube-invertibility}, and define
\(\mathsf{D}_{t}:=\mathrm{dist}(\tilde x_t,\calS(\theta_t))\) and
\(\mathcal E_t:=\{\mathsf{D}_{t}\le r(\gamma)\}\).
Choose the clipping radius in \Cref{alg:hg-minsel-lmc} so that
\(R_v=\mathcal{O}((L_{f,1}+\eta_t)/\gamma)\) as in \Cref{lem:off-tube-et}.
Then for all \(t\),
\begin{equation}\label{eq:et2-explicit}
\begin{aligned}
\mathbb E\|e_t\|^2
&\le
3C_x^2\Big(1+\tfrac{2}{\gamma}+\tfrac{2}{\gamma(c+\gamma)}\Big)^{\!2}
\mathbb E\big[\|\tilde x_t-x^\star(\theta_t)\|^2\mathbf 1_{\mathcal E_t}\big]\\
&\quad+\ \tfrac{12C_{\mathrm{lin}}^2}{\gamma^2}\,\eta_t^2
\ +\ 3C_{\mathrm{reg}}^2\,\gamma^2+\ C_{\mathrm{off}}(1+R_v^2)\,
\mathbb E\big[(1+\mathsf{D}_{t}^{2\mathsf n_{\theta x}})\mathbf 1_{\mathcal E_t^c}\big],
\end{aligned}
\end{equation}
where \(C_x,C_{\mathrm{lin}},C_{\mathrm{reg}}<\infty\) are stability constants depending on the global smoothness bounds
and the normal spectral gap of \(\nabla^2_{xx}g\) on \(\Theta\), \(c>0\) is the uniform normal spectral gap from
\Cref{prop:geometry}, and \(C_{\mathrm{off}}<\infty\) depends on
\(C_{\theta x},D_{\theta x},\mathsf n_{\theta x}\) from \Cref{assum_sample_comp}.
\end{lemma}

The last term in \eqref{eq:et2-explicit} is the off-tube contribution: it is active only when the selected candidate lies
farther than \(r(\gamma)\) from \(\calS(\theta_t)\). For ULA/Gibbs candidates, the tube-concentration estimates in
Appendix~\ref{append:selection-error} show that this term is exponentially small in \(r(\gamma)^2/\lambda\). 

It remains to control the selected-point error
\(\mathbb E\|\tilde x_t-x^\star(\theta_t)\|^2\) in the first term of \eqref{eq:et2-explicit}. This is the role of the
next lemma: small \(\lambda\) keeps ULA/Gibbs candidates close to \(\calS(\theta)\), Best-of-\(N\) selection chooses the
candidate with the smallest upper-level value, and the R\'enyi tolerance \(\varepsilon_{\mathrm R}\) measures the error
from using finite-step ULA rather than exact Gibbs samples.

\begin{lemma}[Expected squared selection error for parallel ULA candidates]\label{thm:selection-error-bound}
Fix \(\theta\in\Theta\) and let \(x^\star(\theta)\) denote the unique minimizer of \(f(\theta,\cdot)\) over
\(\calS(\theta)\). Assume \Cref{ass:plcirc,ass:curvature_g,assum_sample_comp,ass:unique_min} hold, \(\dimk\ge1\), and
\(0<\lambda\le\lambda_0\). Let \(X_1,\dots,X_N\) be the outputs of \(N\) independent finite-step ULA chains targeting
\(\gibbs_\theta^\lambda\), let \(\nu_i\) be the marginal law of \(X_i\), and write
\(\varepsilon_{\mathrm R}^2:=\max_{1\le i\le N}R_2(\nu_i\|\gibbs_\theta^\lambda)\). Assume
\(\varepsilon_{\mathrm R}\le1\), and let \(\tilde x\in\arg\min_{1\le i\le N}f(\theta,X_i)\) be the hard-selection output,
with a measurable tie-breaking rule.\footnote{For example, one may choose the smallest index among all candidates attaining
the minimum value of \(f(\theta,X_i)\), and set \(\tilde x\) to that candidate.} Then
\begin{align}
\mathbb E\|\tilde x-x^\star(\theta)\|^2
\ \le\ &
2e^{1/2}C_{\mathrm{tube},2}\,\lambda\log(1+N)\nonumber\\
&\quad
+\ 2\Bigg(\frac{1}{c_{\mathrm{hg}}}+\frac{4\mathsf D^2}{\Delta_{r_0}}\Bigg)
\Bigg(
2L_{f,1}\,C_{\mathrm{tube}}\,\sqrt{\lambda\log(1+N)}
+\ C_1\,N^{-1/\dimk}
\nonumber\\
&\qquad\qquad\qquad\qquad
+\ 4\sqrt{e-1}\,L_{f,1}\sqrt{N\,C_{\mathrm{PI}}}\,\varepsilon_{\mathrm R}
\Bigg).
\label{eq:final-Ex2-bound}
\end{align}
The constants \(L_{f,1}\) and \(\mathsf D\) are from \Cref{assum_sample_comp};
\(C_{\mathrm{PI}}\) is the low-temperature Poincar\'e constant from the discussion in Appendix~\ref{append:lmc-ula};
\(C_{\mathrm{tube}}\), \(C_{\mathrm{tube},2}\), and \(\lambda_0\) are from
\Cref{lem:gibbs-tube-width,lem:tube-width}; \(C_1\) is from \Cref{lem:inner-quality-gap};
\(c_{\mathrm{hg}}\) and \(r_0\) are from \Cref{lem:uniform-hg}; and \(\Delta_{r_0}>0\) is from
\Cref{lem:away-gap-positive}.
\end{lemma}

The proofs of \Cref{thm:et2-bound,thm:selection-error-bound} are given in
Appendix~\ref{append:hg-stability} and Appendix~\ref{append:selection-error}, respectively. Combining them gives the
following parameter-level bound.

\begin{theorem}[Hyper-gradient error bound]\label{thm:main-et2-bound}
Let \Cref{ass:plcirc,ass:curvature_g,assum_sample_comp,ass:unique_min} hold. Suppose \(\dimk\ge1\), \(0<\lambda\le\lambda_0\), \(\lambda\log(1+N)\le1\), \(0<\gamma\le1\), and \(\varepsilon_{\mathrm R}\le1\), where
\(\lambda_0\) is the low-temperature threshold from \Cref{thm:selection-error-bound}.
Choose the clipping radius \(R_v\) so that \(R_v\ge (2/\gamma)(L_{f,1}+\eta_t)\). We have, for all \(t\),
\begin{align}
\mathbb E\|e_t\|^2
=
\mathcal{O}\!\Bigg(
&\frac{\sqrt{\lambda\log(1+N)}+N^{-1/\dimk}+\sqrt{N\,C_{\mathrm{PI}}}\,\varepsilon_{\mathrm R}+\eta_t^2}{\gamma^2}
\ +\ \gamma^2\nonumber\\
&\quad
\ +\ (1+R_v^2)N\big(1+r(\gamma)^{2\mathsf n_{\theta x}}+\lambda^{\mathsf n_{\theta x}}\big)
\exp\!\Big(-\tfrac{c_{\mathrm{tube}}}{4}\tfrac{r(\gamma)^2}{\lambda}\Big)
\Bigg).
\label{eq:main-et2-explicit-arxiv}
\end{align}
Here \(N\)  is the number of parallel ULA candidates in \Cref{alg:hg-minsel-lmc}, and \(\varepsilon_{\mathrm R}\) denotes a uniform square-root R\'enyi tolerance:
\(\max_i R_2(\nu_{t,i}\|\gibbs_{\theta_t}^\lambda)\le \varepsilon_{\mathrm R}^2\), where \(\nu_{t,i}\) is the conditional law of the \(i\)-th candidate.
The hidden constant in the full bound depends only on problem-data constants, not on
\((N,\lambda,\varepsilon_{\mathrm R},\gamma,\eta_t,R_v)\) or \(t\). The proof appears in Appendix~\ref{append:param-tuning}.
\end{theorem}

Each term in \Cref{thm:main-et2-bound} has a distinct origin. The term \(N^{-1/\dimk}\) is the finite-candidate
optimistic-selection error along the \(\dimk\)-dimensional lower-level minimizer manifold. The terms
\(\sqrt{\lambda\log(1+N)}\) and \(\lambda\log(1+N)\) come from low-temperature Gibbs concentration near
\(\calS(\theta)\). The factor \(\sqrt{N\,C_{\mathrm{PI}}}\,\varepsilon_{\mathrm R}\) is the price of using finite-step ULA
candidates instead of exact Gibbs samples, measured by the square-root R\'enyi tolerance \(\varepsilon_{\mathrm R}\).
The \(\eta_t^2\) term is the CG residual, while \(\gamma^2\) is the bias from replacing the pseudoinverse by a
ridge-regularized inverse. The final exponential term comes from bounding the last term in \eqref{eq:et2-explicit}: the
event that the selected candidate lies outside the \(r(\gamma)\)-tube where the ridge stability argument applies.

Finally, we translate this error bound into one explicit oracle-complexity consequence. This is a worst-case
scaling for the analyzed ULA/Gibbs instantiation. We count evaluations of \(\nabla_x g\) used by ULA, evaluations of \(f\) used for hard selection,
and the Hessian--vector products used by CG; the number of ULA iterations is chosen so that the square-root R\'enyi
tolerance \(\varepsilon_{\mathrm R}\) is sufficiently small. 

\begin{theorem}[Total oracle complexity (informal)]\label{corr:oracle-complexity}
Let \Cref{ass:plcirc,ass:curvature_g,assum_sample_comp,ass:unique_min} hold and run \textsc{HG-MS} with $\alpha=1/L_F$.
Suppose \(\dimk\ge 1\).
Initialize every ULA chain from \(\lambda\)-scaled Gaussians (see \Cref{lem:gaussian-init-r3}).
For all sufficiently small \(\varepsilon\in(0,1)\), with the same parameter choices as in  Appendix~\ref{append:param-tuning} ensuring $\sup_t\mathbb E\|e_t\|^2=\mathcal{O}(\varepsilon^2)$,
\Cref{thm:main-et2-bound,prop:outer-inexact} yield an $\varepsilon$-stationarity guarantee after
$T=\mathcal{O}\big(L_F(F(\theta_0)-F_\star)/\varepsilon^2\big)$ outer iterations.
Moreover, under the explicit scaling $N=\lceil\varepsilon^{-4\dimk}\rceil$, $\lambda=\varepsilon^{8}/\log(1+N)$,
the square-root R\'enyi tolerance $\varepsilon_{\mathrm{R}}=\varepsilon^{4+2\dimk}$, \(\gamma=\varepsilon\), \(\eta_t=\varepsilon^2\), \(R_v=\Theta(\varepsilon^{-1})\), and
$K=\widetilde{\mathcal{O}}(\lambda^{-4}\varepsilon_{\mathrm{R}}^{-2})$,
each outer iteration uses $NK$ evaluations of $\nabla_x g$ (ULA), $N$ evaluations of $f$ (hard selection), and \(\#\mathrm{HVP}\) Hessian--vector products for CG.
Hence the total oracle complexity for achieving an $\varepsilon$-stationary point is
\(
T(NK+N+\#\mathrm{HVP})
=
\widetilde{\mathcal{O}}\Big(\varepsilon^{-42-8\dimk}\Big).
\)
\end{theorem}

\begin{remark}[{On the \texorpdfstring{$\dimk$}{k}-dependence}]\label{rem:k-dependence}
The intrinsic dimension enters through the minima-selection subproblem, most visibly via the \(N^{-1/\dimk}\) approximation term in \Cref{thm:main-et2-bound}.
This matches the worst-case intuition that optimizing over a \(\dimk\)-dimensional nonconvex minimizer manifold can reduce to \(\dimk\)-dimensional global optimization and exhibit \(\Omega(\varepsilon^{-\dimk})\)-type hardness \citep[Rem.~1.1]{masiha2025superquantile}.
\end{remark}

%% file: Appendix_neurips.tex
\clearpage
\onecolumn
\appendix
\begingroup
\makeatletter
\def\@endpart{\par\bigskip} 
\part{Appendix}
\endgroup
\parttoc

\clearpage

\section{Technical Preliminaries and Regularity Assumptions}\label{append:theory} 

\subsection{Appendix notation guide}\label{append:notation-guide}
This subsection supplements the notation paragraph in the introduction. It records mathematical notation used repeatedly
in the appendix proofs.

\begin{description}[leftmargin=*,style=nextline]
\item[Distance, balls, and spheres.]
For a nonempty set \(A\subseteq\R^d\), we write
\[
\mathrm{dist}(x,A):=\inf_{y\in A}\|x-y\|.
\]
We denote the closure and interior of \(A\) by \(\overline A\) and \(\mathrm{int}(A)\), respectively.
For \(a\in\R^d\) and \(r\ge0\), \(\mathbb B_d(a;r):=\{x\in\R^d:\|x-a\|\le r\}\) denotes the closed Euclidean ball.
For \(q\ge1\), \(\mathbb S^{q-1}:=\{u\in\R^q:\|u\|=1\}\) denotes the unit sphere in \(\R^q\).

\item[Manifold volume and geodesic distance.]
If \(\mathcal M\subset\R^d\) is an \(r\)-dimensional embedded submanifold, then
\(\ud\mathrm{Vol}_{\mathcal M}(y)\) denotes the Riemannian volume element induced by the ambient Euclidean metric, and
\[
\mathrm{Vol}_{r}(\mathcal M):=\int_{\mathcal M}\ud\mathrm{Vol}_{\mathcal M}(y)
\]
denotes its total \(r\)-dimensional volume. We write \(d_{\mathcal M}(x,y)\) for the geodesic distance between
\(x,y\in\mathcal M\).

\item[Riemannian second-order notation.]
For a smooth scalar function \(h\) on \(\mathcal M\), \(\mathrm{Hess}_{\mathcal M} h(x)\) denotes the Riemannian Hessian
of \(h\) at \(x\). When \(\mathcal M=\calS(\theta)\), \(\mathrm{II}_{x}^{\theta}\) denotes the second fundamental form of
\(\calS(\theta)\) at \(x\).

\item[Lower-level objective shorthand.]
We write
\[
g_\star(\theta):=\min_{z\in\R^d}g(\theta,z).
\]
The common intrinsic dimension of \(\calS(\theta)\) is denoted by \(\dimk\). In normal-coordinate arguments, we write
\(q:=d-\dimk\) for the normal dimension.

\item[Linear algebra.]
For a matrix \(A\), \(\ker(A)\) denotes its null space and \(A^\dagger\) denotes its Moore--Penrose pseudoinverse. If
\(A\) is symmetric, \(\lambda_{\min}(A)\) denotes its smallest eigenvalue. We write \(\|A\|_{\mathrm{op}}\) for the
operator norm.

\item[Asymptotic notation.]
We use \(\mathcal{O}(\cdot)\) for upper bounds up to problem-dependent constants that are independent of the displayed
scaling parameters. We use \(\widetilde{\mathcal{O}}(\cdot)\) to additionally suppress logarithmic factors.

\item[Probability measures.]
For \(x\in\R^d\), \(\delta_x\) denotes the Dirac measure at \(x\). Given candidates \(X_1,\ldots,X_M\), their empirical
measure is
\[
\widehat\nu_M:=\frac{1}{M}\sum_{i=1}^M\delta_{X_i}.
\]
The Gibbs measure is
\[
\gibbs_\theta^\lambda(\d x)\propto \exp\{-g(\theta,x)/\lambda\}\,\d x.
\]
The marginal laws of approximate ULA candidates are denoted by \(\nu_i\).

\item[Lower-tail notation.]
For a real random variable \(Z\), its lower \(u\)-quantile is
\[
q_u(Z):=\inf\{r\in\R:\mathbb P(Z\le r)\ge u\}.
\]
The lower-tail superquantile \(\mathrm{SQ}^{\mathrm{low}}_\delta(Z)\) is defined in \Cref{def:lower-sq}.

\item[Selected candidates and tube events.]
For candidates \(X_1,\ldots,X_M\), the hard-selected point is
\[
\tilde x\in\arg\min_{1\le i\le M} f(\theta,X_i).
\]
We write \(\bar x\in\arg\min_{x\in\calS(\theta)}\|x-\tilde x\|\) for a Euclidean projection of \(\tilde x\) onto
\(\calS(\theta)\). When each candidate is projected separately, \(\Pi(X_i)\) denotes a measurable choice from
\(\arg\min_{x\in\calS(\theta)}\|x-X_i\|\). At outer iteration \(t\),
\[
\tilde x_t:=\hat{x}_N^\lambda(\theta_t),
\qquad
\mathsf{D}_{t}:=\mathrm{dist}(\tilde x_t,\calS(\theta_t)),
\qquad
\mathcal E_t:=\{\mathsf{D}_{t}\le r(\gamma)\}.
\]
We write \(\mathbf 1_A\) for the indicator of an event or set \(A\).
\end{description}

\subsection{ULA sampling and Poincar\'e preliminaries}\label{append:lmc-ula}
For completeness, we record the standard unadjusted Langevin algorithm (ULA) used as the candidate generator in the
ULA/Gibbs instantiation of \textsc{HG-MS}; see, e.g., \citet{chewi2024analysis}.

\begin{algorithm}[h]
\caption{\textsc{ULA}$(\theta,g,x^{(0)},\lambda,K,h)$}
\label{alg:lmc-ula}
\begin{algorithmic}[1]
\Require $\theta$; lower objective $g(\theta,\cdot)$; initial state $x^{(0)}$;
temperature $\lambda$; steps $K$; stepsize $h$.
\State $x\leftarrow x^{(0)}$.
\For{$s=0$ to $K-1$}
    \State Draw $\xi^{(s)}\sim\mathcal N(0,I_d)$.
    \State $x\leftarrow x-h\nabla_x g(\theta,x)+\sqrt{2\lambda h}\,\xi^{(s)}$.
\EndFor
\State \textbf{return} $x$.
\end{algorithmic}
\end{algorithm}

\paragraph{Poincar\'e constant.}
{
A probability measure \(\mu\) over \(\R^d\) is said to satisfy the Poincar\'e inequality with constant \(C_{\mathrm{PI}}\) if, for every test function \(\varphi \in H^1(\mu)\),
\begin{equation*}
    \int \Bigl(\varphi-\int \varphi \,\d \mu\Bigr)^2 \d \mu \;\le\; C_{\mathrm{PI}} \int \|\nabla \varphi\|^2 \,\d \mu,
\end{equation*}
where \(H^1(\mu)\) denotes the Sobolev space weighted by \(\mu\). Importantly, for Gibbs measures of the form \(\gibbs_\theta^\lambda(\d x)\propto \exp(-g(\theta,x)/\lambda)\,\d x\), \citet{gong2024poincare} show that under \(\PLcirc\) the corresponding Poincar\'e constant admits a bound independent of \(\lambda\) for all sufficiently small \(\lambda\le \lambda_0\). Accordingly, when \(C_{\mathrm{PI}}\) appears later in our sampling and complexity bounds, it refers to this low-temperature uniform bound inherited from \(\PLcirc\).}

\subsection{Regularity assumptions for sampling and selection}\label{append:regularity-assumption}
The assumptions in this subsection are used for the finite-time convergence and complexity analysis in
\Cref{sec:theory}; they are not needed for the local differentiability statements in \Cref{sec:hyp_diff}.
Their roles are as follows.
\Cref{ass:curvature_g} gives uniform third-derivative control of \(g(\theta,\cdot)\) in a fixed tube around
\(\calS(\theta)\). In \Cref{ass:curvature_g}, \(\rho\) denotes the radius of this fixed tube and \(L_{g,3}\) denotes the
uniform third-derivative bound in the tube. This is used in two places: first, to obtain the explicit tube radius
\[
r(\gamma)=\min\left\{\rho,\frac{\gamma}{2\max\{L_{g,3},1\}}\right\}
\]
for ridge invertibility in \Cref{lem:tube-invertibility}; and second, through the second-fundamental-form bound in
\Cref{prop:second-fund-from-g}, to justify the curvature/volume control used in the nearest-neighbor part of the
selection analysis.
\Cref{assum_sample_comp} collects the global analytic controls needed because the ULA and Gibbs candidates are not
compactly supported: smoothness and Lipschitz bounds make the hyper-gradient stability constants finite; polynomial
growth of the derivative quantities \(\partial_\theta g\) and \(\nabla^2_{\theta x}g\) gives finite moments and controls
the off-tube term in \Cref{lem:off-tube-et}; and the coercive quadratic-growth condition gives well-defined
low-temperature Gibbs measures and Gaussian-type tube tails in \Cref{lem:gibbs-tube-width}. The global Lipschitz
condition on \(f\) already implies the linear growth of \(f\) needed for tail estimates, so no separate polynomial-growth
assumption on \(f\) is imposed. Finally, the global intrinsic non-degeneracy in
\Cref{ass:unique_min} yields the uniform restricted-curvature constant \(m_{\mathrm R}\) in
\Cref{prop:uniform-riem-nondeg}. This constant is used only in the selection-error proof:
\Cref{lem:hg-from-nondeg,lem:uniform-hg,lem:barx-distance} convert a small validation-value gap on the manifold into a
small distance to \(x^\star(\theta)\).

Our sharp selection-rate analysis also uses volume comparison bounds for small geodesic balls on the compact manifolds
\(\calS(\theta)\), which require the curvature control supplied by \Cref{ass:curvature_g}.

\begin{tcolorbox}[colback=gray!20, colframe=gray!50, boxrule=0pt, arc=0mm, left=0mm, right=0mm, top=0mm, bottom=0mm]
\begin{assumption}[regularity of $g$]\label{ass:curvature_g}
There exist constants \(\rho>0\) and \(L_{g,3}<\infty\) such that for all \(\theta\in\Theta\),
\(g(\theta,\cdot)\) is \(\mathcal{C}^{3}\) on the tube \(\{x:\mathrm{dist}(x,\calS(\theta))\le \rho\}\) and
\[
\sup_{\mathrm{dist}(x,\calS(\theta))\le \rho}\ \big\|\nabla^3_{xxx}g(\theta,x)\big\|_{\mathrm{op}}\ \le\ L_{g,3}.
\]
\end{assumption}
\end{tcolorbox}

\begin{tcolorbox}[colback=gray!20, colframe=gray!50, boxrule=0pt, arc=0mm, left=0mm, right=0mm, top=0mm, bottom=0mm, breakable]
\begin{assumption}\label{assum_sample_comp}
The following statements hold for all \(\theta\in\Theta\).
\begin{itemize}[leftmargin=*]\setlength{\itemsep}{0.5pt}
\item \textbf{Regularity of \(f\).}
\(f\) is \(\mathcal{C}^{2}\) with respect to \((\theta,x)\). Moreover, \(f\) is \(L_{f,1}\)-Lipschitz and \(L_{f,2}\)-smooth with respect to \((\theta,x)\).
Since \(\Theta\) is compact, this global Lipschitz condition implies the derived linear-growth bound
\(|f(\theta,x)|\le C(1+\|x\|)\), uniformly over \(\theta\in\Theta\).

\item \textbf{Regularity of \(g\).}
\(g\) is \(\mathcal{C}^{3}\) with respect to \((\theta,x)\).
For every \(\theta\), the map \(x\mapsto g(\theta,x)\) is \(\mathcal{C}^{2}\) and \(L_{g,2}\)-smooth (i.e., \(\nabla_x g(\theta,\cdot)\) is \(L_{g,2}\)-Lipschitz).
Moreover, for all \(x\in\R^{d}\),
\[
\|\partial_{\theta}g(\theta,x)\| \;\le\; C_{g}\|x\|^{\mathsf{n}_{g}} + D_{g},
\]
for constants \(C_{g},D_{g}\ge 0\) and an integer \(\mathsf{n}_{g}\ge 1\).
We also assume polynomial growth of the mixed derivative used in the hyper-gradient:
\[
\|\nabla^2_{\theta x}g(\theta,x)\| \;\le\; C_{\theta x}\|x\|^{\mathsf{n}_{\theta x}} + D_{\theta x},
\]
for constants \(C_{\theta x},D_{\theta x}\ge 0\) and an integer \(\mathsf{n}_{\theta x}\ge 0\).
The Hessian \(\partial_x^2 g(\theta,x)\) is continuous with respect to \((\theta,x)\).
Moreover, there exist constants \(\mathsf{D}>0\) and \(\mu_{\mathsf{qg}}>0\) such that for all \(\|x\|\ge \mathsf{D}\),
\[
\frac{\mu_{\mathsf{qg}}}{2}\,\mathrm{dist}^{2}\!\bigl(x,\calS(\theta)\bigr)
\;\le\;
g(\theta,x)-\min_{z\in\R^{d}} g(\theta,z).
\]
W.l.o.g.\ we take \(\mathsf{D}\) large enough so that \(\calS(\theta)\subseteq \mathbb B_d(0;\mathsf{D})\) for all \(\theta\in\Theta\).
\end{itemize}
\end{assumption}
\end{tcolorbox}

\begin{proposition}[Uniform restricted curvature from global non-degeneracy]\label{prop:uniform-riem-nondeg}
Let \(f\in\mathcal{C}^{2}\), \(g\in\mathcal{C}^{3}\), and suppose \Cref{ass:plcirc,ass:unique_min} hold.
For every \(\theta\in\Theta\), let
\[
\bar f_\theta:\calS(\theta)\to\R,
\qquad
\bar f_\theta(x):=f(\theta,x),
\]
denote the restriction of the upper-level objective to the lower-level solution manifold.
There exists a constant \(m_{\mathrm R}>0\), independent of \(\theta\), such that
\[
\big\langle v,\mathrm{Hess}_{\calS(\theta)}\bar f_\theta(x^\star(\theta))[v]\big\rangle
\ \ge\
m_{\mathrm R}\|v\|^2,
\qquad
\forall \theta\in\Theta,\ \forall v\in\calT_{x^\star(\theta)}^\theta .
\]
\end{proposition}
\begin{proof}
If \(\dimk=0\), the tangent spaces are trivial and the claim is vacuous; take any \(m_{\mathrm R}>0\).
Assume \(\dimk\ge1\). For each \(\bar\theta\in\Theta\), \Cref{thm_hyp_obj_diff} and \Cref{ass:unique_min} give a
neighborhood \(U_{\bar\theta}\), a \(\mathcal{C}^{1}\) selected branch \(x^\star(\theta)\), and a \(\mathcal{C}^{2}\) local chart
\(\psi_{\bar\theta}(\theta,u)\) for \(\calS(\theta)\) near \(x^\star(\theta)\). Let \(u^\star(\theta)\) be the local
coordinate satisfying \(x^\star(\theta)=\psi_{\bar\theta}(\theta,u^\star(\theta))\), and write
\[
\tilde f_{\bar\theta}(\theta,u):=f(\theta,\psi_{\bar\theta}(\theta,u)),
\qquad
A_{\bar\theta}(\theta):=D_u\psi_{\bar\theta}(\theta,u^\star(\theta)).
\]
For a tangent vector \(v=A_{\bar\theta}(\theta)w\), the second variation of the pullback equals the Riemannian Hessian
quadratic form, while \(\|v\|^2=w^\top A_{\bar\theta}(\theta)^\top A_{\bar\theta}(\theta)w\). Hence the local restricted
curvature can be written as the smallest generalized eigenvalue
\[
\mu_{\bar\theta}(\theta):=
\min_{w\neq0}
\frac{w^\top\nabla^2_{uu}\tilde f_{\bar\theta}(\theta,u^\star(\theta))w}
     {w^\top A_{\bar\theta}(\theta)^\top A_{\bar\theta}(\theta)w}.
\]
This quantity is continuous on \(U_{\bar\theta}\) and is positive at every point by the intrinsic non-degeneracy in
\Cref{ass:unique_min}. Shrinking \(U_{\bar\theta}\) if necessary, there is a smaller neighborhood
\(V_{\bar\theta}\) with \(\overline{V_{\bar\theta}}\subset U_{\bar\theta}\) on which
\(\inf_{\theta\in V_{\bar\theta}}\mu_{\bar\theta}(\theta)>0\).
Since \(\Theta\) is compact, choose a finite subcover \(V_{\bar\theta_1},\ldots,V_{\bar\theta_J}\) and set
\[
m_{\mathrm R}:=\min_{1\le j\le J}\inf_{\theta\in V_{\bar\theta_j}}\mu_{\bar\theta_j}(\theta)>0.
\]
This gives the stated uniform lower bound.
\end{proof}

\Cref{prop:uniform-riem-nondeg} is the uniform version of the local intrinsic non-degeneracy condition in
\Cref{def:degenerate-optimistic}. The Hessian of the restricted objective includes the curvature of the constraint
manifold, as shown by the identity below.
With the convention for the second fundamental form used in \Cref{prop:second-fund-from-g}, for \(v\in\calT_{x^\star(\theta)}^\theta\),
\begin{equation}\label{eq:riem-hess-restricted}
\big\langle v,\mathrm{Hess}_{\calS(\theta)}\bar f_\theta(x^\star(\theta))[v]\big\rangle
=
\langle v,\nabla^2_{xx}f(\theta,x^\star(\theta))v\rangle
+
\langle \nabla_x f(\theta,x^\star(\theta)),
\mathrm{II}_{x^\star(\theta)}^\theta(v,v)\rangle .
\end{equation}
Here \(\mathrm{II}_{x^\star(\theta)}^\theta\) denotes the second fundamental form of the embedded manifold \(\calS(\theta)\) at \(x^\star(\theta)\).
Thus the conclusion of \Cref{prop:uniform-riem-nondeg} holds, for instance, if the ambient tangent Hessian has a uniform
positive margin and the second-fundamental-form correction is not too negative.
Concretely, it is sufficient that for some \(m_{\mathrm{amb}}>0\) and \(\kappa<m_{\mathrm{amb}}\),
\[
\langle v,\nabla^2_{xx}f(\theta,x^\star(\theta))v\rangle
\ge m_{\mathrm{amb}}\|v\|^2,
\qquad
\langle \nabla_x f(\theta,x^\star(\theta)),
\mathrm{II}_{x^\star(\theta)}^\theta(v,v)\rangle
\ge -\kappa\|v\|^2
\]
uniformly over \(\theta\in\Theta\) and \(v\in\calT_{x^\star(\theta)}^\theta\).
Using \Cref{prop:second-fund-from-g}, a further sufficient condition is
\[
m_{\mathrm{amb}}
>
\sup_{\theta\in\Theta}\|\nabla_x f(\theta,x^\star(\theta))\|\frac{L_{g,3}}{c},
\]
where \(c\) is the uniform normal spectral gap in \Cref{prop:geometry}.

We close this subsection by comparing these assumptions with common assumptions in related analyses.
The smoothness and Lipschitz-type parts of \Cref{assum_sample_comp} are standard in hyper-gradient analyses of bilevel
optimization: they play the same role as the uniform derivative bounds used to control implicit-gradient stability and
inexact lower-level solves \citep{kwon2023penalty,chen2023bilevel}. Many bilevel convergence results additionally avoid
tail issues by imposing a bounded lower-level domain, bounded gradients, or bounded iterates. Our sampling-based method
does not have compactly supported candidates, because the Gibbs and ULA laws live on \(\R^d\). We therefore replace those
bounded-domain requirements with derivative polynomial-growth bounds, the linear-growth consequence of Lipschitzness of
\(f\), and coercive quadratic growth away from \(\calS(\theta)\). This is the standard type of tail-control condition
used in low-temperature Gibbs/Langevin analyses around minimizer sets
\citep{hasenpflug2024wasserstein,chewi2024analysis}. It is weaker than assuming all relevant quantities are bounded on
\(\R^d\), but strong enough to make the Gibbs normalizing constants finite, to prove tube concentration, and to control
the off-tube contribution in the hyper-gradient error.

The manifold-specific assumptions have a different purpose. \Cref{ass:curvature_g} is a quantitative version of local
regularity near the minimizer manifold; it is analogous to the Hessian-Lipschitz or third-order smoothness assumptions
often used when one needs stable local quadratic models, but here it is imposed only in a tube around \(\calS(\theta)\).
Together with the normal spectral gap from \(\PLcirc\), it provides the uniform curvature bounds required for the
manifold-volume and nearest-neighbor estimates used in the selection analysis; related manifold Gibbs estimates appear in
\citet{masiha2025superquantile}. Finally, \Cref{prop:uniform-riem-nondeg} is the manifold analogue of a uniform
second-order sufficient condition for the upper-level restricted problem \(f(\theta,\cdot)|_{\calS(\theta)}\). This
conclusion is vacuous in the classical singleton-minimizer regime, but in the non-singleton manifold regime it is what
converts validation-value suboptimality into distance to the unique optimistic selection. Thus the additional assumptions
in this subsection are not substitutes for the local differentiability theory in \Cref{sec:hyp_diff}; they are uniform,
quantitative conditions needed for the finite-sample \BoN selection rate and the resulting oracle-complexity bound.

\section{Related Work}\label{sec:related_work}
Bilevel optimization has a long history in mathematical programming; see, e.g.,
\citet{vicente1994bilevel,dempe2002foundations,colson2007overview,luo1996mathematical,dempe2020bilevel}.
Classical optimistic and pessimistic formulations already recognize that nonunique lower-level solutions create a genuine
\emph{selection} issue \citep{dempe2007new,ye1997exact,ye1997necessary}. Standard single-level reformulations, whether via
value functions or KKT/MPEC systems \citep{ye1995optimality,ye2010new,dempe2012kkt,zemkoho2021comparison}, provide an
important optimization viewpoint. Our focus here is a different question: when does the induced hyper-objective admit a
usable hyper-gradient?

\paragraph{Classical hyper-gradient methods.}
Most classical hyper-gradient methods study the regime in which the lower-level minimizer is a singleton, so the
lower-level response can be treated as a single-valued and differentiable map of \(\theta\), typically because their
assumptions (such as strong convexity) guarantee that the relevant lower-level Hessian is invertible. In this regime, the
hyper-objective \(F(\theta)=f(\theta,x(\theta))\) is differentiated exactly by implicit differentiation, while truncated
backpropagation is used to approximate its hyper-gradient by differentiating through finitely many lower-level iterations
\citep{pedregosa2016hypergrad,franceschi2017forward,lorraine2020optimizing,shaban2019truncated}. Closely
related methodological work studies how to differentiate through parameterized optimization problems more generally.
\citet{ablin2020super} analyze gradients of functions defined as minima and compare analytic implicit gradients with
automatic differentiation through iterative solvers. \citet{arbel2021amortized} develop amortized inexact implicit
differentiation for stochastic bilevel optimization with strongly convex inner problems. \citet{blondel2022efficient} develop a general framework for implicit differentiation of parameterized optimization problems: given an optimality system defining a single-valued solution map, they differentiate the solution by applying the implicit function theorem to that system. Their focus is modular implicit-differentiation machinery rather than minima selection from a non-singleton lower-level solution set.
These works are closely related in technique, but they still largely assume a well-defined single-valued solution map, and therefore do not address the minima-selection issue that arises when the lower-level minimizer set is non-singleton.

\paragraph{Degenerate implicit differentiation and pseudoinverse formulas.}
A separate line studies bilevel differentiation when the lower-level Hessian is singular. Most relevant for us,
\citet{arbel2022non} study bilevel games in which the lower-level problem may have multiple critical points and a
selection map chooses the particular critical point reached by a prescribed lower-level process (e.g., gradient flow).
They show that, under a parametric Morse--Bott condition, this algorithmic critical-point selection map remains
differentiable near local minima even when the lower Hessian is singular, and its derivative is characterized by a
pseudoinverse-based implicit-differentiation formula.
In a different overparameterized setting, \citet{vicol2021implicit,vicol2022implicitbias} consider lower-level objectives
with nonunique minimizers and study the minimum-norm lower-level solution selected by the optimization dynamics. Because
this selected solution is defined through a singular linear system, its response Jacobian is characterized by the
minimum-norm solution of that system, which is expressed using the Moore--Penrose pseudoinverse.

The pseudoinverse formulas in these works are technically related to ours, but the object being differentiated is
different. \citet{arbel2022non} differentiate a prescribed algorithmic critical-point selection map, and the
overparameterized line \citep{vicol2021implicit,vicol2022implicitbias} differentiates a minimum-norm or algorithm-induced
lower-level response. In contrast, we differentiate the exact optimistic value
\[
F(\theta)=\min_{x\in\calS(\theta)} f(\theta,x),
\]
where \(\calS(\theta)\) is the manifold of global lower-level minimizers and the selected point is determined by the
upper-level objective itself. Thus our contribution is to show that, under local \PLcirc\ and local uniqueness of the
optimistic minimizer, this hard optimistic value function is differentiable and its exact hyper-gradient admits a
pseudoinverse formula.

\paragraph{Alternative formulations, selection-based models, and nonsmooth bilevel objectives.}
Several recent works address nonconvex or set-valued lower levels by modifying either the object being optimized or the
solution concept. \citet{liu2021towards} propose a trajectory-based approximation framework for nonconvex followers,
introducing an initialization auxiliary and a pessimistic trajectory truncation rule so that the upper-level method
optimizes through finite lower-level dynamics rather than through a globally solved lower-level problem. \citet{arbel2022non}
introduce bilevel games with critical-point selection maps, where the lower-level object is a selected critical point
reached by a prescribed optimization process, and they study how finite-budget unrolled methods relate to equilibria of
these selection-based games. Under a Morse-type parametric qualification, \citet{bolte2025bilevel} show that the graphs of
lower-level critical points and local minima decompose into finitely many smooth branches and use this structure to define
branch-wise hyper-objectives and gradient methods along individual branches. \citet{jiang2025correspondence} argue that,
for general nonconvex lower levels, the classical hyper-objective may be unsettled if global minimizers are unattainable,
and instead propose a correspondence-driven hyper-objective tied to the output of a prescribed lower-level algorithm,
together with Gaussian smoothing and projected stochastic-gradient guarantees. \citet{chen2025setsmoothness} take a
different nonsmooth route: rather than recovering classical differentiability, they introduce a set-smoothness condition
that yields weak convexity/concavity of the induced hyper-objective and enables computation of approximate Clarke
hyper-stationary points.

Another recent line develops Moreau-envelope single-level reformulations as an alternative to classical value-function and
KKT-based reductions. \citet{bai2025optimality} derive directional necessary optimality conditions in this framework,
\citet{gao2023moreau} propose a difference-of-weakly-convex reformulation with a sequentially convergent algorithm, and
\citet{Lu25tsp} develop a two-sided smoothed primal--dual method for general nonconvex bilevel problems with KKT-type
convergence guarantees. Our focus is different from all of these directions: we retain the classical optimistic
hyper-objective \(F(\theta)=\min_{x\in\calS(\theta)} f(\theta,x)\) itself and ask when, despite the lower-level solution set
being manifold-valued, that object is differentiable and even locally smooth.

A recent sampling-based relaxation \citep{masiha2025superquantile} replaces hard minima selection by a
superquantile--Gibbs surrogate and studies a zeroth-order method that optimizes the resulting nonsmooth Lipschitz
objective through sampling-based function-value information rather than exact hyper-gradients. That approach is designed
to circumvent the nondifferentiability of hard minima selection by changing the objective being optimized. The present
paper takes a different route: we keep the hard optimistic objective, prove classical \(\mathcal{C}^{1}\) regularity and
local smoothness under uniqueness and nondegeneracy of the optimistic minimizer, derive the corresponding exact
pseudoinverse hyper-gradient, and then build a practical select-then-differentiate algorithm with an end-to-end
stationarity guarantee. In this sense, the present paper can be viewed as a differentiable counterpart to that
relaxation-based, zeroth-order line.

\paragraph{Relaxing strong convexity through global regularity or convexity assumptions.}
Recent works on bilevel optimization study nonconvex lower-level problems under \emph{global} P\L\ or closely related
uniform error-bound assumptions
\citep{chen2023bilevel,kwon2023penalty,xiao2023generalized,huang2024optimal,shen2025penalty,liu2022bome}. One attraction
of this line is that such global regularity helps guarantee a smooth and analytically tractable hyper-objective
\citep{kwon2023penalty,chen2023bilevel}. From the viewpoint of the lower-level solution set, however, these assumptions are
quite restrictive. For a \(\mathcal{C}^{2}\) lower-level objective satisfying a global P\L\ condition, the minimizer set
\(\calS(\theta)\) can only be either a singleton or an unbounded subset of \(\R^d\)
\citep{criscitiello2025,masiha2025superquantile}. Therefore, once one adds the boundedness or coercivity assumptions
typically used in this literature, the lower-level problem collapses back to the singleton-minimizer regime, and bounded
non-singleton minimizer manifolds are excluded.

This restriction is also visible in the specific assumptions used by these papers. \citet{huang2024optimal} assume a
nondegenerate lower Hessian at global minimizers, which effectively enforces local uniqueness; \citet{liu2022bome}
explicitly assume that \(\calS(\theta)\) is a singleton; and \citet{kwon2023penalty} impose a proximal error bound for
smooth \(g\) together with coercivity, which again yields a singleton solution
\citep[cf. Prop.~C.1]{chen2023bilevel}. Other extensions obtain differentiability or algorithmic guarantees under
convexity of the lower-level objective in \(x\) \citep{xiao2023generalized,shen2025penalty}; this allows non-strongly
convex followers, but it still restricts \(\calS(\theta)\) to a convex geometry rather than the nonconvex compact
manifold-valued setting that motivates our work. A related direction is \citet{chen2025setsmoothness}, who study approximate Clarke hyper-stationarity via a new
set-smoothness property of the lower-level solution map. Their framework assumes, in particular, that for every
\(\theta\), the lower-level solution set \(\calS(\theta)\) is nonempty, closed, and convex, together with a uniform
error-bound condition (equivalently, under smoothness, a global P\L\ condition). Under these assumptions, they show that
the hyper-objective is weakly convex/concave and develop an algorithm for approximate Clarke hyper-stationarity. From the
perspective of lower-level geometry, this is still quite different from our setting: their analysis is tailored to
convex solution sets, whereas we study compact nonconvex manifold-valued minimizer sets under the local
\(\PLcirc\) condition.

By contrast, our analysis is based on the local \(\PLcirc\) condition of \citet{gong2024poincare}, which is specifically
designed to allow a connected compact set of global lower-level minimizers with \(\mathcal{C}^{2}\) manifold structure. This
regime remains rich enough to cover overparameterized deep-learning landscapes where connected minima have been observed
\citep{draxler2018essentially,garipov2018loss,nguyen2019connected}, and it keeps minima selection genuinely relevant rather
than collapsing the problem back to the singleton-minimizer case.

\subsection{Comparison with penalty-based differentiability analyses}\label{append:penalty-diff}
We compare our standing assumptions to penalty-based differentiability analyses such as \citet{kwon2023penalty} and \citet{chen2023bilevel}, which study the perturbed lower-level objective
\[
h_\sigma(\theta,x)\;:=\;g(\theta,x)+\sigma f(\theta,x),\qquad \sigma>0.
\]

\paragraph{What is assumed in penalty-based differentiability works.}
A representative sufficient condition in this line of work is a \emph{$\sigma$-uniform} error-bound / P\L-type regularity of $h_\sigma(\theta,\cdot)$ in the lower-level variable $x$ for all $\sigma\in[0,\sigma_0]$, with constants that do not degenerate as $\sigma\downarrow 0$.
For example, \citet{kwon2023penalty} assume a (small-error) proximal error-bound condition for $h_\sigma(\theta,\cdot)$ that holds uniformly for all $\sigma\in[0,\sigma_0]$, together with coercivity and smoothness assumptions.
In the smooth unconstrained setting, such a proximal error bound is equivalent to a P\L\ inequality \citep[Prop.~C.1]{chen2023bilevel}.
In particular, since $\sigma=0$ is included, these assumptions enforce a P\L\ / error-bound condition for $g(\theta,\cdot)$ itself.

\paragraph{Implication: singleton lower-level minimizers.}
For \(\mathcal{C}^{2}\) objectives, a global P\L\ condition implies that the minimizer set is either a singleton or unbounded \citep{criscitiello2025}.
Therefore, under standard assumptions that rule out unbounded minimizers (e.g., coercivity, or compactness of $\calS(\theta)$), the lower-level minimizer must be unique \citep{masiha2025superquantile}.
In this regime, the optimistic bilevel objective reduces to the classical singleton-minimizer setting, where \minimaselection is redundant.
Several recent analyses in nonconvex bilevel optimization fall into this category (global P\L\ / error bounds together with coercivity/compactness) and hence effectively assume a singleton lower-level minimizer; see, e.g., \citet{huang2024optimal,kwon2023penalty}.
Other works enforce uniqueness more directly (without explicitly invoking global P\L), e.g., \citet{liu2022bome}.

\paragraph{Why our assumptions are weaker than $\sigma$-uniform local P\L\ for \(g+\sigma f\).}
The key obstruction is already visible by testing a putative P\L\ inequality on points of the minimizer manifold.
On \(\calS(\theta)\), the lower objective \(g\) is constant and \(\nabla_x g=0\). Thus the value gap of
\(h_\sigma=g+\sigma f\) between two manifold points is scaled by \(\sigma\), while the gradient contribution from
\(f\) is also scaled by \(\sigma\). The following lemma shows that any P\L\ constant on a neighborhood containing both
the optimistic point and a higher-value manifold point must therefore be \(\mathcal{O}(\sigma)\).

\begin{lemma}[No \(\sigma\)-uniform local P\L\ constant on neighborhoods containing a higher-value manifold point]\label{lem:no-sigma-uniform-pl}
Fix \(\theta\) and let \(x^\star\in\calS(\theta)\) be the unique optimistic minimizer.
Assume there exists \(\bar x\in\calS(\theta)\) such that \(f(\theta,\bar x)>f(\theta,x^\star)\) (i.e., \(f(\theta,\cdot)\) is not constant on \(\calS(\theta)\)).
Let \(U\subset\R^d\) be any neighborhood containing both \(x^\star\) and \(\bar x\).
If \(h_\sigma(\theta,\cdot)\) satisfies a (local) P\L\ inequality on \(U\) with constant \(\mu_\sigma>0\), i.e.,
\[
h_\sigma(\theta,x)-\inf_{y\in U} h_\sigma(\theta,y)\ \le\ (2\mu_\sigma)^{-1}\,\|\nabla_x h_\sigma(\theta,x)\|^2
\qquad\forall x\in U,
\]
then
\begin{equation}\label{eq:mu-sigma-upper-bound}
\mu_\sigma
\ \le\
\frac{\sigma}{2}\cdot
\frac{\|\nabla_x f(\theta,\bar x)\|^2}{f(\theta,\bar x)-f(\theta,x^\star)}.
\end{equation}
In particular, \(\mu_\sigma=\mathcal{O}(\sigma)\) as \(\sigma\downarrow 0\), so no constant bounded away from \(0\) can hold
uniformly over \(\sigma\in(0,\sigma_0]\) on any such neighborhood.
\end{lemma}

\begin{proof}
Fix such a neighborhood \(U\) and point \(\bar x\in\calS(\theta)\cap U\).
Since \(x^\star\in U\), we have \(\inf_{y\in U}h_\sigma(\theta,y)\le h_\sigma(\theta,x^\star)\), hence
\[
h_\sigma(\theta,\bar x)-h_\sigma(\theta,x^\star)
\ \le\
h_\sigma(\theta,\bar x)-\inf_{y\in U}h_\sigma(\theta,y).
\]
Applying the P\L\ inequality at \(x=\bar x\) and using that \(\bar x\in\calS(\theta)\) (so \(\nabla_x g(\theta,\bar x)=0\) and \(g(\theta,\bar x)=g(\theta,x^\star)\)) gives
\[
\sigma\big(f(\theta,\bar x)-f(\theta,x^\star)\big)
\ \le\
(2\mu_\sigma)^{-1}\,\|\sigma\nabla_x f(\theta,\bar x)\|^2,
\]
which rearranges to \eqref{eq:mu-sigma-upper-bound}.
\end{proof}

When \(\dimk>0\) and \(x^\star\) is non-degenerate in the sense of
\Cref{def:degenerate-optimistic}, such higher-value points occur arbitrarily close to \(x^\star\): along any nonzero
tangent direction, the restricted objective \(f(\theta,\cdot)|_{\calS(\theta)}\) increases quadratically for sufficiently
small geodesic displacement. Thus the obstruction above is genuinely local in the non-singleton manifold regime.

\section{Proof of \Cref{sec:hyp_diff}}\label{append_hyper_diff}
This appendix provides the proofs for \Cref{sec:hyp_diff}.

\subsection{Proof of \Cref{thm:F-diff-unique}}\label{append:F-diff-unique}
		Example~\ref{ex:unique-nonsmooth} shows that even when the optimistic minimizer is unique,
		the selection map $\theta\mapsto x^\star(\theta)$ can fail to be differentiable at degenerate points.
		Nevertheless, uniqueness is still enough to ensure that the \emph{value function}
		$F(\theta)=\min_{x\in\calS(\theta)}f(\theta,x)$ is differentiable, and when uniqueness holds in a neighborhood it is in fact $\mathcal{C}^{1}$.

	The following is a standard envelope/Danskin lemma; see, e.g., \citet{danskin1967theory} and
	\citet[Prop.~2.1]{oyama2018nondifferentiability}. We provide the proof for completeness.
	\begin{lemma}[A basic envelope/Danskin lemma]\label{lem:envelope-unique}
	Let \(V\subset\R^m\) be an open neighborhood of \(\theta_0\), let \(U\subset\R^\ell\) be compact, and let
	\(\phi:V\times U\to\R\) be continuous. Suppose that \(\phi(\cdot,u)\) is differentiable in \(\theta\) for all
	\(u\in U\), and that \((\theta,u)\mapsto\nabla_\theta\phi(\theta,u)\) is continuous on a neighborhood of
	\((\theta_0,u_0)\).
	Define \(F(\theta):=\min_{u\in U}\phi(\theta,u)\) for \(\theta\in V\).
	If $\arg\min_{u\in U}\phi(\theta_0,u)=\{u_0\}$, then $F$ is differentiable at $\theta_0$ and
	\[
	\nabla F(\theta_0)=\nabla_\theta\phi(\theta_0,u_0).
	\]
	\end{lemma}
	\begin{proof}
	We first note that every local choice of minimizers converges to the unique minimizer at \(\theta_0\).
	Let $\theta_n\to\theta_0$ and pick any minimizer $u_n\in\arg\min_{u\in U}\phi(\theta_n,u)$.
	By compactness of $U$, along a subsequence (not relabeled) we have $u_n\to \bar u\in U$.
	For any fixed $u\in U$, optimality of $u_n$ gives $\phi(\theta_n,u_n)\le \phi(\theta_n,u)$.
	Taking $n\to\infty$ and using continuity of $\phi$ yields $\phi(\theta_0,\bar u)\le \phi(\theta_0,u)$ for all $u\in U$,
	so $\bar u\in\arg\min_{u\in U}\phi(\theta_0,u)=\{u_0\}$.
	Thus every convergent subsequence of $u_n$ converges to $u_0$, hence $u_n\to u_0$.

	We now prove Fr\'echet differentiability. Let \(h\to0\) in \(\R^m\) with \(\theta_0+h\in V\), and let
	\(u_h\in\arg\min_{u\in U}\phi(\theta_0+h,u)\). By the first part, \(u_h\to u_0\) as \(h\to0\).
	Set \(a:=\nabla_\theta\phi(\theta_0,u_0)\). Optimality of \(u_h\) and \(u_0\) gives
	\[
	\phi(\theta_0+h,u_h)-\phi(\theta_0,u_h)
	\ \le\
	F(\theta_0+h)-F(\theta_0)
	\ \le\
	\phi(\theta_0+h,u_0)-\phi(\theta_0,u_0).
	\]
	It remains to show that both the lower and upper bounds have first-order expansion \(\langle a,h\rangle+o(\|h\|)\).
	The upper bound satisfies this immediately from differentiability of \(\phi(\cdot,u_0)\) at \(\theta_0\):
	\[
	\phi(\theta_0+h,u_0)-\phi(\theta_0,u_0)-\langle a,h\rangle=o(\|h\|).
	\]
	For the lower bound, take \(\|h\|\) small enough so that
	\(\{(\theta_0+s h,u_h):s\in[0,1]\}\) lies in the neighborhood where \(\nabla_\theta\phi\) is continuous, and write
	\[
	\phi(\theta_0+h,u_h)-\phi(\theta_0,u_h)-\langle a,h\rangle
	=
	\int_0^1
	\big\langle
	\nabla_\theta\phi(\theta_0+s h,u_h)-a,\ h
	\big\rangle\,\mathrm ds.
	\]
	We claim that the integrand is uniformly \(o(\|h\|)\). Indeed, since \(u_h\to u_0\) and
	\(\theta_0+s h\to\theta_0\) uniformly in \(s\in[0,1]\), continuity of \(\nabla_\theta\phi\) at
	\((\theta_0,u_0)\) gives
	\[
	\sup_{s\in[0,1]}
	\big\|
	\nabla_\theta\phi(\theta_0+s h,u_h)-a
	\big\|
	\to0.
	\]
	Consequently,
	\[
	\left|
	\phi(\theta_0+h,u_h)-\phi(\theta_0,u_h)-\langle a,h\rangle
	\right|
	\le
	\|h\|
	\sup_{s\in[0,1]}
	\big\|
	\nabla_\theta\phi(\theta_0+s h,u_h)-a
	\big\|
	=o(\|h\|).
	\]
	Combining the upper and lower bounds yields
	\[
	F(\theta_0+h)-F(\theta_0)-\langle a,h\rangle=o(\|h\|),
	\]
	which proves differentiability of \(F\) at \(\theta_0\) with gradient \(a=\nabla_\theta\phi(\theta_0,u_0)\).
	\end{proof}

	Lemma~\ref{lem:parametric-S-chart} is the bridge from minimizing over the moving manifold \(\calS(\theta)\) to
	minimizing over fixed local coordinates: near \(x_0\), points in \(\calS(\theta)\) can be written as
	\(\psi(\theta,u)\) with \(u\) in a fixed neighborhood of \(0\in\R^{\dimk}\).
	\begin{lemma}[Local parameter-dependent chart for the minimizer manifold]\label{lem:parametric-S-chart}
Let \(g\in\mathcal{C}^{3}\) and let \Cref{ass:plcirc} hold. Fix
\(\theta_0\in\Theta\) and \(x_0\in\calS(\theta_0)\). Then, after shrinking the
neighborhood if needed, there exist neighborhoods \(\mathcal U_\theta\) of \(\theta_0\),
		\(\mathcal U_u\) of \(0\in\R^{\dimk}\), and \(\mathcal W\) of \(x_0\), together with a
		\(\mathcal{C}^{2}\) map
		\[
		\psi:\mathcal U_\theta\times\mathcal U_u\to\mathcal W,
		\qquad
		\psi(\theta_0,0)=x_0,
		\]
		such that \(u\mapsto\psi(\theta,u)\) is an embedded chart and
		\[
		\calS(\theta)\cap\mathcal W=\{\psi(\theta,u):u\in\mathcal U_u\},
		\qquad
		\theta\in\mathcal U_\theta .
		\]
		\end{lemma}
		\begin{proof}
		Let \(\calT_0:=\calT_{x_0}^{\theta_0}\), \(\calN_0:=\calN_{x_0}^{\theta_0}\), and choose orthonormal matrices
		\(U_{\calT_0}\in\R^{d\times \dimk}\) and \(U_{\calN_0}\in\R^{d\times(d-\dimk)}\) spanning these spaces. For
		\((u,v)\) near \((0,0)\), write \(x(u,v):=x_0+U_{\calT_0}u+U_{\calN_0}v\) and define
		\[
		\Phi(\theta,u,v):=U_{\calN_0}^{\top}\nabla_x g(\theta,x(u,v)).
		\]
		By \Cref{prop:geometry}, \(U_{\calN_0}^{\top}\nabla^2_{xx}g(\theta_0,x_0)U_{\calN_0}\) is positive definite.
		The implicit function theorem therefore gives a \(\mathcal{C}^{2}\) map
		\(h(\theta,u)\), defined for \((\theta,u)\) near \((\theta_0,0)\), such that
		\(\Phi(\theta,u,h(\theta,u))=0\). Set
		\[
		\psi(\theta,u):=x_0+U_{\calT_0}u+U_{\calN_0}h(\theta,u).
		\]
		For each fixed \(\theta\), this parametrizes the \(\dimk\)-dimensional manifold
		\[
		\widetilde{\calS}(\theta):=\{x:\ U_{\calN_0}^{\top}\nabla_x g(\theta,x)=0\}
		\]
		near \(x_0\). This last statement is local in the following precise sense. The IFT gives open neighborhoods
		\(\mathcal U_\theta^{0}\) of \(\theta_0\), \(\mathcal U_u^{0}\) of \(0\in\R^{\dimk}\), and
		\(\mathcal U_v^{0}\) of \(0\in\R^{d-\dimk}\) such that every solution of
		\(U_{\calN_0}^{\top}\nabla_x g(\theta,x_0+U_{\calT_0}u+U_{\calN_0}v)=0\) with
		\((\theta,u,v)\in\mathcal U_\theta^{0}\times\mathcal U_u^{0}\times\mathcal U_v^{0}\) is uniquely represented as
		\(v=h(\theta,u)\). We now choose a connected open set \(\mathcal U_u\) containing \(0\) whose closure is contained in
		\(\mathcal U_u^{0}\), shrink the parameter neighborhood to
		\(\mathcal U_\theta\subseteq\mathcal U_\theta^{0}\), and choose an ambient neighborhood \(\mathcal W\) of \(x_0\)
		small enough so that every \(x\in\mathcal W\) can be written uniquely as
		\(x=x_0+U_{\calT_0}u+U_{\calN_0}v\) with \(u\in\mathcal U_u^{0}\) and \(v\in\mathcal U_v^{0}\), and such that
		the points of the graph over \(\mathcal U_u\) lie in \(\mathcal W\). With these choices, the local piece of
		\(\widetilde{\calS}(\theta)\) inside \(\mathcal W\) is exactly the graph:
		for every \(\theta\in\mathcal U_\theta\),
		\[
		\widetilde{\calS}(\theta)\cap\mathcal W
		=
		\{\psi(\theta,u):u\in\mathcal U_u\}.
		\]
		Finally, since every point of \(\calS(\theta)\) is a lower-level minimizer and hence satisfies the full first-order
		condition \(\nabla_x g(\theta,x)=0\), it satisfies the projected condition defining
		\(\widetilde{\calS}(\theta)\) and therefore
		\(\calS(\theta)\cap\mathcal W\subseteq\widetilde{\calS}(\theta)\cap\mathcal W\).

		It remains to show that this inclusion is equality on a smaller neighborhood. By the fixed-dimension conclusion in
		\Cref{prop:geometry} and the local Hausdorff continuity of the solution manifolds
		\citep[Lem.~3.2]{masiha2025superquantile}, after shrinking \(\mathcal U_\theta\) and \(\mathcal W\), for every
		\(\theta\in\mathcal U_\theta\),
		\[
		\dim\calS(\theta)=\dimk
		\qquad\text{and}\qquad
		\calS(\theta)\cap\mathcal W\neq\emptyset .
		\]
		For such a \(\theta\), define
		\[
		Z_\theta:=\{u\in\mathcal U_u:\psi(\theta,u)\in\calS(\theta)\}.
		\]
		The set \(Z_\theta\) records which points of the graph for \(\widetilde{\calS}(\theta)\) are genuine lower-level
		minimizers. It is nonempty by the local Hausdorff-continuity step above: since
		\(\calS(\theta)\cap\mathcal W\neq\emptyset\) and
		\(\calS(\theta)\cap\mathcal W\subseteq\widetilde{\calS}(\theta)\cap\mathcal W\), at least one point of
		\(\calS(\theta)\cap\mathcal W\) has the form \(\psi(\theta,u)\). It is closed relative to \(\mathcal U_u\) because
		\[
		Z_\theta=\{u\in\mathcal U_u:\psi(\theta,u)\in\calS(\theta)\}
		=
		\psi_\theta^{-1}(\calS(\theta))\cap\mathcal U_u,
		\qquad \psi_\theta(u):=\psi(\theta,u),
		\]
		where \(\psi_\theta\) is continuous and \(\calS(\theta)\) is compact, hence closed, by \Cref{prop:geometry}.

		We next show that \(Z_\theta\) is open. Take \(u_1\in Z_\theta\) and set \(x_1:=\psi(\theta,u_1)\). Near \(x_1\),
		both \(\calS(\theta)\) and \(\widetilde{\calS}(\theta)\) are embedded \(\dimk\)-dimensional submanifolds, and we already
		know that \(\calS(\theta)\subseteq\widetilde{\calS}(\theta)\). Consider the inclusion map
		\[
		i:\calS(\theta)\hookrightarrow\widetilde{\calS}(\theta),\qquad i(x)=x.
		\]
		This map is smooth because both manifolds are embedded submanifolds of the same ambient space. Its derivative at
		\(x_1\) is the identity map on tangent vectors:
		\[
		D i(x_1):\calT_{x_1}\calS(\theta)\to \calT_{x_1}\widetilde{\calS}(\theta),\qquad v\mapsto v.
		\]
		Since the two tangent spaces both have dimension \(\dimk\), this derivative is an isomorphism. The inverse function
		theorem for manifolds therefore gives a neighborhood \(\mathcal O_S\) of \(x_1\) in \(\calS(\theta)\) such that
		\(i(\mathcal O_S)\) is an open neighborhood of \(x_1\) in \(\widetilde{\calS}(\theta)\). But \(i\) is just inclusion,
		so \(i(\mathcal O_S)=\mathcal O_S\subseteq\calS(\theta)\). Thus some neighborhood of \(x_1\) inside
		\(\widetilde{\calS}(\theta)\) is contained in \(\calS(\theta)\). Pulling this neighborhood back through the chart
		\(u\mapsto\psi(\theta,u)\), we obtain a neighborhood of \(u_1\) contained in \(Z_\theta\). Since this argument works
		for every \(u_1\in Z_\theta\), each point of \(Z_\theta\) has a relative neighborhood in \(\mathcal U_u\) contained in
		\(Z_\theta\). Thus \(Z_\theta\) is open in \(\mathcal U_u\).

		Now \(Z_\theta\) is nonempty and is both closed and open relative to \(\mathcal U_u\). If
		\(Z_\theta\neq\mathcal U_u\), then \(Z_\theta\) and \(\mathcal U_u\setminus Z_\theta\) would form a separation of
		\(\mathcal U_u\): two nonempty disjoint relatively open sets whose union is \(\mathcal U_u\). This contradicts the
		connectedness of \(\mathcal U_u\). Hence \(Z_\theta=\mathcal U_u\). Therefore every point of the local graph
		\(\{\psi(\theta,u):u\in\mathcal U_u\}\) belongs to \(\calS(\theta)\). Together with the earlier inclusion
		\(\calS(\theta)\cap\mathcal W\subseteq\widetilde{\calS}(\theta)\cap\mathcal W\), this gives
		\[
		\calS(\theta)\cap\mathcal W
		=
		\{\psi(\theta,u):u\in\mathcal U_u\},
		\]
		after shrinking \(\mathcal W\) within the chart neighborhood if needed.
		\end{proof}

		\begin{proof}[Proof of \Cref{thm:F-diff-unique}]
		Let $x^\star(\theta)$ denote the unique optimistic minimizer for $\theta\in\mathcal U$.
		Set \(x_0:=x^\star(\theta_0)\).
		We prove that $F$ is $\mathcal{C}^{1}$ locally by (i) building a smooth chart for $\calS(\theta)$ near $x_0$,
		(ii) reducing the bilevel problem to a fixed-domain minimization and applying an envelope/Danskin lemma (which yields both differentiability and continuity of $\nabla F$),
		and (iii) computing the resulting gradient in the pseudoinverse form.

		\paragraph{Step 1: A $\theta$-dependent chart for $\calS(\theta)$ near $x_0$.}
			By \Cref{lem:parametric-S-chart}, after shrinking neighborhoods if needed there are an open neighborhood
			\(\mathcal U_\theta\subseteq\mathcal U\), a bounded open ball \(\mathcal U_u\subset\R^{\dimk}\) around \(0\), an open
			neighborhood \(\mathcal W\) of \(x_0\), and a \(\mathcal{C}^{2}\) map
		\[
		\psi:\mathcal U_\theta\times\mathcal U_u\to\mathcal W
		\]
		such that \(\psi(\theta_0,0)=x_0\), each \(u\mapsto\psi(\theta,u)\) is an embedded chart, and
		\[
		\calS(\theta)\cap\mathcal W=\{\psi(\theta,u):u\in\mathcal U_u\},
		\qquad \theta\in\mathcal U_\theta .
		\]

			\paragraph{Step 2: Reduce to a fixed-domain minimization and show $F\in\mathcal{C}^{1}$.}
			Step 1 only gives a $\mathcal{C}^{2}$ chart near $x_0$. Since $u\mapsto \psi(\theta,u)$ parametrizes only the part of $\calS(\theta)$ near $x_0$,
			we first show that the minimizer stays in this chart patch for $\theta$ close to $\theta_0$.
			
			\emph{Step 2a: $x^\star(\theta)$ stays near $x_0$.}
			Let $\theta_n\to\theta_0$ with $\theta_n\in\mathcal U_\theta$ and set $x_n:=x^\star(\theta_n)\in\calS(\theta_n)\subseteq\mathcal V$.
			By compactness of $\mathcal V$, along a subsequence (not relabeled) we have $x_n\to \bar x\in\mathcal V$.
			For any fixed $y\in\R^d$, since $x_n\in\arg\min_x g(\theta_n,x)$ we have $g(\theta_n,x_n)\le g(\theta_n,y)$.
			Taking $n\to\infty$ and using continuity of $g$ yields $g(\theta_0,\bar x)\le g(\theta_0,y)$ for all $y$, hence $\bar x\in\calS(\theta_0)$.
			
			Moreover, $\psi(\theta_n,0)\in\calS(\theta_n)$ is a feasible competitor for the optimistic problem at $\theta_n$, so
			\[
			f(\theta_n,x_n)=F(\theta_n)\ \le\ f(\theta_n,\psi(\theta_n,0)).
			\]
			Sending $n\to\infty$ and using continuity of $f$ and $\psi$ gives $f(\theta_0,\bar x)\le f(\theta_0,x_0)=F(\theta_0)$.
			By uniqueness of $x_0\in\arg\min_{x\in\calS(\theta_0)} f(\theta_0,x)$ and $\bar x\in\calS(\theta_0)$, we conclude $\bar x=x_0$.
			Thus every convergent subsequence of $x_n$ converges to $x_0$, hence $x_n\to x_0$.
			After possibly shrinking $\mathcal U_\theta$, we may assume $x^\star(\theta)$ lies in the chart patch for all $\theta\in\mathcal U_\theta$.
			
				\emph{Step 2b: fixed-domain reduction.}
				Write $x^\star(\theta)=\psi(\theta,u^\star(\theta))$ for some $u^\star(\theta)\in\mathcal U_u$.
				The convergence \(x^\star(\theta)\to x_0\) from Step~2a and the fact that \(u\mapsto\psi(\theta,u)\) is a local
				chart imply \(u^\star(\theta)\to0\) as \(\theta\to\theta_0\). Hence, after shrinking
				\(\mathcal U_\theta\) once more, we may choose a compact neighborhood \(U\subset\mathcal U_u\) of \(0\) such that
				\(u^\star(\theta)\in U\) for all \(\theta\in\mathcal U_\theta\).
				Define $\tilde f(\theta,u):=f(\theta,\psi(\theta,u))$.
				Then, for all $\theta\in\mathcal U_\theta$,
				\[
			F(\theta)=\min_{x\in\calS(\theta)} f(\theta,x)=\min_{u\in U}\tilde f(\theta,u).
			\]
			
			\emph{Step 2c: differentiability via an envelope lemma.}
			Since $f$ is $\mathcal{C}^{1}$ and $\psi$ is $\mathcal{C}^{2}$, $\tilde f$ is continuous and $\mathcal{C}^{1}$ in $\theta$ with continuous $\nabla_\theta\tilde f$.
			We claim that for every $\theta\in\mathcal U_\theta$, the minimizer of $u\mapsto \tilde f(\theta,u)$ over $U$ is unique and equals $u^\star(\theta)$.
			Indeed, by Step~2b,
			\[
			F(\theta)=\min_{u\in U}\tilde f(\theta,u),
			\qquad
			\tilde f(\theta,u^\star(\theta))=f(\theta,x^\star(\theta))=F(\theta),
			\]
			so $u^\star(\theta)$ is a minimizer.
			Conversely, if $\bar u\in U$ also minimizes $\tilde f(\theta,\cdot)$, then
			\[
			f(\theta,\psi(\theta,\bar u))=\tilde f(\theta,\bar u)=F(\theta),
			\]
			so $\psi(\theta,\bar u)\in\arg\min_{x\in\calS(\theta)} f(\theta,x)$.
			By the uniqueness assumption in \Cref{thm:F-diff-unique}, this implies
			$\psi(\theta,\bar u)=x^\star(\theta)=\psi(\theta,u^\star(\theta))$.
			Since $u\mapsto \psi(\theta,u)$ is a chart parametrization on the patch, it is injective, hence $\bar u=u^\star(\theta)$.
			Applying \Cref{lem:envelope-unique} at each $\theta\in\mathcal U_\theta$ yields that $F$ is differentiable on $\mathcal U_\theta$ and
			\begin{equation}\label{eq:env-grad}
			\nabla F(\theta)=\nabla_\theta \tilde f(\theta,u^\star(\theta))
			\qquad\text{for all }\theta\in\mathcal U_\theta.
			\end{equation}
			Importantly, \eqref{eq:env-grad} does \emph{not} require differentiability of $\theta\mapsto u^\star(\theta)$.
			
			\emph{Step 2d: continuity of $\nabla F$ (hence $F\in\mathcal{C}^{1}$).}
			We show that $\theta\mapsto u^\star(\theta)$ is continuous on $\mathcal U_\theta$.
			Fix $\bar\theta\in\mathcal U_\theta$ and let $\theta_n\to\bar\theta$ with $\theta_n\in\mathcal U_\theta$.
			By compactness of $U$, along a subsequence (not relabeled) we have $u^\star(\theta_n)\to \bar u\in U$.
			Optimality of $u^\star(\theta_n)$ gives $\tilde f(\theta_n,u^\star(\theta_n))\le \tilde f(\theta_n,u)$ for all $u\in U$.
			Sending $n\to\infty$ and using continuity of $\tilde f$ yields $\tilde f(\bar\theta,\bar u)\le \tilde f(\bar\theta,u)$ for all $u\in U$.
			Thus $\bar u\in\arg\min_{u\in U}\tilde f(\bar\theta,u)=\{u^\star(\bar\theta)\}$, hence $\bar u=u^\star(\bar\theta)$.
			Therefore $u^\star(\theta_n)\to u^\star(\bar\theta)$, proving continuity.
			Since $\nabla_\theta\tilde f$ is continuous, \eqref{eq:env-grad} implies that $\nabla F$ is continuous on $\mathcal U_\theta$.
			Setting $\mathcal U_0:=\mathcal U_\theta$, we conclude that $F$ is $\mathcal{C}^{1}$ on $\mathcal U_0$.

			\paragraph{Step 3: Compute $\nabla F(\theta)$ and conclude \eqref{eq:hg_formula}.}
			Fix any $\theta\in\mathcal U_0$ and write $u^\star:=u^\star(\theta)$ and $x^\star:=x^\star(\theta)=\psi(\theta,u^\star)\in\calS(\theta)$.
		By \eqref{eq:env-grad} and the chain rule for $\tilde f(\theta,u)=f(\theta,\psi(\theta,u))$ (with $u$ held fixed),
		\[
		\nabla F(\theta)
		=
		\nabla_\theta f(\theta,x^\star)
		+
		\big(\nabla_\theta \psi(\theta,u^\star)\big)^\top \nabla_x f(\theta,x^\star).
		\]
		Since $x^\star$ minimizes $f(\theta,\cdot)$ over the manifold $\calS(\theta)$ (without boundary),
		its Riemannian gradient vanishes, i.e.,
		$P_{\calT_{x^\star}^\theta}\nabla_x f(\theta,x^\star)=0$, hence $\nabla_x f(\theta,x^\star)\in \calN_{x^\star}^\theta$.
		
		On the other hand, for all $(\theta,u)\in\mathcal U_0\times \mathcal U_u$ we have $\psi(\theta,u)\in\calS(\theta)$, so
		$\nabla_x g(\theta,\psi(\theta,u))=0$.
		Differentiating this identity with respect to $\theta$ (with $u$ held fixed) gives
		\[
		H(\theta,\psi(\theta,u))\,\nabla_\theta \psi(\theta,u)+\nabla_{x\theta}^2 g(\theta,\psi(\theta,u))=0.
		\]
		Evaluating at $u=u^\star$ and multiplying by $H(\theta,x^\star)^\dagger$ yields
		\[
		P_{\calN_{x^\star}^\theta}\nabla_\theta \psi(\theta,u^\star)
		=
		-H(\theta,x^\star)^\dagger\nabla_{x\theta}^2 g(\theta,x^\star),
		\]
		where we used that $H(\theta,x^\star)^\dagger H(\theta,x^\star)=P_{\calN_{x^\star}^\theta}$ by \Cref{prop:geometry}(ii).
		Since $\nabla_x f(\theta,x^\star)\in \calN_{x^\star}^\theta$, we can insert the projector and obtain
			\begin{align*}
			    \big(\nabla_\theta \psi(\theta,u^\star)\big)^\top \nabla_x f(\theta,x^\star)
			&=
			-\big(\nabla_{x\theta}^2 g(\theta,x^\star)\big)^\top
			H(\theta,x^\star)^\dagger\nabla_x f(\theta,x^\star)\\
			&=
			-\nabla_{\theta x}^2 g(\theta,x^\star)H(\theta,x^\star)^\dagger\nabla_x f(\theta,x^\star),
			\end{align*}
		where we used $\nabla^2_{\theta x}g:=(\nabla^2_{x\theta}g)^\top$ and the symmetry of \(H(\theta,x^\star)^\dagger\).
			Substituting into the expression for $\nabla F(\theta)$ gives \eqref{eq:hg_formula}.
			\end{proof}
\begin{remark}[A convenient sufficient condition for $F\in\mathcal{C}^{1}(\Theta)$]\label{rem:global-C1}
				Assume $\Theta$ is compact and $\arg\min_{x\in\calS(\theta)} f(\theta,x)$ is a singleton for every $\theta\in\Theta$.
				If, in addition, the lower-level solution manifolds are uniformly bounded in the sense that there exists a compact set $\mathcal V\subset\R^d$ with $\calS(\theta)\subseteq\mathcal V$ for all $\theta\in\Theta$
				(e.g., $\mathcal V=\mathbb B_d(0;\mathsf D)$ under \Cref{assum_sample_comp}),
				then $F$ is $\mathcal{C}^{1}$ on $\Theta$ and \eqref{eq:hg_formula} holds for all $\theta\in\Theta$.
				Indeed, \Cref{thm:F-diff-unique} applies at each $\theta_0\in\Theta$; compactness of $\Theta$ allows us to extract a finite subcover of the resulting local $\mathcal{C}^{1}$ neighborhoods.
				\end{remark}

\subsection{Proof of \Cref{thm_hyp_obj_diff}}
\begin{proof}[Proof of \Cref{thm_hyp_obj_diff}]
Let \(x_0:=x^\star(\theta_0)\). We prove the result in four steps. First, we write the
moving manifold \(\calS(\theta)\) in fixed local coordinates. Second, we use intrinsic non-degeneracy and the implicit
function theorem to obtain a \(\mathcal C^1\) local selected branch. Third, we use uniqueness at \(\theta_0\) to show that
this local branch is the global optimistic branch for nearby parameters. Finally, we differentiate through lower-level
stationarity to obtain the pseudoinverse formula and local Lipschitzness of \(\nabla F\).

\paragraph{Step 1: Reduce the moving-manifold problem to fixed local coordinates.}
By \Cref{lem:parametric-S-chart}, after shrinking neighborhoods if necessary, there are an open neighborhood
\(\mathcal U_\theta\) of \(\theta_0\), a bounded convex open ball \(\mathcal U_u\subset\R^{\dimk}\) around \(0\) whose closure
lies in a slightly larger coordinate domain, an open neighborhood \(\mathcal W\) of \(x_0\), and a \(\mathcal C^2\) map
\[
\psi:\mathcal U_\theta\times\mathcal U_u\to\mathcal W,
\qquad
\psi(\theta_0,0)=x_0,
\]
such that, for every \(\theta\in\mathcal U_\theta\),
\[
\calS(\theta)\cap\mathcal W=\{\psi(\theta,u):u\in\mathcal U_u\}.
\]
We keep the same symbol \(\psi\) for its extension to this slightly larger coordinate domain and shrink
\(\mathcal U_u\), if needed, so that \(u\mapsto\psi(\theta,u)\) is injective on a neighborhood of
\(\overline{\mathcal U_u}\) for all \(\theta\in\mathcal U_\theta\). In these coordinates, the local restriction of the
upper-level objective to \(\calS(\theta)\) is
\[
\tilde f(\theta,u):=f(\theta,\psi(\theta,u)).
\]
Since \(f\in\mathcal C^2\) and \(\psi\in\mathcal C^2\), we have \(\tilde f\in\mathcal C^2\).

\paragraph{Step 2: Non-degeneracy gives a smooth local minimizer branch.}
At \(\theta=\theta_0\), the point \(u=0\) corresponds to \(x_0\). Since \(x_0\) minimizes \(f(\theta_0,\cdot)\) over
\(\calS(\theta_0)\), \(u=0\) is a local minimizer of \(\tilde f(\theta_0,\cdot)\), and hence
\[
\nabla_u\tilde f(\theta_0,0)=0.
\]
Let \(A:=D_u\psi(\theta_0,0)\). Then \(A\) has full column rank and
\(\mathrm{range}(A)=\calT_{x_0}^{\theta_0}\), because \(u\mapsto\psi(\theta_0,u)\) is an embedded chart for
\(\calS(\theta_0)\) around \(x_0\). Because \(u=0\) is a critical point of the restricted objective, the
second-derivative terms of the chart do not contribute to the Hessian quadratic form at \(u=0\). Therefore, for every
\(w\neq0\), with \(v:=Aw\),
\[
w^\top\nabla^2_{uu}\tilde f(\theta_0,0)w
=
\big\langle v,\mathrm{Hess}_{\calS(\theta_0)}\bar f_{\theta_0}(x_0)[v]\big\rangle .
\]
The non-degeneracy assumption makes the right-hand side positive for all \(w\neq0\), so
\(\nabla^2_{uu}\tilde f(\theta_0,0)\succ0\). Applying the \(\mathcal C^1\) implicit function theorem to
\[
\nabla_u\tilde f(\theta,u)=0
\]
gives, after shrinking \(\mathcal U_\theta\), a unique \(\mathcal C^1\) map
\[
u^\star:\mathcal U_\theta\to\mathcal U_u,\qquad u^\star(\theta_0)=0,
\]
such that \(\nabla_u\tilde f(\theta,u^\star(\theta))=0\). By continuity of \(\nabla^2_{uu}\tilde f\), we can further
shrink \(\mathcal U_\theta\) and choose a closed ball \(\mathcal B_u\subset\mathcal U_u\), with
\(u^\star(\theta)\in\mathrm{int}(\mathcal B_u)\), such that
\[
\nabla^2_{uu}\tilde f(\theta,u)\succ0,
\qquad
\theta\in\mathcal U_\theta,\ u\in\mathcal B_u .
\]
Since \(\mathcal B_u\) is convex, \(\tilde f(\theta,\cdot)\) is strictly convex on \(\mathcal B_u\). Thus
\(u^\star(\theta)\) is the unique minimizer of \(\tilde f(\theta,\cdot)\) on \(\mathcal B_u\). Define
\[
x_{\mathrm{loc}}(\theta):=\psi(\theta,u^\star(\theta)).
\]
Then \(x_{\mathrm{loc}}\) is \(\mathcal C^1\), belongs to \(\calS(\theta)\), and uniquely minimizes \(f(\theta,\cdot)\)
over the local branch \(\psi(\theta,\mathcal B_u)\).

\paragraph{Step 3: The local branch is the global optimistic branch.}
The previous step only controls the part of \(\calS(\theta)\) near \(x_0\). We now rule out competitors on the rest of
\(\calS(\theta)\). Choose a closed ball \(\mathcal B_u^0\) around \(0\) such that
\(\mathcal B_u^0\subset\mathrm{int}(\mathcal B_u)\), and then choose an open Euclidean ball
\(\mathcal W_0:=\{x:\|x-x_0\|<r\}\) such that \(\overline{\mathcal W_0}\subset\mathcal W\) and
\[
\calS(\theta_0)\cap\overline{\mathcal W_0}
\subset \psi(\theta_0,\mathrm{int}(\mathcal B_u^0)).
\]
After shrinking \(\mathcal U_\theta\), we may assume
\[
\calS(\theta)\cap\overline{\mathcal W_0}\subseteq \psi(\theta,\mathcal B_u),
\qquad \theta\in\mathcal U_\theta .
\]
Indeed, otherwise there would be \(\theta_n\to\theta_0\) and points
\(z_n=\psi(\theta_n,u_n)\in\calS(\theta_n)\cap\overline{\mathcal W_0}\) with \(u_n\notin\mathcal B_u\). Passing to a
subsequence, \(u_n\to\bar u\in\overline{\mathcal U_u}\setminus\mathrm{int}(\mathcal B_u)\), and
\[
z_n=\psi(\theta_n,u_n)\to\psi(\theta_0,\bar u)\in\calS(\theta_0)\cap\overline{\mathcal W_0}.
\]
This contradicts
\(\calS(\theta_0)\cap\overline{\mathcal W_0}\subset\psi(\theta_0,\mathrm{int}(\mathcal B_u^0))\), the inclusion
\(\mathcal B_u^0\subset\mathrm{int}(\mathcal B_u)\), and injectivity of the chart at \(\theta_0\).

Since \(x_0\) is the unique minimizer of \(f(\theta_0,\cdot)\) over the compact set \(\calS(\theta_0)\), there is a
positive gap away from \(\mathcal W_0\):
\[
\Delta_0:=
\min_{x\in\calS(\theta_0)\setminus\mathcal W_0}
\big(f(\theta_0,x)-f(\theta_0,x_0)\big)>0,
\]
with the convention that the minimum over an empty set is \(+\infty\).

We claim that, after shrinking \(\mathcal U_\theta\) once more, no optimistic minimizer at parameter
\(\theta\in\mathcal U_\theta\) lies outside \(\mathcal W_0\). Suppose otherwise. Then there are \(\theta_n\to\theta_0\)
and \(y_n\in\calS(\theta_n)\setminus\mathcal W_0\) such that
\[
f(\theta_n,y_n)\le f(\theta_n,x_{\mathrm{loc}}(\theta_n)).
\]
By \Cref{ass:plcirc}, all lower-level minimizers lie in the fixed compact set \(\mathcal V\). Passing to a subsequence,
\(y_n\to\bar y\in\mathcal V\). To see that \(\bar y\in\calS(\theta_0)\), fix any \(z\in\calS(\theta_0)\). Since
\(y_n\) minimizes \(g(\theta_n,\cdot)\), \(g(\theta_n,y_n)\le g(\theta_n,z)\). Taking \(n\to\infty\) gives
\[
g(\theta_0,\bar y)\le g(\theta_0,z)=\min_x g(\theta_0,x),
\]
so \(\bar y\in\calS(\theta_0)\). Also \(\bar y\notin\mathcal W_0\), because \(\mathcal W_0\) is open and
\(y_n\notin\mathcal W_0\). Taking limits in the displayed inequality and using
\(x_{\mathrm{loc}}(\theta_n)\to x_0\), we get
\[
f(\theta_0,\bar y)\le f(\theta_0,x_0),
\]
which contradicts the positive gap \(\Delta_0\). Therefore every global optimistic minimizer for \(\theta\) close to
\(\theta_0\) lies in \(\mathcal W_0\), hence in the charted branch \(\psi(\theta,\mathcal B_u)\). Since
\(u^\star(\theta)\) is the unique minimizer on this branch, the global optimistic minimizer is unique and equals
\[
x^\star(\theta):=x_{\mathrm{loc}}(\theta).
\]
Thus \(x^\star\) is \(\mathcal C^1\) near \(\theta_0\), and
\[
F(\theta)=f(\theta,x^\star(\theta)).
\]

\paragraph{Step 4: Differentiate the selected branch.}
Because \(x^\star(\theta)\in\calS(\theta)\), lower-level stationarity gives
\[
\nabla_x g(\theta,x^\star(\theta))=0.
\]
Differentiating this identity with respect to \(\theta\) yields
\[
\nabla_{x\theta}^{2}g(\theta,x^\star(\theta))
+\nabla_{xx}^{2}g(\theta,x^\star(\theta))\,\nabla_\theta x^\star(\theta)=0.
\]
Let \(H(\theta):=\nabla_{xx}^{2}g(\theta,x^\star(\theta))\). By \Cref{prop:geometry},
\[
\ker H(\theta)=\calT_{x^\star(\theta)}^\theta,
\qquad
H(\theta)^\dagger H(\theta)=P_{\calN_{x^\star(\theta)}^\theta}.
\]
Multiplying the differentiated stationarity identity by \(H(\theta)^\dagger\) gives the normal component of the branch
sensitivity:
\[
P_{\calN_{x^\star(\theta)}^\theta}\nabla_\theta x^\star(\theta)
=
-H(\theta)^\dagger\nabla_{x\theta}^{2}g(\theta,x^\star(\theta)).
\]
Since \(x^\star(\theta)\) minimizes \(f(\theta,\cdot)\) over \(\calS(\theta)\), the Riemannian first-order condition gives
\(\nabla_x f(\theta,x^\star(\theta))\in\calN_{x^\star(\theta)}^\theta\). Hence the tangent component of
\(\nabla_\theta x^\star(\theta)\) does not contribute to the chain rule:
\[
\begin{aligned}
\nabla F(\theta)
&=\nabla_\theta f(\theta,x^\star(\theta))
+\big(\nabla_\theta x^\star(\theta)\big)^\top\nabla_x f(\theta,x^\star(\theta))\\
&=\nabla_\theta f(\theta,x^\star(\theta))
-\nabla_{\theta x}^{2}g(\theta,x^\star(\theta))\,
H(\theta)^\dagger\nabla_x f(\theta,x^\star(\theta)).
\end{aligned}
\]
This is the pseudoinverse hyper-gradient formula. Finally, \(H(\theta)\) has constant rank along the branch and its
nonzero eigenvalues are bounded away from zero by \Cref{prop:geometry}. Since \(g\in\mathcal C^3\) and
\(x^\star\in\mathcal C^1\), the map \(\theta\mapsto H(\theta)\) is \(\mathcal C^1\), and the Moore--Penrose
pseudoinverse is \(\mathcal C^1\) on the constant-rank stratum. Therefore \(\theta\mapsto H(\theta)^\dagger\) is
\(\mathcal C^1\) after shrinking the neighborhood if needed. The displayed formula, together with \(f\in\mathcal C^2\)
and \(x^\star\in\mathcal C^1\), implies \(\nabla F\in\mathcal C^1\). On a compact smaller neighborhood of \(\theta_0\),
\(D(\nabla F)\) is bounded, so \(\nabla F\) is locally Lipschitz.
\end{proof}

\section{Proofs for Section~\ref{sec:alg}}\label{append:alg-proofs}
This appendix contains the proof of the global smoothness result used in the algorithmic development of \Cref{sec:alg}.

\subsection{Proof of \Cref{prop:F-smooth}}\label{append:F-smooth-proof}
\begin{proof}[Proof of \Cref{prop:F-smooth}]
For each $\bar\theta\in\Theta$, \Cref{ass:unique_min} and \Cref{thm_hyp_obj_diff} give an ambient open neighborhood
$W_{\bar\theta}\subseteq\R^m$ on which the local selected branch exists and $F$ is smooth.
Choose a smaller open neighborhood $\mathcal U_{\bar\theta}$ of $\bar\theta$ such that
$\overline{\mathcal U_{\bar\theta}}\cap\Theta\subset W_{\bar\theta}$, and set
$V_{\bar\theta}:=\mathcal U_{\bar\theta}\cap\Theta$.
On this compact patch there is a finite constant $L_{\bar\theta}$ such that
\begin{equation}\label{eq:local-smoothness-cover}
\|\nabla F(\theta)-\nabla F(\theta')\|
\le
L_{\bar\theta}\|\theta-\theta'\|,
\qquad
\forall\,\theta,\theta'\in V_{\bar\theta}.
\end{equation}

The local constant above can be obtained in the standard hyper-gradient way from the closed-form formula.
On such a local patch, let $x^\star_{\bar\theta}$ denote the selected $\mathcal{C}^{1}$ branch and define
$z(\theta):=(\theta,x^\star_{\bar\theta}(\theta))$. Write, for compactness,
\[
a(z):=\nabla_\theta f(z),\qquad
A(z):=\nabla^2_{\theta x}g(z),\qquad
P(z):=\big[\nabla^2_{xx}g(z)\big]^\dagger,\qquad
b(z):=\nabla_x f(z).
\]
After shrinking the local patch if necessary, choose finite constants $\kappa_{\bar\theta}$,
$B_{A,\bar\theta},B_{P,\bar\theta},B_{b,\bar\theta}$ and
$L_{a,\bar\theta},L_{A,\bar\theta},L_{P,\bar\theta},L_{b,\bar\theta}$ such that
\[
\|z(\theta)-z(\theta')\|
\le
\kappa_{\bar\theta}\|\theta-\theta'\|,
\]
and, on the local graph $\{z(\theta):\theta\in V_{\bar\theta}\}$, the functions
$a,A,P,b$ have Lipschitz constants $L_{a,\bar\theta},L_{A,\bar\theta},L_{P,\bar\theta},L_{b,\bar\theta}$,
while $A,P,b$ are bounded by $B_{A,\bar\theta},B_{P,\bar\theta},B_{b,\bar\theta}$.
The pseudoinverse constants are finite because, along this local graph, the Hessian
$\nabla^2_{xx}g$ has constant rank and its nonzero eigenvalues are bounded away from zero by the normal spectral gap in
\Cref{prop:geometry}. We use the hyper-gradient formula
\[
\nabla F(\theta)=a(z(\theta))-A(z(\theta))P(z(\theta))b(z(\theta))
\]
as follows. Fix $\theta,\theta'\in V_{\bar\theta}$ and abbreviate
$z:=z(\theta)$ and $z':=z(\theta')$. Then
\begin{align*}
\nabla F(\theta)-\nabla F(\theta')
&=
a(z)-a(z')
-\big(A(z)P(z)b(z)-A(z')P(z')b(z')\big),
\end{align*}
and the product difference is decomposed by adding and subtracting
$A(z')P(z)b(z)$ and $A(z')P(z')b(z)$:
\begin{align*}
A(z)P(z)b(z)-A(z')P(z')b(z')
&=
\big(A(z)-A(z')\big)P(z)b(z)\\
&\quad
+A(z')\big(P(z)-P(z')\big)b(z)\\
&\quad
+A(z')P(z')\big(b(z)-b(z')\big).
\end{align*}
Therefore, using the local Lipschitz and boundedness constants above,
\begin{align*}
\|\nabla F(\theta)-\nabla F(\theta')\|
&\le
\|a(z)-a(z')\|
+\|A(z)-A(z')\|\,\|P(z)\|\,\|b(z)\|\\
&\quad
+\|A(z')\|\,\|P(z)-P(z')\|\,\|b(z)\|\\
&\quad
+\|A(z')\|\,\|P(z')\|\,\|b(z)-b(z')\|\\
&\le
\Big(
L_{a,\bar\theta}
+L_{A,\bar\theta}B_{P,\bar\theta}B_{b,\bar\theta}
+B_{A,\bar\theta}L_{P,\bar\theta}B_{b,\bar\theta}
+B_{A,\bar\theta}B_{P,\bar\theta}L_{b,\bar\theta}
\Big)\|z-z'\|.
\end{align*}
Combining this with $\|z-z'\|\le\kappa_{\bar\theta}\|\theta-\theta'\|$ yields the explicit local bound
\begin{equation}\label{eq:local-LF-closed-form}
L_{\bar\theta}
\;:=\;
\kappa_{\bar\theta}
\Big(
L_{a,\bar\theta}
+L_{A,\bar\theta}B_{P,\bar\theta}B_{b,\bar\theta}
+B_{A,\bar\theta}L_{P,\bar\theta}B_{b,\bar\theta}
+B_{A,\bar\theta}B_{P,\bar\theta}L_{b,\bar\theta}
\Big),
\end{equation}
which is one admissible choice in \eqref{eq:local-smoothness-cover}.

The sets $\{V_{\bar\theta}\}_{\bar\theta\in\Theta}$ form an open cover of the compact metric space $\Theta$ (with the relative topology).
Choose a finite subcover $V_1,\ldots,V_J$, with corresponding local constants $L_1,\ldots,L_J$, and define
\begin{equation}\label{eq:explicit-LF}
L_F
:=
\max_{1\le j\le J} L_j
\quad
\left(
\text{equivalently, one may take }
L_j
=
\sup_{\substack{\theta,\theta'\in V_j\\ \theta\neq\theta'}}
\frac{\|\nabla F(\theta)-\nabla F(\theta')\|}{\|\theta-\theta'\|}
\right).
\end{equation}
This quantity is finite by the local smoothness bound \eqref{eq:local-smoothness-cover}.

It remains to show that the local constants globalize without any additional Lipschitz assumption on the selected branch.
Let $\delta>0$ be a Lebesgue number of the finite cover $\{V_j\}_{j=1}^J$ of $\Theta$: every subset of $\Theta$ with diameter
strictly smaller than $\delta$ is contained in some $V_j$.
Fix $\theta,\theta'\in\Theta$. If $\theta=\theta'$, the claim is trivial.
Otherwise, since $\Theta$ is convex, the line segment
\[
\theta_s:=\theta+s(\theta'-\theta),\qquad s\in[0,1],
\]
is contained in $\Theta$. Choose an integer $M$ such that $\|\theta-\theta'\|/M<\delta$, and set
$\theta_\ell:=\theta_{\ell/M}$ for $\ell=0,\ldots,M$.
For each $\ell$, the two-point set $\{\theta_\ell,\theta_{\ell+1}\}$ has diameter below $\delta$, so it lies in some
$V_{j_\ell}$. Applying the local bound on that set gives
\[
\|\nabla F(\theta_{\ell+1})-\nabla F(\theta_\ell)\|
\le
L_{j_\ell}\|\theta_{\ell+1}-\theta_\ell\|
\le
L_F\|\theta_{\ell+1}-\theta_\ell\|.
\]
Summing over $\ell=0,\ldots,M-1$ yields
\[
\|\nabla F(\theta')-\nabla F(\theta)\|
\le
L_F\sum_{\ell=0}^{M-1}\|\theta_{\ell+1}-\theta_\ell\|
=
L_F\|\theta'-\theta\|,
\]
which proves the global smoothness claim.
\end{proof}

\section{Proofs and Technical Details for Section~\ref{sec:theory}}\label{append:outer-proof}
This appendix gives the complete proof pipeline behind the convergence analysis in \Cref{sec:theory}. 
First, \Cref{append_outer_loop_proof} proves the inexact projected-gradient reduction in \Cref{prop:outer-inexact}: once \(F\) is smooth, convergence to stationarity is controlled by the average hyper-gradient error \(\mathbb E\|e_t\|^2\). 
Second, \Cref{append:hg-stability} proves the hyper-gradient stability bound \Cref{thm:et2-bound}. This step compares the exact pseudoinverse hyper-gradient at \(x^\star(\theta_t)\) with the ridge-regularized, CG-approximated hyper-gradient evaluated at the selected candidate \(\tilde x_t\), separating the error into selected-point error, CG residual, ridge bias, and an off-tube moment term.

The remaining subsections bound these terms in algorithmic quantities. 
\Cref{append:selection-error} proves the selection estimate \Cref{thm:selection-error-bound}: it combines curvature control for \(\calS(\theta)\), Gibbs/ULA tube concentration, a hard-selection value-gap argument, and the uniform restricted-curvature bound in \Cref{prop:uniform-riem-nondeg} to bound \(\mathbb E\|\tilde x_t-x^\star(\theta_t)\|^2\). 
Finally, \Cref{append:param-tuning} substitutes the stability and selection estimates into a single parameter-level bound for \(\mathbb E\|e_t\|^2\), chooses \((N,\lambda,\varepsilon_{\mathrm R},\gamma,\eta_t,R_v)\) so that this error is \(\mathcal{O}(\varepsilon^2)\), and derives the conservative oracle-complexity bound reported in \Cref{sec:theory}.

\subsection{Outer-loop guarantee}\label{append_outer_loop_proof}
\begin{proof}[Proof of \Cref{prop:outer-inexact}]
Fix $T\ge 1$ and write $g_t:=\widehat h_t=\nabla F(\theta_t)+e_t$.
Define the (used) gradient mapping
\[
\widetilde{\mathcal G}_t
:=\alpha^{-1}\big(\theta_t-\mathrm{Proj}_{\Theta}(\theta_t-\alpha g_t)\big)
=\alpha^{-1}(\theta_t-\theta_{t+1}).
\]

\paragraph{Step 1: a descent inequality in terms of $\widetilde{\mathcal G}_t$.}
Let $\Delta_t:=\theta_{t+1}-\theta_t=-\alpha\,\widetilde{\mathcal G}_t$.
By $L_F$-smoothness of $F$ on $\Theta$ and convexity of $\Theta$, we have
\[
F(\theta_{t+1})
\le
F(\theta_t)+\langle \nabla F(\theta_t),\Delta_t\rangle+\frac{L_F}{2}\|\Delta_t\|^2.
\]
Since $\theta_{t+1}=\mathrm{Proj}_{\Theta}(\theta_t-\alpha g_t)$ and $\theta_t\in\Theta$, the projection
optimality condition yields
\[
\langle g_t,\Delta_t\rangle\ \le\ -\frac{1}{\alpha}\|\Delta_t\|^2.
\]
	Using $\nabla F(\theta_t)=g_t-e_t$, Cauchy--Schwarz, and the weighted Young inequality
	\[
	uv\le \frac{1}{4\alpha}u^2+\alpha v^2,\qquad u,v\ge0,
	\]
	gives
\begin{align*}
F(\theta_{t+1})
&\le
F(\theta_t)+\langle g_t,\Delta_t\rangle-\langle e_t,\Delta_t\rangle+\frac{L_F}{2}\|\Delta_t\|^2\\
&\le
F(\theta_t)-\frac{1}{\alpha}\|\Delta_t\|^2+\|e_t\|\,\|\Delta_t\|+\frac{L_F}{2}\|\Delta_t\|^2\\
&\le
F(\theta_t)-\frac{1}{\alpha}\|\Delta_t\|^2+\frac{1}{4\alpha}\|\Delta_t\|^2+\alpha\|e_t\|^2+\frac{L_F}{2}\|\Delta_t\|^2\\
&=
F(\theta_t)-\Big(\frac{3}{4\alpha}-\frac{L_F}{2}\Big)\|\Delta_t\|^2+\alpha\|e_t\|^2.
\end{align*}
Using $\alpha\le 1/L_F$ gives $\frac{3}{4\alpha}-\frac{L_F}{2}\ge \frac{1}{4\alpha}$, hence
\[
F(\theta_{t+1})
\le
F(\theta_t)-\frac{1}{4\alpha}\|\Delta_t\|^2+\alpha\|e_t\|^2
=
F(\theta_t)-\frac{\alpha}{4}\|\widetilde{\mathcal G}_t\|^2+\alpha\|e_t\|^2.
\]
Summing over $t=0,\dots,T-1$ and using $F(\theta_T)\ge F_\star$ yields
\begin{equation}\label{eq:sum-Gtilde}
\sum_{t=0}^{T-1}\|\widetilde{\mathcal G}_t\|^2
\le
\frac{4\,(F(\theta_0)-F_\star)}{\alpha}
+4\sum_{t=0}^{T-1}\|e_t\|^2.
\end{equation}

\paragraph{Step 2: relate $\widetilde{\mathcal G}_t$ to the true gradient mapping.}
By nonexpansiveness of projection,
\[
\|\mathcal G_{\Theta}(\theta_t,\nabla F(\theta_t);\alpha)-\widetilde{\mathcal G}_t\|
=
\frac{1}{\alpha}\big\|\mathrm{Proj}_{\Theta}(\theta_t-\alpha g_t)-\mathrm{Proj}_{\Theta}(\theta_t-\alpha\nabla F(\theta_t))\big\|
\le \|e_t\|.
\]
Therefore $\|\mathcal G_{\Theta}(\theta_t,\nabla F(\theta_t);\alpha)\|^2\le 2\|\widetilde{\mathcal G}_t\|^2+2\|e_t\|^2$,
and combining with \eqref{eq:sum-Gtilde} gives
\[
\sum_{t=0}^{T-1}\|\mathcal G_{\Theta}(\theta_t,\nabla F(\theta_t);\alpha)\|^2
\le
\frac{8\,(F(\theta_0)-F_\star)}{\alpha}
+10\sum_{t=0}^{T-1}\|e_t\|^2.
\]
Dividing by $T$ and taking expectations yields \eqref{eq:outer-inexact}.
\end{proof}

\subsection[Hyper-gradient stability]{Hyper-gradient stability and proof of \Cref{thm:et2-bound}}\label{append:hg-stability}

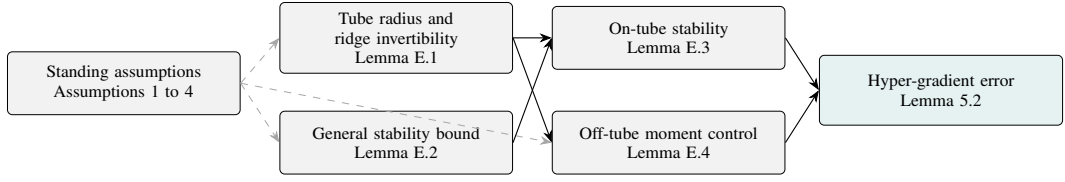
\begin{figure}[t]
\centering
\resizebox{\textwidth}{!}{%
\begin{tikzpicture}[
    >=Stealth,
    box/.style={
        draw,
        rounded corners=2pt,
        align=center,
        font=\scriptsize,
        inner xsep=4pt,
        inner ysep=4pt,
        minimum height=9mm,
        text width=3.05cm
	    },
	    auxbox/.style={box, fill=gray!10},
	    mainbox/.style={box, fill=teal!10, text width=3.2cm, minimum height=10mm},
	    arrow/.style={->, line width=0.45pt},
	    auxarrow/.style={->, dashed, line width=0.4pt, draw=gray!70}
]

	\node[auxbox] (assump) at (0,0)
	{Standing assumptions\\
	\Cref{ass:plcirc,ass:curvature_g,ass:unique_min,assum_sample_comp}};
	
	\node[auxbox] (tubeinv) at (3.9,0.65)
	{Tube radius and ridge invertibility\\
	\Cref{lem:tube-invertibility}};
	\node[auxbox] (detstab) at (3.9,-0.85)
	{General stability bound\\
	\Cref{lem:hg-stability-detailed}};
	\node[auxbox] (ontube) at (7.8,0.65)
	{On-tube stability\\
	\Cref{lem:hg-stability}};
	\node[auxbox] (offtube) at (7.8,-0.85)
	{Off-tube moment control\\
	\Cref{lem:off-tube-et}};
	
	\node[mainbox] (etmain) at (11.7,-0.1)
	{Hyper-gradient error\\
	\Cref{thm:et2-bound}};
	
	\draw[auxarrow] (assump.east) -- (tubeinv.west);
	\draw[auxarrow] (assump.east) -- (detstab.west);
	\draw[auxarrow] (assump.east) -- (offtube.west);
	\draw[arrow] (tubeinv.east) -- (ontube.west);
	\draw[arrow] (detstab.east) -- (ontube.west);
	\draw[arrow] (tubeinv.east) -- (offtube.west);
	\draw[arrow] (ontube.east) -- (etmain.west);
	\draw[arrow] (offtube.east) -- (etmain.west);
	
	\end{tikzpicture}%
	}
	\caption{\textbf{Hyper-gradient error proof for \Cref{thm:et2-bound}.}
	This figure summarizes the proof strategy presented in Appendix~\ref{append:hg-stability}: tube invertibility, deterministic stability, its on-tube specialization, and off-tube clipping control are combined to prove Lemma~\ref{thm:et2-bound}.}
	\label{fig:appendix-d-hg-flow}
	\end{figure}

We prove \Cref{thm:et2-bound} in the same order as the terms in \eqref{eq:et2-explicit}.
First, \Cref{lem:tube-invertibility} gives a quantitative tube radius \(r(\gamma)\) on which the ridge matrix is uniformly invertible.
Second, \Cref{lem:hg-stability-detailed} proves a deterministic stability estimate for a ridge-regularized, inexact linear solve.
Third, \Cref{lem:hg-stability} specializes this estimate to the tube event \(\mathcal E_t\).
Fourth, \Cref{lem:off-tube-et} controls the complement \(\mathcal E_t^c\) using the clipping safeguard and the
polynomial-growth condition on \(\nabla^2_{\theta x}g\) in \Cref{assum_sample_comp}.
The final proof then splits the second moment over \(\mathcal E_t\) and \(\mathcal E_t^c\).

\Cref{fig:appendix-d-hg-flow} summarizes the proof strategy and dependencies for this section.

\begin{lemma}[Tube radius ensuring ridge invertibility]\label{lem:tube-invertibility}
Let \Cref{ass:plcirc,ass:curvature_g} hold.
Recall from \Cref{ass:curvature_g} that \(\rho\) is the radius of the tube around \(\calS(\theta)\) on which the
third-derivative bound \(L_{g,3}\) holds. Let \(\bar L_{g,3}:=\max\{L_{g,3},1\}\) and define
\[
r(\gamma):=\min\left\{\rho,\frac{\gamma}{2\bar L_{g,3}}\right\}.
\]
Then for all $\theta\in\Theta$ and all
$x\in\R^d$ satisfying $\mathrm{dist}(x,\calS(\theta))\le r(\gamma)$, we have
\[
\nabla^2_{xx}g(\theta,x)+\gamma I \ \succeq\ \tfrac{\gamma}{2}\,I,
\qquad\text{and hence}\qquad
\big\|\big(\nabla^2_{xx}g(\theta,x)+\gamma I\big)^{-1}\big\|\ \le\ \tfrac{2}{\gamma}.
\]
\end{lemma}
\begin{proof}
Let \(y\in\calS(\theta)\) be a nearest projection of \(x\), so \(\|x-y\|=\mathrm{dist}(x,\calS(\theta))\le r(\gamma)\).
By \Cref{prop:geometry}, \(\nabla^2_{xx}g(\theta,y)\succeq 0\).
Since \(r(\gamma)\le\rho\), the segment from \(y\) to \(x\) stays inside the \(\rho\)-tube around \(\calS(\theta)\).
Assumption~\ref{ass:curvature_g} therefore gives
\[
\|\nabla^2_{xx}g(\theta,x)-\nabla^2_{xx}g(\theta,y)\|
\le L_{g,3}\|x-y\|
\le \bar L_{g,3}r(\gamma)
\le \gamma/2.
\]
Weyl's inequality yields \(\lambda_{\min}(\nabla^2_{xx}g(\theta,x))\ge-\gamma/2\).
Thus \(\nabla^2_{xx}g(\theta,x)+\gamma I\succeq(\gamma/2)I\), and the inverse-norm bound follows.
\end{proof}

Before stating the stability lemma, we clarify the notation for the selected branch and the candidate point. By \Cref{ass:unique_min}, for every
\(\theta\in\Theta\) there is a unique optimistic minimizer, which we denote by \(x^\star(\theta)\). For each
\(\bar\theta\in\Theta\), \Cref{thm_hyp_obj_diff} gives a neighborhood of \(\bar\theta\) on which this same selection
coincides with a local \(\mathcal{C}^{1}\) branch. On overlaps, the local branches agree by uniqueness. Thus \(x^\star\) is a
globally defined selection on \(\Theta\), and it is locally \(\mathcal{C}^{1}\) around every point of \(\Theta\).
In \Cref{lem:hg-stability-detailed}, \(\tilde x_t\) is a fixed candidate point at which the inexact implicit gradient is evaluated.
The lemma is deterministic and does not use how \(\tilde x_t\) was generated; in the ULA/Gibbs application below,
\(\tilde x_t=\hat x_N^\lambda(\theta_t)\), the \BoN-selected point returned by the ULA candidate generator.

\begin{lemma}[General stability bound (technical)]\label{lem:hg-stability-detailed}
Let \Cref{ass:plcirc,ass:unique_min,assum_sample_comp} hold, and let $x^\star:\Theta\to\R^d$ denote the unique optimistic
selection described above.
Fix \(\gamma>0\), an iteration $t$, a parameter $\theta_t\in\Theta$, and a candidate point $\tilde x_t\in\R^d$.
Let $\mathcal V\subset\R^d$ be any compact convex set containing $\{x^\star(\theta):\theta\in\Theta\}\cup\{\tilde x_t\}$.
Assume \(\nabla^2_{xx}g(\theta_t,\tilde x_t)+\gamma I\) is invertible and satisfies
\[
\big\|\big(\nabla^2_{xx}g(\theta_t,\tilde x_t)+\gamma I\big)^{-1}\big\|\le \kappa_\gamma .
\]
Let $\tilde v_t$ satisfy the residual bound
\[
\big\|\big(\nabla^2_{xx}g(\theta_t,\tilde x_t)+\gamma I\big)\tilde v_t-\nabla_x f(\theta_t,\tilde x_t)\big\|\ \le\ \eta_t.
\]
Define
\[
\widehat h_t
:=
\nabla_\theta f(\theta_t,\tilde x_t)-\nabla^2_{\theta x}g(\theta_t,\tilde x_t)\tilde v_t,
\qquad
e_t:=\widehat h_t-\nabla F(\theta_t).
\]
Then
\[
\|e_t\|\ \le\ C_x\Bigl(1+\kappa_\gamma+\frac{\kappa_\gamma}{c+\gamma}\Bigr)\,\|\tilde x_t-x^\star(\theta_t)\|
\ +\ C_{\mathrm{lin}}\cdot\kappa_\gamma\,\eta_t \ +\ C_{\mathrm{reg}}\cdot\gamma,
\]
where one admissible choice of constants is
\[
\begin{aligned}
C_{\mathrm{lin}}
&:=\sup_{(\theta,x)\in\Theta\times\mathcal V}\|\nabla^2_{\theta x}g(\theta,x)\|,\\[2pt]
C_x
&:=L_{\theta f,x} + L_{\theta x,x}\,B_{f,x} + B_{\theta x}\,L_{f_x,x} + B_{\theta x}\,L_{g_{xx},x}\,B_{f,x},\\
C_{\mathrm{reg}}
&:=\sup_{\theta\in\Theta}
\frac{\|\nabla^2_{\theta x}g(\theta,x^\star(\theta))\|\,\|\nabla_x f(\theta,x^\star(\theta))\|}{c^2},
\end{aligned}
\]
where $c>0$ is the normal spectral gap constant from \Cref{prop:geometry},
\[
\begin{aligned}
B_{\theta x}
&:=\sup_{(\theta,x)\in\Theta\times\mathcal V}\|\nabla^2_{\theta x}g(\theta,x)\|,\qquad
B_{f,x}
:=\sup_{(\theta,x)\in\Theta\times\mathcal V}\|\nabla_x f(\theta,x)\|,
\end{aligned}
\]
and $L_{\theta f,x},L_{\theta x,x},L_{f_x,x},L_{g_{xx},x}<\infty$ are uniform Lipschitz constants on $\mathcal V$
(uniform over $\theta\in\Theta$) for
$x\mapsto \nabla_\theta f(\theta,x)$, $x\mapsto \nabla^2_{\theta x}g(\theta,x)$,
$x\mapsto \nabla_x f(\theta,x)$, and $x\mapsto \nabla^2_{xx}g(\theta,x)$, respectively.
Such constants exist and are finite because these maps are continuously differentiable and $\Theta\times\mathcal V$ is compact and convex.
\end{lemma}
\begin{proof}[Proof of \Cref{lem:hg-stability-detailed}]
Write $x_t^\star:=x^\star(\theta_t)$.
By \Cref{thm_hyp_obj_diff} and \Cref{ass:unique_min},
\[
\nabla F(\theta_t)
=
\nabla_\theta f(\theta_t,x_t^\star)
-\nabla^2_{\theta x}g(\theta_t,x_t^\star)
\big[\nabla^2_{xx}g(\theta_t,x_t^\star)\big]^\dagger
\nabla_x f(\theta_t,x_t^\star).
\]
Let
\[
v_{t,\gamma}
:=
\big(\nabla^2_{xx}g(\theta_t,\tilde x_t)+\gamma I\big)^{-1}\nabla_x f(\theta_t,\tilde x_t).
\]
Add and subtract the ridge-regularized expressions at \(\tilde x_t\) and \(x_t^\star\), and write
\[
e_t=(I)+(II)+(III),
\]
where
\[
(I):=\nabla^2_{\theta x}g(\theta_t,\tilde x_t)\,(v_{t,\gamma}-\tilde v_t),
\]
\[
\begin{aligned}
(II):={}&
\nabla_\theta f(\theta_t,\tilde x_t)
-\nabla^2_{\theta x}g(\theta_t,\tilde x_t)v_{t,\gamma}
-\nabla_\theta f(\theta_t,x_t^\star)\\
&\quad
+\nabla^2_{\theta x}g(\theta_t,x_t^\star)
\big(\nabla^2_{xx}g(\theta_t,x_t^\star)+\gamma I\big)^{-1}
\nabla_x f(\theta_t,x_t^\star),
\end{aligned}
\]
and
\[
(III):=
\nabla^2_{\theta x}g(\theta_t,x_t^\star)
\Big(
\big[\nabla^2_{xx}g(\theta_t,x_t^\star)\big]^\dagger
-\big(\nabla^2_{xx}g(\theta_t,x_t^\star)+\gamma I\big)^{-1}
\Big)
\nabla_x f(\theta_t,x_t^\star).
\]

For term (II), the \(\nabla_\theta f\) difference is bounded by
\[
\|\nabla_\theta f(\theta_t,\tilde x_t)-\nabla_\theta f(\theta_t,x_t^\star)\|
\le L_{\theta f,x}\|\tilde x_t-x_t^\star\|.
\]
For the product difference, add and subtract
\[
\nabla^2_{\theta x}g(\theta_t,x_t^\star)
\big(\nabla^2_{xx}g(\theta_t,\tilde x_t)+\gamma I\big)^{-1}
\nabla_x f(\theta_t,\tilde x_t)
\]
and
\[
\nabla^2_{\theta x}g(\theta_t,x_t^\star)
\big(\nabla^2_{xx}g(\theta_t,\tilde x_t)+\gamma I\big)^{-1}
\nabla_x f(\theta_t,x_t^\star).
\]
This gives
\[
\begin{aligned}
\MoveEqLeft[2]\Big\|
\nabla^2_{\theta x}g(\theta_t,\tilde x_t)
\big(\nabla^2_{xx}g(\theta_t,\tilde x_t)+\gamma I\big)^{-1}
\nabla_x f(\theta_t,\tilde x_t)\\
&\quad-
\nabla^2_{\theta x}g(\theta_t,x_t^\star)
\big(\nabla^2_{xx}g(\theta_t,x_t^\star)+\gamma I\big)^{-1}
\nabla_x f(\theta_t,x_t^\star)
\Big\|\\
&\le
\|\nabla^2_{\theta x}g(\theta_t,\tilde x_t)-\nabla^2_{\theta x}g(\theta_t,x_t^\star)\|\,
\big\|\big(\nabla^2_{xx}g(\theta_t,\tilde x_t)+\gamma I\big)^{-1}\big\|\,
\|\nabla_x f(\theta_t,\tilde x_t)\|\\
&\quad+
\|\nabla^2_{\theta x}g(\theta_t,x_t^\star)\|\,
\big\|\big(\nabla^2_{xx}g(\theta_t,\tilde x_t)+\gamma I\big)^{-1}\big\|\,
\|\nabla_x f(\theta_t,\tilde x_t)-\nabla_x f(\theta_t,x_t^\star)\|\\
&\quad+
\|\nabla^2_{\theta x}g(\theta_t,x_t^\star)\|\,
\Big\|
\Big[
\big(\nabla^2_{xx}g(\theta_t,\tilde x_t)+\gamma I\big)^{-1}\\
&\hspace{39mm}
-\big(\nabla^2_{xx}g(\theta_t,x_t^\star)+\gamma I\big)^{-1}
\Big]\nabla_x f(\theta_t,x_t^\star)
\Big\|.
\end{aligned}
\]
The first two terms are bounded by
\[
L_{\theta x,x}\kappa_\gamma B_{f,x}\|\tilde x_t-x_t^\star\|
\quad\text{and}\quad
B_{\theta x}\kappa_\gamma L_{f_x,x}\|\tilde x_t-x_t^\star\|.
\]
For the last term, use the resolvent identity to obtain
\[
\begin{aligned}
\MoveEqLeft[2]\Big\|
\Big[
\big(\nabla^2_{xx}g(\theta_t,\tilde x_t)+\gamma I\big)^{-1}
-\big(\nabla^2_{xx}g(\theta_t,x_t^\star)+\gamma I\big)^{-1}
\Big]\nabla_x f(\theta_t,x_t^\star)
\Big\|\\
&\le
\kappa_\gamma\,
\|\nabla^2_{xx}g(\theta_t,\tilde x_t)-\nabla^2_{xx}g(\theta_t,x_t^\star)\|\,
\big\|\big(\nabla^2_{xx}g(\theta_t,x_t^\star)+\gamma I\big)^{-1}
\nabla_x f(\theta_t,x_t^\star)\big\|.
\end{aligned}
\]
Since $x_t^\star\in\calS(\theta_t)$ and \(\nabla_x f(\theta_t,x_t^\star)\in\calN_{x_t^\star}^{\theta_t}\), \Cref{prop:geometry} implies
\[
\big\|\big(\nabla^2_{xx}g(\theta_t,x_t^\star)+\gamma I\big)^{-1}
\nabla_x f(\theta_t,x_t^\star)\big\|
\le
\frac{\|\nabla_x f(\theta_t,x_t^\star)\|}{c+\gamma}
\le
\frac{B_{f,x}}{c+\gamma}.
\]
Moreover,
\[
\|\nabla^2_{xx}g(\theta_t,\tilde x_t)-\nabla^2_{xx}g(\theta_t,x_t^\star)\|
\le
L_{g_{xx},x}\|\tilde x_t-x_t^\star\|.
\]
Therefore, the last product term is bounded by
\[
B_{\theta x}\kappa_\gamma L_{g_{xx},x}B_{f,x}\,
\frac{\|\tilde x_t-x_t^\star\|}{c+\gamma}.
\]
Combining these bounds yields
\[
\|(II)\|
\le C_x\Bigl(1+\kappa_\gamma+\frac{\kappa_\gamma}{c+\gamma}\Bigr)\,\|\tilde x_t-x_t^\star\|.
\]

For term (I), define the residual
\[
r_t:=
\big(\nabla^2_{xx}g(\theta_t,\tilde x_t)+\gamma I\big)\tilde v_t-\nabla_x f(\theta_t,\tilde x_t).
\]
Then
\[
\tilde v_t-v_{t,\gamma}
=
\big(\nabla^2_{xx}g(\theta_t,\tilde x_t)+\gamma I\big)^{-1}r_t,
\]
and hence
\[
\|(I)\|
=
\|\nabla^2_{\theta x}g(\theta_t,\tilde x_t)(\tilde v_t-v_{t,\gamma})\|
\le
\|\nabla^2_{\theta x}g(\theta_t,\tilde x_t)\|\,
\big\|\big(\nabla^2_{xx}g(\theta_t,\tilde x_t)+\gamma I\big)^{-1}\big\|\,
\|r_t\|
\le C_{\mathrm{lin}}\,\kappa_\gamma\,\eta_t.
\]

	For term (III), note that $x_t^\star$ is a minimizer of $f(\theta_t,\cdot)$ over the manifold $\calS(\theta_t)$, so
	the first-order condition implies $\nabla_x f(\theta_t,x_t^\star)\in\calN_{x_t^\star}^{\theta_t}$.
	By \Cref{prop:geometry}, the restriction of $\nabla^2_{xx}g(\theta_t,x_t^\star)$ to $\calN_{x_t^\star}^{\theta_t}$ has
	eigenvalues in $[c,\infty)$.
		Let $\{u_i\}_{i=1}^{d-\dimk}$ be an orthonormal eigenbasis of $\calN_{x_t^\star}^{\theta_t}$ with
		$\nabla^2_{xx}g(\theta_t,x_t^\star)u_i=\lambda_i u_i$ and $\lambda_i\ge c$.
	Since $\nabla_x f(\theta_t,x_t^\star)\in\calN_{x_t^\star}^{\theta_t}$, we can write
	$\nabla_x f(\theta_t,x_t^\star)=\sum_{i=1}^{d-\dimk} \beta_i u_i$.
	On $\calN_{x_t^\star}^{\theta_t}$, the pseudoinverse coincides with the inverse, hence
	\[
	\big[\nabla^2_{xx}g(\theta_t,x_t^\star)\big]^\dagger \nabla_x f(\theta_t,x_t^\star)
	=
	\sum_{i=1}^{d-\dimk}(\beta_i/\lambda_i)u_i,
	\]
	while
	\[
	\big(\nabla^2_{xx}g(\theta_t,x_t^\star)+\gamma I\big)^{-1}\nabla_x f(\theta_t,x_t^\star)
	=
	\sum_{i=1}^{d-\dimk}\big(\beta_i/(\lambda_i+\gamma)\big)u_i.
	\]
		Therefore
		\[
		\begin{aligned}
		&\big(\nabla^2_{xx}g(\theta_t,x_t^\star)+\gamma I\big)^{-1}\nabla_x f(\theta_t,x_t^\star)
		-\big[\nabla^2_{xx}g(\theta_t,x_t^\star)\big]^\dagger\nabla_x f(\theta_t,x_t^\star)\\
		&\qquad =
		\sum_{i=1}^{d-\dimk} \beta_i\Big(\frac{1}{\lambda_i+\gamma}-\frac{1}{\lambda_i}\Big)u_i .
		\end{aligned}
		\]
		Hence
		\[
		\begin{aligned}
		&\big\|\big(\nabla^2_{xx}g(\theta_t,x_t^\star)+\gamma I\big)^{-1}\nabla_x f(\theta_t,x_t^\star)
		-\big[\nabla^2_{xx}g(\theta_t,x_t^\star)\big]^\dagger\nabla_x f(\theta_t,x_t^\star)\big\|\\
		&\qquad\le
		\max_{i=1,\dots,d-\dimk}\Big|\frac{1}{\lambda_i+\gamma}-\frac{1}{\lambda_i}\Big|\,
		\|\nabla_x f(\theta_t,x_t^\star)\|.
		\end{aligned}
		\]
	Using $\big|\frac{1}{\lambda+\gamma}-\frac{1}{\lambda}\big|=\frac{\gamma}{\lambda(\lambda+\gamma)}$ and $\lambda_i\ge c$ gives
	\begin{align*}
	\MoveEqLeft[2]\big\|\big(\nabla^2_{xx}g(\theta_t,x_t^\star)+\gamma I\big)^{-1}\nabla_x f(\theta_t,x_t^\star)
	-\big[\nabla^2_{xx}g(\theta_t,x_t^\star)\big]^\dagger\nabla_x f(\theta_t,x_t^\star)\big\|\\
	&\le \frac{\gamma}{c(c+\gamma)}\,\|\nabla_x f(\theta_t,x_t^\star)\|
	\le \frac{\gamma}{c^2}\,\|\nabla_x f(\theta_t,x_t^\star)\|.
	\end{align*}
	Multiplying by $\|\nabla^2_{\theta x}g(\theta_t,x_t^\star)\|$ and taking the supremum over $\theta\in\Theta$ yields
	$\|(III)\|\le C_{\mathrm{reg}}\gamma$.
	Combining the three bounds gives the claim.
	\end{proof}

\noindent \Cref{lem:hg-stability} follows by taking a fixed compact set containing the \(\rho\)-tube around all \(\calS(\theta)\), restricting to
$\mathrm{dist}(\tilde x_t,\calS(\theta_t))\le r(\gamma)$, and applying \Cref{lem:tube-invertibility} to bound $\kappa_\gamma\le 2/\gamma$.
\begin{lemma}[Stability of the inexact implicit term on the tube]\label{lem:hg-stability}
Assume the hypotheses of Lemma~\ref{lem:hg-stability-detailed} hold for the same \(\theta_t,\tilde x_t,\tilde v_t\), and fix $\gamma>0$ with tube radius $r(\gamma)$ from Lemma~\ref{lem:tube-invertibility}.
Let $\mathcal V=\mathbb B_d(0;\mathsf D+\rho)$.
Fix an iteration $t$ with $\theta_t\in\Theta$ and $\mathrm{dist}(\tilde x_t,\calS(\theta_t))\le r(\gamma)$, so that $\tilde x_t\in\mathcal V$.
Assume clipping is inactive, i.e., the clipped vector used by the algorithm satisfies $v_t=\tilde v_t$.
Then the algorithmic hyper-gradient error $e_t=\widehat h_t-\nabla F(\theta_t)$ obeys
\[
\|e_t\|
\ \le\
A_\gamma\,\|\tilde x_t-x^\star(\theta_t)\|
\ +\ \tfrac{2C_{\mathrm{lin}}}{\gamma}\,\eta_t
\ +\ C_{\mathrm{reg}}\,\gamma,
\]
where $C_{\mathrm{lin}},C_{\mathrm{reg}}$ are as in Lemma~\ref{lem:hg-stability-detailed} and
\[
A_\gamma
:=
C_x\Bigl(1+\tfrac{2}{\gamma}+\tfrac{2}{\gamma(c+\gamma)}\Bigr),
\]
with $c>0$ the normal spectral gap constant from \Cref{prop:geometry} and $C_x$ as in Lemma~\ref{lem:hg-stability-detailed}.
\end{lemma}
\begin{proof}
Because \(r(\gamma)\le\rho\) and \(\calS(\theta_t)\subseteq\mathbb B_d(0;\mathsf D)\), the tube condition implies \(\tilde x_t\in\mathcal V\).
Under this same condition, Lemma~\ref{lem:tube-invertibility} gives $\kappa_\gamma\le 2/\gamma$.
Since clipping is inactive, the algorithmic \(\widehat h_t\), which is formed with \(v_t\), coincides with the quantity in Lemma~\ref{lem:hg-stability-detailed}, which is formed with \(\tilde v_t\).
Applying Lemma~\ref{lem:hg-stability-detailed} with $\mathcal V=\mathbb B_d(0;\mathsf D+\rho)$ and substituting this inverse-norm bound gives the claim.
\end{proof}

\begin{lemma}[Bounding the off-tube term under clipping]\label{lem:off-tube-et}
Fix $\gamma>0$ and let \(r(\gamma)\) be the tube radius from Lemma~\ref{lem:tube-invertibility}.
Assume Algorithm~\ref{alg:hg-minsel-lmc} uses the safeguard $v_t=\mathrm{Proj}_{\mathbb B_d(0;R_v)}(\tilde v_t)$, and assume the CG output $\tilde v_t$ satisfies $\|b_t-(H_t+\gamma I)\tilde v_t\|\le \eta_t$.
Let \(\mathsf{D}_{t}:=\mathrm{dist}(\tilde x_t,\calS(\theta_t))\).
Define the tube event \(\mathcal E_t:=\{\mathsf{D}_{t}\le r(\gamma)\}\).
Under \Cref{assum_sample_comp} and the smoothness of \(F\) on compact \(\Theta\), there exists a constant \(C_{\mathrm{off}}<\infty\), independent of \(t,\gamma,\eta_t,R_v\), such that
\[
\|e_t\|^2
\ \le\
C_{\mathrm{off}}(1+R_v^2)(1+\mathsf{D}_{t}^{2\mathsf n_{\theta x}}).
\]
Consequently,
\[
\mathbb E\big[\|e_t\|^2\mathbf 1_{\mathcal E_t^c}\big]
\ \le\
C_{\mathrm{off}}(1+R_v^2)\,
\mathbb E\big[(1+\mathsf{D}_{t}^{2\mathsf n_{\theta x}})\mathbf 1_{\mathcal E_t^c}\big].
\]
Moreover, on $\mathcal E_t$ Lemma~\ref{lem:tube-invertibility} implies $\|\tilde v_t\|\le (2/\gamma)(L_{f,1}+\eta_t)$, so choosing
$R_v\ge (2/\gamma)(L_{f,1}+\eta_t)$ ensures the safeguard is inactive on $\mathcal E_t$ (i.e., $v_t=\tilde v_t$).
\end{lemma}
\begin{proof}
Since $\|v_t\|\le R_v$ by construction, \Cref{assum_sample_comp} gives
\[
\|\widehat h_t\|
\le \|\nabla_\theta f(\theta_t,\tilde x_t)\|+\|\nabla^2_{\theta x}g(\theta_t,\tilde x_t)\|\,\|v_t\|
\le L_{f,1}+\big(C_{\theta x}\|\tilde x_t\|^{\mathsf n_{\theta x}}+D_{\theta x}\big)R_v.
\]
Since $F$ is smooth on compact \(\Theta\), \(B_F:=\sup_{\theta\in\Theta}\|\nabla F(\theta)\|<\infty\).
Using \((a+b+c)^2\le 3(a^2+b^2+c^2)\), we obtain
\[
\|e_t\|^2\le C(1+R_v^2)(1+\|\tilde x_t\|^{2\mathsf n_{\theta x}})
\]
for a finite constant \(C\).
Because \(\calS(\theta_t)\subseteq\mathbb B_d(0;\mathsf D)\), we have \(\|\tilde x_t\|\le \mathsf D+\mathsf{D}_{t}\), and hence
\((1+\|\tilde x_t\|^{2\mathsf n_{\theta x}})\le C'(1+\mathsf{D}_{t}^{2\mathsf n_{\theta x}})\).
This proves the off-tube moment bound.

For the final claim, on \(\mathcal E_t\), Lemma~\ref{lem:tube-invertibility} and the residual relation
$(\nabla^2_{xx}g(\theta_t,\tilde x_t)+\gamma I)\tilde v_t=\nabla_x f(\theta_t,\tilde x_t)+r_t$ with $\|r_t\|\le\eta_t$ give
\[
\|\tilde v_t\|
\le \frac{2}{\gamma}\big(\|\nabla_x f(\theta_t,\tilde x_t)\|+\eta_t\big)
\le \frac{2}{\gamma}(L_{f,1}+\eta_t),
\]
where the last inequality uses the global Lipschitzness of \(f\).
Under the stated choice of \(R_v\), this gives \(\|\tilde v_t\|\le R_v\), so the Euclidean projection onto \(\mathbb B_d(0;R_v)\) leaves \(\tilde v_t\) unchanged.
\end{proof}

\subsubsection{Proof of \Cref{thm:et2-bound}}\label{append:proof-et2}
\begin{proof}
Fix an iteration $t$.
Split
\[
\mathbb E\|e_t\|^2
=
\mathbb E\big[\|e_t\|^2\mathbf 1_{\mathcal E_t}\big]
\ +\
\mathbb E\big[\|e_t\|^2\mathbf 1_{\mathcal E_t^c}\big].
\]

\paragraph{On the tube.}
On $\mathcal E_t$, our choice of clipping radius ensures $v_t=\tilde v_t$ (Lemma~\ref{lem:off-tube-et}).
Therefore Lemma~\ref{lem:hg-stability} gives
\[
\|e_t\|
\ \le\
A_\gamma\,\|\tilde x_t-x^\star(\theta_t)\|
\ +\ \tfrac{2C_{\mathrm{lin}}}{\gamma}\,\eta_t
\ +\ C_{\mathrm{reg}}\,\gamma.
\]
Using $(a+b+c)^2\le 3(a^2+b^2+c^2)$ and that $\mathbf 1_{\mathcal E_t}\le 1$ yields
\[
\mathbb E\big[\|e_t\|^2\mathbf 1_{\mathcal E_t}\big]
\ \le\
3A_\gamma^2\,\mathbb E\big[\|\tilde x_t-x^\star(\theta_t)\|^2\mathbf 1_{\mathcal E_t}\big]
\ +\ \tfrac{12C_{\mathrm{lin}}^2}{\gamma^2}\,\eta_t^2
\ +\ 3C_{\mathrm{reg}}^2\,\gamma^2.
\]
Recalling $A_\gamma=C_x(1+\tfrac{2}{\gamma}+\tfrac{2}{\gamma(c+\gamma)})$ from Lemma~\ref{lem:hg-stability} gives the first three terms in \eqref{eq:et2-explicit}.

\paragraph{Off the tube.}
Lemma~\ref{lem:off-tube-et} yields
\[
\mathbb E\big[\|e_t\|^2\mathbf 1_{\mathcal E_t^c}\big]
\le
C_{\mathrm{off}}(1+R_v^2)\,
\mathbb E\big[(1+\mathsf{D}_{t}^{2\mathsf n_{\theta x}})\mathbf 1_{\mathcal E_t^c}\big].
\]

Combining the two bounds gives \eqref{eq:et2-explicit}.
\end{proof}


\subsection{Selection Error and Proof of \Cref{thm:selection-error-bound}}\label{append:selection-error}

We prove \Cref{thm:selection-error-bound} by reducing the selected-point error to two simpler quantities.
Fix \(\theta\) and let \(\tilde x\) be the hard-selected candidate. Since \(\tilde x\) is generally not on the minimizer
manifold \(\calS(\theta)\), we let \(\bar x\) be its Euclidean projection onto \(\calS(\theta)\). The proof then follows the
decomposition
\[
\|\tilde x-x^\star(\theta)\|^2
\ \le\
2\,\mathrm{dist}(\tilde x,\calS(\theta))^2
+2\,\|\bar x-x^\star(\theta)\|^2 .
\]
	The first term is an \emph{off-manifold tube error}: it is controlled by showing that Gibbs/ULA candidates remain close to
	\(\calS(\theta)\) with high probability. The second term is an \emph{on-manifold selection error}: hard selection first gives
	a small value gap \(f(\theta,\bar x)-F(\theta)\), and \Cref{prop:uniform-riem-nondeg} converts this value gap
	into squared distance on \(\calS(\theta)\). More concretely, the proof uses the following sequence:
\begin{enumerate}[leftmargin=*]
\item establish geometric and sampling ingredients: curvature control, global quadratic growth of \(g\), tube concentration,
and Poincar\'e/R\'enyi transportation control for approximate ULA samples;
\item prove that hard selection gives a small on-manifold value gap for \(\bar x\);
\item convert the on-manifold value gap into \(\|\bar x-x^\star(\theta)\|^2\) using uniform quadratic growth of
\(f|_{\calS(\theta)}\) and a positive away-from-minimizer gap;
\item combine the tube error and the on-manifold distance bound to obtain \eqref{eq:final-Ex2-bound}.
\end{enumerate}
\Cref{fig:appendix-d-selection-flow} summarizes this dependency structure.

\begin{figure}[t]
\centering
\resizebox{\textwidth}{!}{%
\begin{tikzpicture}[
    >=Stealth,
    proofbox/.style={
        draw,
        rounded corners=2pt,
        align=center,
        font=\scriptsize,
        inner xsep=4pt,
        inner ysep=4pt,
        minimum height=8mm,
        text width=2.85cm,
        fill=gray!8
    },
    mainbox/.style={proofbox, fill=teal!10, text width=3.05cm, minimum height=9mm},
    stepbox/.style={
        draw=teal!70!black,
        rounded corners=3pt,
        line width=0.65pt,
        inner xsep=8pt,
        inner ysep=8pt
    },
    arrow/.style={->, line width=0.45pt},
    auxarrow/.style={->, dashed, line width=0.4pt, draw=gray!70}
]

	\node[proofbox] (assump) at (0,0)
	{Standing assumptions\\
	\Cref{ass:plcirc,ass:curvature_g,assum_sample_comp,ass:unique_min}};
\node[proofbox] (curv) at (0,-1.15)
{Curvature support\\
\Cref{prop:second-fund-from-g}};
\node[proofbox] (qgg) at (0,-2.30)
{Uniform growth of \(g\)\\
\Cref{lem:uniform-qg-g}};
	\node[proofbox] (gibbstube) at (0,-3.45)
	{Gibbs tube concentration\\
	\Cref{lem:gibbs-tube-width}};
	\node[proofbox] (nutube) at (0,-4.60)
	{Candidate tube transfer\\
	\Cref{lem:tube-width}};
	\node[proofbox] (renyi) at (0,-5.75)
	{ULA R\'enyi mixing\\
	\Cref{prop:lmc-renyi}};
	\node[proofbox] (w2) at (0,-6.90)
	{R\'enyi transport control\\
	\Cref{cor:lmc-gibbs-w2}};

\node[proofbox] (lowsq) at (4.20,0)
{Lower-tail SQ\\
\Cref{def:lower-sq}};
\node[proofbox] (tail) at (4.20,-1.15)
{Tail-min control\\
\Cref{lem:tail-to-sq}};
\node[proofbox] (hardgap) at (4.20,-3.45)
{Hard-selection value gap\\
\Cref{lem:inner-quality-gap}};

\node[proofbox] (hglocal) at (8.40,0)
{Local growth on \(\calS(\theta)\)\\
\Cref{lem:hg-from-nondeg}};
\node[proofbox] (ongrowth) at (8.40,-1.15)
{Uniform on-manifold growth\\
\Cref{lem:uniform-hg}};
\node[proofbox] (awaygap) at (8.40,-2.30)
{Away-from-minimizer gap\\
\Cref{lem:away-gap-positive}};
\node[proofbox] (distgap) at (8.40,-3.45)
{Value gap to distance\\
\Cref{lem:barx-distance}};

\node[mainbox] (split) at (12.60,-0.55)
{Triangle split\\
\(\tilde x\to\bar x\to x^\star\)};
\node[mainbox] (final) at (12.60,-1.90)
{Selection error\\
\Cref{thm:selection-error-bound}};

	\node[stepbox, fit=(assump)(curv)(qgg)(gibbstube)(nutube)(renyi)(w2),
	label={[font=\scriptsize\bfseries,text=teal!70!black]above:Step 1: geometric and sampling ingredients}] {};
\node[stepbox, fit=(lowsq)(tail)(hardgap),
label={[font=\scriptsize\bfseries,text=teal!70!black]above:Step 2: hard-selection value gap}] {};
\node[stepbox, fit=(hglocal)(ongrowth)(awaygap)(distgap),
label={[font=\scriptsize\bfseries,text=teal!70!black]above:Step 3: gap-to-distance conversion}] {};
\node[stepbox, fit=(split)(final),
label={[font=\scriptsize\bfseries,text=teal!70!black]above:Step 4: final assembly}] {};

\draw[auxarrow] (assump.east) -- (curv.west);
\draw[auxarrow] (assump.east) -- (qgg.west);
\draw[auxarrow] (assump.east) -- (hglocal.west);
\draw[auxarrow] (assump.east) -- (ongrowth.west);

	\draw[arrow] (qgg) -- (gibbstube);
	\draw[arrow] (gibbstube) -- (nutube);
	\draw[arrow] (renyi) -- (nutube);
	\draw[arrow] (renyi) -- (w2);
	\draw[arrow] (curv.east) -- (hardgap.west);
	\draw[arrow] (gibbstube.east) -- (hardgap.west);
	\draw[arrow] (nutube.east) -- (hardgap.west);
	\draw[arrow] (w2.east) -- (hardgap.west);
\draw[arrow] (lowsq) -- (tail);
\draw[auxarrow] (tail) -- (hardgap);

\draw[arrow] (hglocal) -- (ongrowth);
\draw[arrow] (ongrowth) -- (awaygap);
\draw[arrow] (ongrowth) -- (distgap);
\draw[arrow] (awaygap) -- (distgap);
	\draw[arrow] (nutube.east) -- (split.west);
\draw[arrow] (hardgap.east) -- (final.west);
\draw[arrow] (distgap.east) -- (final.west);
\draw[arrow] (split) -- (final);

\end{tikzpicture}%
}
\caption{\textbf{Four-step selection-error proof for \Cref{thm:selection-error-bound}.}
Each lemma, proposition, corollary, or definition used in Appendix~\ref{append:selection-error} appears in its own box. The larger step boxes group the results according to the proof strategy stated at the beginning of Appendix~\ref{append:selection-error}; dashed arrows mark supporting assumptions or intermediate deterministic superquantile steps.}
\label{fig:appendix-d-selection-flow}
\end{figure}
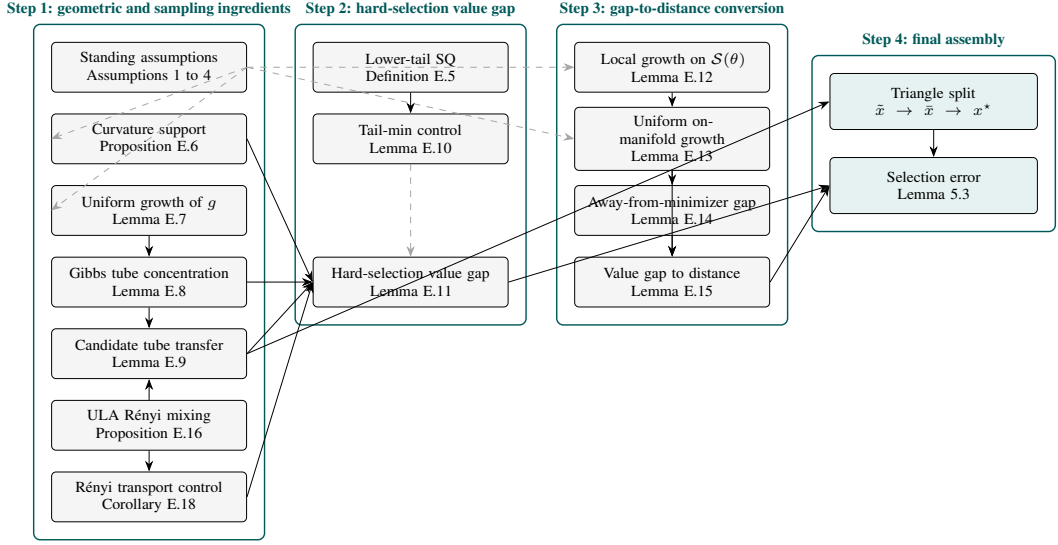
\paragraph{Notation for the selection proof.}
We fix notation for the selection argument and record the lower-tail functional used in the deterministic superquantile bound inside \Cref{lem:inner-quality-gap}.
Fix $\theta\in\Theta$ and recall
\[
\calS(\theta)=\arg\min_{x\in\R^d} g(\theta,x),\qquad
F(\theta)=\min_{x\in\calS(\theta)} f(\theta,x).
\]
Throughout, assume that the constrained problem admits a unique optimistic minimizer
\[
x^\star(\theta)\in\arg\min_{x\in\calS(\theta)} f(\theta,x),
\]
as in \Cref{ass:unique_min}. Given candidates $X_1,\dots,X_M\in\R^d$, our hard selection rule is
\[
\tilde x \in \arg\min_{i\in[1:M]} f(\theta,X_i),
\qquad
\delta:=\frac1M.
\]
Since $\tilde x$ need not lie on $\calS(\theta)$, we also define its projection
\[
\bar x \in \arg\min_{x\in\calS(\theta)}\|x-\tilde x\|,
\qquad\text{so that}\qquad
\|\tilde x-\bar x\|=\mathrm{dist}(\tilde x,\calS(\theta)).
\]

\paragraph{Empirical and Gibbs measures.}
Define the empirical measure of the candidates,
\[
\widehat\nu_M := \frac1M\sum_{i=1}^M \delta_{X_i},
\]
and recall the Gibbs measure used to sample near $\calS(\theta)$,
\[
\gibbs_\theta^\lambda(\d x)\propto \exp\{-g(\theta,x)/\lambda\}\,\d x.
\]

\paragraph{Lower-tail superquantile (best-$\delta$ CVaR).}
We use the lower-tail superquantile functional $\mathrm{SQ}^{\mathrm{low}}_\delta$.

\begin{definition}[Lower $\delta$-superquantile]\label{def:lower-sq}
Let $Z$ be a real-valued random variable with (lower) quantile function
$q_u(Z):=\inf\{t\in\R:\ \mathbb{P}(Z\le t)\ge u\}$ for $u\in(0,1)$.
Define
\[
\mathrm{SQ}^{\mathrm{low}}_\delta(Z):=\frac1\delta\int_0^\delta q_u(Z)\,\d u,
\qquad \delta\in(0,1).
\]
If $X\sim \nu$ and $Z=f(\theta,X)$, we also write
$\mathrm{SQ}^{\mathrm{low}}_\delta(f(\theta,X);\,X\sim\nu)$.
\end{definition}

\subsubsection[Bounded curvature from analytic regularity of g]{Bounded curvature from analytic regularity of $g$}\label{append:curvature-from-g}
\noindent
This subsection proves the second fundamental form bound used in the volume comparison arguments of
\citep[e.g., Lem.~4.1]{masiha2025superquantile}. The proof combines the manifold and normal-gap consequences of
\(\PLcirc\) from \Cref{prop:geometry} with the analytic regularity condition in \Cref{ass:curvature_g}.
Note that \Cref{ass:curvature_g} is stronger than the $L_{g,2}$-smoothness in \Cref{assum_sample_comp}:
it also controls how $\nabla^2_{xx}g(\theta,\cdot)$ varies with $x$ (via a uniform bound on $\nabla^3_{xxx}g$ in a tube around $\calS(\theta)$).
We use this bound later in \Cref{lem:inner-quality-gap}, where the \(M^{-1/\dimk}\) nearest-neighbor term follows from
small-ball volume scaling on \(\calS(\theta)\). Thus \Cref{prop:second-fund-from-g} is the step that verifies, from our
assumptions on \(g\), the curvature regularity needed to invoke the volume comparison result.

\begin{proposition}[Bounded second fundamental form from analytic regularity]\label{prop:second-fund-from-g}
Let \Cref{ass:plcirc,ass:curvature_g} hold, and let \(c>0\) be the uniform normal spectral gap from
\Cref{prop:geometry}. Fix any $\theta\in\Theta$. Then the second fundamental form $\mathrm{II}$ of $\calS(\theta)$ satisfies
\[
\sup_{x\in \calS(\theta)}\ \|\mathrm{II}_x\|_{\mathrm{op}}\ \le\ \frac{L_{g,3}}{c}.
\]
In particular, since \(L_{g,3}\) and \(c\) are uniform over \(\theta\in\Theta\), we obtain a uniform curvature bound over the family \(\{\calS(\theta)\}_{\theta\in\Theta}\).
\end{proposition}

\begin{proof}
Fix $\theta$ and abbreviate \(\mathcal{S}:=\calS(\theta)\).
Let $H(x):=\nabla^2_{xx}g(\theta,x)$.
Fix $x\in \mathcal{S}$ and unit tangent vectors $u,v\in\calT_x^\theta$.
Choose a $\mathcal{C}^{2}$ curve $\gamma:(-\varepsilon,\varepsilon)\to \mathcal{S}$ with $\gamma(0)=x$ and $\dot\gamma(0)=v$,
and choose a $\mathcal{C}^{1}$ tangent vector field $U(\cdot)$ along $\gamma$ with $U(0)=u$ and $U(t)\in\calT_{\gamma(t)}^\theta$.

By \Cref{prop:geometry}, \(\ker H(\gamma(t))=\calT_{\gamma(t)}^\theta\) for all \(t\). Since \(U(t)\) is tangent along
\(\gamma\), we have \(H(\gamma(t))\,U(t)=0\).
Differentiating at $t=0$ gives
\[
0
=\frac{\ud}{\ud t}\Big(H(\gamma(t))\,U(t)\Big)\Big|_{t=0}
=\big(D_v H(x)\big)\,u+H(x)\,\dot U(0),
\]
where $D_v H(x)$ denotes the directional derivative of the Hessian in direction $v$.
Let $P_{\calN_x^\theta}$ denote the orthogonal projection onto the normal space $\calN_x^\theta$.
Since $H(x)$ is symmetric and $\ker H(x)=\calT_x^\theta$, we have $\mathrm{range}(H(x))=(\calT_x^\theta)^\perp=\calN_x^\theta$ and hence
$P_{\calN_x^\theta}H(x)=H(x)P_{\calN_x^\theta}$.
Applying $P_{\calN_x^\theta}$ to the display above yields
\[
H(x)\,P_{\calN_x^\theta}\dot U(0) \;=\; -\,P_{\calN_x^\theta}\big(D_v H(x)\big)\,u.
\]
By the definition of the second fundamental form of an embedded submanifold in Euclidean space,
$\mathrm{II}_x(v,u)=P_{\calN_x^\theta}\dot U(0)$.
Restricting $H(x)$ to $\calN_x^\theta$ and using the uniform normal spectral gap gives
\[
\|\mathrm{II}_x(v,u)\|
\le \big\|\big(H(x)|_{\calN_x^\theta}\big)^{-1}\big\|\;\big\|\big(D_v H(x)\big)\,u\big\|
\le \frac{1}{c}\,\big\|\big(D_v H(x)\big)\,u\big\|.
\]
Finally, by the bounded third derivative assumption,
\(
\|(D_v H(x))\,u\|\le \|\nabla^3_{xxx}g(\theta,x)\|_{\mathrm{op}}\|v\|\,\|u\|\le L_{g,3}.
\)
Therefore $\|\mathrm{II}_x(v,u)\|\le L_{g,3}/c$ for all unit $u,v\in\calT_x^\theta$, and taking the supremum over $x\in \mathcal{S}$ yields the claim.
\end{proof}

\subsubsection{Controlling the off-manifold term $\mathrm{dist}(\tilde x,\calS(\theta))$}

\paragraph{Proof strategy.}
The off-manifold term is controlled in two steps. First, since the hard-selected point \(\tilde x\) is one of the
candidates, \eqref{eq:dist-max-bound} reduces the problem to bounding the largest normal distance among the candidate
points. Second, \Cref{lem:uniform-qg-g} turns the lower-level objective gap \(g(\theta,x)-g_\star(\theta)\) into a
uniform quadratic lower bound in \(\mathrm{dist}(x,\calS(\theta))\). This converts the Gibbs density
\(\exp(-g(\theta,x)/\lambda)\) into a subgaussian tail for the distance to \(\calS(\theta)\), which
\Cref{lem:gibbs-tube-width} integrates to control
\(\mathbb E[\max_i \mathrm{dist}(X_i,\calS(\theta))]\) and
\(\mathbb E[\max_i \mathrm{dist}^2(X_i,\calS(\theta))]\) for exact Gibbs candidates. Finally,
\Cref{lem:tube-width} transfers the same bounds to candidates with R\'enyi-controlled marginal laws \(\nu_i\), which is the
case used for the ULA outputs.

\paragraph{Max-distance reduction.}
Because $\tilde x$ is one of the candidates, it satisfies the deterministic inequality
\begin{equation}\label{eq:dist-max-bound}
\mathrm{dist}(\tilde x,\calS(\theta))
\ \le\ \max_{1\le i\le M}\mathrm{dist}(X_i,\calS(\theta)).
\end{equation}
We now show that, under our standing assumptions, Gibbs samples concentrate around $\calS(\theta)$ with a \emph{subgaussian} tail in the
normal distance $R:=\mathrm{dist}(X,\calS(\theta))$, which yields the sharper moment bound
$\mathbb E[\max_i R_i]\lesssim \sqrt{\lambda\log M}$.

\paragraph{Uniform quadratic growth for $g(\theta,\cdot)$.}
Recall from \Cref{assum_sample_comp} that $g(\theta,\cdot)$ satisfies a quadratic growth bound \emph{outside} a compact ball.
The next lemma shows that, under local \PLcirc\ and normal nondegeneracy on the minimizer manifold, this can be upgraded to a \emph{global}
quadratic growth inequality in terms of the distance to $\calS(\theta)$, uniformly over $\theta\in\Theta$.

\begin{lemma}[Uniform quadratic growth of $g(\theta,\cdot)$ over $\Theta$]\label{lem:uniform-qg-g}
Let \Cref{ass:plcirc,assum_sample_comp} hold.
Then there exists a constant $\mu_{\mathrm{QG}}>0$ such that for all $\theta\in\Theta$ and all $x\in\R^d$,
\begin{equation}\label{eq:global-qg-g}
g(\theta,x)-\min_{z\in\R^d}g(\theta,z)
\ \ge\
\frac{\mu_{\mathrm{QG}}}{2}\,\mathrm{dist}^2\!\bigl(x,\calS(\theta)\bigr).
\end{equation}
\end{lemma}

\begin{proof}
Let $\mathsf D$ be the compactness radius from \Cref{assum_sample_comp}, so that
$\calS(\theta)\subseteq \mathbb B_d(0;\mathsf D)$ for all $\theta\in\Theta$.
Write $g_\star(\theta):=\min_{z}g(\theta,z)$.

\emph{Preliminaries on the minimizer family.}
Since each minimizer set is contained in \(\mathbb B_d(0;\mathsf D)\),
\[
g_\star(\theta)=\min_{\|z\|\le \mathsf D} g(\theta,z).
\]
Thus \(g_\star\) is continuous on the compact set \(\Theta\) by Berge's maximum theorem. Consequently, the solution graph
\[
\mathcal K:=\{(\theta,y):\theta\in\Theta,\ y\in\calS(\theta)\}
=\{(\theta,y)\in\Theta\times \mathbb B_d(0;\mathsf D): g(\theta,y)=g_\star(\theta)\}
\]
is compact. The standard local normal-chart argument for the zero set \(\{\nabla_x g(\theta,x)=0\}\), using the normal
spectral gap in \Cref{prop:geometry} around each \((\theta,y)\in\mathcal K\) and then patched by compactness, gives local Hausdorff continuity of
\(\theta\mapsto\calS(\theta)\) on \(\Theta\). In particular, the map
\((\theta,x)\mapsto \mathrm{dist}(x,\calS(\theta))\) is continuous on \(\Theta\times \mathbb B_d(0;\mathsf D)\), and the
normal spaces \(\calN_y^\theta\) vary continuously over \(\mathcal K\).

\emph{Step 1: local quadratic growth in a uniform tube.}
Let $c>0$ be the normal spectral gap from \Cref{prop:geometry}, i.e.,
$\langle v,\nabla_{xx}^2 g(\theta,y)v\rangle\ge c\|v\|^2$ for all $\theta\in\Theta$, $y\in\calS(\theta)$, and
$v\in\calN_y^\theta$.
Since $(\theta,x)\mapsto \nabla^2_{xx}g(\theta,x)$ is continuous (\Cref{assum_sample_comp}) and the normal bundle over
\(\mathcal K\) is compact, there exists $r_{\mathrm{loc}}>0$ such that
for all $\theta\in\Theta$, all $y\in\calS(\theta)$, and all $x$ with $\|x-y\|\le r_{\mathrm{loc}}$,
\begin{equation}\label{eq:local-normal-gap}
\langle v,\nabla_{xx}^2 g(\theta,x)v\rangle\ \ge\ \tfrac{c}{2}\,\|v\|^2
\qquad \forall v\in\calN_y^\theta.
\end{equation}
Now fix any $\theta\in\Theta$ and $x\in\R^d$ with $R:=\mathrm{dist}(x,\calS(\theta))\le r_{\mathrm{loc}}$.
Let $y\in\calS(\theta)$ be a Euclidean projection of $x$ onto $\calS(\theta)$, so that $\|x-y\|=R$.
By first-order optimality of the projection problem on the manifold, $u:=x-y$ lies in the normal space $\calN_y^\theta$.
Using $\nabla_x g(\theta,y)=0$ (since $y\in\calS(\theta)$) and the integral form of Taylor's theorem,
\[
g(\theta,x)-g(\theta,y)
=\int_0^1 (1-t)\,u^\top \nabla_{xx}^2 g(\theta,y+t u)\,u\,\ud t.
\]
Because $\|t u\|\le R\le r_{\mathrm{loc}}$, \eqref{eq:local-normal-gap} applies along the segment $y+t u$, giving
$u^\top \nabla^2_{xx} g(\theta,y+t u)u\ge (c/2)\|u\|^2$ for all $t\in[0,1]$.
Therefore,
\begin{equation}\label{eq:qg-local}
g(\theta,x)-g_\star(\theta)
\ge
g(\theta,x)-g(\theta,y)
\ge
\tfrac{c}{4}\,\|u\|^2
=\tfrac{c}{4}\,R^2.
\end{equation}

\emph{Step 2: quadratic growth outside the compact ball.}
By \Cref{assum_sample_comp}, there exist constants $\mu_{\mathsf{qg}}>0$ and $\mathsf D>0$ such that for all $\|x\|\ge \mathsf D$,
\begin{equation}\label{eq:qg-outside}
g(\theta,x)-g_\star(\theta)\ \ge\ \tfrac{\mu_{\mathsf{qg}}}{2}\, \mathrm{dist}^2\!\bigl(x,\calS(\theta)\bigr),
\qquad \forall \theta\in\Theta.
\end{equation}

\emph{Step 3: the middle region by compactness.}
Consider the set
\[
\mathcal A
:=\Big\{(\theta,x):\theta\in\Theta,\ \|x\|\le \mathsf D,\ \mathrm{dist}(x,\calS(\theta))\ge r_{\mathrm{loc}}\Big\}.
\]
By the continuity of the distance map just noted, \(\mathcal A\) is closed in
\(\Theta\times \mathbb B_d(0;\mathsf D)\), hence compact. If \(\mathcal A=\emptyset\), the middle region is vacuous; in this
case set \(\mu_{\mathrm{mid}}:=+\infty\). Otherwise, the map
$(\theta,x)\mapsto g(\theta,x)-g_\star(\theta)$ is continuous on $\Theta\times \mathbb B_d(0;\mathsf D)$ and is strictly
positive on $\mathcal A$ (since $\mathrm{dist}(x,\calS(\theta))>0$ implies $x\notin\calS(\theta)$).
In the nonempty case, therefore,
\[
m_{\min}:=\min_{(\theta,x)\in\mathcal A}\big(g(\theta,x)-g_\star(\theta)\big)>0.
\]
Moreover, if $\|x\|\le \mathsf D$ then $\mathrm{dist}(x,\calS(\theta))\le 2\mathsf D$, so on the nonempty set
\(\mathcal A\) we have
\[
g(\theta,x)-g_\star(\theta)\ \ge\ m_{\min}
\ \ge\ \frac{m_{\min}}{(2\mathsf D)^2}\,\mathrm{dist}^2(x,\calS(\theta)).
\]
For this nonempty case, set \(\mu_{\mathrm{mid}}:=2m_{\min}/(2\mathsf D)^2\).

\emph{Step 4: combine the three regions.}
Let
\[
\mu_{\mathrm{QG}}
:=\min\Big\{\tfrac{c}{2},\ \mu_{\mathsf{qg}},\ \mu_{\mathrm{mid}}\Big\}.
\]
Then \eqref{eq:qg-local}, \eqref{eq:qg-outside}, and the middle-region bound above imply \eqref{eq:global-qg-g} for all $x\in\R^d$ and all
$\theta\in\Theta$.
\end{proof}

\paragraph{Subgaussian tube concentration and R\'enyi transfer.}
The uniform quadratic growth bound from \Cref{lem:uniform-qg-g} first gives Gaussian-like concentration of the normal
distance to $\calS(\theta)$ for exact Gibbs samples. We then transfer the same type of control to candidates whose
marginal laws are close to the Gibbs law in order-$2$ R\'enyi divergence.

\begin{lemma}[Tube concentration for Gibbs samples]\label{lem:gibbs-tube-width}
Assume the setting of \Cref{lem:uniform-qg-g} and fix any $\theta\in\Theta$.
Let $X\sim \gibbs_\theta^\lambda$ and define $R:=\mathrm{dist}(X,\calS(\theta))$.
Then there exist constants $\lambda_0,c_{\mathrm{tube}},C_{\mathrm{tube}},C_{\mathrm{tube},2}>0$ (independent of $\theta$ and $\lambda$) such that, for every \(0<\lambda\le\lambda_0\),
\begin{equation}\label{eq:gibbs-tail}
\mathbb P(R\ge t)\ \le\ C_{\mathrm{tube}}\exp\!\Big(-c_{\mathrm{tube}}\,\tfrac{t^2}{\lambda}\Big),
\qquad \forall t\ge 0.
\end{equation}
Consequently,
\[
\mathbb E[R]\le C_{\mathrm{tube}}\sqrt{\lambda},
\qquad
\mathbb E[R^2]\le C_{\mathrm{tube},2}\lambda.
\]

Moreover, if $X_1,\dots,X_M\stackrel{\mathrm{i.i.d.}}{\sim}\gibbs_\theta^\lambda$ and $\tilde x$ is any measurable selection of one of the candidates, then
\begin{equation}\label{eq:dist-tilde-max}
\mathbb E\big[\mathrm{dist}(\tilde x,\calS(\theta))\big]
\ \le\
\mathbb E\Big[\max_{1\le i\le M}\mathrm{dist}(X_i,\calS(\theta))\Big]
\ \le\
C_{\mathrm{tube}}\,\sqrt{\lambda\,\log(1+M)}.
\end{equation}
\begin{equation}\label{eq:dist-tilde-max2}
\mathbb E\big[\mathrm{dist}(\tilde x,\calS(\theta))^2\big]
\ \le\
\mathbb E\Big[\max_{1\le i\le M}\mathrm{dist}(X_i,\calS(\theta))^2\Big]
\ \le\
C_{\mathrm{tube},2}\,\lambda\,\log(1+M).
\end{equation}
\end{lemma}

\begin{proof}[Proof of \Cref{lem:gibbs-tube-width}]
Fix $\theta$ and abbreviate \(\mathcal{S}:=\calS(\theta)\) and $g_\star:=\min_z g(\theta,z)$.
By \Cref{lem:uniform-qg-g}, $g(\theta,x)\ge g_\star+\frac{\mu_{\mathrm{QG}}}{2}\mathrm{dist}^2(x,\mathcal{S})$.
Therefore, for any $t\ge 0$,
\[
\mathbb P(R\ge t)
=\frac{\int_{\mathrm{dist}(x,\mathcal{S})\ge t} e^{-g(\theta,x)/\lambda}\,\ud x}{\int_{\R^d} e^{-g(\theta,x)/\lambda}\,\ud x}
\le
\frac{\int_{\mathrm{dist}(x,\mathcal{S})\ge t} \exp\!\big(-\tfrac{\mu_{\mathrm{QG}}}{2\lambda}\mathrm{dist}^2(x,\mathcal{S})\big)\,\ud x}{\int_{\R^d} e^{-(g(\theta,x)-g_\star)/\lambda}\,\ud x}.
\]
Let \(q:=d-\dimk\) be the normal dimension of \(\mathcal S\). Since \(\mathcal S\) is compact and embedded in
\(\R^d\), we have \(q\ge 1\).
We next use a uniform tubular-coordinate estimate around the family \(\{\calS(\theta)\}_{\theta\in\Theta}\).
By compactness of the solution graph, the uniform normal spectral gap in \Cref{prop:geometry}, and the uniform curvature
control in \Cref{prop:second-fund-from-g}, there are constants \(r_{\mathrm{tube}}>0\), \(0<J_{\min}\le J_{\max}<\infty\),
and \(0<V_{\min}\le V_{\max}<\infty\), independent of \(\theta\), such that the normal map
\[
\Phi_\theta(y,z)=y+z,\qquad y\in\mathcal S,\quad z\in\calN_y^\theta,\quad \|z\|<r_{\mathrm{tube}},
\]
is a diffeomorphism onto the \(r_{\mathrm{tube}}\)-tube around \(\mathcal S\), its Jacobian lies in
\([J_{\min},J_{\max}]\), and the Riemannian volume of \(\mathcal S\) lies in \([V_{\min},V_{\max}]\).

We first lower bound the denominator. Let
\[
Z_\lambda(\theta):=\int_{\R^d} e^{-(g(\theta,x)-g_\star)/\lambda}\,\ud x .
\]
The small-temperature normalizing-constant lower bound used in
\citet[Lem.~D.1]{masiha2025superquantile}, applied here in the normal coordinates above, gives the manifold-version
scaling
\begin{equation}\label{eq:gibbs-den-lower}
Z_\lambda(\theta)\ge c_{\mathrm{den}}\,\lambda^{q/2}
\end{equation}
for all sufficiently small \(\lambda\), with \(c_{\mathrm{den}}>0\) uniform in \(\theta\). The exponent is \(q/2\), rather
than \(d/2\), because the Gibbs mass only concentrates in the \(q=d-\dimk\) normal directions; integration along
\(\mathcal S\) contributes a bounded positive manifold-volume factor.

We now upper bound the numerator. Write \(a:=\mu_{\mathrm{QG}}/(2\lambda)\) and split the domain into the part inside the
uniform tube and the part outside it.

\emph{Inside the tube.}
Every point in the tube has a unique representation \(x=y+z\), where \(y\in\mathcal S\) and
\(z\in\calN_y^\theta\). Here \(\ud\mathrm{Vol}_{\mathcal S}(y)\) denotes the intrinsic \(\dimk\)-dimensional volume
measure on \(\mathcal S\) induced by the ambient Euclidean metric; for example, it is arclength when \(\dimk=1\) and
surface area when \(\dimk=2\). Thus
\[
\int_{\mathcal S}\ud\mathrm{Vol}_{\mathcal S}(y)=\mathrm{Vol}_{\dimk}(\mathcal S).
\]
In these coordinates, \(\mathrm{dist}(y+z,\mathcal S)=\|z\|\) and the Euclidean volume element decomposes as
\[
\ud x=J_\theta(y,z)\,\ud z\,\ud\mathrm{Vol}_{\mathcal S}(y),
\]
where \(J_\theta(y,z)\le J_{\max}\). Therefore, after integrating out the \(y\)-variable using
\(\mathrm{Vol}_{\dimk}(\mathcal S)\le V_{\max}\), and identifying each normal fiber
\(\calN_y^\theta\) with \(\R^q\) through an arbitrary orthonormal basis, the tube contribution is bounded by a
\(q\)-dimensional Gaussian tail:
\begin{align}
\int_{\substack{\mathrm{dist}(x,\mathcal S)\ge t\\ \mathrm{dist}(x,\mathcal S)<r_{\mathrm{tube}}}}
e^{-a\,\mathrm{dist}^2(x,\mathcal S)}\,\ud x
&\le
J_{\max}V_{\max}
\int_{\substack{z\in\R^q:\ \|z\|\ge t\\ \|z\|<r_{\mathrm{tube}}}}
e^{-a\|z\|^2}\,\ud z \nonumber\\
&\le
J_{\max}V_{\max}\,|\mathbb S^{q-1}|
\int_t^\infty e^{-a r^2}r^{q-1}\,\ud r \nonumber\\
&\le C_1\,a^{-q/2}\exp\!\big(-a t^2/2\big).
\label{eq:gibbs-num-near}
\end{align}
The second inequality is just polar coordinates in the \(q\)-dimensional normal variable \(z\): writing
\(z=r\omega\), with \(\omega\in\mathbb S^{q-1}\), gives
\(\ud z=r^{q-1}\ud r\,\ud\sigma(\omega)\), and integrating over \(\omega\) gives the factor
\(|\mathbb S^{q-1}|\). We then upper bound the radial integral by replacing the upper limit
\(r_{\mathrm{tube}}\) with \(\infty\).
For the last inequality, substitute \(u=\sqrt a\,r\) and use
\(\int_s^\infty u^{q-1}e^{-u^2}\,\ud u\le C_q e^{-s^2/2}\) for all \(s\ge0\).

\emph{Outside the tube.}
Let \(R:=\max\{t,r_{\mathrm{tube}}\}\). We give a direct bound for this part. Since
\(\mathcal S\subseteq\mathbb B_d(0;\mathsf D)\), every point with \(\|x\|>\mathsf D+R\) satisfies
\[
\mathrm{dist}(x,\mathcal S)\ge \|x\|-\mathsf D.
\]
Hence
\begin{align*}
\int_{\mathrm{dist}(x,\mathcal S)\ge R}
e^{-a\,\mathrm{dist}^2(x,\mathcal S)}\,\ud x
&\le
\int_{\substack{\mathrm{dist}(x,\mathcal S)\ge R\\ \|x\|\le \mathsf D+R}}
e^{-aR^2}\,\ud x
\ +\
\int_{\|x\|>\mathsf D+R}
e^{-a(\|x\|-\mathsf D)^2}\,\ud x  \\
&\le
C_d(\mathsf D+R)^d e^{-aR^2}
\ +\
C_d\int_R^\infty (\mathsf D+s)^{d-1}e^{-a s^2}\,\ud s,
\end{align*}
where the last line uses polar coordinates with \(s=\|x\|-\mathsf D\) for the radial tail.
Because \(R\ge r_{\mathrm{tube}}>0\), the polynomial factors are dominated by the Gaussian decay. More explicitly, using
the standard one-dimensional Gaussian-tail bound
\[
\int_R^\infty (\mathsf D+s)^{d-1}e^{-a s^2}\,\ud s
\le C_{d,\mathsf D,r_{\mathrm{tube}}}\,a^{-1}(1+R^{d-2})e^{-aR^2}
\]
for \(R\ge r_{\mathrm{tube}}\), and increasing the constant to also cover the bounded-annulus term, there exists
\(a_0<\infty\) such that for all \(a\ge a_0\),
\begin{align}
\int_{\mathrm{dist}(x,\mathcal S)\ge \max\{t,r_{\mathrm{tube}}\}}
e^{-a\,\mathrm{dist}^2(x,\mathcal S)}\,\ud x
&\le C_5\,a^{-q/2}\exp\!\big(-a t^2/2\big).
\label{eq:gibbs-num-far}
\end{align}
Indeed, the last step follows because
\[
\sup_{\substack{a\ge a_0\\ R\ge r_{\mathrm{tube}}}}
a^{q/2}(\mathsf D+R)^d e^{-aR^2/2}
<\infty,
\qquad
\sup_{\substack{a\ge a_0\\ R\ge r_{\mathrm{tube}}}}
a^{q/2-1}(1+R^{d-2}) e^{-aR^2/2}
<\infty,
\]
and \(R\ge t\).
Choose \(\lambda_0>0\) small enough so that the denominator lower bound above holds and
\(a=\mu_{\mathrm{QG}}/(2\lambda)\ge a_0\) whenever \(0<\lambda\le\lambda_0\).
Combining \eqref{eq:gibbs-num-near} and \eqref{eq:gibbs-num-far},
\[
\int_{\mathrm{dist}(x,\mathcal{S})\ge t} e^{-a\,\mathrm{dist}^2(x,\mathcal{S})}\,\ud x
\le C_{\mathrm{num}}\,a^{-q/2}\exp\!\big(-a t^2/2\big)
\le C_{\mathrm{num}}'\,\lambda^{q/2}\exp\!\Big(-\tfrac{\mu_{\mathrm{QG}}}{4\lambda}t^2\Big).
\]
Dividing this numerator bound by \eqref{eq:gibbs-den-lower} yields \eqref{eq:gibbs-tail}.
Integrating \eqref{eq:gibbs-tail} over $t\ge 0$ yields $\mathbb E[R]\le C_{\mathrm{tube}}\sqrt{\lambda}$.
Similarly,
\[
\mathbb E[R^2]
=2\int_0^\infty t\,\mathbb P(R\ge t)\,\ud t
\le
2C_{\mathrm{tube}}\int_0^\infty t\,\exp\!\Big(-c_{\mathrm{tube}}\,\tfrac{t^2}{\lambda}\Big)\,\ud t
\le C_{\mathrm{tube},2}\lambda
\]
after increasing the constant if needed.

	Finally, set \(R_i:=\mathrm{dist}(X_i,\mathcal{S})\). Since \(\tilde x\) is one of the candidates,
	\[
	\mathrm{dist}(\tilde x,\mathcal{S})\le \max_{1\le i\le M}R_i.
	\]
	By a union bound and \eqref{eq:gibbs-tail},
		\[
		\mathbb P\!\Big(\max_{1\le i\le M} R_i\ge t\Big)
		\le \sum_{i=1}^M \mathbb P(R_i\ge t)
		\le M\,C_{\mathrm{tube}}\exp\!\Big(-c_{\mathrm{tube}}\,\tfrac{t^2}{\lambda}\Big).
		\]
		To bound the expectation, use the tail integral representation
		\[
		\mathbb E\Big[\max_{1\le i\le M} R_i\Big]
		=\int_0^\infty \mathbb P\!\Big(\max_{1\le i\le M} R_i\ge t\Big)\,\ud t.
		\]
		Let $t_0:=\sqrt{\frac{\lambda}{c_{\mathrm{tube}}}\log(1+M)}$. Splitting the integral and using the bound above yields
		\[
		\mathbb E\Big[\max_{1\le i\le M} R_i\Big]
		\le t_0 + \int_{t_0}^\infty M\,C_{\mathrm{tube}}\exp\!\Big(-c_{\mathrm{tube}}\,\tfrac{t^2}{\lambda}\Big)\,\ud t.
		\]
		Using $\int_{t_0}^\infty e^{-a t^2}\,\ud t\le \frac{1}{2 a t_0}\,e^{-a t_0^2}$ for $a>0$ (with $a=c_{\mathrm{tube}}/\lambda$) gives
			\begin{align*}
			\int_{t_0}^\infty M\,C_{\mathrm{tube}}\exp\!\Big(-c_{\mathrm{tube}}\,\tfrac{t^2}{\lambda}\Big)\,\ud t
			&\le M\,C_{\mathrm{tube}}\cdot\frac{\lambda}{2c_{\mathrm{tube}}\,t_0}\exp\!\Big(-c_{\mathrm{tube}}\,\tfrac{t_0^2}{\lambda}\Big)\\
			&= \frac{M}{1+M}\cdot\frac{C_{\mathrm{tube}}}{2c_{\mathrm{tube}}}\cdot\frac{\lambda}{t_0}\\
			&\le \frac{C_{\mathrm{tube}}}{2c_{\mathrm{tube}}}\,\frac{\sqrt{\lambda}}{\sqrt{\log(1+M)}}.
			\end{align*}
		Since \(M\ge1\), \(\log(1+M)\ge \log 2\), and we conclude that
		$\mathbb E[\max_i R_i]\le C_{\mathrm{tube}}\sqrt{\lambda\log(1+M)}$ after adjusting $C_{\mathrm{tube}}$.
		Similarly,
		\[
		\mathbb E\Big[\max_{1\le i\le M} R_i^2\Big]
		=
		2\int_0^\infty t\,\mathbb P\!\Big(\max_{1\le i\le M} R_i\ge t\Big)\,\ud t.
		\]
		Splitting at the same $t_0$ and using $\int_{t_0}^\infty t e^{-a t^2}\,\ud t=\frac{1}{2a}e^{-a t_0^2}$ with $a=c_{\mathrm{tube}}/\lambda$, we get
		\[
		\mathbb E\Big[\max_{1\le i\le M} R_i^2\Big]
		\le
		t_0^2+\frac{M C_{\mathrm{tube}}\lambda}{c_{\mathrm{tube}}}\exp\!\Big(-c_{\mathrm{tube}}\,\tfrac{t_0^2}{\lambda}\Big)
		\le
		C_{\mathrm{tube},2}\lambda\log(1+M),
		\]
		which yields \eqref{eq:dist-tilde-max2} after increasing $C_{\mathrm{tube},2}$.
\end{proof}

The next lemma is the form used for the candidates produced by the ULA/Gibbs instantiation of
\Cref{alg:hg-minsel-lmc}. Those candidates are not assumed to be exact Gibbs samples: after finitely many ULA steps, the
\(i\)-th candidate has marginal law \(\nu_i\). The
lemma only requires each \(\nu_i\) to be close to the ideal Gibbs law in order-\(2\) R\'enyi divergence, and then transfers
the Gibbs tube tail from \Cref{lem:gibbs-tube-width} by a Cauchy--Schwarz change-of-measure argument. This R\'enyi premise
is not an additional modeling assumption for ULA candidates: \Cref{prop:lmc-renyi} below gives the required bound
\(R_2(\nu_i\|\gibbs_\theta^\lambda)\le \varepsilon_{\mathrm R}^2\) once each ULA chain is run for the corresponding number
of steps. Thus \Cref{lem:tube-width} is the bridge from exact Gibbs concentration to the finite-time sampler used in the
algorithm.

\begin{lemma}[Tube bounds for candidates with R\'enyi-controlled marginals]\label{lem:tube-width}
Assume the setting of \Cref{lem:gibbs-tube-width}. Let $X_1,\dots,X_M$ be candidates whose marginal laws
$\nu_1,\dots,\nu_M$ are not necessarily identical and satisfy
\[
\max_{1\le i\le M}R_2(\nu_i\|\gibbs_\theta^\lambda)\le \varepsilon_{\mathrm{R}}^2.
\]
Let $\tilde x$ be the hard-selection output
\[
\tilde x\in\arg\min_{1\le i\le M} f(\theta,X_i),
\]
with any measurable tie-breaking rule. Then
\begin{equation}\label{eq:nu-tail-max}
\mathbb P\!\Big(\max_{1\le i\le M}\mathrm{dist}(X_i,\calS(\theta))\ge t\Big)
\le
M\,e^{\varepsilon_{\mathrm{R}}^2/2}\sqrt{C_{\mathrm{tube}}}
\exp\!\Big(-\tfrac{c_{\mathrm{tube}}}{2}\tfrac{t^2}{\lambda}\Big),
\qquad t\ge0.
\end{equation}
Consequently,
\begin{equation}\label{eq:dist-tilde-max-lmc}
\mathbb E\big[\mathrm{dist}(\tilde x,\calS(\theta))\big]
\ \le\
\mathbb E\Big[\max_{1\le i\le M}\mathrm{dist}(X_i,\calS(\theta))\Big]
\ \le\
C_{\mathrm{tube}}\,e^{\varepsilon_{\mathrm{R}}^2/2}\sqrt{\lambda\,\log(1+M)},
\end{equation}
and
\begin{equation}\label{eq:dist-tilde-max2-lmc}
\mathbb E\big[\mathrm{dist}(\tilde x,\calS(\theta))^2\big]
\ \le\
\mathbb E\Big[\max_{1\le i\le M}\mathrm{dist}(X_i,\calS(\theta))^2\Big]
\ \le\
C_{\mathrm{tube},2}\,e^{\varepsilon_{\mathrm{R}}^2/2}\lambda\,\log(1+M),
\end{equation}
after increasing \(C_{\mathrm{tube}}\) and \(C_{\mathrm{tube},2}\) by universal numerical factors if necessary.
\end{lemma}

\begin{proof}[Proof of \Cref{lem:tube-width}]
	Fix $\theta$ and abbreviate \(\mathcal{S}:=\calS(\theta)\) and
	\(R_i:=\mathrm{dist}(X_i,\mathcal{S})\). For any candidate index $i$, let
	\(A_t:=\{x:\mathrm{dist}(x,\mathcal{S})\ge t\}\).
	Under $R_2(\nu_i\|\gibbs_\theta^\lambda)\le \varepsilon_{\mathrm{R}}^2$, Cauchy--Schwarz yields
	\[
	\nu_i(A_t)
	=\int \mathbf 1_{A_t}\,\frac{\d \nu_i}{\d \gibbs_\theta^\lambda}\,\d \gibbs_\theta^\lambda
	\le e^{\varepsilon_{\mathrm{R}}^2/2}\,\gibbs_\theta^\lambda(A_t)^{1/2}.
	\]
		Combining this with the Gibbs tail \eqref{eq:gibbs-tail} gives
		$\mathbb P(R_i\ge t)\le e^{\varepsilon_{\mathrm{R}}^2/2}\sqrt{C_{\mathrm{tube}}}\exp\!\big(-\tfrac{c_{\mathrm{tube}}}{2}\tfrac{t^2}{\lambda}\big)$,
		and hence (by a union bound)
		\[
		\mathbb P\!\Big(\max_{1\le i\le M} R_i\ge t\Big)
		\le
		M\,e^{\varepsilon_{\mathrm{R}}^2/2}\sqrt{C_{\mathrm{tube}}}\exp\!\Big(-\tfrac{c_{\mathrm{tube}}}{2}\tfrac{t^2}{\lambda}\Big).
		\]
		This proves \eqref{eq:nu-tail-max}. To integrate the tail, set
		\[
		A_{\mathrm R}:=e^{\varepsilon_{\mathrm{R}}^2/2}\sqrt{C_{\mathrm{tube}}},
		\qquad
		\beta:=\frac{c_{\mathrm{tube}}}{2\lambda},
		\qquad
		t_0:=\sqrt{\beta^{-1}\log(1+M)}.
		\]
		Then \eqref{eq:nu-tail-max} gives
		\[
		\mathbb P\!\Big(\max_{1\le i\le M}R_i\ge t\Big)
		\le \min\{1,\;M A_{\mathrm R}e^{-\beta t^2}\}.
		\]
		Using the tail integral formula and splitting at \(t_0\),
		\begin{align*}
		\mathbb E\Big[\max_{1\le i\le M} R_i\Big]
		&\le t_0+M A_{\mathrm R}\int_{t_0}^{\infty}e^{-\beta t^2}\,\ud t\\
		&\le t_0+\frac{M A_{\mathrm R}}{2\beta t_0}e^{-\beta t_0^2}\\
		&\le C(1+A_{\mathrm R})\sqrt{\lambda\log(1+M)}
		\le C' e^{\varepsilon_{\mathrm{R}}^2/2}\sqrt{\lambda\log(1+M)},
		\end{align*}
		where the constants absorb \(C_{\mathrm{tube}}\), \(c_{\mathrm{tube}}\), and the fact that \(\log(1+M)\ge\log 2\).
		After increasing \(C_{\mathrm{tube}}\), this proves \eqref{eq:dist-tilde-max-lmc}.
		
		Similarly,
		\[
		\mathbb E\Big[\max_{1\le i\le M} R_i^2\Big]
		=
		2\int_0^\infty t\,\mathbb P\!\Big(\max_{1\le i\le M} R_i\ge t\Big)\,\ud t.
		\]
		Splitting at the same \(t_0\) gives
		\begin{align*}
		\mathbb E\Big[\max_{1\le i\le M} R_i^2\Big]
		&\le t_0^2+2M A_{\mathrm R}\int_{t_0}^{\infty}t e^{-\beta t^2}\,\ud t\\
		&=t_0^2+\frac{M A_{\mathrm R}}{\beta}e^{-\beta t_0^2}\\
		&\le C(1+A_{\mathrm R})\lambda\log(1+M)
		\le C' e^{\varepsilon_{\mathrm{R}}^2/2}\lambda\log(1+M).
		\end{align*}
		After increasing \(C_{\mathrm{tube},2}\), this proves \eqref{eq:dist-tilde-max2-lmc}.
	Finally, since the hard-selection output \(\tilde x\in\arg\min_i f(\theta,X_i)\) is one of the candidates,
	\[
	\mathrm{dist}(\tilde x,\mathcal{S})
	\le \max_{1\le i\le M}R_i,
	\]
	which gives the first inequalities in \eqref{eq:dist-tilde-max-lmc} and \eqref{eq:dist-tilde-max2-lmc}.
\end{proof}

\subsubsection{From hard selection to a value gap on $\calS(\theta)$}
With the tube controls in place, we next bound the value gap of the projected selected point \(\bar x\). The deterministic superquantile bound is useful for interpreting hard selection, but the final rate below uses a direct hard-min argument to avoid an unnecessary factor \(M\).

\paragraph{Tail-min control by an empirical lower superquantile.}
Hard selection is a deterministic lower-tail operation, and its output is always below the corresponding
(lower) superquantile under the empirical law.

\begin{lemma}[Tail-selection $\Rightarrow$ superquantile control]\label{lem:tail-to-sq}
Fix $\theta$, candidates $X_1,\dots,X_M$, and $\delta\in(0,1)$. Let
\(
\tilde x \in \arg\min_{i\in[1:M]} f(\theta,X_i).
\)
Then
\[
f(\theta,\tilde x)\ \le\ \mathrm{SQ}^{\mathrm{low}}_\delta\big(f(\theta,X);\,X\sim\widehat\nu_M\big).
\]
\end{lemma}
\begin{proof}
Under the empirical law \(\widehat\nu_M\), the random variable \(f(\theta,X)\) takes values among
\(\{f(\theta,X_i)\}_{i=1}^M\). Let \(m:=\min_{1\le i\le M}f(\theta,X_i)=f(\theta,\tilde x)\).
For every \(u\in(0,1)\), no value below \(m\) can have empirical probability at least \(u\), hence
\(q_u(f(\theta,X))\ge m\). Averaging this pointwise inequality over \(u\in(0,\delta)\) gives the claim.
\end{proof}

\paragraph{Lifting $\tilde x$ back to $\calS(\theta)$ via Lipschitzness.}
We will compare $x^\star(\theta)$ to $\bar x\in\calS(\theta)$ (the projection of $\tilde x$) and separate
(i) an off-manifold error $\|\tilde x-\bar x\|$ from (ii) an on-manifold suboptimality gap at $\bar x$.
\begin{lemma}[Suboptimality gap for R\'enyi-controlled candidates]\label{lem:inner-quality-gap}
Fix $\theta$ and let $x^\star(\theta)$ be the unique minimizer of $f(\theta,\cdot)$ over $\calS(\theta)$.
Assume \(\dimk\ge1\) and that $x\mapsto f(\theta,x)$ is $L_{f,1}$-Lipschitz (as in \Cref{assum_sample_comp}).
Given candidates \(X_1,\ldots,X_M\), let
\[
\widehat\nu_M:=\frac1M\sum_{i=1}^M\delta_{X_i},
\qquad
\tilde x\in\arg\min_{1\le i\le M}f(\theta,X_i),
\]
with any measurable tie-breaking rule, and let $\bar x$ be any measurable Euclidean projection of \(\tilde x\) onto
\(\calS(\theta)\).
Then, for every \(\delta\in(0,1)\),
\begin{align}
0 \ \le\ f(\theta,\bar x)-f(\theta,x^\star(\theta))
\ \le\ & L_{f,1}\,\mathrm{dist}(\tilde x,\calS(\theta))
\nonumber\\
&\quad + \Big[\mathrm{SQ}^{\mathrm{low}}_\delta\big(f(\theta,X);\,X\sim\widehat\nu_M\big)
	        - f(\theta,x^\star(\theta))\Big].
\label{eq:value-gap-basic}
\end{align}
Moreover, for any reference measure $\nu$ on $\R^d$ with finite first moment,
\begin{align}
\mathrm{SQ}^{\mathrm{low}}_\delta\big(f(\theta,X);\,X\sim\widehat\nu_M\big)
- \mathrm{SQ}^{\mathrm{low}}_\delta\big(f(\theta,X);\,X\sim\nu\big)
\ \le\ \frac{L_{f,1}}{\delta}\;\mathsf{W}_1(\widehat\nu_M,\nu),
\label{eq:sq-w1-lip}
\end{align}
and consequently, choosing $\nu=\gibbs_\theta^\lambda$ yields
\begin{align}
f(\theta,\bar x)-f(\theta,x^\star(\theta))
\ \le\ & L_{f,1}\,\mathrm{dist}(\tilde x,\calS(\theta))
\nonumber\\
&\quad + \frac{L_{f,1}}{\delta}\;\mathsf{W}_1(\widehat\nu_M,\gibbs_\theta^\lambda)
+ \Big[\mathrm{SQ}^{\mathrm{low}}_\delta\big(f(\theta,X);\,X\sim\gibbs_\theta^\lambda\big)
	        - f(\theta,x^\star(\theta))\Big].
\label{eq:value-gap-gibbs}
\end{align}

	Moreover, assume additionally that \(0<\lambda\le\lambda_0\), where \(\lambda_0\) is chosen small enough for both the
	tube bound in \Cref{lem:gibbs-tube-width} and the projected-Gibbs small-ball bound used below, and that
	\(X_1,\dots,X_M\) are independent candidates with marginal laws
\(\nu_1,\dots,\nu_M\), that \(\gibbs_\theta^\lambda\) satisfies a Poincar\'e inequality with constant at most
\(C_{\mathrm{PI}}\), and that
\[
\varepsilon_{\mathrm R}^2:=\max_{1\le i\le M}R_2(\nu_i\|\gibbs_\theta^\lambda)<\infty.
\]
Assume \(\varepsilon_{\mathrm R}\le1\).
Then, for the hard-selection output \(\tilde x\) and projection \(\bar x\) defined above, there exists a constant
$C_1>0$ (depending only on the on-manifold geometry near $x^\star(\theta)$; under the uniform assumptions above it can be
chosen uniformly over \(\theta\)) such that
		\begin{equation}\label{eq:inner-gap-hardmin}
		\mathbb E\big[f(\theta,\bar x)-F(\theta)\big]
		\ \le\
		2L_{f,1}\,C_{\mathrm{tube}}\,\sqrt{\lambda\log(1+M)}
		\ +\ C_1\,M^{-1/\dimk}
		\ +\ 4\sqrt{e-1}\,L_{f,1}\sqrt{M\,C_{\mathrm{PI}}}\,\varepsilon_{\mathrm R},
		\end{equation}
		where $C_{\mathrm{tube}}$ is the tube constant from \Cref{lem:gibbs-tube-width}.
\end{lemma}

\begin{proof}
	We first prove the deterministic bounds. Since \(\bar x\in\calS(\theta)\) and \(x^\star(\theta)\) minimizes
	\(f(\theta,\cdot)\) over \(\calS(\theta)\), the left inequality in \eqref{eq:value-gap-basic} holds. By Lipschitzness and
	the definition of \(\bar x\),
	\[
	f(\theta,\bar x)
	\le
	f(\theta,\tilde x)+L_{f,1}\,\mathrm{dist}(\tilde x,\calS(\theta)).
	\]
	Using \Cref{lem:tail-to-sq} to bound \(f(\theta,\tilde x)\) by the empirical lower superquantile gives
	\eqref{eq:value-gap-basic}.
	
	To prove \eqref{eq:sq-w1-lip}, use the variational representation
	\[
	\mathrm{SQ}^{\mathrm{low}}_\delta(Z)
	=
	\sup_{\tau\in\R}\left\{\tau-\frac1\delta\mathbb E[(\tau-Z)_+]\right\}.
	\]
	For any coupling of \(X\sim\widehat\nu_M\) and \(Y\sim\nu\), the map \(z\mapsto(\tau-f(\theta,z))_+\) is
	\(L_{f,1}\)-Lipschitz, and therefore, for every \(\tau\),
	\[
	\tau-\frac1\delta\mathbb E[(\tau-f(\theta,X))_+]
	\le
	\tau-\frac1\delta\mathbb E[(\tau-f(\theta,Y))_+]
	+\frac{L_{f,1}}{\delta}\mathbb E\|X-Y\|.
	\]
	Taking the supremum over \(\tau\) and then the infimum over couplings gives \eqref{eq:sq-w1-lip}. Combining
	\eqref{eq:value-gap-basic} and \eqref{eq:sq-w1-lip} with \(\nu=\gibbs_\theta^\lambda\) gives
	\eqref{eq:value-gap-gibbs}.

		The deterministic bound \eqref{eq:value-gap-gibbs} is valid for any $\delta$, but when $\delta=1/M$ it yields the loose factor
		$\delta^{-1}=M$ multiplying the empirical discrepancy term. We therefore analyze the hard-min selection directly.
		Write $\tilde x\in\arg\min_i f(\theta,X_i)$.
		By Lipschitzness and the definition of $\bar x$,
	\[
	f(\theta,\bar x)\ \le\ f(\theta,\tilde x)+L_{f,1}\,\mathrm{dist}(\tilde x,\calS(\theta)).
	\]
	Taking expectations gives
	\begin{equation}\label{eq:inner-gap-split}
	\mathbb E\big[f(\theta,\bar x)-F(\theta)\big]
	\ \le\
	L_{f,1}\,\mathbb E\big[\mathrm{dist}(\tilde x,\calS(\theta))\big]
	+\mathbb E\Big[\min_{1\le i\le M} f(\theta,X_i)\Big]-F(\theta).
	\end{equation}
	
		\paragraph{Step 1: hard-min bound under exact Gibbs sampling.}
		Assume first that $X_1,\dots,X_M\stackrel{\mathrm{i.i.d.}}{\sim}\gibbs_\theta^\lambda$.
		For each $i$, let $\Pi(X_i)\in\arg\min_{x\in\calS(\theta)}\|x-X_i\|$ be a Euclidean projection onto $\calS(\theta)$ (choose any measurable selection).
	By Lipschitzness,
	\[
	f(\theta,X_i)\ \le\ f(\theta,\Pi(X_i)) + L_{f,1}\,\mathrm{dist}(X_i,\calS(\theta)).
	\]
	Taking the minimum over $i$ and using $\min_i (a_i+b_i)\le \min_i a_i + \max_i b_i$ gives
	\[
	\min_{1\le i\le M} f(\theta,X_i)
	\ \le\
	\min_{1\le i\le M} f(\theta,\Pi(X_i))
	+\ L_{f,1}\max_{1\le i\le M}\mathrm{dist}(X_i,\calS(\theta)).
	\]
	Plugging this into \eqref{eq:inner-gap-split} and using $\mathrm{dist}(\tilde x,\calS(\theta))\le \max_i \mathrm{dist}(X_i,\calS(\theta))$ yields
	\begin{equation}\label{eq:inner-gap-max}
	\mathbb E\big[f(\theta,\bar x)-F(\theta)\big]
	\ \le\
	2L_{f,1}\,\mathbb E\Big[\max_{1\le i\le M}\mathrm{dist}(X_i,\calS(\theta))\Big]
	+\ \Big(\mathbb E\Big[\min_{1\le i\le M} f(\theta,\Pi(X_i))\Big]-F(\theta)\Big).
	\end{equation}
		The tube bound in \Cref{lem:gibbs-tube-width} gives
	\[
	\mathbb E\Big[\max_{1\le i\le M}\mathrm{dist}(X_i,\calS(\theta))\Big]
	\le C_{\mathrm{tube}}\sqrt{\lambda\log(1+M)}.
	\]

		For the on-manifold term, note that the projected points \(\Pi(X_i)\) lie on \(\calS(\theta)\). Under our standing
		regularity assumptions, the projected Gibbs law converges, as \(\lambda\downarrow0\), to the limiting on-manifold law
		\(\gibbs_\theta^0\), whose density with respect to Riemannian volume on \(\calS(\theta)\) is continuous and strictly
		positive near \(x^\star(\theta)\); see \citep[Prop.~3.7]{masiha2025superquantile}. Therefore, after decreasing
		\(\lambda_0\) if needed, the projected Gibbs laws satisfy a uniform small-ball lower bound near \(x^\star(\theta)\).
		Combined with geodesic volume scaling on \(\calS(\theta)\) \citep[Lem.~4.1]{masiha2025superquantile}--whose curvature
		requirement is verified in \Cref{prop:second-fund-from-g}--there are constants \(c_0>0\) and \(r_{\mathrm{nn}}>0\)
		such that, for every \(0<\lambda\le\lambda_0\),
		\[
		\mathbb P\big(d_{\calS(\theta)}(\Pi(X_1),x^\star(\theta))\le r\big)\ge c_0 r^{\dimk},
		\qquad 0<r\le r_{\mathrm{nn}}.
		\]
	Let \(D_M:=\min_{1\le i\le M}d_{\calS(\theta)}(\Pi(X_i),x^\star(\theta))\). The small-ball bound gives
	\[
	\mathbb P(D_M>r)
	\le (1-c_0r^{\dimk})^M
	\le \exp(-Mc_0r^{\dimk}),
	\qquad 0<r\le r_{\mathrm{nn}},
	\]
		Let \(\mathrm{diam}_{\calS}\) denote the finite geodesic diameter of \(\calS(\theta)\). Integrating this tail bound gives
		\begin{align*}
		\mathbb E[D_M]
		&=\int_0^{\mathrm{diam}_{\calS}}\mathbb P(D_M>r)\,\ud r\\
		&\le \int_0^{r_{\mathrm{nn}}} \exp(-Mc_0r^{\dimk})\,\ud r
		+\mathrm{diam}_{\calS}\exp(-Mc_0r_{\mathrm{nn}}^{\dimk})\\
		&\le C_{\mathrm{nn}}M^{-1/\dimk},
		\end{align*}
		where the last line uses \(\dimk\ge1\) and absorbs the exponentially small term into the same polynomial rate for
		all \(M\ge1\).
	Since \(f(\theta,\cdot)\) is \(L_{f,1}\)-Lipschitz on \(\calS(\theta)\),
	\[
	\mathbb E\Big[\min_{1\le i\le M} f(\theta,\Pi(X_i))\Big]-F(\theta)
	\le L_{f,1}\mathbb E[D_M]
	\le C_1M^{-1/\dimk}.
	\]
	Moreover, the display preceding \eqref{eq:inner-gap-max} also gives
	\begin{align*}
	&L_{f,1}\,\mathbb E\Big[\max_{1\le i\le M}\mathrm{dist}(X_i,\calS(\theta))\Big]
	+\mathbb E\Big[\min_{1\le i\le M} f(\theta,X_i)\Big]-F(\theta)\\
	&\qquad\le
	2L_{f,1}\,\mathbb E\Big[\max_{1\le i\le M}\mathrm{dist}(X_i,\calS(\theta))\Big]
	+\mathbb E\Big[\min_{1\le i\le M} f(\theta,\Pi(X_i))\Big]-F(\theta).
	\end{align*}
	Substituting the tube and nearest-neighbor bounds into this display gives the exact-Gibbs reference bound
	\begin{equation}\label{eq:inner-gap-exact-gibbs}
	\begin{aligned}
	&L_{f,1}\,\mathbb E\Big[\max_{1\le i\le M}\mathrm{dist}(X_i,\calS(\theta))\Big]\\
	&\qquad
	+\mathbb E\Big[\min_{1\le i\le M} f(\theta,X_i)\Big]-F(\theta)\\
	&\le
	2L_{f,1}\,C_{\mathrm{tube}}\,\sqrt{\lambda\log(1+M)}
	+\ C_1\,M^{-1/\dimk}.
	\end{aligned}
	\end{equation}

	\paragraph{Step 2: transfer to R\'enyi-controlled candidates.}
		We now return to independent candidates $X_1,\dots,X_M$ with marginal laws $\nu_i$ satisfying
		$R_2(\nu_i\|\gibbs_\theta^\lambda)\le\varepsilon_{\mathrm R}^2$.
		For each $i$, choose an optimal coupling of \(X_i\sim\nu_i\) with a reference variable
		\(X_i^\lambda\sim\gibbs_\theta^\lambda\) so that
		\[
		\mathbb E\|X_i-X_i^\lambda\|^2
		\le \mathsf W_2(\nu_i,\gibbs_\theta^\lambda)^2.
		\]
		Taking the product of these pairwise couplings gives a joint coupling whose first marginal is
		\(\nu_1\otimes\cdots\otimes\nu_M\), the law of the actual independent candidates, and whose second marginal is
		\((\gibbs_\theta^\lambda)^{\otimes M}\). Thus the reference variables \(X_1^\lambda,\ldots,X_M^\lambda\) are i.i.d.\
		from \(\gibbs_\theta^\lambda\), and expectations of functions of \(X_1,\ldots,X_M\) are unchanged under this coupling.
	We use the two deterministic comparisons
	\[
	\begin{aligned}
	\mathrm{dist}(\tilde x,\calS(\theta))
	&\le \max_{1\le i\le M}\mathrm{dist}(X_i,\calS(\theta))\\
	&\le \max_{1\le i\le M}\mathrm{dist}(X_i^\lambda,\calS(\theta))
	+\max_{1\le i\le M}\|X_i-X_i^\lambda\|
	\end{aligned}
	\]
	and
	\[
	\min_{1\le i\le M}f(\theta,X_i)
	\le \min_{1\le i\le M}f(\theta,X_i^\lambda)+L_{f,1}\max_{1\le i\le M}\|X_i-X_i^\lambda\|,
	\]
	and combine them with \eqref{eq:inner-gap-split}, applied to the actual candidates, to obtain
	\begin{align*}
	\mathbb E\big[f(\theta,\bar x)-F(\theta)\big]
	&\le
	L_{f,1}\,\mathbb E\Big[\max_{1\le i\le M}\mathrm{dist}(X_i^\lambda,\calS(\theta))\Big]
	+\mathbb E\Big[\min_{1\le i\le M} f(\theta,X_i^\lambda)\Big]-F(\theta)\\
	&\quad
	+2L_{f,1}\,\mathbb E\Big[\max_{1\le i\le M}\|X_i-X_i^\lambda\|\Big].
	\end{align*}
	The first two terms on the right are controlled by the exact-Gibbs reference bound \eqref{eq:inner-gap-exact-gibbs}. It remains
	to bound the coupling error
	\[
	2L_{f,1}\,\mathbb E\Big[\max_{1\le i\le M}\|X_i-X_i^\lambda\|\Big].
	\]
	Finally, $\max_i a_i\le (\sum_i a_i^2)^{1/2}$ gives
	\[
	\mathbb E\Big[\max_{1\le i\le M}\|X_i-X_i^\lambda\|\Big]
	\le
	\Big(\sum_{i=1}^M \mathbb E\|X_i-X_i^\lambda\|^2\Big)^{1/2}
	\le
	\Bigg(\sum_{i=1}^M \mathsf W_2(\nu_i,\gibbs_\theta^\lambda)^2\Bigg)^{1/2},
	\]
	and the Poincar\'e/R\'enyi transportation inequality \citep{liu2020poincare} gives
	\[
	\mathsf W_2(\nu_i,\gibbs_\theta^\lambda)
	\le
	2\,C_{\mathrm{PI}}^{1/2}\big(e^{R_2(\nu_i\|\gibbs_\theta^\lambda)}-1\big)^{1/2}
	\le
	2\sqrt{(e-1)C_{\mathrm{PI}}}\,\varepsilon_{\mathrm R}.
	\]
	The last inequality uses \(\varepsilon_{\mathrm R}\le1\), so
	\(e^{\varepsilon_{\mathrm R}^2}-1\le (e-1)\varepsilon_{\mathrm R}^2\).
	Therefore
	\[
	2L_{f,1}\,
	\mathbb E\Big[\max_{1\le i\le M}\|X_i-X_i^\lambda\|\Big]
	\le
	4\sqrt{e-1}\,L_{f,1}\sqrt{M\,C_{\mathrm{PI}}}\,\varepsilon_{\mathrm R}.
	\]
	Combining this estimate with \eqref{eq:inner-gap-exact-gibbs} proves \eqref{eq:inner-gap-hardmin}.
\end{proof}

\subsubsection{Converting the on-manifold gap to distance (link to local identifiability)}

\paragraph{Quadratic growth on $\calS(\theta)$ from Riemannian non-degeneracy.}
To convert an on-manifold value gap into a distance bound, we use a local growth inequality for the restriction of
$f(\theta,\cdot)$ to the manifold $\calS(\theta)$ around its optimistic minimizer $x^\star(\theta)$.
The next lemma is the local deterministic step: positive Riemannian Hessian of the restricted objective implies
quadratic growth along the manifold.

\begin{lemma}[Local quadratic growth from Riemannian non-degeneracy]\label{lem:hg-from-nondeg}
Fix $\theta$ and suppose that $f(\theta,\cdot)$ is $\mathcal{C}^{2}$ in a neighborhood of $\calS(\theta)$ and that
$x^\star(\theta)\in\calS(\theta)$ is a local minimizer of \(f(\theta,\cdot)|_{\calS(\theta)}\).
Assume moreover that the Riemannian Hessian of the restricted objective is positive definite at \(x^\star(\theta)\), i.e.,
there exists \(m_\theta>0\) such that
\[
\big\langle v,\mathrm{Hess}_{\calS(\theta)}\bar f_\theta(x^\star(\theta))[v]\big\rangle
\ \ge\
m_\theta\|v\|^2,
\qquad
\forall v\in\calT_{x^\star(\theta)}^\theta,
\]
where \(\bar f_\theta:=f(\theta,\cdot)|_{\calS(\theta)}\).
Then there exist constants $c_{\mathrm{hg}}>0$ and $r_0>0$ such that for all $x\in\calS(\theta)$ with
$d_{\calS(\theta)}(x,x^\star(\theta))\le r_0$,
\begin{equation}\label{eq:quadratic-growth-onS}
f(\theta,x)-f(\theta,x^\star(\theta))
\ \ge\
c_{\mathrm{hg}}\; d_{\calS(\theta)}(x,x^\star(\theta))^2.
\end{equation}
\end{lemma}

\begin{proof}
Let $\mathcal M:=\calS(\theta)$ be the embedded submanifold with the induced Riemannian metric, and define the restriction
$\bar f:\mathcal M\to\R$ by $\bar f(x):=f(\theta,x)$.
Since $f(\theta,\cdot)$ is $\mathcal{C}^{2}$ on a neighborhood of $\mathcal M$ and $\mathcal M$ is $\mathcal{C}^{2}$ embedded,
$\bar f$ is $\mathcal{C}^{2}$ on $\mathcal M$.

By the local-minimizer assumption, $x^\star:=x^\star(\theta)$ is a (local) minimizer of $\bar f$ on $\mathcal M$, so the Riemannian gradient vanishes:
$\mathrm{grad}_{\mathcal M}\bar f(x^\star)=0$.
Moreover, the assumed Riemannian non-degeneracy gives positive curvature of the restricted objective in every tangent direction.
Define the minimal Riemannian Hessian eigenvalue
\[
m
\ :=\
\min_{\substack{v\in\mathcal T_{x^\star}\mathcal M\\ \|v\|=1}}
\big\langle v,\ \mathrm{Hess}_{\mathcal M}\bar f(x^\star)[v]\big\rangle
\ >\ 0.
\]
Since $\mathrm{Hess}_{\mathcal M}\bar f$ is continuous, there exists a neighborhood $\mathcal U\subset\mathcal M$ of $x^\star$
such that for all $y\in\mathcal U$ and all $w\in\mathcal T_y\mathcal M$,
\begin{equation}\label{eq:hess-lb-local}
\big\langle w,\ \mathrm{Hess}_{\mathcal M}\bar f(y)[w]\big\rangle\ \ge\ \tfrac{m}{2}\,\|w\|^2.
\end{equation}
Let $r_0>0$ be such that the closed geodesic ball
\(\{y\in\mathcal M:d_{\mathcal M}(y,x^\star)\le r_0\}\) is contained in $\mathcal U$.

Fix any $x\in\mathcal M$ with $d_{\mathcal M}(x,x^\star)\le r_0$.
Since $\mathcal M$ is compact (and hence complete), the Hopf--Rinow theorem guarantees the existence of a minimizing geodesic
$\gamma:[0,1]\to\mathcal M$ from $x^\star$ to $x$.
See, e.g., \citep[Hopf--Rinow Theorem]{lee2006riemannian}.
Parameterize $\gamma$ at constant speed so that $\|\dot\gamma(t)\|=d_{\mathcal M}(x,x^\star)$ for all $t\in[0,1]$.
Because $\gamma$ is minimizing, each subsegment is minimizing, and hence
\(d_{\mathcal M}(\gamma(t),x^\star)\le d_{\mathcal M}(x,x^\star)\le r_0\) for all \(t\in[0,1]\). Thus
\(\gamma([0,1])\subset\mathcal U\).

Define $\phi(t):=\bar f(\gamma(t))$.
Then $\phi'(t)=\langle \mathrm{grad}_{\mathcal M}\bar f(\gamma(t)),\dot\gamma(t)\rangle$ and in particular $\phi'(0)=0$.
Differentiating once more and using that $\gamma$ is a geodesic yields the standard identity
\[
\phi''(t)
\ =\
\big\langle \dot\gamma(t),\ \mathrm{Hess}_{\mathcal M}\bar f(\gamma(t))[\dot\gamma(t)]\big\rangle.
\]
Therefore, by \eqref{eq:hess-lb-local} and constant speed,
\[
\phi''(t)
\ \ge\
\tfrac{m}{2}\,\|\dot\gamma(t)\|^2
\ =\
\tfrac{m}{2}\,d_{\mathcal M}(x,x^\star)^2
\qquad \forall t\in[0,1].
\]
Using $\phi'(0)=0$ and the integral form of Taylor's theorem,
\[
\phi(1)-\phi(0)
=\int_0^1 (1-t)\,\phi''(t)\,\ud t
\ge \int_0^1 (1-t)\,\tfrac{m}{2}\,d_{\mathcal M}(x,x^\star)^2\,\ud t
=\tfrac{m}{4}\,d_{\mathcal M}(x,x^\star)^2.
\]
Recalling $\phi(1)=\bar f(x)$ and $\phi(0)=\bar f(x^\star)$ gives \eqref{eq:quadratic-growth-onS} with $c_{\mathrm{hg}}:=m/4$.
\end{proof}

\paragraph{Uniform version over $\Theta$.}
The restricted-curvature constant from \Cref{prop:uniform-riem-nondeg} makes the quadratic-growth constants uniform over
\(\Theta\).

\begin{lemma}[Uniform quadratic growth on $\calS(\theta)$ over $\Theta$]\label{lem:uniform-hg}
Assume the setting of \Cref{sec:theory} and let \(m_{\mathrm R}>0\) be the uniform restricted-curvature constant from
\Cref{prop:uniform-riem-nondeg}.
Let $x^\star:\Theta\to\R^d$ be the global selected minimizer from \Cref{ass:unique_min}.
Then there exist constants $c_{\mathrm{hg}}>0$ and $r_0>0$ such that for all $\theta\in\Theta$ and all
$x\in\calS(\theta)$ with $d_{\calS(\theta)}(x,x^\star(\theta))\le r_0$,
\[
f(\theta,x)-f(\theta,x^\star(\theta))
\ \ge\
c_{\mathrm{hg}}\; d_{\calS(\theta)}(x,x^\star(\theta))^2.
\]
\end{lemma}
\begin{proof}
If \(\dimk=0\), then each connected compact zero-dimensional manifold \(\calS(\theta)\) is a singleton. Hence the only
point \(x\in\calS(\theta)\) is \(x^\star(\theta)\), and the claim is trivial for any fixed positive
\(c_{\mathrm{hg}}\) and \(r_0\). We therefore assume \(\dimk\ge1\).

For each $\theta\in\Theta$, let $\mathcal M_\theta:=\calS(\theta)$ and define the restriction
$\bar f_\theta:\mathcal M_\theta\to\R$ by $\bar f_\theta(x):=f(\theta,x)$.
By \Cref{ass:unique_min}, \(x^\star(\theta)\) is the unique minimizer of \(\bar f_\theta\) on \(\mathcal M_\theta\).
By \Cref{prop:uniform-riem-nondeg}, for every \(\theta\in\Theta\),
\[
\big\langle v,\ \mathrm{Hess}_{\mathcal M_\theta}\bar f_\theta(x^\star(\theta))[v]\big\rangle
\ge m_{\mathrm R}\|v\|^2,
\qquad
\forall v\in\mathcal T_{x^\star(\theta)}\mathcal M_\theta .
\]

We next make the neighborhood on which this lower bound persists uniform. For each \(\bar\theta\in\Theta\), continuity of
the Riemannian Hessian in local charts gives a parameter neighborhood \(V_{\bar\theta}\) and a radius
\(r_{\bar\theta}>0\) such that, for every \(\theta\in V_{\bar\theta}\), every
\(y\in\mathcal M_\theta\) with \(d_{\mathcal M_\theta}(y,x^\star(\theta))\le r_{\bar\theta}\), and every
\(w\in\mathcal T_y\mathcal M_\theta\),
\[
\big\langle w,\ \mathrm{Hess}_{\mathcal M_\theta}\bar f_\theta(y)[w]\big\rangle\ \ge\ \tfrac{m_{\mathrm R}}{2}\,\|w\|^2.
\]
Choose a finite subcover \(V_{\bar\theta_1},\ldots,V_{\bar\theta_J}\) of \(\Theta\) and set
\[
r_0:=\min_{1\le j\le J} r_{\bar\theta_j}>0.
\]
Then, for all $\theta\in\Theta$, all $y\in\mathcal M_\theta$ with $d_{\mathcal M_\theta}(y,x^\star(\theta))\le r_0$, and all $w\in\mathcal T_y\mathcal M_\theta$,
\[
\big\langle w,\ \mathrm{Hess}_{\mathcal M_\theta}\bar f_\theta(y)[w]\big\rangle\ \ge\ \tfrac{m_{\mathrm R}}{2}\,\|w\|^2.
\]
Fix any $\theta\in\Theta$ and $x\in\mathcal M_\theta$ with $d_{\mathcal M_\theta}(x,x^\star(\theta))\le r_0$.
Since $\mathcal M_\theta$ is compact, the Hopf--Rinow theorem yields a minimizing geodesic $\gamma:[0,1]\to\mathcal M_\theta$ from $x^\star(\theta)$ to $x$.
Parameterize $\gamma$ at constant speed so that $\|\dot\gamma(t)\|=d_{\mathcal M_\theta}(x,x^\star(\theta))$ for all $t\in[0,1]$.
Since \(\gamma\) is minimizing, \(d_{\mathcal M_\theta}(\gamma(t),x^\star(\theta))\le r_0\) for all \(t\in[0,1]\), so the
uniform Hessian lower bound above applies along the whole curve.
As in the proof of \Cref{lem:hg-from-nondeg}, define $\phi(t):=\bar f_\theta(\gamma(t))$.
Then $\phi'(0)=0$ and
\[
\phi''(t)
\ =\
\big\langle \dot\gamma(t),\ \mathrm{Hess}_{\mathcal M_\theta}\bar f_\theta(\gamma(t))[\dot\gamma(t)]\big\rangle
\ \ge\
\tfrac{m_{\mathrm R}}{2}\,d_{\mathcal M_\theta}(x,x^\star(\theta))^2,
\]
so the same integral argument yields
\[
f(\theta,x)-f(\theta,x^\star(\theta))
\ =\
\phi(1)-\phi(0)
\ \ge\
\tfrac{m_{\mathrm R}}{4}\,d_{\mathcal M_\theta}(x,x^\star(\theta))^2.
\]
The claim follows with $c_{\mathrm{hg}}:=m_{\mathrm R}/4$.
\end{proof}

\paragraph{Removing the locality restriction.}
The quadratic growth bound from \Cref{lem:uniform-hg} only applies within a geodesic neighborhood of radius $r_0$ around $x^\star(\theta)$.
To remove the explicit restriction $d_{\calS(\theta)}(\bar x, x^\star(\theta))\le r_0$, we use a value gap away from this neighborhood.

\begin{lemma}[Uniform value gap away from the selected minimizer]\label{lem:away-gap-positive}
Let $(c_{\mathrm{hg}},r_0)$ be the constants from \Cref{lem:uniform-hg} and define
\[
\calS_{r_0}(\theta)
:=
\{x\in\calS(\theta): d_{\calS(\theta)}(x,x^\star(\theta))\ge r_0\},
\qquad
\Delta_{r_0}
:=
\inf_{\theta\in\Theta}\ \min_{x\in\calS_{r_0}(\theta)}\big(f(\theta,x)-F(\theta)\big),
\]
with the convention that empty sets are ignored in the infimum. If all \(\calS_{r_0}(\theta)\) are empty, choose any fixed positive value for \(\Delta_{r_0}\).
Then \(\Delta_{r_0}>0\).
\end{lemma}
\begin{proof}
If all \(\calS_{r_0}(\theta)\) are empty, the claim is immediate by convention.
Otherwise, consider
\[
\mathcal A_{r_0}
:=
\{(\theta,x):\theta\in\Theta,\ x\in\calS(\theta),\ d_{\calS(\theta)}(x,x^\star(\theta))\ge r_0\}.
\]
Under the uniform regularity assumptions used in \Cref{lem:uniform-hg}, the compact manifolds \(\calS(\theta)\) vary continuously in local charts over compact \(\Theta\); hence their graph is compact and the geodesic distance to the continuous selected branch \(x^\star(\theta)\) is continuous on this graph. Therefore \(\mathcal A_{r_0}\) is compact. It is also disjoint from the selected graph \(\{(\theta,x^\star(\theta)):\theta\in\Theta\}\).

The gap map \((\theta,x)\mapsto f(\theta,x)-F(\theta)\) is continuous and strictly positive on \(\mathcal A_{r_0}\): if it vanished at some \((\theta,x)\in\mathcal A_{r_0}\), then \(x\) would be another minimizer of \(f(\theta,\cdot)\) over \(\calS(\theta)\), contradicting the uniqueness in \Cref{ass:unique_min}. A positive continuous function on a compact set has a positive minimum, which is exactly \(\Delta_{r_0}\).
\end{proof}

\begin{lemma}[Bounding $\mathbb E\|\bar x-x^\star(\theta)\|^2$ without a locality condition]\label{lem:barx-distance}
Let $(c_{\mathrm{hg}},r_0)$ be the constants from \Cref{lem:uniform-hg}, and let \(\Delta_{r_0}>0\) be the uniform away-gap from \Cref{lem:away-gap-positive}.
Then for any $\theta\in\Theta$,
\[
\begin{aligned}
\mathbb E\|\bar x-x^\star(\theta)\|^2
&\ \le\
\Bigg(\frac{1}{c_{\mathrm{hg}}}+\frac{4\mathsf D^2}{\Delta_{r_0}}\Bigg)\,
\mathbb E\big[f(\theta,\bar x)-F(\theta)\big].
\end{aligned}
\]
\end{lemma}
\begin{proof}
Fix $\theta\in\Theta$ and define the suboptimality gap $\Delta(\bar x):=f(\theta,\bar x)-F(\theta)\ge 0$.
Consider the event $\mathcal A:=\{\Delta(\bar x)<\Delta_{r_0}\}$.
By definition of $\Delta_{r_0}$, on $\mathcal A$ we must have $d_{\calS(\theta)}(\bar x,x^\star(\theta))<r_0$, so
\Cref{lem:uniform-hg} gives
\[
d_{\calS(\theta)}(\bar x,x^\star(\theta))
\ \le\
\left(\frac{\Delta(\bar x)}{c_{\mathrm{hg}}}\right)^{1/2}.
\]
Since the ambient Euclidean distance is bounded by the geodesic distance, this implies
$\|\bar x-x^\star(\theta)\|^2\le \Delta(\bar x)/c_{\mathrm{hg}}$ on $\mathcal A$.
On $\mathcal A^c$, use $\|\bar x-x^\star(\theta)\|^2\le 4\mathsf D^2$ (since $\calS(\theta)\subseteq \mathbb B_d(0;\mathsf D)$).
Therefore,
\[
\mathbb E\|\bar x-x^\star(\theta)\|^2
\ \le\
\frac{1}{c_{\mathrm{hg}}}\mathbb E\big[\Delta(\bar x)\big]
\ +\ 4\mathsf D^2\,\mathbb P(\mathcal A^c).
\]
Finally, $\mathbb P(\mathcal A^c)=\mathbb P(\Delta(\bar x)\ge \Delta_{r_0})\le \mathbb E[\Delta(\bar x)]/\Delta_{r_0}$ by Markov,
yielding the claim.
\end{proof}

\subsubsection{Sampler accuracy from R\'enyi divergence}
The hard-selection value-gap bound above is stated directly for candidates whose laws are R\'enyi-close to the ideal Gibbs law.
The next statements explain how finite-step ULA supplies this R\'enyi accuracy and recall the Poincar\'e/R\'enyi transportation control used in the proof of \Cref{lem:inner-quality-gap}.

We restate a quantitative order-$2$ R\'enyi mixing bound for ULA from \citet[Prop.~5.6]{masiha2025superquantile} and record its implication for Wasserstein error under a Poincar\'e inequality.

\begin{proposition}[ULA convergence in order-$2$ R\'enyi divergence {\citep[Prop.~5.6]{masiha2025superquantile}}]\label{prop:lmc-renyi}
Consider the unadjusted Langevin algorithm (ULA) iteration
\[
X^{k+1}=X^{k}-h\nabla G(X^{k})+\sqrt{2h}\,\xi^{k},
\qquad \xi^{k}\sim\mathcal N(0,I_d),
\]
targeting $\pi(\mathrm dx)\propto e^{-G(x)}\,\mathrm dx$.
Assume $\nabla G$ is $L_{G,2}$-Lipschitz and that $\pi$ satisfies a Poincar\'e inequality with constant $C_{\mathrm{PI}}$.
Let $\widehat\mu_n$ be the law of $X^{n}$.
Then for any sufficiently small $\varepsilon\in(0,1)$ and an appropriate stepsize choice (as required by the ULA theory),
one can ensure $R_2(\widehat\mu_n\|\pi)\le \varepsilon^2$ after
\begin{equation}\label{eq:lmc-renyi-steps}
n \;=\; \Theta\!\left(
C_{\mathrm{PI}}^{2}L_{G,2}^{2}\,d\,\varepsilon^{-2}\Big(R_3(\widehat\mu_0\|\pi)^2+\log^2(1/\varepsilon)\Big)
\right)
\end{equation}
iterations, where $R_q(\cdot\|\cdot)$ denotes the order-$q$ R\'enyi divergence.
Moreover, using the transportation-variance implication of the Poincar\'e inequality \citep{liu2020poincare} to relate
\(W_2\) and \(R_2\), we obtain
\begin{equation}\label{eq:lmc-w2-from-renyi}
\mathsf W_2(\widehat\mu_n,\pi)
\ \le\ 2\,C_{\mathrm{PI}}^{1/2}\,\big(e^{R_2(\widehat\mu_n\|\pi)}-1\big)^{1/2}
\ \le\ 2\,C_{\mathrm{PI}}^{1/2}\,C^{1/2}\,\varepsilon,
\end{equation}
where one may take the numerical constant $C:=2(e^{1/2}-1)$.
\end{proposition}

\begin{lemma}[Order-$3$ R\'enyi warm start for \(\lambda\)-scaled Gaussian initialization]\label{lem:gaussian-init-r3}
Let \Cref{assum_sample_comp} hold, and let \(g_\star(\theta):=\min_x g(\theta,x)\).
Fix a compact set \(\mathcal Z\subset\R^d\) and a variance parameter
\[
0<\tau<\frac{3}{2L_{g,2}}.
\]
For \(z\in\mathcal Z\), let
\[
\widehat\mu_{0,z}^{\lambda}:=\mathcal N(z,\tau\lambda I_d),
\qquad
0<\lambda\le1.
\]
Then there are constants \(C_{\mathrm{init}},C_{\mathrm{init}}'<\infty\), independent of
\(\theta,z,\lambda\), such that
\begin{equation}\label{eq:gaussian-init-r3}
R_3\!\left(\widehat\mu_{0,z}^{\lambda}\middle\|\gibbs_\theta^\lambda\right)
\le
\frac{C_{\mathrm{init}}}{\lambda}
\ +\ C_{\mathrm{init}}'\bigl(1+\log(1/\lambda)\bigr).
\end{equation}
Moreover, if \(z\in\calS(\theta)\), then the stronger bound
\begin{equation}\label{eq:gaussian-init-r3-ons}
R_3\!\left(\widehat\mu_{0,z}^{\lambda}\middle\|\gibbs_\theta^\lambda\right)
\le
C_{\mathrm{init}}'\bigl(1+\log(1/\lambda)\bigr)
\end{equation}
holds after increasing \(C_{\mathrm{init}}'\) if necessary.
\end{lemma}

\begin{proof}
Write \(p_{z,\lambda}\) for the density of \(\widehat\mu_{0,z}^{\lambda}\), and write
\[
Z_\theta^\lambda:=\int_{\R^d}\exp\{-g(\theta,x)/\lambda\}\,\ud x
\]
for the Gibbs normalizing constant. Since
\(\gibbs_\theta^\lambda(\ud x)=(Z_\theta^\lambda)^{-1}e^{-g(\theta,x)/\lambda}\ud x\), the order-$3$ R\'enyi
divergence is
\[
R_3\!\left(\widehat\mu_{0,z}^{\lambda}\middle\|\gibbs_\theta^\lambda\right)
=
\frac12\log\left[
(Z_\theta^\lambda)^2
\int_{\R^d}p_{z,\lambda}(x)^3\exp\{2g(\theta,x)/\lambda\}\,\ud x
\right].
\]
First, by \Cref{lem:uniform-qg-g} and \(\calS(\theta)\subseteq\mathbb B_d(0;\mathsf D)\),
\[
Z_\theta^\lambda
\le
e^{-g_\star(\theta)/\lambda}
\int_{\R^d}
\exp\!\left\{-\frac{\mu_{\mathrm{QG}}}{2\lambda}\mathrm{dist}(x,\calS(\theta))^2\right\}\ud x
\le
C_Z e^{-g_\star(\theta)/\lambda}
\]
for all \(0<\lambda\le1\), with \(C_Z\) uniform in \(\theta\). The last inequality follows because
\(\mathrm{dist}(x,\calS(\theta))\ge(\|x\|-\mathsf D)_+\), and the resulting Gaussian tail is integrable uniformly for
\(\lambda\le1\).

Next, set \(u=x-z\). By \(L_{g,2}\)-smoothness in \(x\),
\[
g(\theta,z+u)
\le
g(\theta,z)+\langle \nabla_x g(\theta,z),u\rangle+\frac{L_{g,2}}{2}\|u\|^2.
\]
Since
\[
p_{z,\lambda}(z+u)^3=(2\pi\tau\lambda)^{-3d/2}
\exp\!\left\{-\frac{3}{2\tau\lambda}\|u\|^2\right\},
\]
and \(a:=3/(2\tau)-L_{g,2}>0\), Gaussian integration gives
\[
\begin{aligned}
&\int_{\R^d}p_{z,\lambda}(x)^3\exp\{2g(\theta,x)/\lambda\}\,\ud x\\
&\qquad\le
C\,\lambda^{-d}
\exp\!\left\{
\frac{2g(\theta,z)}{\lambda}
+\frac{\|\nabla_x g(\theta,z)\|^2}{a\lambda}
\right\}.
\end{aligned}
\]
Combining the two displays yields
\[
R_3\!\left(\widehat\mu_{0,z}^{\lambda}\middle\|\gibbs_\theta^\lambda\right)
\le
\frac{g(\theta,z)-g_\star(\theta)}{\lambda}
\ +\ \frac{\|\nabla_x g(\theta,z)\|^2}{2a\lambda}
\ +\ \frac d2\log(1/\lambda)+C.
\]
The first two terms are uniformly bounded by \(C_{\mathrm{init}}/\lambda\), because
\(\Theta\times\mathcal Z\) is compact and \(g_\star\) is continuous on \(\Theta\).
This proves \eqref{eq:gaussian-init-r3}. If \(z\in\calS(\theta)\), then
\(g(\theta,z)=g_\star(\theta)\) and \(\nabla_x g(\theta,z)=0\), so only the logarithmic term remains, proving
\eqref{eq:gaussian-init-r3-ons}.
\end{proof}

\begin{corr}[Specialization to Gibbs sampling (R\'enyi accuracy implies Wasserstein control)]\label{cor:lmc-gibbs-w2}
Apply \Cref{prop:lmc-renyi} with $G(x)=g(\theta,x)/\lambda$ and $\pi=\gibbs_\theta^\lambda$.
Then $L_{G,2}=L_{g,2}/\lambda$ and the ULA update
\[
X^{k+1}=X^{k}-h\nabla_x g(\theta,X^{k})+\sqrt{2\lambda h}\,\xi^{k}
\]
corresponds to the same ULA scheme with stepsize $h' = \lambda h$ for $G$.
In particular, for any target R\'enyi tolerance $\varepsilon_{\mathrm R}\in(0,1)$, running the chain for
\[
n
=
\widetilde{\mathcal{O}}\!\left(
d\,C_{\mathrm{PI}}^{2}\lambda^{-2}\varepsilon_{\mathrm R}^{-2}
\bigl(1+R_3(\widehat\mu_0\|\gibbs_\theta^\lambda)^2\bigr)
\right)
\]
iterations ensures \(R_2(\widehat\mu_n\|\gibbs_\theta^\lambda)\le \varepsilon_{\mathrm R}^{2}\).
If, moreover, \(\widehat\mu_0=\widehat\mu_{0,z}^{\lambda}\) is the \(\lambda\)-scaled Gaussian initialization in
\Cref{lem:gaussian-init-r3} with \(z\in\mathcal Z\), then it suffices to run
\begin{equation}\label{eq:lmc-n-final}
n\ =\ \widetilde{\mathcal{O}}\!\left(d\,C_{\mathrm{PI}}^{2}\,\lambda^{-4}\,\varepsilon_{\mathrm R}^{-2}\right)
\end{equation}
iterations. If the Gaussian center satisfies \(z\in\calS(\theta)\), then the sharper logarithmic bound
\eqref{eq:gaussian-init-r3-ons} gives the improved sufficient count
\[
n\ =\ \widetilde{\mathcal{O}}\!\left(d\,C_{\mathrm{PI}}^{2}\,\lambda^{-2}\,\varepsilon_{\mathrm R}^{-2}\right).
\]
In all cases, once
\[
R_2(\widehat\mu_n\|\gibbs_\theta^\lambda)\ \le\ \varepsilon_{\mathrm R}^{2},
\]
we have
\[
\mathsf W_2(\widehat\mu_n,\gibbs_\theta^\lambda)
\ \le\
2\sqrt{(e-1)C_{\mathrm{PI}}}\,\varepsilon_{\mathrm R},
\qquad\text{and hence}\qquad
\mathsf W_1(\widehat\mu_n,\gibbs_\theta^\lambda)
\ \le\
2\sqrt{(e-1)C_{\mathrm{PI}}}\,\varepsilon_{\mathrm R}.
\]
\end{corr}

\subsubsection{Final assembly of the selection-error bound}
We now assemble the tube bound, the hard-selection value gap, the on-manifold growth estimate, and the ULA accuracy bound into a single bound on the squared (Euclidean) selection error
$\mathbb E\|\tilde x-x^\star(\theta)\|^2$.

\begin{proof}[Proof of \Cref{thm:selection-error-bound}]
Let \(\bar x\in\arg\min_{x\in\calS(\theta)}\|x-\tilde x\|\) be a measurable Euclidean projection of \(\tilde x\) onto
\(\calS(\theta)\).
By the Euclidean triangle inequality and $(a+b)^2\le 2a^2+2b^2$,
\[
\|\tilde x-x^\star(\theta)\|^2
\le
2\|\tilde x-\bar x\|^2+2\|\bar x-x^\star(\theta)\|^2
=
2\,\mathrm{dist}(\tilde x,\calS(\theta))^2+2\|\bar x-x^\star(\theta)\|^2.
\]
Taking expectations and applying \eqref{eq:dist-tilde-max2-lmc} from \Cref{lem:tube-width}, together with
\(e^{\varepsilon_{\mathrm R}^2/2}\le e^{1/2}\) since \(\varepsilon_{\mathrm R}\le1\), gives
\[
\mathbb E\|\tilde x-x^\star(\theta)\|^2
\le
2e^{1/2}C_{\mathrm{tube},2}\,\lambda\log(1+N)
+2\,\mathbb E\|\bar x-x^\star(\theta)\|^2.
\]
Then \Cref{lem:barx-distance} yields
\[
\mathbb E\|\bar x-x^\star(\theta)\|^2
\le
\Bigg(\frac{1}{c_{\mathrm{hg}}}+\frac{4\mathsf D^2}{\Delta_{r_0}}\Bigg)\,
\mathbb E\big[f(\theta,\bar x)-F(\theta)\big].
\]
Finally, the R\'enyi-controlled hard-selection bound \eqref{eq:inner-gap-hardmin} from \Cref{lem:inner-quality-gap} gives
\[
\mathbb E\big[f(\theta,\bar x)-F(\theta)\big]
\le
2L_{f,1}\,C_{\mathrm{tube}}\,\sqrt{\lambda\log(1+N)}
+C_1\,N^{-1/\dimk}
+4\sqrt{e-1}\,L_{f,1}\sqrt{N\,C_{\mathrm{PI}}}\,\varepsilon_{\mathrm R},
\]
which gives \eqref{eq:final-Ex2-bound}.
\end{proof}

\paragraph{Plug-in form.}
With \Cref{prop:lmc-renyi}, one can run the sampler until the candidate laws satisfy
$R_2(\nu_i\|\gibbs_\theta^\lambda)\le \varepsilon_{\mathrm R}^2$.
This yields an explicit end-to-end bound on $\mathbb E\|\tilde x-x^\star(\theta)\|^2$ in terms of:
(i) the number of candidates $N$,
(ii) the Gibbs temperature $\lambda$,
(iii) the achieved R\'enyi tolerance $\varepsilon_{\mathrm R}$ through the additive
\(\sqrt{N\,C_{\mathrm{PI}}}\varepsilon_{\mathrm R}\) contribution,
and (iv) the on-manifold approximation term $N^{-1/\dimk}$.

\subsection{Parameter choices and oracle complexity}\label{append:param-tuning}
This appendix states the parameter-level hyper-gradient error bound obtained by combining \eqref{eq:et2-explicit} with the selection-error bound \eqref{eq:final-Ex2-bound}, and then records a convenient parameter choice ensuring
$\sup_t \mathbb E\|e_t\|^2=\mathcal{O}(\varepsilon^2)$.
Together with \Cref{prop:outer-inexact}, this yields the $\varepsilon$-stationarity guarantee and conservative polynomial oracle bound reported in \Cref{sec:theory}.

\begin{lemma}[Hyper-gradient error in terms of algorithmic parameters]\label{lem:app-et2-bigO}
Using the notation of \Cref{thm:et2-bound,thm:selection-error-bound}, suppose the assumptions therein hold.
Assume additionally that \(\varepsilon_{\mathrm R}\le1\) and \(0<\gamma\le1\).
Then, uniformly over \(t\),
\begin{equation}\label{eq:app-et2-bigO}
\begin{aligned}
\mathbb E\|e_t\|^2
&=
\mathcal{O}\!\Bigg(
\frac{1}{\gamma^2}
\Big[
\lambda\log(1+N)
\ +\ \sqrt{\lambda\log(1+N)}
\ +\ N^{-1/\dimk}
\ +\ \sqrt{N\,C_{\mathrm{PI}}}\,\varepsilon_{\mathrm R}
\ +\ \eta_t^2
\Big]
\ +\ \gamma^2\\
&\qquad
\ +\ (1+R_v^2)N
\big(1+r(\gamma)^{2\mathsf n_{\theta x}}+\lambda^{\mathsf n_{\theta x}}\big)\\
&\qquad\qquad\qquad
\times\exp\!\Big(-\tfrac{c_{\mathrm{tube}}}{4}\tfrac{r(\gamma)^2}{\lambda}\Big)
\Bigg).
\end{aligned}
\end{equation}
\end{lemma}

\begin{proof}
We first condition on the current outer iterate \(\theta_t\). The constants in \Cref{thm:selection-error-bound} and
\Cref{lem:tube-width} are uniform over \(\theta\in\Theta\), so the resulting conditional bounds can be averaged over
\(\theta_t\) without changing the constants. Combining \eqref{eq:et2-explicit} with \eqref{eq:final-Ex2-bound} from
\Cref{thm:selection-error-bound} and the bound
$\mathbb E\big[\|\tilde x_t-x^\star(\theta_t)\|^2\mathbf 1_{\mathcal E_t}\big]\le \mathbb E\|\tilde x_t-x^\star(\theta_t)\|^2$,
we obtain the following bound, uniformly over \(t\):
\begin{equation}\label{eq:et2-plugged}
\begin{aligned}
\mathbb E\|e_t\|^2
\le\;&
3C_x^2 B_\gamma^2\,\mathcal R_{\mathrm{sel}}(N,\lambda,\varepsilon_{\mathrm R})
\ +\ \frac{12C_{\mathrm{lin}}^2}{\gamma^2}\,\eta_t^2
\ +\ 3C_{\mathrm{reg}}^2\gamma^2\\
&\quad
\ +\ C_{\mathrm{off}}(1+R_v^2)\,
\mathbb E\big[(1+\mathsf{D}_{t}^{2\mathsf n_{\theta x}})\mathbf 1_{\mathcal E_t^c}\big],
\end{aligned}
\end{equation}
where
\[
B_\gamma
:=
1+\frac{2}{\gamma}+\frac{2}{\gamma(c+\gamma)}
\]
and
\begin{equation}\label{eq:Rsel-expanded}
\begin{aligned}
\mathcal R_{\mathrm{sel}}(N,\lambda,\varepsilon_{\mathrm R})
:=\;&
2e^{1/2}C_{\mathrm{tube},2}\,\lambda\log(1+N)\\
&+
2\Bigg(\frac{1}{c_{\mathrm{hg}}}
\ +\ \frac{4\mathsf D^2}{\Delta_{r_0}}\Bigg)
\Big(
2L_{f,1}C_{\mathrm{tube}}\sqrt{\lambda\log(1+N)}
\ +\ C_1N^{-1/\dimk}\\
&\hspace{4.9cm}
\ +\ 4\sqrt{e-1}\,L_{f,1}\sqrt{N\,C_{\mathrm{PI}}}\,\varepsilon_{\mathrm R}
\Big).
\end{aligned}
\end{equation}

To control the off-tube moment, the candidate-marginal tube bound from Appendix~\ref{append:selection-error}
(Lemma~\ref{lem:tube-width}) and \(\varepsilon_{\mathrm R}\le1\) yield, for
\(\mathsf{D}_{t}=\mathrm{dist}(\tilde x_t,\calS(\theta_t))\),
\[
\mathbb P(\mathsf{D}_{t}>s)
\le
N\,e^{1/2}\,\sqrt{C_{\mathrm{tube}}}
\exp\!\Big(-\tfrac{c_{\mathrm{tube}}}{2}\tfrac{s^2}{\lambda}\Big),
\qquad s\ge 0.
\]
Indeed, \(\mathsf{D}_{t}\le \max_i\mathrm{dist}(X_{t,i},\calS(\theta_t))\), so the maximum-tail bound in
\Cref{lem:tube-width} applies. Integrating the tail gives
\begin{equation}\label{eq:offtube-tail-moment}
\begin{aligned}
\mathbb E\big[(1+\mathsf{D}_{t}^{2\mathsf n_{\theta x}})\mathbf 1_{\mathcal E_t^c}\big]
\le\;&
C_{\mathrm{tail}}\,N
\big(1+r(\gamma)^{2\mathsf n_{\theta x}}+\lambda^{\mathsf n_{\theta x}}\big)\\
&\quad\times
\exp\!\Big(-\tfrac{c_{\mathrm{tube}}}{4}\tfrac{r(\gamma)^2}{\lambda}\Big),
\end{aligned}
\end{equation}
where \(C_{\mathrm{tail}}\) depends only on \(\mathsf n_{\theta x}\) and problem-data constants. One way to see
\eqref{eq:offtube-tail-moment} is to use the identity
\[
\mathbb E\big[\mathsf{D}_{t}^{2p}\mathbf 1_{\{\mathsf{D}_{t}>r\}}\big]
=
r^{2p}\mathbb P(\mathsf{D}_{t}>r)
\ +\ \int_r^\infty 2p\,s^{2p-1}\mathbb P(\mathsf{D}_{t}>s)\,\ud s,
\]
with \(p=\mathsf n_{\theta x}\) and \(r=r(\gamma)\). In particular, if
\(\mathbb P(\mathsf{D}_{t}>s)\le K\exp(-a s^2/\lambda)\), then
\[
\mathbb P(\mathsf{D}_{t}>r)
\ +\ \mathbb E\big[\mathsf{D}_{t}^{2p}\mathbf 1_{\{\mathsf{D}_{t}>r\}}\big]
\le
C_pK\big(1+r^{2p}+\lambda^p\big)
\exp\!\Big(-\tfrac{a}{2}\tfrac{r^2}{\lambda}\Big),
\]
which gives \eqref{eq:offtube-tail-moment} with \(a=c_{\mathrm{tube}}/2\) and
\(K=N e^{1/2}\sqrt{C_{\mathrm{tube}}}\).

Substituting \eqref{eq:offtube-tail-moment} into \eqref{eq:et2-plugged} gives
\begin{equation}\label{eq:et2-before-simplification}
\begin{aligned}
\mathbb E\|e_t\|^2
\le\;&
C\,B_\gamma^2
\Big[
\lambda\log(1+N)
\ +\ \sqrt{\lambda\log(1+N)}
\ +\ N^{-1/\dimk}
\ +\ \sqrt{N\,C_{\mathrm{PI}}}\,\varepsilon_{\mathrm R}
\Big]\\
&\quad
+\frac{12C_{\mathrm{lin}}^2}{\gamma^2}\eta_t^2
\ +\ 3C_{\mathrm{reg}}^2\gamma^2\\
&\quad
\ +\ C(1+R_v^2)N
\big(1+r(\gamma)^{2\mathsf n_{\theta x}}+\lambda^{\mathsf n_{\theta x}}\big)
\exp\!\Big(-\tfrac{c_{\mathrm{tube}}}{4}\tfrac{r(\gamma)^2}{\lambda}\Big),
\end{aligned}
\end{equation}
for a problem-data constant \(C<\infty\) independent of
\(N,\lambda,\varepsilon_{\mathrm R},\gamma,\eta_t,R_v\), and \(t\).
Since \(c>0\) is fixed and \(0<\gamma\le1\),
\[
B_\gamma^2
=
\mathcal{O}\!\Big(\frac{1}{\gamma^2}\Big).
\]
	Applying this simplification to \eqref{eq:et2-before-simplification}, and absorbing numerical and problem-data constants,
	gives \eqref{eq:app-et2-bigO}.
	Since \Cref{thm:main-et2-bound} is the same bound with the exponentially small off-tube term abbreviated as
	\(\mathsf{OT}_{\gamma,\lambda,N,R_v}\), this proves the hyper-gradient error bound stated in \Cref{sec:theory}.
	\end{proof}

\paragraph{Choosing parameters to achieve $\mathbb E\|e_t\|^2=\mathcal{O}(\varepsilon^2)$.}
In the non-singleton regime \(\dimk\ge1\), fix a sufficiently small target $\varepsilon\in(0,1)$.
As in \Cref{thm:selection-error-bound}, \(\varepsilon_{\mathrm R}\) denotes the square-root R\'enyi tolerance, so the sampler accuracy condition is \(R_2(\nu_i\|\gibbs_{\theta_t}^\lambda)\le \varepsilon_{\mathrm R}^2\).
We take \(\varepsilon\) small enough that the choices below satisfy \(0<\lambda\le\lambda_0\), \(\varepsilon_{\mathrm R}\le1\), and \(0<\gamma\le1\).
To ensure $\sup_t\mathbb E\|e_t\|^2=\mathcal{O}(\varepsilon^2)$ it suffices to make each term on the right-hand side of \eqref{eq:app-et2-bigO} of order at most $\varepsilon^2$.
One convenient choice is:
(i) set $\gamma=\varepsilon$ so that $C_{\mathrm{reg}}\gamma=\mathcal{O}(\varepsilon)$,
(ii) solve the ridge system to residual $\eta_t=\gamma\varepsilon=\varepsilon^2$, so that \(\eta_t^2/\gamma^2=\mathcal{O}(\varepsilon^2)\),
(iii) choose \(R_v\ge (2/\gamma)(L_{f,1}+\eta_t)\), so \(R_v=\mathcal{O}(\varepsilon^{-1})\), and
(iv) choose the sampling/selection parameters so that the remaining bracketed terms satisfy
\[
\lambda\log(1+N)
\ +\ \sqrt{\lambda\log(1+N)}
\ +\ N^{-1/\dimk}
\ +\ \sqrt{N\,C_{\mathrm{PI}}}\,\varepsilon_{\mathrm R}
= \mathcal{O}(\varepsilon^4).
\]
A sufficient scaling is
\[
N=\left\lceil\varepsilon^{-4\dimk}\right\rceil,
\qquad
\lambda=\frac{\varepsilon^{8}}{\log(1+N)},
\qquad
\varepsilon_{\mathrm R}=\varepsilon^{4+2\dimk}.
\]
Then $\lambda\log(1+N)=\varepsilon^{8}$, $\sqrt{\lambda\log(1+N)}=\varepsilon^{4}$, and $N^{-1/\dimk}=\mathcal{O}(\varepsilon^{4})$.
Moreover, treating \(C_{\mathrm{PI}}\) as a problem-data constant,
\[
\sqrt{N\,C_{\mathrm{PI}}}\,\varepsilon_{\mathrm R}
=
\mathcal{O}\!\Big(\varepsilon^{-2\dimk}\cdot \varepsilon^{4+2\dimk}\Big)
=
\mathcal{O}(\varepsilon^{4}).
\]
Thus the displayed sampling/selection term is $\mathcal{O}(\varepsilon^4)$, and the corresponding contribution to
$\mathbb E\|e_t\|^2$ is $\mathcal{O}(\varepsilon^2)$ after multiplication by the prefactor $\mathcal{O}(1/\gamma^2)=\mathcal{O}(1/\varepsilon^2)$.
It remains to check the off-tube moment term. By Lemma~\ref{lem:tube-invertibility},
\[
r(\gamma)=\min\left\{\rho,\frac{\gamma}{2\bar L_{g,3}}\right\},
\]
so for all sufficiently small \(\varepsilon\), \(r(\gamma)\ge c_r\varepsilon\) for a problem-data constant \(c_r>0\).
Consequently,
\[
\frac{r(\gamma)^2}{\lambda}
\ \ge\
c_r'\,\varepsilon^{-6}\log(1+N).
\]
Since \(R_v=\mathcal{O}(\varepsilon^{-1})\), the off-tube term
\[
(1+R_v^2)N
\big(1+r(\gamma)^{2\mathsf n_{\theta x}}+\lambda^{\mathsf n_{\theta x}}\big)
\exp\!\Big(-\tfrac{c_{\mathrm{tube}}}{4}\tfrac{r(\gamma)^2}{\lambda}\Big)
\]
is bounded by a polynomial in \(\varepsilon^{-1}\) times
\(\exp(-c\,\varepsilon^{-6}\log(1+N))\) for a problem-data constant \(c>0\).
It is therefore super-polynomially small in \(\varepsilon^{-1}\) and, in particular, is \(\mathcal{O}(\varepsilon^2)\).

Substituting this control into \Cref{prop:outer-inexact} (with stepsize $1/L_F$) yields an $\varepsilon$-stationary guarantee after
\[
T=\mathcal{O}\!\left(\frac{L_F\,(F(\theta_0)-F_\star)}{\varepsilon^2}\right)
\]
outer iterations.

\paragraph{Total oracle complexity.}
Per outer iteration, \textsc{HG-MS} uses (a) $NK$ evaluations of $\nabla_x g$ for the sampler (\Cref{alg:lmc-ula}),
(b) $N$ evaluations of $f(\theta_t,X_{t,i})$ for hard selection, and (c) $\#\mathrm{HVP}$ Hessian--vector products for CG to reach residual tolerance $\eta_t$.
First consider the sampler/function-evaluation cost,
\[
\mathcal{O}\big(T(NK+N)\big).
\]
Assume each ULA chain is initialized from the \(\lambda\)-scaled Gaussian warm start in
\Cref{lem:gaussian-init-r3}, with centers in a fixed compact set. If the ULA sampler is run for $K$ steps so that the
square-root R\'enyi tolerance satisfies
$\varepsilon_{\mathrm{R}}=\varepsilon^{4+2\dimk}$, then \Cref{cor:lmc-gibbs-w2} gives
$K=\widetilde{\mathcal{O}}(d\,C_{\mathrm{PI}}^{2}\,\lambda^{-4}\varepsilon_{\mathrm{R}}^{-2})$.
Under the above scaling of \((N,\lambda,\varepsilon_{\mathrm R})\), this gives
$K=\widetilde{\mathcal{O}}\!\big(d\,C_{\mathrm{PI}}^{2}\,\log^{4}(1+N)\,\varepsilon^{-40-4\dimk}\big)$ and hence
\[
T\,N\,K
=\widetilde{\mathcal{O}}\!\big(d\,C_{\mathrm{PI}}^{2}\,N\log^{4}(1+N)\,\varepsilon^{-42-4\dimk}\big)
=\widetilde{\mathcal{O}}\!\big(d\,C_{\mathrm{PI}}^{2}\,\varepsilon^{-42-8\dimk}\big),
\]
which dominates the additional \(TN\) hard-selection evaluations.
If HVPs are counted as oracle calls, the total accounting becomes
\[
\mathcal{O}\big(T(NK+N+\#\mathrm{HVP})\big)
=
\widetilde{\mathcal{O}}\!\big(d\,C_{\mathrm{PI}}^{2}\,\varepsilon^{-42-8\dimk}\big)
+\mathcal{O}(T\,\#\mathrm{HVP}).
\]
The following remark computes the CG/HVP count and shows that the displayed sampler term dominates it under the parameter choice above.

\begin{remark}[HVP count for the CG solve]\label{rem:hvp-count}
On the tube event, the ridge system solved by CG has coefficient matrix
\[
A_t:=\nabla^2_{xx}g(\theta_t,\tilde x_t)+\gamma I.
\]
Lemma~\ref{lem:tube-invertibility} gives \(\lambda_{\min}(A_t)\ge\gamma/2\), while the \(L_{g,2}\)-smoothness of
\(g(\theta,\cdot)\) gives \(\lambda_{\max}(A_t)\le L_{g,2}+\gamma\). Hence
\[
\kappa(A_t)
\le
\frac{2(L_{g,2}+\gamma)}{\gamma}.
\]
The standard CG estimate for an SPD system \(Av=b\),
\[
\|v_k-v_\star\|_A
\le
2\left(\frac{\sqrt{\kappa(A)}-1}{\sqrt{\kappa(A)}+1}\right)^k
\|v_0-v_\star\|_A,
\]
see \citet[Thm.~5.3]{nocedal2006numerical}, implies the residual bound
\[
\|b-Av_k\|
\le
2\sqrt{\kappa(A)}
\left(\frac{\sqrt{\kappa(A)}-1}{\sqrt{\kappa(A)}+1}\right)^k
\|b\|,
\]
when initialized at \(v_0=0\).
Each CG iteration applies \(A_t\) once, which requires one Hessian-vector product with
\(\nabla^2_{xx}g(\theta_t,\tilde x_t)\) plus vector operations.
Therefore, to reach the residual tolerance \(\eta_t\) in $\|b_t-(H_t+\gamma I)\tilde v_t\|\le \eta_t$, it suffices to take
\[
\#\mathrm{HVP}
=
\mathcal{O}\!\left(
\sqrt{\kappa(A_t)}
\log\frac{\sqrt{\kappa(A_t)}\,\|\nabla_x f(\theta_t,\tilde x_t)\|}{\eta_t}
\right)
=
\mathcal{O}\!\left(
\gamma^{-1/2}\log\frac{\gamma^{-1/2}L_{f,1}}{\eta_t}
\right),
\]
where the last step uses \(\|\nabla_x f(\theta_t,\tilde x_t)\|\le L_{f,1}\).
With \(\gamma=\varepsilon\) and \(\eta_t=\varepsilon^2\), this gives
\[
\#\mathrm{HVP}=\widetilde{\mathcal{O}}(\varepsilon^{-1/2}),
\qquad
T\,\#\mathrm{HVP}=\widetilde{\mathcal{O}}(\varepsilon^{-5/2}),
\]
which is negligible compared with the sampler-dominated cost
\(\widetilde{\mathcal{O}}(d\,C_{\mathrm{PI}}^{2}\varepsilon^{-42-8\dimk})\).
The final multiplication by \(\nabla^2_{\theta x}g(\theta_t,\tilde x_t)v_t\) requires one additional mixed Hessian-vector product per outer iteration, which is also dominated by the sampler term.
\end{remark}

\section{Experimental details}\label{append:exp-details}
\subsection{Common protocol}
\paragraph{Metrics.}
For data hyper-cleaning we report \emph{accuracy} (\% correct).
For imbalanced loss tuning we report \emph{balanced accuracy} (macro recall), i.e., the average of per-class recall,
reported as a percentage.
For controlled GPT-2 source reweighting we report token-level validation/test loss and source-weight recovery error.
For real-world LLM mathematical reasoning we report GSM8K and MATH accuracy (\% correct), and for instruction following
we report MT-Bench score on the standard \(0\)--\(10\) scale.
\paragraph{Confidence intervals.}
For the MNIST experiments, tables/plots report mean $\pm$ $95\%$ confidence intervals across independent seeds, using a two-sided $t$-interval.
The real-world LLM results in \Cref{tab:llm-realworld-math} are single-seed matched-budget comparisons. We complement them
with uncertainty estimates computed from the saved evaluation files; see \Cref{tab:llm-eval-uncertainty}. For GSM8K and
MATH, we treat each test example as correct or incorrect and compute a Wilson \(95\%\) binomial interval. For MT-Bench, we
resample the saved question-level GPT-4o judge scores over the \(80\) standard questions. The table reports the range of
\(95\%\) confidence-interval half-widths across the cells in each benchmark, measured in accuracy percentage points for
GSM8K/MATH and MT-Bench score points for instruction following. These intervals quantify uncertainty from the finite
evaluation set and do not estimate variance from rerunning source reweighting or final fine-tuning with different seeds.

\begin{table}[t]
  \centering
  \caption{\textbf{Evaluation uncertainty for the real-world LLM tables.}
  The ranges summarize the \(95\%\) confidence-interval half-widths over the reported cells in the corresponding benchmark. GSM8K/MATH intervals use Wilson binomial intervals over saved binary correctness records; MT-Bench intervals use nonparametric bootstrap intervals over saved question-level judge scores.}
  \label{tab:llm-eval-uncertainty}
  \begin{adjustbox}{max width=\linewidth}
  \begin{tabular}{llcc}
    \toprule
    Benchmark & CI construction & Evaluation units & 95\% half-width range \\
    \midrule
    GSM8K, 7B/9B math table & Wilson binomial, \(n=1319\) & accuracy points & \(2.07\)--\(2.64\) \\
    MATH, 7B/9B math table & Wilson binomial, \(n=4989\) & accuracy points & \(1.22\)--\(1.35\) \\
    GSM8K, 32B panel & Wilson binomial, \(n=1319\) & accuracy points & \(1.42\)--\(2.23\) \\
    MATH, 32B panel & Wilson binomial, \(n=4989\) & accuracy points & \(1.37\)--\(1.39\) \\
    MT-Bench instruction table & Question-level bootstrap, \(n=80\) & MT-Bench points & \(0.30\)--\(0.50\) \\
    \bottomrule
  \end{tabular}
  \end{adjustbox}
\end{table}
\paragraph{Time budget.}
For time-budget plots and budget-sweep diagnostics, we measure wall-clock elapsed time for each algorithm. Fixed-budget
tables compare methods after the same per-method wall-clock budget within each benchmark/backbone setting.
\paragraph{Code.}
The experimental code and configuration files are available at
\url{https://github.com/msmasiha1997/Experiment_HG_MS_paper}.

For the MNIST experiments, we compare \textsc{HG-MS} against
\textsc{HG-baseline} \citep{lorraine2020optimizing}, V-PBGD and G-PBGD \citep{kwon2023penalty,shen2023penalty}, and
IAPTT-GM \citep{xiao2023generalized,shaban2019truncated}, and we report fixed-budget means $\pm$ $95\%$ confidence
intervals across seeds. Throughout, $\mathrm{CE}(\cdot,\cdot)$ denotes cross-entropy (negative log-likelihood).

\paragraph{Lower-level solvers and candidate generators.}
\Cref{tab:lower-solver-summary} summarizes the lower-level optimizer used by the iterative baselines and the candidate
generator used by \textsc{HG-MS} in each experiment. For both MNIST benchmarks, the lower-level optimizer family is
included in the same validation-loss pilot sweep for \textsc{HG-MS} and for the bilevel baselines: each \textsc{HG-MS}
branch and each baseline lower-level solver is run with the better of AdamW and SGD under the selected configuration. For the
controlled GPT-2 and real-world LLM experiments, the corresponding lower-level optimizer choices are stated in the table
and detailed in the benchmark-specific subsections below.

\begin{table}[t]
  \centering
  \caption{\textbf{Lower-level solvers and \textsc{HG-MS} candidate generators.}
  Summary of the optimizer used to generate \textsc{HG-MS} candidates and the corresponding lower-level optimizer used by
  iterative baselines. The iteration counts are lower-level optimizer updates per outer update or per candidate branch,
  up to the benchmark-specific time budget; detailed learning rates, batch sizes, and regularization values are given in
  the corresponding subsections.}
  \label{tab:lower-solver-summary}
  \begin{adjustbox}{max width=\linewidth}
  \begin{tabular}{p{0.21\linewidth}p{0.29\linewidth}p{0.31\linewidth}p{0.19\linewidth}}
    \toprule
    Experiment & \textsc{HG-MS} candidate generator & Lower optimizer for iterative baselines & Lower-level updates \\
    \midrule
    MNIST data hyper-cleaning
    & AdamW/SGD selected by validation-loss pilot sweep; reported \textsc{HG-MS} uses \(8\) optimizer branches
    & AdamW/SGD selected by validation-loss pilot sweep for \textsc{HG-baseline}, V-PBGD, G-PBGD, and IAPTT-GM; reported \textsc{HG-baseline} uses AdamW unroll
    & \textsc{HG-MS}/\textsc{HG-baseline}: \(20\); V-PBGD/G-PBGD/IAPTT-GM: same wall-clock budget \\
    \addlinespace
    MNIST imbalanced loss tuning
    & AdamW/SGD selected by validation-loss pilot sweep; reported \textsc{HG-MS} uses one base model plus \(4\) perturbed optimizer branches
    & AdamW/SGD selected by validation-loss pilot sweep for \textsc{HG-baseline}, V-PBGD, G-PBGD, and IAPTT-GM; reported \textsc{HG-baseline}/\textsc{HG-MS} lower optimizer is AdamW
    & \(60\) warm-start; \textsc{HG-baseline}: \(30\); \textsc{HG-MS} extra branches: \(5\) by default; V-PBGD: \(20\); IAPTT-GM: \(5\) \\
    \addlinespace
    Controlled GPT-2 source reweighting
    & SGD optimizer branches on the full GPT-2 parameter state
    & \textsc{HG-baseline} and \textsc{ScaleBiO}: SGD lower optimizer under the same validation-loss pilot-sweep protocol
    & Default \(1\); verified multilingual \textsc{HG-MS}: \(8\); configurations selected by validation loss \\
    \addlinespace
    Real-world LLM math, 7B/8B/9B
    & \(N=4\) one-update SGD LoRA branches
    & \textsc{HG-baseline} and \textsc{ScaleBiO}: one SGD lower-model update
    & \(1\) for \textsc{HG-MS} and iterative bilevel baselines \\
    \addlinespace
    Real-world LLM instruction following
    & \(N=8\) one-update SGD LoRA branches
    & \textsc{HG-baseline} and \textsc{ScaleBiO}: one SGD lower-model update
    & \(1\) for \textsc{HG-MS} and iterative bilevel baselines \\
    \addlinespace
    Real-world \texttt{Qwen2.5-32B} math
    & \(N=4\) one-update SGD LoRA branches
    & \textsc{HG-baseline} and \textsc{ScaleBiO}: one SGD lower-model update
    & \(1\) for \textsc{HG-MS} and iterative bilevel baselines \\
    \bottomrule
  \end{tabular}
  \end{adjustbox}
\end{table}

\subsection{Data hyper-cleaning (MNIST label noise)}\label{subsec:dhc}
\paragraph{Setup.}
We study \emph{data hyper-cleaning} on MNIST~\citep{lecun1998mnist}, a standard bilevel benchmark in which the upper-level variable assigns a
soft weight to each (potentially corrupted) training label.
The lower-level (follower) trains a classifier by minimizing a weighted training loss on the (possibly noisy) training
set, while the upper-level (leader) minimizes the validation loss on a small clean validation set.
We use a 2-layer sigmoid MLP with architecture \(784\to300\to10\) and follow the split used in our code: we subsample the MNIST \emph{training}
split and randomly partition it into $5{,}000$ lower-level training points, $500$ validation points, and $10{,}000$ test
points (all drawn from the same underlying MNIST training set).
We corrupt a fraction $\rho\in\{0.4,0.6,0.8\}$ of the training labels uniformly at random.
This yields a challenging regime (small training set with substantial label noise) in which explicit hyper-cleaning is
meaningful and differences between upper-level updates become visible under a fixed time budget.

\paragraph{Bilevel model.}
Let $\theta\in\R^{n_{\mathrm{tr}}}$ be training-example hyper-weights and let $y$ denote the follower network parameters.
Writing $w=\sigma(\theta)$ (sigmoid reparameterization), we consider the bilevel problem
\[
\min_{\theta\in\R^{n_{\mathrm{tr}}}}\ \min_{y\in\arg\min_{z} g(\theta,z)} f(\theta,y),
\]
which is the optimistic formulation: for each $\theta$, among all minimizers of the weighted training loss $g(\theta,\cdot)$ we select
the one with smallest validation loss $f(\theta,\cdot)$.
The lower-level objective is a weighted training loss
\[
g(\theta,y)\;=\;\frac{1}{n_{\mathrm{tr}}}\sum_{i=1}^{n_{\mathrm{tr}}} w_i\cdot \mathrm{CE}\big(h_y(x_i^{\mathrm{tr}}),\ \widetilde c_i^{\mathrm{tr}}\big)\;+\;\omega_{\mathrm{dhc}}\|y\|_2^2,
\]
and the upper-level objective is the validation loss on clean labels,
\[
f(y)\;=\;\frac{1}{n_{\mathrm{val}}}\sum_{j=1}^{n_{\mathrm{val}}}\mathrm{CE}\big(h_y(x_j^{\mathrm{val}}),\ c_j^{\mathrm{val}}\big).
\]

\paragraph{Methods and optimization.}
We use the following MNIST baselines: \textsc{HG-baseline}, V-PBGD, G-PBGD, and IAPTT-GM.
All methods use the same 2-layer sigmoid MLP architecture.
This label-noise regime can produce lower-level candidates with similar corrupted-training losses but different clean
validation losses. \textsc{HG-MS} is designed to select the candidate with the smallest clean validation loss before
computing the hyper-gradient.

\paragraph{Implementation details.}
The AdamW/SGD lower-optimizer choice is selected by the validation-loss pilot sweep described in the common protocol.
For \textsc{HG-baseline}, the selected lower-level solver is a differentiable AdamW unroll with \(20\) inner steps per outer update,
lower learning rate \(5\times10^{-3}\), AdamW default betas, and weight decay \(0\). The upper optimizer is AdamW on the
data-weight logits with learning rate \(0.1\) and weight decay \(0\). The lower regularization coefficient above is
\(\omega_{\mathrm{dhc}}=10^{-4}\).
For \textsc{HG-MS}, we use \(8\) candidate branches, and each branch uses the selected lower-optimizer family from the
same AdamW/SGD sweep. Each branch takes \(20\) lower updates per outer step with branch
learning rate \(2\times10^{-2}\), minibatch size \(256\), and branch weight decay \(0\). The reported runs use no parameter initialization noise by default
\((\texttt{branch\_init\_noise\_std}=0)\) and no learning-rate diversity across branches. The upper learning rate starts at
\(0.1\) and uses the default cosine schedule to \(0.01\); the explicit \(\rho=0.8\), \(120\)s script uses an inverse
schedule with decay steps \(50\). The conjugate-gradient solve in \textsc{HG-MS} uses ridge \(10^{-2}\), residual tolerance
\(10^{-3}\), at most \(30\) iterations, and full-batch HVP/CG when \(\texttt{cg\_batch\_size}=0\).

\paragraph{Time-budget comparison.}
\Cref{fig:dhc-time} reports mean and $95\%$ confidence intervals under a fixed wall-clock budget
\emph{per algorithm} on a single NVIDIA H100 GPU (60s for $\rho\in\{0.4,0.6\}$; 120s for $\rho=0.8$),
using $5$ runs for each corruption rate.
In the hardest setting $\rho=0.8$, \textsc{HG-MS} improves the mean test accuracy from $64.5\%$ (\textsc{HG-baseline})
to $66.1\%$, and outperforms the penalty/alternating baselines under the same time budget.

\begin{table}[t]
  \centering
  \caption[\textbf{Data hyper-cleaning (MNIST): train/test accuracies.}]{%
  \textbf{Data hyper-cleaning (MNIST): train/test accuracies under different corruption rates.}
  Mean $\pm$ 95\% CI accuracies (in \%) at a fixed wall-clock budget per algorithm:
  60s for $\rho\in\{0.4,0.6\}$ (5 runs each), and 120s for $\rho=0.8$ (5 runs).}
  \label{tab:dhc-acc}
  \begin{adjustbox}{max width=\linewidth}
  \begin{tabular}{lcccccc}
    \toprule
    & \multicolumn{2}{c}{$\rho=0.4$} & \multicolumn{2}{c}{$\rho=0.6$} & \multicolumn{2}{c}{$\rho=0.8$} \\
    Method & Train & Test & Train & Test & Train & Test \\
    \midrule
    HG-baseline & 79.69$\pm$1.78 & 81.22$\pm$1.45 & 70.19$\pm$7.56 & 73.09$\pm$4.86 & 63.48$\pm$2.48 & 64.54$\pm$2.62 \\
    HG-MS       & \textbf{89.00$\pm$1.60} & \textbf{86.98$\pm$1.02} & \textbf{81.12$\pm$2.71} & \textbf{80.08$\pm$2.52} & \textbf{66.47$\pm$2.69} & \textbf{66.08$\pm$3.51} \\
    V-PBGD      & 61.62$\pm$4.56 & 60.43$\pm$4.69 & 62.12$\pm$2.65 & 61.23$\pm$3.44 & 60.29$\pm$2.16 & 60.07$\pm$2.08 \\
    G-PBGD      & 70.73$\pm$4.48 & 69.21$\pm$4.93 & 67.40$\pm$7.94 & 66.26$\pm$7.86 & 58.84$\pm$3.87 & 58.77$\pm$4.93 \\
    IAPTT-GM    & 71.46$\pm$2.76 & 71.04$\pm$3.45 & 69.86$\pm$2.74 & 69.68$\pm$3.29 & 63.84$\pm$2.95 & 63.48$\pm$2.14 \\
    \bottomrule
  \end{tabular}
  \end{adjustbox}
\end{table}

\begin{figure*}[t]
  \centering
  \begin{subfigure}{0.32\linewidth}
    \centering
    \includegraphics[width=\linewidth]{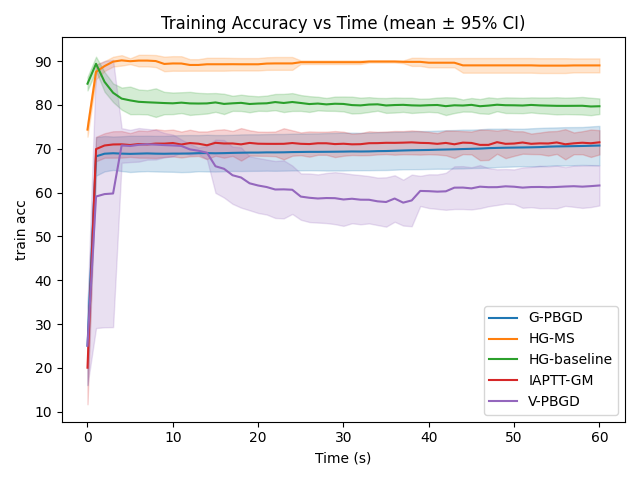}
    \caption{Train, $\rho=0.4$.}
  \end{subfigure}\hfill
  \begin{subfigure}{0.32\linewidth}
    \centering
    \includegraphics[width=\linewidth]{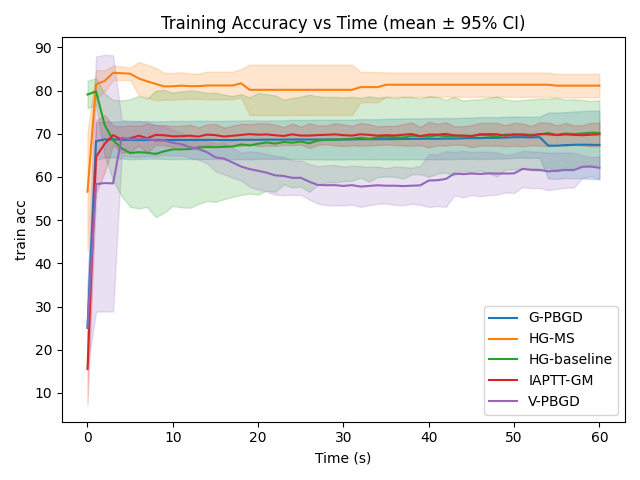}
    \caption{Train, $\rho=0.6$.}
  \end{subfigure}
  \hfill
  \begin{subfigure}{0.32\linewidth}
    \centering
    \includegraphics[width=\linewidth]{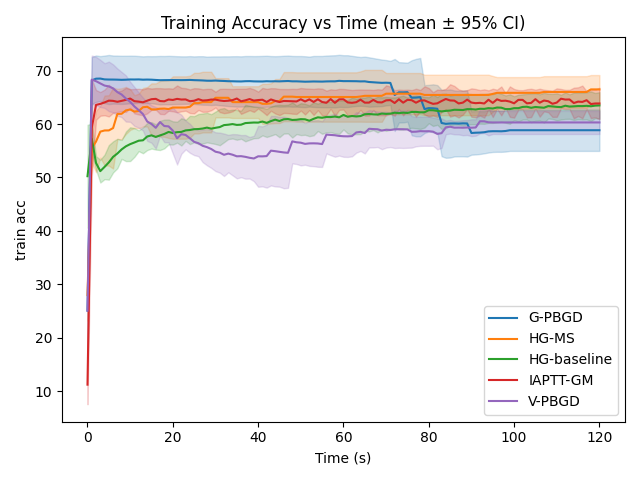}
    \caption{Train, $\rho=0.8$.}
  \end{subfigure}

  \vspace{0.25em}

  \begin{subfigure}{0.32\linewidth}
    \centering
    \includegraphics[width=\linewidth]{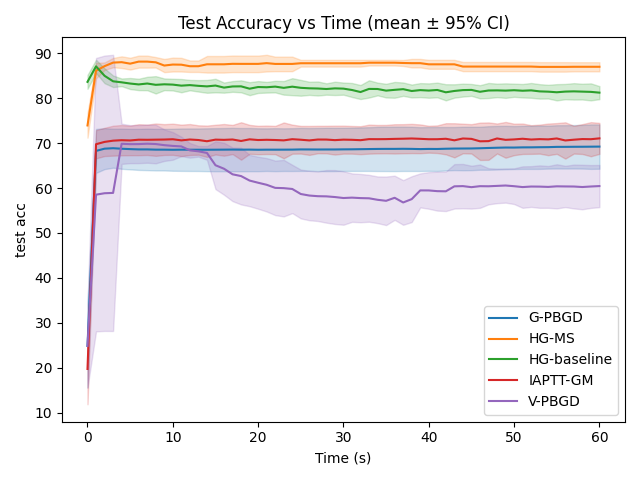}
    \caption{Test, $\rho=0.4$.}
  \end{subfigure}\hfill
  \begin{subfigure}{0.32\linewidth}
    \centering
    \includegraphics[width=\linewidth]{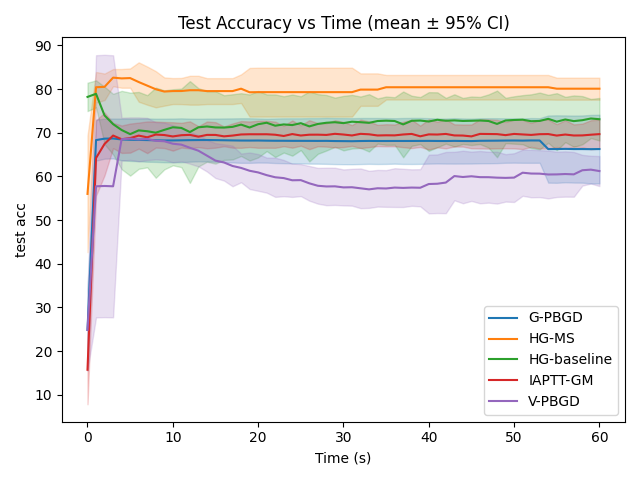}
    \caption{Test, $\rho=0.6$.}
  \end{subfigure}\hfill
  \begin{subfigure}{0.32\linewidth}
    \centering
    \includegraphics[width=\linewidth]{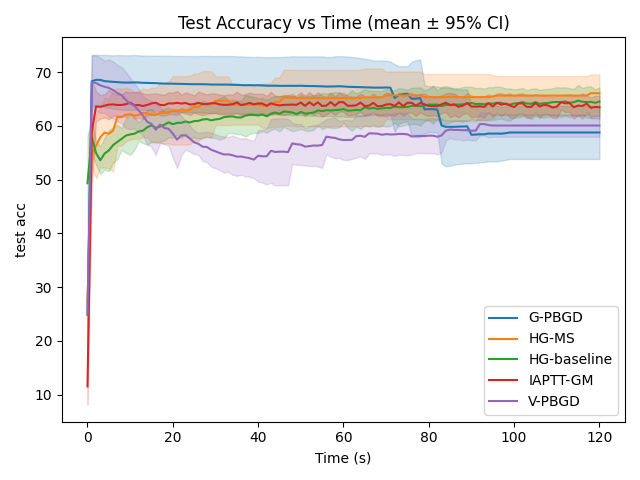}
    \caption{Test, $\rho=0.8$.}
  \end{subfigure}

  \caption{\textbf{Data hyper-cleaning (MNIST).} Train (top row) and test (bottom row) accuracies vs.\ time at
  corruption rates $\rho\in\{0.4,0.6,0.8\}$. Accuracies are plotted as mean $\pm$ $95\%$ confidence intervals (fixed
  budget per algorithm: $60\,\mathrm{s}$ for $\rho\in\{0.4,0.6\}$ with 5 runs, and $120\,\mathrm{s}$ for $\rho=0.8$ with 5 runs).}
  \label{fig:dhc-time}
\end{figure*}

\subsection{Parametric loss tuning for imbalanced data (MNIST)}\label{subsec:imb}
\paragraph{Setup.}
We study bilevel \emph{loss tuning} for class-imbalanced MNIST classification.
From the MNIST training split, we construct an imbalanced lower-level training set by keeping each example of digit
$c\in\{0,\ldots,9\}$ with probability $p_c=b^{c}$ (we use $b=0.3$), which yields a long-tailed label distribution in which
larger digits are much rarer.
We also hold out a small \emph{class-balanced} validation set (1000 images total, roughly 100 per digit) for the upper-level
objective, and we report performance on the standard MNIST test set.
Because the training distribution is imbalanced while validation/test are class-balanced, we evaluate using
\emph{balanced accuracy} (macro-averaged recall).
The upper-level variable parameterizes a class-wise transformation of the logits inside the cross-entropy loss,
\[
\ell_{\delta,\gamma}(c;\,x)\;=\;\mathrm{CE}\!\Big(\,h(x)\odot \gamma+\delta,\;c\,\Big),
\]
where $h(x)\in\R^{10}$ are the classifier logits, $\delta\in\R^{10}$ is a per-class bias, and $\gamma\in\R^{10}$ is a
per-class scale. In the implementation, we optimize unconstrained raw variables
\((\tilde\delta,\tilde\gamma)\in\R^{10}\times\R^{10}\) and use the transformed loss parameters
\[
\delta=B_\delta\tanh(\tilde\delta),
\qquad
\gamma=\mathbf 1+B_\gamma\tanh(\tilde\gamma),
\]
with fixed bounds \(B_\delta=0.5\) and \(B_\gamma=0.15\). Thus every class scale is positive and the unshifted scale
\(\gamma=\mathbf 1\) is attained at \(\tilde\gamma=0\).
The lower level minimizes the imbalanced training loss $g(\delta,\gamma,\cdot)$, while the upper level minimizes a
class-balanced validation objective $f(\delta,\gamma,\cdot)$ (implemented via inverse-frequency class weights).
In our implementation, the same logit shift/scale $(\delta,\gamma)$ is applied in both $g$ and $f$, so $\theta$ affects the upper-level objective both directly and through the trained network parameters.
Training uses data augmentation (denoted $\tilde x$), while validation and test examples are left unchanged.

\paragraph{Bilevel model.}
Let \(\theta=(\tilde\delta,\tilde\gamma)\in\R^{10}\times\R^{10}\) be the raw upper-level variable, let
\(\delta=\delta(\theta)\) and \(\gamma=\gamma(\theta)\) denote the transformed loss parameters above, and let $y$ denote
the CNN parameters.
Let $n_{\mathrm{tr}}$ and $n_{\mathrm{val}}$ denote the sizes of the (imbalanced) training set and the (class-balanced) validation set.
We solve the optimistic bilevel problem
\[
\min_{\theta}\ \min_{y\in\arg\min_{z} g(\theta,z)} f(\theta,y),
\]
i.e., for each $\theta$, among all minimizers of the training loss $g(\theta,\cdot)$ we select the one with the smallest
validation loss $f(\theta,\cdot)$.
The lower-level objective is
\[
g(\theta,y)\;=\;\frac{1}{n_{\mathrm{tr}}}\sum_{i=1}^{n_{\mathrm{tr}}}\mathrm{CE}\!\big(h_y(\tilde x_i)\odot\gamma+\delta,\ c_i\big)\;+\;\frac{\omega_y}{2}\|y\|_2^2,
\]
and the upper-level objective is
\[
f(\theta,y)\;=\;\frac{1}{n_{\mathrm{val}}}\sum_{j=1}^{n_{\mathrm{val}}}u_{c_j}\cdot \mathrm{CE}\!\big(h_y(x_j)\odot\gamma+\delta,\ c_j\big)\;+\;\frac{\omega_\delta}{2}\|\delta\|_2^2\;+\;\frac{\omega_\gamma}{2}\|\gamma-\mathbf 1\|_2^2,
\]
where $u_c$ is an inverse-frequency class weight computed on the validation set, and the regularizers are applied to the
transformed loss parameters \(\delta(\theta)\) and \(\gamma(\theta)\).
In the reported \textsc{HG-MS}/\textsc{HG-baseline} implementation, the regularization coefficients are
\(\omega_y=5\times10^{-4}\) and \(\omega_\delta=\omega_\gamma=10^{-2}\).
Here $\tilde x_i$ denotes the train-time augmented view of $x_i$, while validation and test examples are uncorrupted.

\paragraph{Methods.}
We use the common baselines listed in the overview above.
For \textsc{HG-MS}, we maintain a small ensemble of lower-level candidates trained on \(g(\cdot)\): one base run plus
\(4\) optimizer branches with slightly perturbed optimizer settings. At each outer step, \textsc{HG-MS} selects the
candidate that minimizes the validation objective \(f(\cdot)\). In this finite-budget nonconvex training pipeline,
different branches can produce candidate models with different validation losses and balanced accuracies, so explicit
selection can change the point through which the hyper-gradient is computed.

\paragraph{Implementation details.}
The CNN has architecture
\(\mathrm{Conv2d}(1,32,5,\mathrm{pad}=2)\to\mathrm{ReLU}\to\mathrm{MaxPool}\),
\(\mathrm{Conv2d}(32,64,5,\mathrm{pad}=2)\to\mathrm{ReLU}\to\mathrm{MaxPool}\), flatten,
\(\mathrm{Linear}(64\cdot7\cdot7,256)\to\mathrm{ReLU}\to\mathrm{Linear}(256,10)\).
The train and validation batch sizes are both \(64\), with full-batch validation used when
\(\texttt{val\_full\_batch=True}\).
The AdamW/SGD lower-optimizer choice is selected by the validation-loss pilot sweep described in the common protocol for
all \textsc{HG-MS} branches and for the lower-level solvers used by \textsc{HG-baseline}, V-PBGD, G-PBGD, and IAPTT-GM.
For \textsc{HG-baseline}/\textsc{HG-MS}, the selected lower optimizer is AdamW with learning rate \(10^{-3}\), weight decay
\(5\times10^{-4}\), betas \((0.9,0.999)\), cosine warm restart scheduler \(T_0=20\), minimum learning rate
\(0.2\) times the base learning rate, and gradient clipping at \(1.0\). The upper optimizer is AdamW by default with
learning rate \(10^{-2}\), upper \(L_2\) coefficient \(10^{-4}\), and upper-gradient clipping at \(0.5\).
We use \(60\) lower warm-start steps and then \(30\) lower steps per outer update in the default
\texttt{minsel\_effective}/time-budget setting. \textsc{HG-MS} uses the base model plus \(4\) extra branches, for \(5\)
candidates total. The branches use initialization noise \(0.01\); in the default \texttt{minsel\_effective} setting,
branch learning-rate and weight-decay multipliers are log-spaced from \(0.5\) to \(2.0\), while the showcase settings use
the wider range \([0.25,4.0]\). In time-budget mode, the default number of branch lower steps is \(5\) unless overridden by the
environment.

For the penalty/alternating baselines, the reported configurations are selected by the same validation-loss pilot
protocol. V-PBGD and G-PBGD use train/validation batch sizes \(64/64\), lower regularization
\(5\times10^{-4}\), upper regularizers \(\omega_\delta=\omega_\gamma=10^{-2}\), \(lrx=2\times10^{-2}\),
\(lry=5\times10^{-2}\), \(\gamma_{\max}=10\), and momentum \(0.9\); V-PBGD additionally uses
\(\texttt{lr\_inner}=5\times10^{-2}\) and \(\texttt{inner\_itr}=20\). IAPTT-GM uses the default time-budget settings
\(K_{\mathrm{inner}}=5\), \(\texttt{lr\_inner}=0.03\), \(lrx=lrz=5\times10^{-4}\), lower regularization
\(5\times10^{-4}\), \(\omega_\delta=\omega_\gamma=10^{-4}\), the same CNN model, and the Adam-style inner option enabled.

\paragraph{Relation to the manifold theory.}
This benchmark tests the select-then-differentiate principle in a finite-budget nonconvex training pipeline, where the
inner problems are not solved exactly and different optimizer branches can reach candidate models with different
validation losses and balanced accuracies. This is precisely the practical regime in which explicit candidate selection
can matter: \textsc{HG-MS} first selects the candidate with the best upper-level value and then computes the hyper-gradient
through that selected candidate. The experiment is therefore complementary to the manifold theory in \Cref{sec:hyp_diff}.
Because the explicit \(L_2\)-regularization may make the exact global lower-level minimizers isolated (\(\dimk=0\)), we do
not use this benchmark as direct evidence for positive-dimensional minimizer manifolds. Rather, it shows that the same
select-then-differentiate mechanism can be useful even in the singleton/isolated-minimizer limit when lower-level training
is finite-budget and nonconvex.

\paragraph{Time-budget comparison.}
\Cref{fig:imb-acc} reports mean and $95\%$ confidence intervals under a fixed wall-clock budget of $120\,\mathrm{s}$
\emph{per algorithm} on a single NVIDIA H100 GPU.
In this regime, \textsc{HG-MS} improves test balanced accuracy by about $0.97$ percentage points over
\textsc{HG-baseline} at the same wall-clock budget.

\begin{table}[t]
  \centering
  \caption{\textbf{Imbalanced loss tuning (MNIST): final balanced accuracies.}
  Mean $\pm$ $95\%$ CI balanced accuracy (\%) at the last time point ($120\,\mathrm{s}$ per algorithm).}
  \label{tab:imb-acc}
  \begin{adjustbox}{max width=\linewidth}
  \begin{tabular}{lcc}
    \toprule
    Method & Train BalAcc & Test BalAcc \\
    \midrule
    HG-baseline       & 99.46$\pm$0.20 & 95.48$\pm$0.77 \\
    HG-MS (ours)      & \textbf{99.51$\pm$0.26} & \textbf{96.45$\pm$0.11} \\
    V-PBGD            & 91.31$\pm$2.19 & 90.54$\pm$1.95 \\
    G-PBGD            & 75.88$\pm$2.17 & 77.40$\pm$1.41 \\
    IAPTT-GM          & 82.74$\pm$2.62 & 82.13$\pm$2.30 \\
    \bottomrule
  \end{tabular}
  \end{adjustbox}
\end{table}

\begin{figure}[t]
  \centering
  \includegraphics[width=\linewidth]{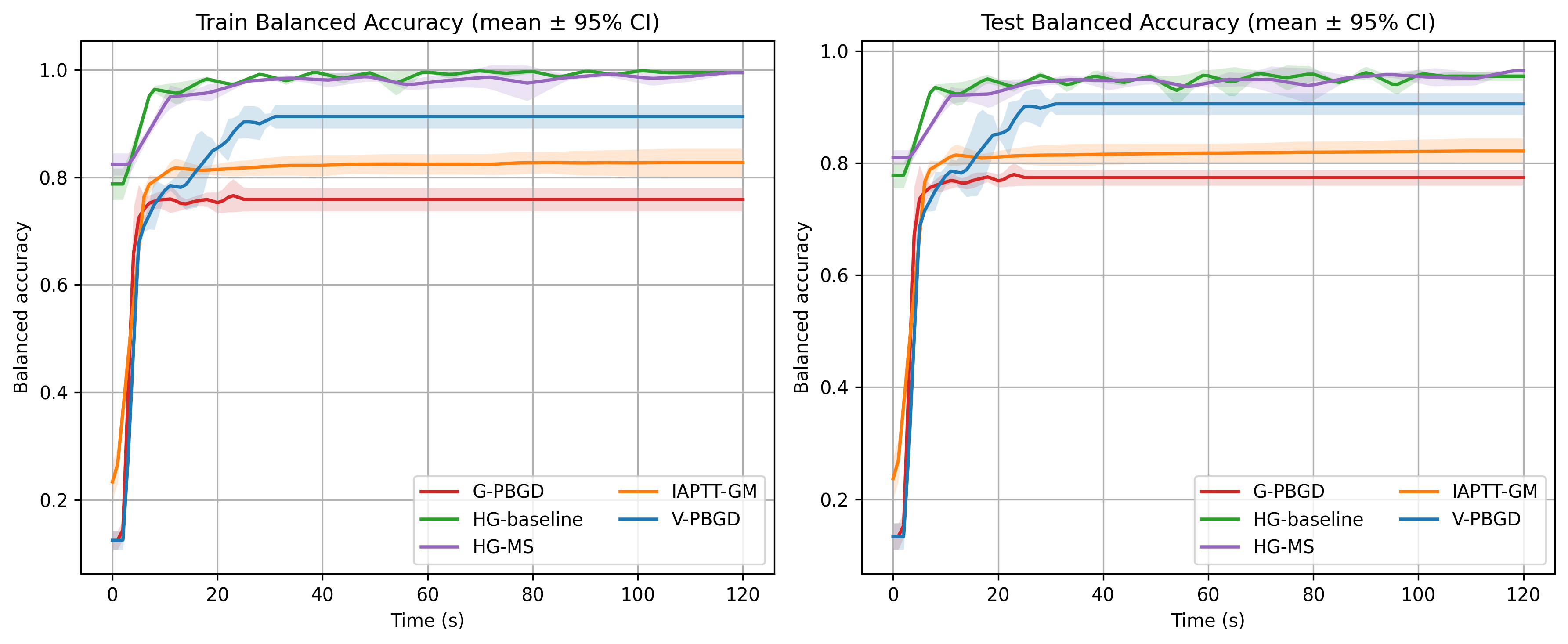}
  \caption{\textbf{Imbalanced loss tuning (MNIST).} Train/test balanced accuracy vs.\ time (mean and $95\%$ CI; 120s per algorithm).}
  \label{fig:imb-acc}
\end{figure}

\subsection{Completed outer updates in the MNIST fixed-budget runs}\label{append:mnist-outer-counts}
To make the fixed-budget comparisons more transparent, \Cref{tab:mnist-outer-updates} reports the completed outer
updates counted from the saved logs. For data hyper-cleaning, we report the mean over \(5\) runs with the min--max range
in parentheses. For imbalanced loss tuning, we report the mean over \(3\) repetitions at \(120\)s; a min--max range is
included when it was recorded. These counts are not meant to equalize the cost of one outer step across methods:
\textsc{HG-MS} completes fewer outer updates because each update is substantially heavier, as it trains multiple
candidate branches, performs explicit minima selection, and then computes the selected hyper-gradient.

\begin{table}[t]
  \centering
  \caption{\textbf{Completed outer updates in MNIST fixed-budget experiments.}
  Entries are mean completed outer updates, with min--max ranges in parentheses when available.}
  \label{tab:mnist-outer-updates}
  \begin{adjustbox}{max width=\linewidth}
  \begin{tabular}{lccc}
    \toprule
    \multicolumn{4}{c}{\textbf{Data hyper-cleaning: mean over 5 runs}} \\
    \midrule
    Method & \(\rho=0.4\), \(60\)s & \(\rho=0.6\), \(60\)s & \(\rho=0.8\), \(120\)s \\
    \midrule
    \textsc{HG-MS} & \(312\ (310--313)\) & \(311\ (310--312)\) & \(626\ (625--628)\) \\
    \textsc{HG-baseline} & \(1{,}191\ (1{,}185--1{,}199)\) & \(1{,}200\ (1{,}185--1{,}205)\) & \(2{,}402\ (2{,}380--2{,}416)\) \\
    V-PBGD & \(18{,}747\ (17{,}983--19{,}007)\) & \(18{,}680\ (17{,}826--18{,}904)\) & \(37{,}016\ (36{,}995--37{,}104)\) \\
    G-PBGD & \(18{,}381\ (18{,}337--18{,}452)\) & \(18{,}303\ (18{,}271--18{,}324)\) & \(36{,}233\ (36{,}196--37{,}284)\) \\
    IAPTT-GM & \(5{,}263\ (5{,}148--5{,}427)\) & \(5{,}287\ (5{,}075--5{,}712)\) & \(10{,}745\ (10{,}051--11{,}556)\) \\
    \midrule
    \multicolumn{4}{c}{\textbf{Imbalanced loss tuning: mean over 3 repetitions at \(120\)s}} \\
    \midrule
    Method & \multicolumn{3}{c}{Completed outer updates} \\
    \midrule
    \textsc{HG-MS} & \multicolumn{3}{c}{\(155\ (154--156)\)} \\
    \textsc{HG-baseline} & \multicolumn{3}{c}{\(620\)} \\
    V-PBGD & \multicolumn{3}{c}{\(11{,}164\)} \\
    G-PBGD & \multicolumn{3}{c}{\(10{,}998\)} \\
    IAPTT-GM & \multicolumn{3}{c}{\(3{,}568\)} \\
    \bottomrule
  \end{tabular}
  \end{adjustbox}
\end{table}

\subsection{Controlled GPT-2 source reweighting}\label{subsec:llm}\label{append:exp-llm}
\paragraph{Bilevel formulation.}
We use the two-source small-scale LLM tasks from ScaleBiO \citep{pan2025scalebio} to test whether the outer variable learns a nontrivial source mixture. Let \(\theta\in\R^2\) be source logits, \(\omega=\softmax(\theta)\in\Delta^1\) the induced source weights, and \(y\) the GPT-2 parameter vector. The lower-level and upper-level losses are
\[
g(\theta,y)
\;:=\;
\sum_{s=1}^2 \omega_s\,
\frac{1}{n_s}\sum_{i=1}^{n_s}\ell(y;\xi^{\mathrm{tr}}_{s,i}),
\qquad
f(y)
\;:=\;
\frac{1}{n_{\mathrm{val}}}\sum_{j=1}^{n_{\mathrm{val}}}\ell(y;\xi^{\mathrm{val}}_{j}).
\]
Here \(\mathcal D^{\mathrm{tr}}_s=\{\xi^{\mathrm{tr}}_{s,i}\}_{i=1}^{n_s}\) are the two training sources, \(\mathcal D^{\mathrm{val}}=\{\xi^{\mathrm{val}}_j\}_{j=1}^{n_{\mathrm{val}}}\) is the target validation set, and \(\ell(y;\xi)\) is the token-level language-model negative log-likelihood. Thus the lower level trains on a source-weighted mixture, while the upper level evaluates held-out target loss.

\paragraph{Backbone and shared code path.}
The controlled experiments use GPT-2 (124M) \citep{radford2019language} as the lower-level model, with maximum sequence
length \(128\). Unlike the real-world LLM experiments below, this controlled GPT-2 code path does \emph{not} use LoRA:
it trains the full GPT-2 trainable parameter state loaded by \texttt{AutoModelForCausalLM}. Thus LoRA rank, scaling, and
dropout are not applicable to this controlled experiment.
The implementation follows the select-then-differentiate structure of \Cref{alg:hg-minsel-lmc}: at each outer step, it
constructs lower-level candidate GPT-2 parameter states, scores those candidates on the upper-level validation objective,
selects the candidate with the smallest validation loss, and computes the source-weight hyper-gradient through the
selected candidate. These branches are the practical counterpart of the \(N\) candidate samples in \Cref{alg:hg-minsel-lmc};
they are not exact ULA samples, and we do not inject Langevin Gaussian noise in these runs. Candidate diversity instead
comes from finite-budget stochastic training effects such as minibatch sampling, data order, optimizer stochasticity, and
branch-level perturbations. Thus the formal convergence guarantee applies to the ULA-based version analyzed in
\Cref{sec:theory}, while the LLM experiments test the same select-then-differentiate principle in the standard training
pipeline. For the outer update, we score candidate branches on shared validation data rather than on separate tiny random
batches; otherwise branch selection is dominated by sampling noise instead of true held-out performance. All methods share
the same single-GPU PyTorch implementation in the repository, so the comparison isolates the outer optimization rule
rather than changes in training infrastructure.
\paragraph{Task construction.}
For \emph{data denoising}, we follow the two-source construction of \citet{pan2025scalebio}: the training set contains $1000$ clean Alpaca examples~\citep{taori2023alpaca} and $9000$ Alpaca examples with synthetically corrupted outputs, where each target sequence is replaced by ``.''. The validation and test sets each contain $1000$ clean Alpaca examples. Thus the desired source weight is $\omega^\star=(1,0)$. For \emph{multilingual training}, the training set contains $40{,}000$ Alpaca-GPT4-ZH examples and $40{,}000$ Alpaca-GPT4-EN examples~\citep{peng2023instruction}, while the validation and test sets each contain $600$ ZH plus $400$ EN examples, so the target source weight is $\omega^\star=(0.6,0.4)$.
\paragraph{Methods and optimization details.}
We compare \textsc{Uniform}, \textsc{ScaleBiO}, \textsc{HG-baseline}, and \textsc{HG-MS}. \textsc{Uniform} fixes
\(\omega=(0.5,0.5)\). Source logits are initialized at \(\theta=0\), hence the initial learned mixture is uniform, unless
a fixed source mixture is explicitly passed to the script. For each tunable method, we selected the reported
controlled-task configuration from a small method-specific pilot grid using only held-out validation loss. The grid varied
learning-rate and update parameters appropriate to each method; test loss was not inspected for hyperparameter selection.
In the reported \textsc{HG-MS} configurations, branch selection uses only the upper-level validation loss.

The default controlled script uses SGD as the lower optimizer, batch size \(64\), selection validation batch size \(128\),
evaluation batch size \(64\), one inner step, two branches, branch noise \(10^{-3}\), model learning rate \(10^{-5}\),
outer learning rate \(10^{-2}\), weight decay \(10^{-2}\), and source-logit gradient clipping at \(10\). The verified
multilingual \textsc{HG-MS} configuration uses model learning rate \(10^{-3}\), outer learning rate \(1.5\), batch
size \(8\), selection batch size \(64\), \(8\) selection batches, \(8\) inner steps, and \(2\) candidate branches. In both
controlled tasks, the outer hyper-gradient is computed through the selected branch using shared validation batches.
\paragraph{Metrics.}
For both tasks we report token-level validation and test losses, together with the source-weight recovery error $\|\omega-\omega^\star\|_1$. In denoising, the latter remains a direct diagnostic because the clean-data optimum is unambiguous. In multilingual training, however, we treat held-out loss as the primary downstream metric and use $\|\omega-\omega^\star\|_1$ only as a calibration diagnostic, since finite-horizon nonconvex training can make the loss-optimal source mixture differ from the nominal validation composition.

\paragraph{Results.}
\Cref{tab:llm-controlled} in the main text reports the completed controlled runs and discusses the resulting source
mixtures and validation/test losses.

\subsection{Detailed protocols for the real-world LLM experiments}\label{append:exp-llm-real}
\paragraph{Shared setup for the main real-world tables.}
The real-world LLM tables in the main text use single-seed, single-H100 LoRA pipelines in which the outer variable is the source-logit vector $\theta$ and the simplex weights are $\omega=\softmax(\theta)$. Unless stated otherwise, every reported run uses one H100 in bf16 precision\footnote{A 16-bit floating-point format is used to train/evaluate neural networks faster and with less GPU memory than float32. The bf16 has a much larger exponent range, so it is usually more numerically stable for LLM training. On H100 GPUs, bf16 is well supported and is the standard choice for large-model LoRA fine-tuning.}, and each method/backbone pair is run independently on that single-GPU path. For every method, the learned source weights are converted into a final sampled supervised-fine-tuning set, and the target model is then fine-tuned for one fixed epoch. The final supervised fine-tuning optimizer is AdamW with learning rate \(8\times10^{-6}\), no scheduler or warmup, default AdamW weight decay \(0.01\), and no explicit max-gradient-norm clipping in the checked code paths. Thus the comparison is intended to isolate the source-reweighting rule rather than the final fine-tuning budget. All main-table runs reported in \Cref{tab:llm-realworld-math,tab:llm-realworld-instr-mtbench} use seed $2424$. We interpret the real-world LLM results as controlled matched-budget comparisons under a fixed implementation and seed; \Cref{tab:llm-eval-uncertainty} reports evaluation-set uncertainty, while the MNIST experiments provide multi-seed confidence intervals. The non-iterative baselines (\textsc{Uniform}, \textsc{LESS}, and \textsc{RHO-LOSS}) use the standard adapted pipeline without an outer reweighting horizon. For both tables, the iterative methods (\textsc{ScaleBiO}, \textsc{HG-baseline}, and \textsc{HG-MS}) use matched per-backbone wall-clock reweighting budgets of approximately $9$h for \texttt{Qwen-2-7B}, $12$h for \texttt{Llama-3-8B}, and $15$h for \texttt{Gemma-2-9B}.

\paragraph{Choice of wall-clock budgets.}
The iterative LLM methods are finite-budget procedures: the reweighting horizon determines the learned source distribution, which is then used to sample the final fine-tuning set. We choose the wall-clock budgets to give the iterative methods substantial reweighting time under the single-H100 implementation, while keeping the comparison feasible and matched. The budget is fixed for a given task/backbone and then shared by \textsc{ScaleBiO}, \textsc{HG-baseline}, and \textsc{HG-MS}; because the backbones have different per-step costs, matching is by wall-clock time rather than by the number of completed outer updates.

The \(9\)h budget used for the main \texttt{Qwen-2-7B} mathematical-reasoning comparison was fixed before evaluating the final checkpoints on the GSM8K and MATH test sets. The explicit \texttt{Qwen-2-7B} math sweep in \Cref{tab:llm-budget-sensitivity-qwen} is included as an audit trail and sensitivity diagnostic, not as a test-set-based procedure for choosing the main stopping budget. Longer 15h and 21h runs are therefore reported to show budget sensitivity under the same pipeline, while the main table keeps the pre-specified 9h budget. The same fixed-budget principle is used for the other LLM task/backbone budgets: the chosen horizons are intended to give the iterative algorithms sufficient reweighting time under a matched wall-clock constraint, not to reward additional wall-clock after inspecting downstream test metrics.

\paragraph{Relation to the ScaleBiO protocol.}
Although the table layouts follow the corresponding ScaleBiO tables, the numbers below are adapted single-H100 LoRA reproductions under our shared code path and final fine-tuning protocol, not direct copies of the original ScaleBiO results.
Concretely, \Cref{tab:llm-realworld-math} follows the layout of Table~4 of \citet{pan2025scalebio}; in \Cref{tab:llm-realworld-math-large}, panel~(a) follows the layout of their Table~5 and panel~(b) follows the layout of their Table~2.
The comparison with ScaleBiO should therefore be read as an adapted reproduction under our single-H100 LoRA implementation rather than as a direct transcription of the original ScaleBiO tables.
For dataset transparency, \Cref{tab:instr-source-portfolio,tab:math-source-portfolio} mirror the dataset-summary tables from Appendix~B.2 of \citet{pan2025scalebio}; the reported sizes, dataset types, and licenses follow their Tables~10 and~11, while the experimental protocol text below states where our adapted reproduction differs.

\begin{table*}[t]
  \centering
  \caption{\textbf{Instruction-following source portfolio.}
  Training sources used for the eighteen-source instruction-following benchmark, following Table~10 of \citet{pan2025scalebio}. The first nine rows are the high-quality/general instruction or chat sources; the last nine rows are multilingual Alpaca variants.}
  \label{tab:instr-source-portfolio}
  \begin{adjustbox}{max width=\textwidth}
  \begin{tabular}{llcl}
    \toprule
    Source & \#Samples & Kind & License \\
    \midrule
    AlpacaGPT4 & 52K & Instruction & Apache-2.0 \\
    ShareGPT4 & 6K & Conversation & Apache-2.0 \\
    SlimOrca & 518K & Instruction & MIT \\
    AlpacaChat & 20K & Conversation & Apache-2.0 \\
    OpenOrcaGPT4 & 1M & Instruction & MIT \\
    WildChat & 1M & Conversation & AI2 ImpACT \\
    LMSYS-Chat & 1M & Conversation & LMSYS-Chat-1M \\
    GPTeacher & 89K & Instruction & MIT \\
    Airoboros & 59K & Conversation & CC-BY-4.0 \\
    \midrule
    Alpaca-es & 52K & Instruction & CC-BY-4.0 \\
    Alpaca-de & 50K & Instruction & Apache-2.0 \\
    Alpaca-ja & 52K & Instruction & CC-BY-NC-SA-4.0 \\
    Alpaca-ko & 50K & Instruction & CC-BY-NC-4.0 \\
    Alpaca-ru & 30K & Instruction & CC-BY-4.0 \\
    Alpaca-it & 52K & Instruction & CC-BY-NC-SA-4.0 \\
    Alpaca-fr & 55K & Instruction & Apache-2.0 \\
    Alpaca-zh & 49K & Instruction & CC-BY-4.0 \\
    Alpaca-pt & 52K & Instruction & CC-BY-NC-4.0 \\
    \bottomrule
  \end{tabular}
  \end{adjustbox}
\end{table*}

\begin{table*}[t]
  \centering
  \caption{\textbf{Mathematical-reasoning source portfolio and validation sources.}
  Training sources above the midrule and validation/tuning sources below the midrule, following Table~11 of \citet{pan2025scalebio}.}
  \label{tab:math-source-portfolio}
  \begin{adjustbox}{max width=\textwidth}
  \begin{tabular}{llcl}
    \toprule
    Source & \#Samples & Kind & License \\
    \midrule
    DART-Math & 591K & Math & MIT \\
    SlimOrca & 518K & Instruction & MIT \\
    UltraInteract SFT & 289K & Reasoning & MIT \\
    MathInstruct & 262K & Reasoning & MIT \\
    Orca-Math word problems & 200K & Math & MIT \\
    WizardLM Evol-Instruct V2 & 196K & Instruction & MIT \\
    Magicoder Evol-Instruct & 110K & Coding & Apache-2.0 \\
    ShareGPT-Vicuna unfiltered & 94K & Instruction & Apache-2.0 \\
    GPTeacher-General-Instruct & 89K & Instruction & MIT \\
    GPT4-LLM-Cleaned & 55K & Instruction & Apache-2.0 \\
    \midrule
    GSM8K & 7.5K & Math & MIT \\
    Competition Math & 12.5K & Math & MIT \\
    Self-OSS-Instruct & 50.7K & Coding & ODC-By \\
    MMLU & 116K & Science & MIT \\
    ARC-Easy & 5.2K & Instruction & CC-BY-SA-4.0 \\
    ARC-Challenge & 2.6K & Instruction & CC-BY-SA-4.0 \\
    \bottomrule
  \end{tabular}
  \end{adjustbox}
\end{table*}

\paragraph{Baselines and compared methods.}
We compare against both non-bilevel data-selection methods and bilevel source-reweighting methods. \textsc{Uniform Weighting} is the source-agnostic control: it samples the final fine-tuning set uniformly across the available sources (10K examples for instruction following and 20K examples for mathematical reasoning) and then uses the same final supervised fine-tuning and evaluation pipeline as every other method. \textsc{LESS} \citep{xia2024less} is an influence-aware data-selection baseline for targeted instruction tuning; in our adapted pipeline, it scores candidate examples from the same source portfolio using gradient/influence alignment with the target validation data and then selects the highest-scoring examples for final fine-tuning. \textsc{RHO-LOSS} \citep{mindermann2022prioritized} is a reference-model-based prioritized training baseline; it ranks examples by a reducible-loss criterion that subtracts a reference model's estimated irreducible loss from the target model loss, then uses the resulting scores to form the final training set. \textsc{ScaleBiO} \citep{pan2025scalebio} is the main bilevel competitor: it learns source logits through the scalable first-order bilevel reweighting update proposed in that work, converts the learned simplex weights into source counts, and samples the final training set from those counts. \textsc{HG-baseline} is our one-branch implicit hyper-gradient baseline, following the standard hyper-gradient principle of differentiating through a single lower-level trajectory \citep{lorraine2020optimizing}; it uses the same implementation as \textsc{HG-MS} but removes the multi-branch minima-selection step. \textsc{HG-MS} is our proposed select-then-differentiate method, which trains several lower-level candidate branches under the current source weights, selects the branch with the best held-out objective, and computes the source-weight hyper-gradient through that selected branch.

ScaleBiO also reports SOBA \citep{dagreou2022framework} as a bilevel baseline, but it runs out of memory in their real-world LLM tables. We do not include SOBA as a main-table row in our single-H100 reproduction for the same practical reason: its auxiliary-state and Hessian-system bookkeeping is not competitive with the memory budget needed for LoRA fine-tuning of the target LLMs. This omission keeps the main comparison focused on methods that can complete the shared single-H100 final-training pipeline.

\paragraph{HG-MS implementation details for the real-world LLM runs.}
All real-world LLM \textsc{HG-MS} experiments instantiate the select-then-differentiate structure of
\Cref{alg:hg-minsel-lmc} with source logits \(\theta\) as the upper-level variable, LoRA adapter weights as the
lower-level variable, and \(N\) candidate LoRA branches at each outer step. Concretely, each branch starts from the
current adapter state, takes one SGD lower-level update on the source-weighted training loss \(g(\theta,\cdot)\), is scored on
held-out validation batches for \(f\), and the selected branch is the branch with the smallest validation selection
objective. These branches are generated by the same one-step SGD lower-model update used by the iterative bilevel baselines
rather than by an exact ULA sampler with injected Langevin noise. Branch diversity comes from independent minibatches/seeds and Gaussian parameter
noise for branch index \(>0\); the real-world LLM code paths do not use branch learning-rate multiplier or optimizer
diversity. The hyper-gradient update for \(\theta\) is then computed through this selected branch, matching the
selection-and-differentiation order of \Cref{alg:hg-minsel-lmc}.

The branch counts used in the reported real-world LLM experiments are fixed as follows. For the mathematical-reasoning
experiments using the ten-source portfolio in \Cref{tab:math-source-portfolio}, \textsc{HG-MS} uses \(N=4\) branches for
the reported \texttt{Llama-3-8B}, \texttt{Qwen-2-7B}, and \texttt{Gemma-2-9B} rows; the one-branch hyper-gradient baseline
uses \(N=1\). The large-model \texttt{Qwen2.5-32B} configuration uses \(N=4\) branches. For the instruction-following
experiments using the eighteen-source portfolio in \Cref{tab:instr-source-portfolio}, the reported \textsc{HG-MS} runs
use \(N=8\) branches and \textsc{HG-baseline} uses \(N=1\).

All iterative real-world methods initialize source logits at \(\theta=0\), so the initial source distribution is uniform
unless a fixed source mixture is explicitly loaded. For the source-logit update, \textsc{HG-MS} and
\textsc{HG-baseline} use Adam, while \textsc{ScaleBiO} uses SGD. This is separate from the lower-model update: for the
real-world LLM experiments, \textsc{HG-MS}, \textsc{HG-baseline}, and \textsc{ScaleBiO} all use one SGD lower-model update
per branch or inner step. The method-specific outer learning rates are reported in
\Cref{tab:llm-implementation-hparams}.

\begin{table}[t]
  \centering
  \caption{\textbf{LLM reweighting implementation hyperparameters.}
  Outer learning rates for source-logit updates in the real-world LLM runs.}
  \label{tab:llm-implementation-hparams}
  \begin{adjustbox}{max width=\linewidth}
  \begin{tabular}{llc}
    \toprule
    Experiment & Method & Outer LR \\
    \midrule
    Math, \texttt{Qwen-2-7B} & \textsc{HG-MS}/\textsc{HG-baseline} & \(0.05\) \\
    Math, \texttt{Qwen-2-7B} & \textsc{ScaleBiO} & \(1.0\) \\
    Math, \texttt{Llama-3-8B}/\texttt{Gemma-2-9B} & \textsc{HG-MS}/\textsc{HG-baseline} & \(0.05\) \\
	Math, \texttt{Llama-3-8B}/\texttt{Gemma-2-9B} & \textsc{ScaleBiO} & \(0.1\) \\
    Instruction & \textsc{HG-MS} & \(0.05\) \\
    Instruction & \textsc{HG-baseline} & \(10.0\) \\
    Instruction & \textsc{ScaleBiO} & \(0.01\) \\
    Large-model \texttt{Qwen2.5-32B} & \textsc{HG-MS}/\textsc{HG-baseline}	 & \(0.03\) \\
    Large-model \texttt{Qwen2.5-32B} & \textsc{ScaleBiO} & \(0.05\) \\
    \bottomrule
  \end{tabular}
  \end{adjustbox}
\end{table}

The corresponding lower-model optimizer settings are summarized in \Cref{tab:llm-lower-model-hparams}. Each iterative row
uses one SGD lower-model update per branch or inner step; the table separates this lower-model learning rate from the
source-logit learning rates in \Cref{tab:llm-implementation-hparams}. For \textsc{ScaleBiO},
\(\texttt{scalebio\_alpha}\) is the coefficient multiplying its first-order source-reweighting surrogate, not an outer
learning rate.

\begin{table}[t]
  \centering
  \caption{\textbf{Lower-model optimizer settings in the real-world LLM runs.}
  All rows use one SGD lower-model update per branch or inner step.}
  \label{tab:llm-lower-model-hparams}
  \begin{adjustbox}{max width=\textwidth}
  \begin{tabular}{llcl}
    \toprule
    Experiment & Method & Lower-model LR & Additional settings \\
    \midrule
    Math, \texttt{Qwen-2-7B} & \textsc{HG-MS}/\textsc{HG-baseline} & \(10^{-5}\) & \textsc{HG-MS}: \(N=4\), branch noise \(10^{-3}\) \\
    Math, \texttt{Qwen-2-7B} & \textsc{ScaleBiO} & \(10^{-4}\) & \(\texttt{scalebio\_alpha}=10\) \\
    Math, \texttt{Llama-3-8B} & \textsc{HG-MS}/\textsc{HG-baseline} & $10^{-4}$ & \textsc{HG-MS}: \(N=4\), branch noise $10^{-5}$ \\
    Math, \texttt{Llama-3-8B} & \textsc{ScaleBiO} & $10^{-4}$ & \(\texttt{scalebio\_alpha}=10\) \\
    Math, \texttt{Gemma-2-9B} & \textsc{HG-MS}/\textsc{HG-baseline} & $10^{-4}$ & \textsc{HG-MS}: \(N=4\), branch noise $10^{-5}$ \\
    Math, \texttt{Gemma-2-9B} & \textsc{ScaleBiO} & $10^{-4}$ & \(\texttt{scalebio\_alpha}=10\) \\
    Instruction & \textsc{HG-MS}/\textsc{HG-baseline} & \(2\times10^{-4}\) & \textsc{HG-MS}: \(N=8\), branch noise \(5\times10^{-4}\) \\
    Instruction & \textsc{ScaleBiO} & \(10^{-5}\) & \(\texttt{scalebio\_alpha}=100\) \\
    Large-model \texttt{Qwen2.5-32B} & \textsc{HG-MS}/\textsc{HG-baseline} & \(8\times10^{-6}\) & HG-MS: \(N=4\), branch noise \(5\times10^{-4}\) \\
    Large-model \texttt{Qwen2.5-32B} & \textsc{ScaleBiO} & \(5\times10^{-5}\) & \(\texttt{scalebio\_alpha}=100\) \\
    \bottomrule
  \end{tabular}
  \end{adjustbox}
\end{table}

In the real-world code paths, the ScaleBiO source-logit update uses the softmax source weights and a scaled difference
between weighted and unweighted source training losses: in the functional implementation this is
\(\alpha(\texttt{train\_loss\_w}-\texttt{train\_loss\_u})\), while in the 32B FSDP implementation it is
\(\alpha\big((\softmax(\theta)\odot\ell_w)^\top\mathbf 1-(\softmax(\theta)\odot\ell_u)^\top\mathbf 1\big)\), followed by
the configured source-logit optimizer and optional gradient clipping. We did not find additional normalization beyond the
softmax source weights and this optional clipping.

\paragraph{Final inference and judging.}
Across methods in a given real-world table, final checkpoints are evaluated with the same prompts, decoding code, and answer-extraction or judging procedure. The reported GSM8K/MATH evaluations and MT-Bench answer-generation stage use one generated response per prompt with greedy decoding (\texttt{do\_sample=false}), with no sampling, beam search, or best-of-\(N\) inference. The \BoN terminology in the theory and algorithmic description refers to the minima-selection mechanism inside \textsc{HG-MS}, not to the downstream LLM inference protocol. For instruction following, model answers are generated once and then scored by the same GPT-4o MT-Bench judge \citep{zheng2023judging} with temperature zero.

\paragraph{Main mathematical-reasoning protocol (\Cref{tab:llm-realworld-math}).}
The math benchmark uses the same ten-source portfolio as \citet{pan2025scalebio}; the full source and validation-source inventory is reported in \Cref{tab:math-source-portfolio}. The ten training sources are DART-Math, SlimOrca, UltraInteract, MathInstruct, Orca-Math, WizardLM, Magicoder, ShareGPT-Vicuna, GPTeacher, and GPT4-LLM-Cleaned. Each backbone in the main table is equipped with LoRA adapters on the attention projections (\texttt{q\_proj}, \texttt{k\_proj}, \texttt{v\_proj}, \texttt{o\_proj}) and is run through the same single-H100 adapted pipeline. The math pipeline uses LoRA rank $16$, LoRA scaling $32$, dropout $0.05$, maximum sequence length $512$, and greedy generation-based evaluation. The target evaluation benchmarks are GSM8K \citep{cobbe2021training}, a grade-school math word-problem benchmark, and MATH \citep{hendrycks2021measuring}, a competition-style mathematical problem-solving benchmark. During reweighting, we cache $500$ training examples per source and use a GSM8K-only validation path; after reweighting, we sample a final $20$K-example training set from the learned source mixture and fine-tune for one epoch with learning rate \(8\times 10^{-6}\), micro-batch size $1$, and gradient accumulation $32$. The completed main-table runs evaluate the final checkpoint on the full GSM8K and MATH test sets. For the iterative methods, the outer-stage wall-clock budget is matched within each backbone and is approximately $9$h for \texttt{Qwen-2-7B}, $12$h for \texttt{Llama-3-8B}, and $15$h for \texttt{Gemma-2-9B}.

The \texttt{Qwen2.5-32B} results in \Cref{tab:llm-realworld-math-large} use the same ten-source portfolio and final GSM8K/MATH evaluation target. The completed iterative rows apply the corresponding source-reweighting method to the 32B model using the same LoRA code path, SDPA attention, gradient checkpointing, and matched large-model reweighting setup; the completed \textsc{Uniform} row uses equal source counts. For the iterative 32B rows, the outer reweighting budget was fixed in advance at \(15\)h per method before final GSM8K/MATH test evaluation. Each completed row then uses the resulting source mixture to form the full \(20\)K-example final fine-tuning set and evaluates on the full GSM8K/MATH test sets. The completed \textsc{HG-MS} row attains \(92.58\%\) on GSM8K and \(53.32\%\) on MATH, the best completed scores in both columns. \textsc{RHO-LOSS} and \textsc{LESS} are marked as OOM in this large-model setting and remain in the table for comparability with the ScaleBiO large-model layout.

\paragraph{Main instruction-following protocol (\Cref{tab:llm-realworld-instr-mtbench}).}
The instruction-following benchmark uses the released eighteen-source ScaleBiO portfolio, whose source inventory is reported in \Cref{tab:instr-source-portfolio}:
\texttt{AlpacaGPT4}, \texttt{ShareGPT4}, \texttt{SlimOrca}, \texttt{AlpacaChat}, \texttt{OpenOrcaGPT4}, \texttt{WildChat}, \texttt{LMSYS-Chat}, \texttt{GPTeacher}, \texttt{Airoboros}, and the nine translated Alpaca variants (\texttt{es}, \texttt{de}, \texttt{ja}, \texttt{ko}, \texttt{ru}, \texttt{it}, \texttt{fr}, \texttt{zh}, \texttt{pt}). All completed instruction-following runs use LoRA adapters \citep{hu2022lora} on \texttt{q\_proj}, \texttt{k\_proj}, \texttt{v\_proj}, and \texttt{o\_proj}, maximum sequence length \(256\), LoRA rank \(4\), scaling \(8\), LoRA dropout $0.05$, and one epoch of final supervised fine-tuning with learning rate \(8\times 10^{-6}\), micro-batch size $1$, and gradient accumulation $64$. Gradient checkpointing is enabled for the iterative jobs when needed by the single-H100 memory budget. During reweighting, we cache up to roughly $3000$ examples per source from the released ScaleBiO assets, use a final $10$K-example sampled training set, and cap MT-Bench prompts at $1024$ input tokens with up to $512$ generated tokens. The iterative methods use the matched per-backbone wall-clock reweighting budgets stated above before final fine-tuning: approximately $9$h for \texttt{Qwen-2-7B}, $12$h for \texttt{Llama-3-8B}, and $15$h for \texttt{Gemma-2-9B}. The final evaluation is MT-Bench \citep{zheng2023judging} on the standard $80$ FastChat questions, using a GPT-4o judge. In our implementation, MT-Bench generations are produced once per finished run and judged offline in a separate pass so that all methods share the same prompt set and judge model. The standard FastChat MT-Bench questions, generated answers, and GPT-4o scores used for the final reported evaluation are never used for source reweighting, checkpoint selection, hyperparameter tuning, or early stopping; those choices use only the released ScaleBiO MT-Bench-style validation and tuning splits described below.

For the instruction-following table, the time-budget comparison is therefore fixed at the reweighting stage:
\textsc{ScaleBiO}, \textsc{HG-baseline}, and \textsc{HG-MS} receive the same single-H100 outer-loop budget within each
backbone: approximately $9$h for \texttt{Qwen-2-7B}, $12$h for \texttt{Llama-3-8B}, and $15$h for
\texttt{Gemma-2-9B}. This is followed by the same one-epoch final fine-tuning and the same MT-Bench generation/judging pipeline. The
non-iterative baselines do not use this outer-loop budget because they do not run bilevel source reweighting.

Following the ScaleBiO setup, the instruction portfolio intentionally mixes broadly useful instruction/chat data with multilingual Alpaca variants that are less directly aligned with MT-Bench. Each source is filtered by the maximum sequence length before entering the cached reweighting pool. The released ScaleBiO MT-Bench-style validation split is used as the outer validation set for source reweighting, and a separate MT-Bench-style tuning split is used to select the best reweighting checkpoint and tune method hyperparameters. This matches the role of the MT-Bench-style reference and tuning data described in Appendix~B.2 of \citet{pan2025scalebio}, while our final reported score is computed only after one response is generated for each of the standard FastChat MT-Bench questions and judged by GPT-4o.

The method-specific data use is as follows. \textsc{Uniform} samples the final $10$K training examples uniformly across the eighteen sources. \textsc{LESS} and \textsc{RHO-LOSS} score examples from the same portfolio using their adapted influence/reference-model criteria and then form the final training set. \textsc{ScaleBiO}, \textsc{HG-baseline}, and \textsc{HG-MS} learn source logits during the matched per-backbone reweighting stage, convert the selected source distribution into integer source counts, sample the final $10$K training examples accordingly, and then run the identical one-epoch LoRA fine-tuning and MT-Bench evaluation pipeline.

\paragraph{Final source-mixture audit for mathematical reasoning.}
The model-specific tables in this paragraph report the final mathematical-reasoning source proportions stored in the run
metadata for the main LLM tables; see
\Cref{tab:llm-main-math-mixtures-llama,tab:llm-main-math-mixtures-qwen,tab:llm-main-math-mixtures-gemma,tab:llm-large-math-mixtures}.
These diagnostics are separate from the downstream test scores: they are included to make the learned or selected source
mixtures auditable. For \textsc{Uniform}, the mixture is fixed by construction, so we report the common per-source mass
rather than an arbitrary top-three ranking. For \textsc{LESS} and \textsc{RHO-LOSS}, the entries are empirical source
proportions in the selected final SFT set, not learned source logits. For \textsc{ScaleBiO}, \textsc{HG-baseline}, and
\textsc{HG-MS}, the entries are the largest components of the source distribution used to sample the final SFT set.
We focus the source-mixture audit on mathematical reasoning, where the source shifts are most directly connected to the
GSM8K/MATH improvements; the saved instruction-following source weights are included in the released experiment artifacts.

\begin{table}[t]
  \centering
  \caption{\textbf{Final mathematical-reasoning source mixtures for \texttt{Llama-3-8B}.}
  Each non-uniform row reports the three largest final source proportions over the ten-source math portfolio.}
  \label{tab:llm-main-math-mixtures-llama}
  \small
  \begin{adjustbox}{max width=\textwidth}
  \begin{tabular}{ll}
    \toprule
    Method & Final source mixture \\
    \midrule
    Uniform Weighting & uniform over ten sources, each 0.100 \\
    RHO-LOSS & SlimOrca 0.352; GPT4-LLM-Cleaned 0.202; WizardLM 0.142 \\
    LESS & Orca-Math 0.119; MathInstruct 0.117; DART-Math 0.104 \\
    ScaleBiO & Magicoder 0.829; SlimOrca 0.096; GPTeacher 0.044 \\
    HG-baseline & MathInstruct 0.809; SlimOrca 0.091; GPT4-LLM-Cleaned 0.005 \\
    HG-MS & MathInstruct 0.902; DART-Math 0.036; Orca-Math 0.010 \\
    \bottomrule
  \end{tabular}
  \end{adjustbox}
\end{table}

\begin{table}[t]
  \centering
  \caption{\textbf{Final mathematical-reasoning source mixtures for \texttt{Qwen-2-7B}.}
  Each non-uniform row reports the three largest final source proportions over the ten-source math portfolio.}
  \label{tab:llm-main-math-mixtures-qwen}
  \small
  \begin{adjustbox}{max width=\textwidth}
  \begin{tabular}{ll}
    \toprule
    Method & Final source mixture \\
    \midrule
    Uniform Weighting & uniform over ten sources, each 0.100 \\
    RHO-LOSS & MathInstruct 0.299; SlimOrca 0.215; GPT4-LLM-Cleaned 0.133 \\
    LESS & MathInstruct 0.140; DART-Math 0.128; Orca-Math 0.125 \\
    ScaleBiO & MathInstruct 0.996; SlimOrca 0.001; GPT4-LLM-Cleaned 0.001 \\
    HG-baseline & MathInstruct 0.996; GPT4-LLM-Cleaned 0.003; SlimOrca 0.000 \\
    HG-MS & DART-Math 0.940; MathInstruct 0.056; Orca-Math 0.004 \\
    \bottomrule
  \end{tabular}
  \end{adjustbox}
\end{table}

\begin{table}[t]
  \centering
  \caption{\textbf{Final mathematical-reasoning source mixtures for \texttt{Gemma-2-9B}.}
  Each non-uniform row reports the three largest final source proportions over the ten-source math portfolio.}
  \label{tab:llm-main-math-mixtures-gemma}
  \small
  \begin{adjustbox}{max width=\textwidth}
  \begin{tabular}{ll}
    \toprule
    Method & Final source mixture \\
    \midrule
    Uniform Weighting & uniform over ten sources, each 0.100 \\
    RHO-LOSS & SlimOrca 0.267; GPT4-LLM-Cleaned 0.178; WizardLM 0.144 \\
    LESS & Orca-Math 0.110; MathInstruct 0.109; DART-Math 0.103 \\
    ScaleBiO & UltraInteract 0.859; Magicoder 0.083; SlimOrca 0.033 \\
    HG-baseline & MathInstruct 0.975; SlimOrca 0.004; GPT4-LLM-Cleaned 0.004 \\
    HG-MS & MathInstruct 0.941; DART-Math 0.016; Orca-Math 0.010 \\
    \bottomrule
  \end{tabular}
  \end{adjustbox}
\end{table}

\begin{table}[t]
  \centering
  \caption{\textbf{Final source mixtures for the \texttt{Qwen2.5-32B} mathematical-reasoning rows.}
  Completed non-\texttt{OOM} rows use the ten-source math portfolio; \textsc{RHO-LOSS} and \textsc{LESS} do not have completed large-model final mixtures under the reported memory budget.}
  \label{tab:llm-large-math-mixtures}
  \begin{adjustbox}{max width=0.9\textwidth}
  \begin{tabular}{ll}
    \toprule
    Method & Final source mixture \\
    \midrule
    Uniform Weighting & uniform over ten sources, each 0.100 \\
    RHO-LOSS & \texttt{OOM}; no completed final source mixture \\
    LESS & \texttt{OOM}; no completed final source mixture \\
    ScaleBiO & GPT4-LLM-Cleaned 0.166; SlimOrca 0.120; GPTeacher 0.105 \\
    HG-baseline & Orca-Math 0.159; UltraInteract 0.155; DART-Math 0.153 \\
    HG-MS & DART-Math 0.165; Orca-Math 0.135; MathInteract 0.115 \\
    \bottomrule
  \end{tabular}
  \end{adjustbox}
\end{table}

\paragraph{Run diagnostics.}
To make the fixed-budget mathematical-reasoning comparison auditable, each iterative run stores the elapsed reweighting
wall-clock time, the number of completed outer steps, and the learned source weights used to sample the final fine-tuning
set. \Cref{tab:llm-budget-sensitivity-qwen} gives the budget-sensitivity reporting table for the \texttt{Qwen-2-7B} math
setting, and \Cref{tab:llm-branch-ablation-qwen} reports the corresponding fixed-budget branch-count sensitivity check.
We fill the completed diagnostic rows from the available run metadata.

\Cref{tab:llm-budget-sensitivity-qwen} should be read as a budget ledger rather than as an additional benchmark with a new protocol. Each row uses the same \texttt{Qwen-2-7B} math-reasoning pipeline described above, changes only the outer reweighting budget for the iterative method, then performs the same one-epoch final fine-tuning and greedy GSM8K/MATH evaluation. The ``Outer steps'' column records how many source-weight updates fit into the wall-clock budget, while the ``Top-3 learned source weights'' column reports the largest entries of the final source distribution used to sample the final training set.

\begin{table*}[t]
  \centering
  \caption{\textbf{Budget-sensitivity ledger for \texttt{Qwen-2-7B} math reasoning.}}
  \label{tab:llm-budget-sensitivity-qwen}
  \begin{adjustbox}{max width=\textwidth}
  \begin{tabular}{llclc}
    \toprule
    Method & Budget & Outer steps & Top-3 learned source weights & GSM8K / MATH \\
    \midrule
    ScaleBiO  & 3h  & 10527 & Magicoder 0.848; GPTeacher 0.037; WizardLM 0.026 & 66.19 / 30.70 \\
    HG-baseline & 3h  & 1950  & GPT4-LLM-Cleaned 0.565; MathInstruct 0.370; SlimOrca 0.025 & 61.41 / 29.97 \\
    HG-MS     & 3h  & 1233  & MathInstruct 0.993; GPT4-LLM-Cleaned 0.002; SlimOrca 0.002 & 68.61 / 33.75 \\
    \midrule
    ScaleBiO  & 9h  & 13804 & MathInstruct 0.996; SlimOrca 0.001; GPT4-LLM-Cleaned 0.001 & 75.74 / 36.93 \\
    HG-baseline & 9h  & 6129  & MathInstruct 0.996; GPT4-LLM-Cleaned 0.003; SlimOrca 0.000 & 68.46 / 34.27 \\
    HG-MS     & 9h  & 3754  & DART-Math 0.940; MathInstruct 0.056; Orca-Math 0.004 & \textbf{82.03} / \textbf{38.98} \\
    \midrule
    ScaleBiO  & 15h & 23039 & Magicoder 0.980; DART-Math 0.010; MathInstruct 0.003 & 70.66 / 29.65 \\
    HG-baseline & 15h & 10150 & MathInstruct 1.000; GPT4-LLM-Cleaned 0.000; SlimOrca 0.000 & 70.05 / 34.14 \\
    HG-MS     & 15h & 6283  & DART-Math 0.828; Orca-Math 0.156; MathInstruct 0.016 & 81.96 / 37.38 \\
    \midrule
    ScaleBiO  & 21h & 32208 & Magicoder 0.978; DART-Math 0.010; MathInstruct 0.005 & 69.96 / 30.73 \\
    HG-baseline & 21h & 14220 & MathInstruct 0.992; GPT4-LLM-Cleaned 0.002; SlimOrca 0.001 & 69.73 / 34.09 \\
    HG-MS     & 21h & 8813  & DART-Math 0.830; Orca-Math 0.142; MathInstruct 0.026 & 81.98 / 37.28 \\
    \bottomrule
  \end{tabular}
  \end{adjustbox}
\end{table*}

\paragraph{Interpretation of the budget ledger.}
The learned distributions differ substantially across methods. \textsc{ScaleBiO} and \textsc{HG-baseline} usually become nearly single-source solutions: at 9h both put almost all mass on MathInstruct (\(0.996\)), and the longer \textsc{HG-baseline} runs remain similarly concentrated on MathInstruct. In contrast, \textsc{HG-MS} moves to a different source mixture. Its main 9h run places most mass on DART-Math (\(0.940\)) while keeping a secondary MathInstruct component (\(0.056\)); at 15h and 21h it remains DART-Math dominated but assigns a larger secondary mass to Orca-Math (\(0.156\) and \(0.142\), respectively). Thus the gain of \textsc{HG-MS} is not just another collapse to the same source preferred by \textsc{ScaleBiO}/\textsc{HG-baseline}; minima selection changes which training sources are emphasized in the final fine-tuning set.

The main \texttt{Qwen-2-7B} math row uses the 9h \textsc{HG-MS} run because that stopping budget was fixed before the final GSM8K/MATH test evaluation. The 15h and 21h rows were run under the same final fine-tuning and evaluation protocol to check sensitivity to a longer reweighting horizon. In this run family, extending the \textsc{HG-MS} budget does not improve final accuracy even though it continues to update the source mixture; this is a post-hoc diagnostic observation, not the criterion used to choose the main-table stopping budget.

\paragraph{Fixed-budget branch-count sensitivity for \textsc{HG-MS}.}
To probe the role of explicit minima selection, we add a fixed-budget branch-count sensitivity study for the \texttt{Qwen-2-7B} math-reasoning setting. The study keeps the ten-source portfolio, seed, LoRA configuration, GSM8K-only validation path, nine-hour reweighting budget, final \(20\)K-example fine-tuning set, and full GSM8K/MATH evaluation fixed, while changing the number \(N\) of candidate lower-level branches inside \textsc{HG-MS}. Because the comparison is wall-clock matched, changing \(N\) does not isolate branch count alone: larger \(N\) gives a richer minima-selection set per outer update but leaves fewer source-weight updates within the same budget. The \(N=1\) row corresponds to the one-branch \textsc{HG-baseline}, the \(N=4\) row is the completed main-table \textsc{HG-MS} configuration, and the \(N=2\) and \(N=8\) rows test whether the observed gain comes from a moderate amount of branch diversity or from simply increasing branch count. We therefore report outer steps and interpret this single-seed table as a finite-budget diagnostic rather than as an exhaustive tuning study over \(N\).

\begin{table*}[t]
  \centering
  \caption{\textbf{Fixed-budget branch-count sensitivity for \texttt{Qwen-2-7B} math reasoning.}
  All rows use the same nine-hour reweighting budget and one-epoch final fine-tuning protocol; \(N\) controls the number of candidate branches per outer update, so larger \(N\) trades richer branch selection against fewer source-weight updates.}
  \label{tab:llm-branch-ablation-qwen}
  \begin{adjustbox}{max width=\textwidth}
  \begin{tabular}{lclcc}
    \toprule
    Method & Branches \(N\) & Top-3 learned source weights & Outer steps & GSM8K / MATH \\
    \midrule
    HG-baseline & 1 & MathInstruct 0.996; GPT4-LLM-Cleaned 0.003; SlimOrca 0.000 & 6129 & 68.46 / 34.27 \\
    HG-MS       & 2 & Orca-Math 0.492; DART-Math 0.423; MathInstruct 0.085 & 5211 & 76.65 / 31.96 \\
    HG-MS       & 4 & DART-Math 0.940; MathInstruct 0.056; Orca-Math 0.004 & 3754 & \textbf{82.03} / \textbf{38.98} \\
    HG-MS       & 8 & DART-Math 0.760; Magicoder 0.190; Orca-Math 0.008 & 2516 & 66.19 / 29.37 \\
    \bottomrule
  \end{tabular}
  \end{adjustbox}
  
\end{table*}